\documentclass{memo-l}

\usepackage{morefloats, latexsym, enumerate, hyperref, pdftricks, framed,color,enumitem,tikz,epstopdf,mathtools}
\usetikzlibrary{decorations.markings}
\usepackage{multirow,array}

\pdfoutput=1

\newcommand{\Ex}{\mathbb{E}}
\def\Pr{\mathbb{P}}
\def\P{\mathbb{P}}
\newcommand\R{\mathbb{R}}
\newcommand\TV{\mathrm{TV}}

\newcommand\tr{\operatorname{trace}}
\newcommand\dom{\operatorname{Dom}}
\newcommand{\divergence}{\operatorname{div}}

\newcommand\threebars{|\!|\!|}

\DeclareMathOperator{\sign}{sign}
\DeclareMathOperator{\diag}{diag}
\DeclareMathOperator{\sech}{sech}

\theoremstyle{plain}
\newtheorem{theorem}{Theorem}
\newtheorem{prop}[theorem]{Proposition}
\newtheorem{lemma}[theorem]{Lemma}
\newtheorem{cor}[theorem]{Corollary}

\theoremstyle{definition}
\newtheorem{defn}{Definition}
\newtheorem{example}{Example}
\newtheorem{assumption}{Assumption}
\newtheorem{algo}{Algorithm}

\theoremstyle{remark}
\newtheorem{rem}{Remark}

\numberwithin{equation}{section}
\numberwithin{section}{chapter}
\numberwithin{theorem}{section}
\numberwithin{defn}{section}
\numberwithin{assumption}{chapter}
\numberwithin{rem}{section}
\numberwithin{figure}{chapter}
\numberwithin{algo}{chapter}

\begin{document}

\frontmatter

\title{Continuous-time Random Walks for the Numerical Solution of Stochastic Differential Equations}

\author{Nawaf Bou-Rabee}
\address{Department of Mathematical Sciences, 
Rutgers University Camden, 311 N 5th Street, Camden, NJ 08102}
\email{nawaf.bourabee@rutgers.edu}
\thanks{The work of N.~B-R. was supported in part by NSF grant DMS-1212058.}

\author{Eric Vanden-Eijnden}
\address{Courant Institute of Mathematical Sciences,
New York University, 251 Mercer Street, New York, NY 10012-1185}
\email{eve2@cims.nyu.edu}
\thanks{}




\maketitle

\tableofcontents

\begin{abstract} This paper introduces time-continuous numerical schemes to simulate stochastic differential equations (SDEs) arising in mathematical finance, population dynamics, chemical kinetics, epidemiology, biophysics, and polymeric fluids. These schemes are obtained by spatially discretizing the Kolmogorov equation associated with the SDE in such a way that the resulting semi-discrete equation generates a Markov jump process that can be realized exactly using a Monte Carlo method. In this construction the spatial increment of the approximation can be bounded uniformly in space, which guarantees that the schemes are numerically stable for both finite and long time simulation of SDEs. By directly analyzing the generator of the approximation, we prove that the approximation has a sharp stochastic Lyapunov function when applied to an SDE with a drift field that is locally Lipschitz continuous and weakly dissipative. We use this stochastic Lyapunov function to extend a local semimartingale representation of the approximation. This extension permits to analyze the complexity of the approximation. Using the theory of semigroups of linear operators on Banach spaces, we show that the approximation is (weakly) accurate in representing finite and infinite-time statistics, with an order of accuracy identical to that of its generator. The proofs are carried out in the context of both fixed and variable spatial step sizes. Theoretical and numerical studies confirm these statements, and provide evidence that these schemes have several advantages over standard methods based on time-discretization. In particular, they are accurate, eliminate nonphysical moves in simulating SDEs with boundaries (or confined domains), prevent exploding trajectories from occurring when simulating stiff SDEs, and solve first exit problems without time-interpolation errors.   \end{abstract}


\mainmatter

\chapter{Introduction}

\section{Motivation}

Stochastic differential equations (SDEs) are commonly used to model the effects of random fluctuations in applications such as molecular dynamics \cite{ScSt1978, BrBrKa1984,AlTi1987, FrSm2002}, molecular motors \cite{berg1993random}, microfluidics \cite{La2005, Sh2005}, atmosphere/ocean sciences \cite{MaWa2006}, epidemiology \cite{MaMaRe2002,RuPi2006}, population dynamics \cite{Al2008,Br2010}, and mathematical finance \cite{KaSh1998,PlHe2006}.  In most of these applications, the SDEs cannot be solved analytically and numerical methods are used to approximate their solutions.  By and large the development of such methods has paralleled what has been done in the context of ordinary differential equations (ODEs) and this work has led to popular schemes such as Euler-Maruyama, Milstein's, etc \cite{KlPl1995,Ta1995,MiTr2004}. While the design and analysis of these schemes differ for ODEs and SDEs due to peculiarities of stochastic calculus, the strategy is the same in both cases and relies on discretizing in time the solution of the ODE or the SDE.  For ODEs, this is natural: their solution is unique and smooth, and it makes sense to interpolate their path at a given order of accuracy between snapshots taken at discrete times. This objective is what ODE integrators are designed for. The situation is somewhat different for SDEs, however. While their solutions (in a pathwise sense) may seem like an ODE solution when the Wiener process driving them is given in advance, these solutions are continuous but not differentiable in general, and an ensemble of trajectories originates from any initial point once one considers different realizations of the Wiener process.

The simulation of SDEs rather than ODEs presents additional challenges for standard integration schemes based on time-discretization. One is related to the long time stability of schemes, and their capability to capture the invariant distribution of the SDE when it exists. Such questions are usually not addressed in the context of ODEs: indeed, except in the simplest cases, it is extremely difficult to make precise statements about the long time behavior of their solutions since this typically involves addressing very hard questions from dynamical systems theory. The situation is different with SDEs: it is typically simpler to prove their ergodicity with respect to some invariant distribution, and one of the goals of the integrator is often to sample this distribution when it is not available in closed analytical form. This aim, however, is one that is difficult to achieve with standard integrators, because they typically fail to be ergodic even if the underlying SDE is. Indeed the numerical drift in standard schemes can become destabilizing if the drift field is only locally Lipschitz continuous, which is often the case in applications. These destabilizing effects typically also affect the accuracy of the scheme on finite time intervals, making the problem even more severe. Another important difference between SDEs and ODEs is that the solutions of the former can be confined to a certain region of their state space, and it is sometimes necessary to impose boundary conditions at the edge of this domain. Imposing these boundary conditions, or even guaranteeing that the numerical solution of the scheme remains in the correct domain, is typically difficult with standard integration schemes.

In view of these difficulties, it is natural to ask whether one can take a different viewpoint to simulating SDEs, and adopt an approach that is more tailored to the probabilistic structure of their solution to design stable and (weakly) accurate schemes for their simulation that are both provably ergodic when the underlying SDE is and faithful to the geometry of their domain of integration. The aim of this paper is to propose one such novel strategy. The basic idea is to discretize their solution in space rather than in time via discretization of their  infinitesimal generator.  Numerous approximations of this second-order partial differential operator are permissible including finite difference or finite volume discretizations \cite{elston1996numerical,WaPeEl2003,Ph2008,MeScVa2009,LaMeHaSc2011}.  As long as this discretization satisfies a realizability condition, namely that the discretized operator be the generator of a Markov jump process on a discrete state space, it defines a process which can be exactly simulated using the Stochastic Simulation Algorithm (SSA) \cite{gillespie1977exact, elston1996numerical, WaPeEl2003, LaMeHaSc2011}.  Note that this construction alleviates the curse of dimensionality that limits numerical PDE approaches to low dimensional systems while at the same time permitting to borrow design and analysis tools that have been developed in the numerical PDE context.  The new method gets around the issues of standard integrators by permitting the displacement of the approximation to be bounded uniformly and adaptively in space.  This feature leads to simple schemes that can provably sample the stationary distribution of general SDEs, and that remain in the domain of definition of the SDE, by construction.  In the next section, we describe the main results of the paper and discuss in detail the applications where these SDE problems arise.

\section{Main Results}

%
%

If the state space of the approximating Markov jump process, or approximation for short, is finite-dimensional, then a large collection of analysis tools can be transported from the theory of numerical PDE methods to assess the stability and accuracy of the approximation. However, if the state space is not finite-dimensional, then the theoretical framework presented in this paper and described below can be used instead. It is important to underscore that the Monte Carlo method we use to realize this approximation is local in character, and hence, our approach applies to SDE problems which are unreachable by standard methods for numerically solving PDEs. In fact, the approximation does not have to be tailored to a grid. Leveraging this flexibility, in this paper we derive realizable discretizations for SDE problems with general domains using simple finite difference methods such as upwinded and central difference schemes \cite{strikwerda2004finite,leveque2007finite,gustafsson2013time}, though we stress that other types of discretizations that satisfy the realizablility condition also fit our framework.

%
%

Since an outcome of this construction is the generator of the approximation, the proposed methods have a transparent probabilistic structure that make them straightforward to analyze using probabilistic tools.  We carry out this analysis in the context of an SDE with unbounded coefficients and  an unbounded domain.     Our structural assumptions permit the drift field of the SDE to be locally Lipschitz continuous, but require the drift field to satisfy weak dissipativity and polynomial growth conditions.  Our definition of stability is that the approximation has a stochastic Lyapunov function whenever the underlying SDE has one.   This definition is natural when one deals with SDEs with locally Lipschitz drift fields.   Indeed, global existence and uniqueness theorems for such SDEs typically rely on the existence of a stochastic Lyapunov function \cite{Kh2012,mao2007stochastic}.

%
%

In this context, here are the main results of the paper regarding realizable discretizations with gridded state spaces.  We emphasize that we do not assume that the set of all grid points is finite.
\begin{enumerate}
\item
We use Harris Theorem to prove that the approximation is geometrically ergodic with respect to an invariant probability measure \cite{MeTw2009,HiMaSt2002,HaMa2011}.   To invoke Harris theorem, we prove that the approximation preserves a sharp stochastic Lyapunov function from the true dynamics.  Sharp, here, means that this stochastic Lyapunov function dominates the exponential of the magnitude of the drift field, and can therefore be used to prove that a large class of observables (including all moments of the approximation) are integrable.  
\item
In addition to playing an important role in proving existence of an invariant probability measure and geometric ergodicity,  this stochastic Lyapunov function also helps solve a martingale problem associated to the approximation \cite{kurtz1981approximation,ReYo1999,EtKu2009}.  This solution justifies a global semimartingale representation of the approximation, and implies that Dynkins formula holds for the conditional expectation of observables that satisfy a mild growth condition.  We apply this semimartingale representation to analyze the complexity of the approximation.  Specifically, we obtain an estimate for the average number of computational steps on a finite-time interval in terms of the spatial step size parameter.  
\item
We use a variation of constants formula to quantify the global error of the approximation \cite{Pa1983}.  This formula is a continuous-time analog of the Talay-Tubaro expansion of the global error of a discrete-time SDE method \cite{TaTu1990}.  An analysis of this formula reveals that the order of accuracy of the approximation in representing finite-time conditional expectations and equilibrium expectations is given by the order of accuracy of its generator.  Geometric ergodicity and finite-time accuracy imply that the approximation can accurately represent long-time dynamics, such as equilibrium correlation functions \cite{BoVa2012}.   
\end{enumerate}

%
%

\noindent
This paper also expands on these results by considering space-discrete generators that use variable spatial step sizes, and whose state space may not be confined to a grid.  In these cases, we embed the `gridless' state space of the approximation into $\mathbb{R}^n$, and the main issues are:
\begin{itemize}
\item the semigroup of the approximation may not be irreducible with respect to the standard topology on $\mathbb{R}^n$; and 
\item the semigroup of the approximation may not even satisfy a Feller property because its associated generator is an unbounded operator and the sample paths of its associated process are discontinuous in space \cite{MeTw2009,DaZa1996}.  
\end{itemize}
To deal with this lack of irreducibility and regularity, we use a weight function to mollify the spatial discretization so that the resulting generator is a bounded linear operator.  This mollification needs to be done carefully in order for the approximation to preserve a sharp stochastic Lyapunov function from the true dynamics.  In short, if the effect of the mollification is too strong, then the approximation does not have a sharp stochastic Lyapunov function; however, if the mollification is too weak, then the operator associated to the approximation is unbounded.  

The semigroup associated to this mollified generator is Feller, but not strongly Feller, because the driving noise is still discrete and does not have a regularizing effect.  However, since the process is Feller and has a stochastic Lyapunov function, we are able to invoke the Krylov-Bogolyubov Theorem to obtain existence of an invariant probability measure for the approximation \cite{DaZa1996}.  This existence of an invariant probability measure is sufficient to prove accuracy of the approximation with respect to equilibrium expectations. To summarize, here are the main results of the paper regarding realizable discretizations with gridless state spaces.
\begin{enumerate}
\item  The approximation is stable in the sense that it has a sharp stochastic Lyapunov function, and accurate in the sense that the global error of the approximation is identical to the accuracy of its generator.  The proof of both of these properties entails analyzing the action of the generator on certain test functions, which is straightforward to do.
\item  A semimartingale representation for suitable observables of the approximation is proven to hold globally.  This representation implies Dynkins formula holds for this class of observables.
\item  The approximation has an invariant probability measure, which may not be unique.  This existence is sufficient to prove that equilibrium expectations of the SDE solution are accurately represented by the approximation.
\end{enumerate}
To be clear, the paper does not prove or disprove uniqueness of an invariant probability measure of the approximation on a gridless state space.  One way to resolve this open question is to show that the semigroup associated to the approximation is asymptotically strong Feller \cite{hairer2006ergodicity}.  

%
%

Besides these theoretical results, below we also apply the new approach to a variety of examples to show that the new approach is not only practical, but also permits to alleviate several issues that commonly arise in applications and impede existing schemes.

\begin{enumerate}

\item {\em Stiffness problems.}  Stiffness in SDEs can cause the numerical trajectory to artificially ``explode.'' For a precise statement and proof of this divergence in the context of forward Euler applied to a general SDE see \cite{hutzenthaler2011strong}.  This numerical instability is a well-known issue with explicit discretizations of nonlinear SDEs \cite{Ta2002, HiMaSt2002, MiTr2005,Hi2011,HuJeKl2012}, and it is especially acute if the SDE is stiff e.g.~due to coefficients with limited regularity, which is typically the case in applications.  Intuitively speaking, this instability is due to explicit integrators being only {\em conditionally stable}, i.e.~for any fixed time-step they are numerically stable if and only if the numerical solution remains within a sufficiently large compact set.  Since noise can drive the numerical solution outside of this compact set, exploding trajectories occur almost surely.  The proposed approximation can be constructed to be unconditionally stable, since its spatial step size can be  bounded uniformly in space.  As a result, the mean time lag between successive jumps adapts to the stiffness of the drift coefficients.

\item {\em Long-time issues.} SDEs are often used to sample their stationary distribution, which is generically not known, meaning there is no explicit formula for its associated density even up to a normalization constant.  Examples include polymer dynamics in extensional, shear, and mixed planar flows \cite{LaHuSmCh1999, HuShLa2000, HuShBaCh2002, JedeGr2002, ScBaShCh2003, ScShCh2004}; polymers in confined flows \cite{ChLa2002, JeDiScGrde2003a,JeDiScGrde2003b,JeScdeGr2004, WoShKh2004a, WoShKh2004b,HeMadeGr2006a, HeMadeGr2006b,HedeGr2007,HeMadeGr2008, Gr2011, de2011, ZhdeGr2012}; polymer conformational transitions in electrophoretic flows \cite{kim2006brownian,randall2006methods,kim2007design,trahan2009simulation,hsieh2011simulation,hsieh2012simulation}; polymer translocation through a nanopore \cite{lubensky1999driven,loebl2003simulation,matysiak2006dynamics,huopaniemi2006langevin,allen2006simulating}; and molecular systems with multiple heat baths at different temperatures \cite{eckmann1999non,eckmann2000non,rey2001exponential,MaVa2006,hairer2009slow,lin2010nonequilibrium,abrams2010large} as in temperature accelerated molecular dynamics \cite{MaVa2006,vanden2009some,abrams2010large}.  Unknown stationary distributions also arise in a variety of diffusions with nongradient drift fields like those used in atmosphere/ocean dynamics or epidemiology.  The stationary distribution in these situations is implicitly defined as the steady-state solution to a PDE, which cannot be solved in practice.  Long-time simulations are therefore needed to sample from the stationary distribution of these SDEs to compute quantities such as the equilibrium concentration of a species in population dynamics or, the polymer contribution to the stress tensor or equilibrium radius of gyration in Brownian dynamics. This however requires to use numerical schemes that are ergodic when the underlying SDEs they approximate are, which is by no means guaranteed.    We show that the proposed approximation is long-time stable as a consequence of preserving a stochastic Lyapunov function from the true dynamics.   In fact, this stochastic Lyapunov function can be used to prove that the approximation on a gridded state space is geometrically ergodic with respect to a unique invariant probability measure nearby the stationary distribution of the SDE, as we do in this paper.

\item {\em Influence of boundaries.}  SDEs with boundaries arise in chemical kinetics, population dynamics, epidemiology, and mathematical finance, where the solution (e.g., population densities or interest rates) must always be non-negative \cite{gillespie1977exact,CoInRo1985,RuPi2006,colizza2006role,Al2008,Br2010,LoKoKi2010, andersen2010simulation}.  Boundaries also arise in Brownian dynamics simulations.  For example, in bead-spring chain models upper bounds are imposed on bond lengths in order to model the experimentally observed extension of semi-flexible polymers \cite{Ot1996}.  In this context standard integrators may produce updated states that are not within the domain of the SDE problem, which in this paper we refer to as nonphysical moves.  The proposed approximation can eliminate such nonphysical moves by allowing the spatial step size to be set such that no update lies outside the domain of definition of the SDE.

 \item {\em First Exit Problems.}  It is known that standard integrators do not work well in simulating first exits from a given region because, simply put, they fail to account for the probability that the system exits the region in between each discrete time step~\cite{gobet2004exact,gobet2007discrete,gobet2010stopped}.  Building this capability into time integrators is nontrivial, e.g., one technique requires locally solving a boundary value problem per step, which is prohibitive to do in many dimensions \cite{Ma1999}.   By using a grid whose boundary conforms to the boundary of the first exit problem, the proposed approximation does not make such time interpolation errors. 
 \end{enumerate}

In addition, the examples illustrate other interesting properties of realizable discretizations such as their accuracy with respect to the spectrum of the infinitesimal generator of non-symmetric diffusions.

\section{Relation to Other Works}

When the diffusion matrix is diagonally dominant, H.~Kushner developed realizable finite difference approximations to SDEs and used probabilistic methods to prove that these finite difference schemes are convergent \cite{Ku1970,
Ku1973, Ku1975, Ku1976A, Ku1976B, Ku1977}.  When the SDE is self-adjoint and the noise is additive, realizable discretizations combined with SSA have also been developed, see \cite{elston1996numerical, WaPeEl2003,Ph2008,MeScVa2009,LaMeHaSc2011}, and to be specific, this combined approach seems to go back to at least \cite{elston1996numerical}.  Among these papers, the most general and geometrically flexible realizable discretizations seem to be the finite volume methods developed in \cite{LaMeHaSc2011}.  This paper generalizes these realizable discretizations to diffusion processes with multiplicative noise.  

Our approach is  related to multiscale methods for approximating SDEs that leverage the spectrum of suitably defined stochastic matrices \cite{Sc1999,ScHuDe2001,CoLaLeMaNaWaZu2005A,CoLa2006,NaLaCoKe2006,Ph2008,MeScVa2009,MeScVa2009,LaMeHaSc2011}.  However, to be clear, this paper is primarily about approximating the SDE solution, not coarse-graining the dynamics of the SDE.

Our approach also seems to be foreshadowed by Metropolis integrators for SDEs, in the sense that the proportion of time a Metropolis integrator spends in a given region of state space is set by the equilibrium distribution of the SDE, not the time step size \cite{RoTw1996A,BoOw2010, BoVa2010,BoVa2012,BoHa2013,BoDoVa2014, Bo2014}.  The idea behind Metropolis integrators is to combine an explicit time integrator for the SDE with a Monte Carlo method so that the resulting scheme is ergodic with respect to the stationary distribution of the SDE and finite-time accurate \cite{BoVa2010}.  However, a drawback of these methods is that they are typically limited to diffusions that are self-adjoint \cite{BoDoVa2014}, whereas the methods presented in this paper do not require this structure.  Moreover, Metropolis integrators are typically ergodic with respect to the stationary distribution of the SDE, but they may not be geometrically ergodic when the SDE is \cite{RoTw1996A,BoVa2010,BoHa2013}.   Intuitively speaking, the reason Metropolis integrators lack geometric ergodicity is that they use as proposal moves conditionally stable integrators for the SDE.  Hence, outside the region of stability of these integrators the Metropolis rejection rate tends to one.  In these unstable regions, Metropolis integrators may get stuck, and as a consequence, typically do not have a stochastic Lyapunov function.    In contrast, the proposed approximations do have a stochastic Lyapunov function for a large class of diffusion processes.  


Our approach is also related to, but different from, variable time step integrators for SDEs \cite{gaines1997variable,szepessy2001adaptive,LaMaSt2007}.   These integrators are not all adapted to the filtration given by the noise driving the approximation.  Let us comment on the adapted (or non-anticipating) technique proposed in \cite{LaMaSt2007}, which is based on a simple and elegant design principle.  The technique has the nice feature that it does not require iterated stochastic integrals.  The method adjusts its time step size at each step according to an estimate of the `deterministic accuracy' of the integrator, that is, how well the scheme captures the dynamics in the zero noise limit. The time step size is reduced until this estimate is below a  certain tolerance.  The aim of this method is to achieve stability by being deterministically accurate.  In fact, as a consequence of this accuracy, the approximation `inherits' a stochastic Lyapunov function from the SDE and is provably geometrically ergodic.  This paper \cite{LaMaSt2007} also addresses what happens if the noise is large compared to the drift.   The difference between these variable time step integrators and the proposed  approximation is that the latter sets its mean time step according to the Kolmogorov equation and achieves stability by bounding its spatial step size.

\section{Organization of Paper}  The remainder of this paper is structured as follows. \begin{description}
\item[Chapter \ref{chap:methods}] details our strategy to spatially discretize an SDE problem.
\begin{itemize}[leftmargin=-.25in]
\item \S\ref{sec:realizability} defines what it means for a spatial discretization to be realizable.
\item \S\ref{sec:gridded_vs_gridless} describes the differences and pros/cons between a gridded and gridless state space.
\item \S\ref{sec:scalar_case} derives realizable discretizations in one dimension using an upwinded finite difference, finite volume, and a probabilistic discretization.
\item \S\ref{sec:realizable_methods_2D} derives realizable discretizations in two dimension using an upwinded finite difference method.
\item \S\ref{sec:realizable_methods_nD} presents realizable discretizations in $n$-dimension, which are based on central and upwinded finite differences on a gridless state space.
\item \S\ref{sec:scaling} discusses and analyzes the scaling of the approximation with respect to the dimension of the SDE problem.
\item \S\ref{sec:generalization} generalizes realizable discretizations in $n$-dimension to diffusion matrices that can be written as a sum of rank one matrices. 
\item \S\ref{sec:diagonally_dominant_case} considers the special case of weakly diagonally dominant diffusion matrices, and proves that a second-order accurate, upwinded finite difference discretization is realizable on a gridded state space.
\end{itemize}
\item[Chapter \ref{chap:numerics}] applies realizable discretizations to SDE problems.
\begin{itemize}[leftmargin=-.25in]
\item \S\ref{sec:cubic_oscillator} (resp.~\S\ref{sec:lognormal_1d}) provides numerical evidence that the approximation is stable (i.e.~geometrically ergodic) and accurate in representing moments at finite times, mean first passage times, exit probabilities, and the stationary distribution of a one-dimensional process with additive (resp.~multiplicative) noise.
\item \S\ref{sec:downhill_time} presents an asymptotic analysis of the mean-holding time of the approximation, which is used to validate some numerical results.
\item \S\ref{sec:cir_process} shows that the approximation eliminates nonphysical moves in the Cox-Ingersoll-Ross model from mathematical finance, is accurate in representing the stationary distribution of this process, and can efficiently handle singularities in SDE problems by using adaptive mesh refinement.
\item \S\ref{sec:planar_sdes} illustrates that the approximation accurately represents the spectrum of the generator in a variety of non-symmetric diffusion processes with additive noise including SDE problems with nonlinear drift fields, internal discontinuities, and singular drift fields.
\item \S\ref{sec:lognormal_2d} shows that the approximation on a gridded state space (with variable step sizes) is accurate in representing the stationary distribution of a two-dimensional process  with multiplicative noise.
\item  \S\ref{sec:lv_process} applies the approximation to a Lotka-Volterra process from population dynamics, which shows that the approximation eliminates nonphysical moves, is suitable for long-time simulation, and can capture extinction.
 \item   \S\ref{sec:colloidal_cluster} tests the approximation on a Brownian dynamics simulation of a colloidal cluster immersed in an unbounded solvent, and shows that the jump size can be set according to a characteristic length scale in an SDE problem, namely the Lennard-Jones radius.
\end{itemize}
\item[Chapter \ref{chap:analysis}] analyzes a realizable discretization on a gridded state space.
\begin{itemize}[leftmargin=-.25in]
\item \S\ref{sec:preliminaries} precisely describes the theoretical context, and in particular, hypotheses on the SDE problem are given in Assumption~\ref{assumptions_on_drift}.
\item \S\ref{sec:geometric_ergodicity} derives a sharp stochastic Lyapunov function for the SDE solution and the approximation, and uses this Lyapunov function to show that both the SDE solution and the approximation are geometrically ergodic with respect to an invariant probability measure.
\item \S\ref{sec:properties} solves a martingale problem associated to the approximation, justifies a Dynkins formula for a large class of observables of the approximation, and applies this solution to analyze basic properties of realizations of the approximation.
\item \S\ref{sec:accuracy} proves that the approximation is accurate with respect to finite-time conditional expectations and equilibrium expectations by using a continuous-time analog of the Talay-Tubaro expansion of the global error.
\end{itemize}
\item[Chapter \ref{chap:gridless}] analyzes a realizable discretization on a gridless state space.
\begin{itemize}[leftmargin=-.25in]
\item \S\ref{sec:context} introduces a mollified generator that uses variable step sizes.
\item \S\ref{sec:feller} shows that this generator defines a bounded operator with respect to Borel bounded functions and induces a Feller semigroup.
\item \S\ref{sec:generator_accuracy} proves that this generator is accurate with respect to the infinitesimal generator of the SDE.
\item \S\ref{sec:lyapunov_gridless} shows that the associated approximation has a sharp stochastic Lyapunov function provided that the mollification is not too strong.
\item \S\ref{sec:accuracy_gridless} proves that the associated approximation is accurate with respect to finite-time conditional expectations and equilibrium expectations by using a continuous-time analog of the Talay-Tubaro expansion of the global error.
\end{itemize}
\item[Chapter \ref{chap:tridiagonal}] analyzes tridiagonal, realizable discretizations on gridded state spaces.
\begin{itemize}[leftmargin=-.25in]
\item \S\ref{sec:nu_1D} derives an explicit formula for the invariant density of the approximation.
\item \S\ref{sec:nu_accuracy_1D} quantifies the accuracy of the stationary density of the approximation.
\item \S\ref{sec:committor_1D} derives an explicit formula for the exit probability (or committor function) of the approximation.
\item \S\ref{sec:mfpt_1D} derives an explicit formula for the mean first passage time of the approximation.
\end{itemize}
\item
[Chapter \ref{chap:conclusion}] concludes the paper.
\end{description}

\section{Acknowledgements}

We wish to thank Aleksandar Donev, Weinan E, Denis Talay, Ra{\'u}l  Tempone and Jonathan Weare for useful discussions.

\chapter{Algorithms} \label{chap:methods}

\section{Realizability Condition} \label{sec:realizability}

Here we present a time-continuous approach to numerically solve an $n$-dimensional SDE with domain $\Omega \subset \mathbb{R}^n$.  The starting point is the infinitesimal generator: \begin{equation} \label{eq:generator}
L f (x) = Df(x)^T \mu(x) +  \tr( D^2 f(x) \sigma(x) \sigma(x)^T ) 
\end{equation}
associated to the stochastic process $Y$ that satisfies the SDE: \begin{equation} \label{eq:sde}
d Y = \mu(Y) dt + \sqrt{2} \sigma(Y) d W \;, \quad Y(0) \in \Omega 
\end{equation}
where $W$ is an $n$-dimensional Brownian motion and $\mu: \Omega \to \mathbb{R}^n$ and $\sigma: \Omega \to \mathbb{R}^{n \times n}$ are the drift and noise fields, respectively.   (The test function $f: \mathbb{R}^n \to \mathbb{R}$ in \eqref{eq:generator} is assumed to be twice differentiable.)  Let $L^*$ be the (formal) adjoint operator to $L$. Given an initial density $p(0,x)$, the Fokker-Planck equation: \begin{equation} \label{eq:fokkerplanck}
\frac{\partial p}{\partial t}  =  L^* p 
\end{equation}
describes the evolution of the probability density of $Y(t)$.  The time evolution of a conditional expectation of an observable is described by the Kolmogorov equation: \begin{equation} \label{eq:kolmogorov_intro}
 \frac{\partial u}{\partial t} = L u
 \end{equation}  with initial condition $u(0,x)$.  A weak solution to the SDE \eqref{eq:sde} can be obtained by solving \eqref{eq:fokkerplanck} or \eqref{eq:kolmogorov_intro}.  In this paper we consider solvers for \eqref{eq:sde} that exploit this connection.  

The essential idea is to approximate the infinitesimal generator of the SDE in \eqref{eq:generator} by a discrete-space generator of the form: \begin{equation} \label{eq:discrete_generator}
 Q f(x) =  \sum_{i=1}^K q(x,y_i(x)) \left( \vphantom{\sum} f(y_i(x)) - f(x) \right) 
\end{equation}
where we have introduced a reaction rate function  $q: \Omega \times \Omega \to \mathbb{R}$, and $K$ reaction channels: $y_i(x) \in \Omega$ for $i=1,\dotsc,K$ and for every $x \in \Omega$.    (This terminology comes from chemical kinetics \cite{gillespie1977exact}.)  Let $h$ be a spatial step size parameter and $p$ be a positive parameter, which sets the order of accuracy of the method.  We require that the generator $Q$ satisfies
\smallskip
 \begin{description}
\item[(Q1) $p$th-order accuracy] \begin{equation*}
Q f(x) = L f(x) + O(h^{p}) \quad \forall~x \in \Omega 
\end{equation*}
\item[(Q2) realizability] \begin{equation*}
q(x,y) \ge 0 \quad \forall~x, y \in \Omega \quad x \ne y 
\end{equation*}
\end{description}
\smallskip
Condition (Q1) is a basic requirement for any spatial discretization of a PDE.   If (Q1) holds, and assuming numerically stability, then the approximation is $p$th-order accurate on finite-time intervals (see e.g.~Theorem~\ref{thm:finite_time_accuracy} for a precise statement).  Condition (Q2) is not a standard requirement, and is related to the structure of the Kolmogorov equation.  In words, (Q2) states that the weights in the finite differences appearing in \eqref{eq:discrete_generator} must be non-negative. We call an operator that satisfies this condition realizable.   Without this requirement the approximation would be a standard numerical PDE solver and limited to low-dimensional systems by the curse of dimensionality.   In contrast, when (Q2) holds the generator $Q$ is realizable, in the sense that it induces a Markov jump process $X$ which can be simulated exactly in time using the Stochastic Simulation Algorithm (SSA) \cite{gillespie1977exact, elston1996numerical, WaPeEl2003, LaMeHaSc2011}.   This simulation does not require explicitly gridding up any part of state space, since its inputs -- including rate functions and reaction channels -- can be computed online.   We stress that the channels in \eqref{eq:discrete_generator} can be picked so that the spatial displacement $(y_i - x)$ is bounded for $i=1,\dotsc,K$, where as a reminder $K$ is the \# of channels.  

\medskip

To emphasize that simulating $X$ is simple, let us briefly recall how SSA works.

%
%

\begin{algo}[Stochastic Simulation Algorithm \cite{gillespie1977exact}] \label{algo:ssa}
Given the current time $t$, the current state of the process $X(t)=x$,  and the $K$ jump rates evaluated at this state $ \{ q(x,y_i(x)) \}_{i=1}^K$,
the algorithm outputs the state of the system  $X(t+ \delta t)$ at time $t + \delta t$ in two sub-steps.

\medskip

\begin{description}
\item[(Step 1)] obtains $\delta t$ by generating an exponentially distributed random variable with parameter \[
\lambda(x) =  \sum_{i=1}^K q(x,y_i(x))  
\]
\item[(Step 2)] updates the state of the system by assuming that the process moves to state $y_i(x)$ with probability: \[
\Pr( X(t + \delta t) = y_i(x)  \mid X(t)=x )  = \frac{q(x,y_i(x))}{\lambda(x)} 
\]  
\end{description}
\end{algo}

Algorithm~\ref{algo:ssa} suggests the following compact way to write the action of the generator $Q$ in \eqref{eq:discrete_generator}: \begin{equation} \label{eq:generator_compact}
Q f(x) = \lambda(x) \; \Ex (f(x+\xi_x) - f(x) )
\end{equation}
where the expectation is over the random $n$-vector $\xi_x$, whose state space is the set of all $K$ reaction channels from the state $x$ and whose probability distribution is \[
\Pr( \xi_x = y_i(x) ) =  \frac{q(x,y_i(x))}{\lambda(x)} \;, \quad 1 \le i \le K \;.
\]  

To summarize, this paper proposes to (weakly) approximate the solution of a general SDE using a Markov process that can be simulated using Algorithm~\ref{algo:ssa}, whose updating procedure entails generating a random time increment in (Step 1) and firing a reaction in (Step 2).   By restricting the firings in (Step 2) to reaction channels that are a bounded distance away from the current state $x$, the approximation is numerically stable, a point we elaborate on in \S\ref{sec:geometric_ergodicity}.   Note that the mean time increment (or mean holding time) in (Step 1) is $\lambda(x)^{-1}$.  Thus, the spatial discretization of $L$ determines the average amount of time the process spends at any state.  Morever, note that the process satisfies: $X(s) = x$ for $s \in [t,t+\delta t)$.  The latter property implies that sample paths of the process are right-continuous in time with left limits (or c\`adl\`ag).  Also note that the process only changes its state by jumps, and hence, is a pure jump process.  Since the time and state updates only depend on the current state $X(t_0)$, and the holding time in (Step 1) is exponentially distributed, this approximation has the Markov property and is adapted to the filtration of the driving noise.  Thus, the approximation is a Markov jump process.  As we show in Chapter~\ref{chap:analysis}, the Markov and non-anticipating property of the approximation greatly simplify the analysis.

\medskip

Next we discuss the state space of the approximation, and then derive schemes that meet the requirements (Q1) and (Q2). 

%
%

\section{Gridded vs Gridless State Spaces} \label{sec:gridded_vs_gridless}

We distinguish between two types of discretization based on the structure of the graph that represents the state space of the approximation.  We say that a graph is connected if there is a finite path joining any pair of its vertices.

\medskip

\begin{defn}
A gridded (resp.~gridless) state space consists of points which form a connected (resp.~disconnected) graph.
\end{defn}

\medskip

\noindent
We stress that both state spaces permit capping the spatial step size of the approximation.  Moreover, the SSA on either state space only requires local evaluation of the SDE coefficients.  Gridded state spaces are appealing theoretically because a generator on such a space can often be represented as a matrix, and in low dimensional problems, the spectral properties of these matrices (e.g.~low-lying eigenvalues, leading left eigenfunction, or a spectral decomposition) can be computed using a standard sparse matrix eigensolver.   When the state space is unbounded this matrix is infinite.  However, this matrix can still be approximated by a finite-dimensional matrix, and the validity of this approximation can be tested by making sure that  numerically computed quantities like the spectrum do not depend on the truncation order.  We will return to this point when we consider concrete examples in Chapter~\ref{chap:numerics}.  In order to be realizable, the matrix $\mathsf{Q}$ associated to the linear operator $Q$ on a gridded state space must satisfy the {\em $Q$-matrix property}: \begin{equation} \label{Qmatrix_property}
\mathsf{Q}_{ij} \ge 0  \quad \text{for all $i \ne j$} \quad \text{and} \quad   \sum_j \mathsf{Q}_{ij} = 0 \;.
\end{equation} Indeed, the $Q$-matrix property is satisfied by the matrix $\mathsf{Q}$ if and only if its associated linear operator $Q$ is realizable.   Traditional finite difference and finite volume discretizations have been successfully used to design $Q$-matrix discretizations when the generator of the SDE is self-adjoint and the noise entering the SDE is additive, see \cite{elston1996numerical, WaPeEl2003,Ph2008,MeScVa2009,LaMeHaSc2011}.  However, $Q$-matrix discretizations on a gridded state space become trickier to design when the noise is multiplicative even in low dimension.   This issue motivates designing spatial discretizations on gridless state spaces.  Let us elaborate on this point by considering realizable discretizations in one, two, and then many dimensions.   With a slight abuse of notation, we may use the same symbol to denote an operator $Q$ and its associated matrix.  The meaning should be clear from the context.  The most general realizable discretization is given in \S\ref{sec:generalization}.  

%
%

\section{Realizable Discretizations in 1D} \label{sec:scalar_case}

SDEs in one dimension possess a symmetry property, which we briefly recall since we expect the proposed schemes to accurately represent this structure.  To prove this property, it is sufficient to assume that the noise coefficient is differentiable and satisfies $\sigma(x) > 0$ for all $x \in \Omega$, where $\Omega$ is the state space of the approximation.   Under these assumptions the generator \eqref{eq:generator} can be written in conservative form as: \begin{equation} \label{eq:generator_1D_conservative_form}
( L f)(x) =  \frac{ 1 }{\nu(x)}( M(x) \nu(x) f'(x))'
\end{equation}
where $M(x) = (\sigma(x))^2$ and $\nu(x)$ is given by: \[
 \nu(x)  = \frac{1}{M(x)} \exp\left(  \int_a^x \frac{\mu(s)}{M(s)} ds \right) 
\]
for some $a \in \Omega$.  Similarly, the SDE \eqref{eq:sde} can be written in self-adjoint form: \begin{equation} \label{sde_1d}  
d Y = - M(Y) U'(Y) dt + M'(Y) dt + \sqrt{2} \sigma(Y) dW 
\end{equation}
where $U(x)$ is the free energy of $\nu(x)$ defined as: \[
U(x) = - \log( \nu(x) ) \;.
\] 
Let $\langle \cdot , \cdot \rangle_{\nu}$ denote the $\nu$-weighted $L^2$ inner product defined as: \[
\langle f , g \rangle_{\nu} = \int_{\Omega} f(x) g(x) \nu(x) dx 
\] 
where the functions $f(x)$ and $g(x)$ are square integrable with respect to the measure $\nu(x) dx$.   The generator $L$ for scalar, elliptic diffusions is self-adjoint with respect to this inner product i.e. \begin{equation} \label{eq:self_adjoint_property_1D}
\langle L f , g \rangle_{\nu} = \langle f , L g \rangle_{\nu} 
\end{equation}
for smooth and compactly supported functions $f(x)$ and $g(x)$.  Note that by taking $g(x)=1$ for every $x \in \Omega$, then $\langle L f , g \rangle_{\nu} = 0$ for all smooth and compactly supported $f(x)$, which implies that \[
L^* \nu(x) = 0 \;.
\]
 Thus,  $\nu(x)$ is an invariant density of the SDE.  Furthermore, if the function $\nu(x)$ is integrable over $\Omega$, i.e. $Z = \int_{\Omega} \nu(x) dx < \infty$, then $\nu(x)/Z$ is an invariant probability density (or stationary density) of the SDE.     

\medskip

We stress that the generator of SDEs in more than one dimension often do not satisfy the self-adjoint property \eqref{eq:self_adjoint_property_1D}.  Thus, to avoid losing generality, we will not assume the generator of the SDE is self-adjoint.  In this context we present finite difference, finite volume, and probabilistic discretizations of $L$.  

\subsection{Finite Difference}

Let $\{ x_i \}$ be a collection of grid points over $\Omega \subset \mathbb{R}$ as shown in Figure~\ref{fig:finite_difference_grid_1D}.  The distance between neighboring points is allowed to be variable. Thus, it helps to define forward, backward and average spatial step sizes: \begin{equation} \label{eq:amr}
\delta x^+_i = x_{i+1} - x_i \;, \quad \delta x^-_i =x_i - x_{i-1} \;, \quad \delta x_i = \frac{\delta x_i^+ + \delta x_i^-}{2}   \;, \quad \text{resp.}
\end{equation}
at every grid point $x_i$.  The simplest way to construct a realizable discretization of $L$ is to use an upwinded and a central finite difference approximation of the first and second-order derivatives in $L$, respectively. To derive this discretization, it is helpful to write the generator $L$ in \eqref{eq:generator_1D_conservative_form} in terms of the functions $M(x)$ and $U(x)$ as follows: \begin{equation} \label{eq:generator_for_upwind}
(L f)(x) = - M(x) U'(x) f'(x) +  (M(x) f'(x))' \;.
\end{equation}  
To obtain a realizable discretization, the first term of \eqref{eq:generator_for_upwind} is discretized using a variable step size, first-order upwind method: \begin{align*}
- M_i U'_i f'_i \approx ( - M_i U'_i \vee 0) \frac{ f_{i+1} - f_i}{ \delta x^+_i} +  ( - M_i U'_i \wedge 0 ) \frac{f_i - f_{i-1}}{\delta x^-_i}   \;.
\end{align*}
Here we employ the shorthand notation: $f_i = f(x_i)$ for all grid points $x_i$ (and likewise for other functions), and \[
a \vee b =  \max(a,b) \quad \text{and} \quad a \wedge b = \min(a,b) 
\]
for any real numbers $a$ and $b$.  By taking the finite difference approximation on the upstream side of the flow, i.e., going backward or forward in the finite difference approximation of $f'_i$  according to the sign of the coefficient $- M_i U'_i$, this discretization is able to simultaneously satisfy first-order accuracy (Q1) and realizability (Q2).  A variable step size, central scheme can be used to approximate the second-order term in \eqref{eq:generator_for_upwind}: \begin{align*}
 (M f')'_i  \approx \frac{1}{ 2 \delta x_i } \left(  (M_{i+1}+M_i) \frac{ f_{i+1} - f_i}{ \delta x^+_i} - (M_{i-1}+M_i) \frac{f_i - f_{i-1}}{\delta x^-_i}  \right) \;.
\end{align*}
It is straightforward to verify that this part of the discretization obeys (Q1) with $p=2$ and since $M(x)>0$ satisfies (Q2) as well. The resulting discrete generator is a $Q$-matrix with off-diagonal entries \begin{equation} \label{eq:tQu_1d}
\begin{dcases}
(\tilde Q_u)_{i,i+1} =  \frac{1 }{ \delta x^+_i }  \left(  ( - M_i U'_i \vee 0) +  \frac{ M_{i+1}+M_i }{2 \; \delta x_i}    \right)   \\
(\tilde Q_u)_{i,i-1} =   \frac{1 }{ \delta x^-_i } \left( - ( - M_i U'_i \wedge 0 ) +  \frac{M_{i-1}+M_i}{2 \; \delta x_i}    \right)  
\end{dcases}
\end{equation}
Note that to specify a $Q$-matrix, it suffices to specify its off-diagonal entries and then use \eqref{Qmatrix_property} to determine its diagonal entries.   Also, note that this scheme does not require computing the derivative of $M(x)$, which appears in the drift coefficient of \eqref{sde_1d}.  The overall accuracy of the scheme is first-order in space.   The approximation also preserves an invariant density, see Proposition~\ref{prop:nu_symmetry_1d} for a precise statement and proof.  Remarkably, the asymptotic mean holding time of this generator is exact when, roughly speaking, the drift is large compared to the noise.  We will prove and numerically validate this statement in \S\ref{sec:downhill_time}.


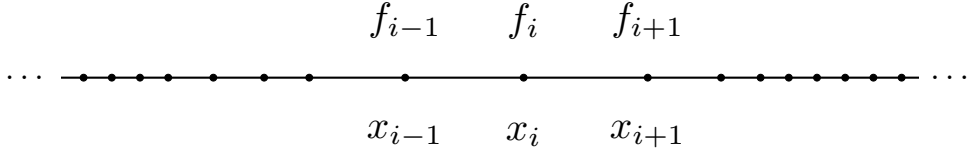
\begin{figure}\centering
\begin{tikzpicture}[scale=1.5]
\filldraw[color=black,fill=black] (0.0,0) circle (0.03);	
\filldraw[color=black,fill=black] (0.25,0) circle (0.03);			
\filldraw[color=black,fill=black] (0.5,0) circle (0.03);	
\filldraw[color=black,fill=black] (0.75,0) circle (0.03);				
\filldraw[color=black,fill=black] (1.0,0) circle (0.03);	
\filldraw[color=black,fill=black] (1.25,0) circle (0.03);			
\filldraw[color=black,fill=black] (1.65,0) circle (0.03);	
\filldraw[color=black,fill=black] (2.1,0) circle (0.03);	
\filldraw[color=black,fill=black] (2.5,0) circle (0.03);	
\filldraw[color=black,fill=black] (3.35,0) circle (0.03);	
\filldraw[color=black,fill=black] (4.4,0) circle (0.03);	
\filldraw[color=black,fill=black] (5.5,0) circle (0.03);	
\filldraw[color=black,fill=black] (6.15,0) circle (0.03);	
\filldraw[color=black,fill=black] (6.5,0) circle (0.03);	
\filldraw[color=black,fill=black] (6.75,0) circle (0.03);	
\filldraw[color=black,fill=black] (7.0,0) circle (0.03);	
\filldraw[color=black,fill=black] (7.25,0) circle (0.03);	
\filldraw[color=black,fill=black] (7.5,0) circle (0.03);	
\filldraw[color=black,fill=black] (7.75,0) circle (0.03);	
\filldraw[color=black,fill=black] (8.0,0) circle (0.03);	
\draw[-, thick](0,0.0) -- (8.0,0.0);
\node[black,scale=1.5] at (3.35,0.5) {$f_{i-1}$};
\node[black,scale=1.5] at (3.35,-0.5) {$x_{i-1}$};
\node[black,scale=1.5] at (4.4,0.5) {$f_i$};
\node[black,scale=1.5] at (4.4,-0.5) {$x_i$};
\node[black,scale=1.5] at (5.5,0.5) {$f_{i+1}$};
\node[black,scale=1.5] at (5.5,-0.5) {$x_{i+1}$};
\node[black, scale=1.25,fill=white] at (0.0,0.0) {$\dotsc$};
\node[black, scale=1.25,fill=white] at (8.2,0.0) {$\dotsc$};
\end{tikzpicture}
\caption{\small { \bf Finite Difference Grid.}
Consider non-uniformly spaced grid points $\{ x_j \}$  and corresponding function values $\{ f_j = f(x_j) \}$.  This diagram depicts a three-point stencil centered at grid point $x_i$ with corresponding function values.   This stencil is used to approximate the generator $L$ using a realizable, finite difference approximation.  See \eqref{eq:tQu_1d} for the entries of the resulting $Q$-matrix.
 }
  \label{fig:finite_difference_grid_1D}
\end{figure}


\subsection{Finite Volume}

Next we present a finite volume discretization of $L$.  This derivation illustrates the approach taken in \cite{LaMeHaSc2011} to construct realizable discretizations for self-adjoint diffusions.   For this purpose discretize the region $\Omega$ using the partition shown in Figure~\ref{fig:finite_volume_grid_1D}.  This partition consists of a set of cells $\{ \Omega_i \}$ with centers $\{ x_i \}$ spaced according to \eqref{eq:amr}.   The cells are defined as $\Omega_i = [x_{i-1/2}, x_{i+1/2}]$ where the midpoint between cell centers $x_{i+1/2} =  (x_i + x_{i+1})/2$ gives the location of each cell edge.  From \eqref{eq:amr} note that the length of each cell is $ x_{i+1/2} - x_{i-1/2} = \delta x_i$.  

Integrating the Fokker-Planck equation \eqref{eq:fokkerplanck} over cell $\Omega_i$ yields: \begin{align*}
\int_{\Omega_i} \frac{\partial p}{\partial t}(t,x) dx &= \int_{\Omega_i} \frac{\partial}{\partial x} \left( \nu(x) M(x) \frac{\partial}{\partial x} \frac{p(t,x)}{\nu(x)} \right) dx  \\
&=  \left. \left( \nu(x) M(x)  \frac{\partial}{\partial x} \frac{p(t,x)}{\nu(x)} \right) \right|^{x_{i+1/2}}_{x_{i-1/2}}  \\ 
&= F_{i+1/2} - F_{i-1/2} 
\end{align*}  Here we have introduced $F_{i+1/2}$ and $F_{i-1/2}$, which are the probability fluxes at the right and left cell edge, respectively.  Approximate the solution $p(t,x)$ to the Fokker-Planck equation by its cell average: \[
\bar p_i(t) \approx \frac{1}{\delta x_i} \int_{\Omega_i} p(t,x) dx \;.
\]  
In terms of which, approximate the probability flux using \begin{equation} \label{eq:probabiliy_flux_approximation}
F_{i+1/2} \approx \overbrace{\frac{\nu_{i+1} + \nu_i}{2}}^{\displaystyle \approx \nu_{i+1/2}} \overbrace{\frac{M_{i+1} + M_i}{2}}^{\displaystyle \approx M_{i+1/2}} \underbrace{\frac{1}{\delta x^+_i } \left( \frac{\bar p_{i+1}}{\nu_{i+1}} - \frac{\bar p_i}{\nu_i} \right)}_{\displaystyle \approx \left( \frac{\partial}{\partial x} \frac{p(t,x)}{\nu(x)} \right)_{i+1/2}}  \;.
\end{equation}  These approximations lead to a semi-discrete Fokker-Planck equation of the form: \[
\frac{d \bar p}{dt} = \tilde Q_{\nu}^{\star} \bar{p}
\]  where $\tilde Q_{\nu}^{\star}$ is the transpose of the matrix $\tilde Q_{\nu}$, whose off-diagonal entries are \begin{equation} \label{nu_symmetric_scheme_1d}
\begin{dcases}
(\tilde Q_{\nu})_{i,i+1} =  \frac{1 }{ \delta x_i \; \delta x^+_i }  \frac{M_{i+1}+M_i}{2}  \frac{\nu_{i+1} + \nu_i}{2 \nu_i}   \\
(\tilde Q_{\nu})_{i,i-1} =   \frac{1 }{ \delta x_i \; \delta x^-_i }    \frac{M_{i-1}+M_i}{2}  \frac{\nu_{i-1} + \nu_i}{2 \nu_i}  
\end{dcases}
\end{equation}
This generator is a second-order accurate approximation of $L$ and it preserves the invariant density $\nu$ of the SDE in the following way.  Let $\bar \nu$ be the invariant density $\nu$ weighted by the length of each cell \[
 \bar \nu_i = \delta x_i \; \nu_i  \;.
\]
It follows from \eqref{nu_symmetric_scheme_1d} that \[
(\tilde Q_{\nu})_{i,i+1} \; \bar \nu_i  = (\tilde Q_{\nu})_{i+1,i} \; \bar \nu_{i+1}  
\]
since from \eqref{eq:amr} $\delta x_i^+ = \delta x_{i+1}^-$.  This identity implies that $\bar \nu$ is an invariant density of $\tilde Q_{\nu}$, see Proposition~\ref{prop:nu_symmetry_1d} for a detailed proof.  Other approximations to the probability flux also have this property e.g.~the midpoint approximation of $\nu_{i+1/2}$ in \eqref{eq:probabiliy_flux_approximation} can be replaced by: \[
\nu_{i+1/2} \approx \nu_i \wedge \nu_{i+1}  \quad \text{or} \quad \nu_{i+1/2} \approx \exp\left(- \frac{ U_{i+1} + U_i }{2} \right) \;. 
\]  
(As a caveat, note that the former choice reduces the order of accuracy of the discretization because it uses the minimum function.)  In fact, the evaluation of $U$ can be bypassed by approximating the forward and backward potential energy differences appearing in \eqref{nu_symmetric_scheme_1d} by the potential force to obtain the following $Q$-matrix approximation: \begin{equation} \label{eq:tQc_1d}
\begin{dcases}
(\tilde Q_c)_{i,i+1} =  \frac{1 }{ \delta x_i \; \delta x^+_i }  \frac{M_{i+1}+M_i}{2} \exp\left( -U'_i \; \frac{\delta x_i^+}{2} \right)    \\
(\tilde Q_c)_{i,i-1} =   \frac{1 }{ \delta x_i \; \delta x^-_i }  \frac{M_{i-1}+M_i}{2} \exp\left( U'_i \; \frac{\delta x_i^-}{2} \right) 
\end{dcases}
\end{equation}
In terms of function evaluations, \eqref{eq:tQc_1d} is a bit cheaper than \eqref{nu_symmetric_scheme_1d}, and this difference in cost increases with dimension.  Under natural conditions on $M(x)$ and $U'(x)$, we will prove that \eqref{eq:tQc_1d} possesses a second-order accurate stationary density, see Proposition~\ref{nu_tQc_1d}.  

\medskip

For a more general theoretical and numerical treatment of realizable finite-volume methods for self-adjoint diffusions with additive noise we refer to \cite{LaMeHaSc2011}.


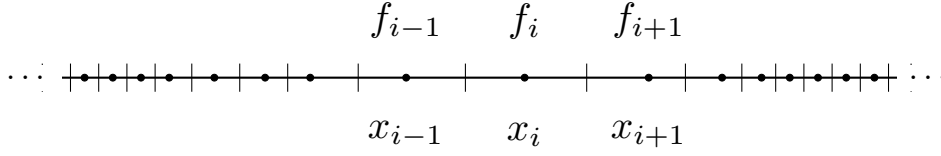
\begin{figure}
\centering
\begin{tikzpicture}[scale=1.5]
\filldraw[color=black,fill=black] (0.0,0) circle (0.03);	
\draw[-, thin](0.125,-0.125) -- (0.125,0.125);
\filldraw[color=black,fill=black] (0.25,0) circle (0.03);		
\draw[-, thin](0.375,-0.125) -- (0.375,0.125);	
\filldraw[color=black,fill=black] (0.5,0) circle (0.03);	
\draw[-, thin](0.625,-0.125) -- (0.625,0.125);	
\filldraw[color=black,fill=black] (0.75,0) circle (0.03);	
\draw[-, thin](0.875,-0.125) -- (0.875,0.125);				
\filldraw[color=black,fill=black] (1.0,0) circle (0.03);	
\draw[-, thin](1.125,-0.125) -- (1.125,0.125);				
\filldraw[color=black,fill=black] (1.25,0) circle (0.03);	
\draw[-, thin](1.45,-0.125) -- (1.45,0.125);				
\filldraw[color=black,fill=black] (1.65,0) circle (0.03);	
\draw[-, thin](1.875,-0.125) -- (1.875,0.125);				
\filldraw[color=black,fill=black] (2.1,0) circle (0.03);	
\draw[-, thin](2.3,-0.125) -- (2.3,0.125);				
\filldraw[color=black,fill=black] (2.5,0) circle (0.03);	
\draw[-, thin](2.925,-0.125) -- (2.925,0.125);
\filldraw[color=black,fill=black] (3.35,0) circle (0.03);	
\draw[-, thin](3.875,-0.125) -- (3.875,0.125);
\filldraw[color=black,fill=black] (4.4,0) circle (0.03);	
\draw[-, thin](4.95,-0.125) -- (4.95,0.125);
\filldraw[color=black,fill=black] (5.5,0) circle (0.03);	
\draw[-, thin](5.825,-0.125) -- (5.825,0.125);
\filldraw[color=black,fill=black] (6.15,0) circle (0.03);	
\draw[-, thin](6.325,-0.125) -- (6.325,0.125);
\filldraw[color=black,fill=black] (6.5,0) circle (0.03);	
\draw[-, thin](6.625,-0.125) -- (6.625,0.125);
\filldraw[color=black,fill=black] (6.75,0) circle (0.03);	
\draw[-, thin](6.875,-0.125) -- (6.875,0.125);
\filldraw[color=black,fill=black] (7.0,0) circle (0.03);	
\draw[-, thin](7.125,-0.125) -- (7.125,0.125);
\filldraw[color=black,fill=black] (7.25,0) circle (0.03);	
\draw[-, thin](7.375,-0.125) -- (7.375,0.125);
\filldraw[color=black,fill=black] (7.5,0) circle (0.03);	
\draw[-, thin](7.625,-0.125) -- (7.625,0.125);
\filldraw[color=black,fill=black] (7.75,0) circle (0.03);	
\draw[-, thin](7.825,-0.125) -- (7.825,0.125);
\filldraw[color=black,fill=black] (8.0,0) circle (0.03);	
\draw[-, thick](0,0.0) -- (8.0,0.0);
\node[black,scale=1.5] at (3.35,0.5) {$f_{i-1}$};
\node[black,scale=1.5] at (3.35,-0.5) {$x_{i-1}$};
\node[black,scale=1.5] at (4.4,0.5) {$f_i$};
\node[black,scale=1.5] at (4.4,-0.5) {$x_i$};
\node[black,scale=1.5] at (5.5,0.5) {$f_{i+1}$};
\node[black,scale=1.5] at (5.5,-0.5) {$x_{i+1}$};
\node[black, scale=1.25,fill=white] at (0.0,0.0) {$\dotsc$};
\node[black, scale=1.25,fill=white] at (8.0,0.0) {$\dotsc$};
\end{tikzpicture}
\caption{\small { \bf Finite Volume Cells.}
Consider non-uniformly spaced points $\{ x_j \}$  and corresponding function values $\{ f_j = f(x_j) \}$.  This diagram depicts a set of cells with cell centers given by the points $\{ x_j \}$ and with cell edges located in between them at $(x_j + x_k)/2$ for $k=j-1,j+1$. This notation is used to derive a realizable, finite volume discretization of the generator $L$.  See \eqref{nu_symmetric_scheme_1d} for the resulting $Q$-matrix.  
 } \label{fig:finite_volume_grid_1D}
\end{figure}


\subsection{Time-Continuous Milestoning} \label{sec:scalar_miletoning}

Finally, we consider a `probabilistic' discretization of the generator $L$.   The idea behind this Q-matrix approximation is to determine the jump rates according to the exit probability and the mean holding time of the SDE solution at each grid point.

To this end, let $\tau_{a}^x = \inf\{ t>0 \mid Y(t) = a , Y(0) = x \}$ be the first time the SDE solution hits the state `a' given that it starts at the state `x'.  Define the exit probability (or the committor function) to be $q(x) = \Pr( \tau_{b}^x < \tau_{a}^x)$, which gives the probability that the process first hits `b' before it hits `a' starting from any state `x' that is between them: $a \le x \le b$.  This function satisfies the following boundary value problem: \begin{equation} \label{eq:committor_bvp}
L q(x) = 0 \;, \quad q(a) = 0 \;, \quad q(b) = 1 \;. 
\end{equation} 
Let $z_i = q(x_i)$ for every grid point $x_i$.    Solving \eqref{eq:committor_bvp} with $a=x_{i-1}$ and $b=x_{i+1}$ yields the following exit probabilities conditional on the process starting at the grid point $x_i$ \begin{equation} \label{eq:exact_rates}
\begin{aligned}
\Pr( \tau_{x_{i+1}}^{x_i} < \tau_{x_{i-1}}^{x_i} ) = \frac{z_{i} - z_{i-1}}{z_{i+1} - z_{i-1}}  \\
\Pr( \tau_{x_{i-1}}^{x_i} < \tau_{x_{i+1}}^{x_i} ) =   \frac{z_{i+1} - z_i}{z_{i+1} - z_{i-1}} 
\end{aligned}
\end{equation} 
We use these expressions below to define the transition rates of the approximation from $x_i$ to its nearest neighbors.

To obtain the mean holding time of the approximation, recall that the mean first passage time (MFPT) to $(x_{i-1}, x_{i+1})^c$ given that $Y(t) = x$ is the unique solution to the following boundary value problem: \begin{equation} \label{eq:mfpt_bvp}
L u(x)  = -1 \;, \quad x_{i-1} < x < x_{i+1}\;,  \quad u(x_{i-1}) = u(x_{i+1}) = 0 \;.
\end{equation}  
In particular, the MFPT to $(x_{i-1},x_{i+1})^c$ given that the process starts at the grid point $x_i$ is given by \[
 \Ex ( \tau_{x_{i-1}}^{x_i} \wedge \tau_{x_{i+1}}^{x_i}  ) = u(x_i) \;.
\]

Since the process initiated at $x_i$ spends on average $u(x_i)$ time units before it first hits a nearest neighbor of $x_i$, the exact mean holding time in state $x_i$ is given by: \begin{equation} \label{eq:intensity}
 ((Q_e)_{i,i+1} + (Q_e)_{i,i-1})^{-1} = u(x_i) \;.
\end{equation}
Similarly, the solution \eqref{eq:exact_rates} specifies the transition probabilities: \begin{equation} \label{eq:transition_probabilities}
\begin{aligned}
\frac{(Q_e)_{i,i+1}}{(Q_e)_{i,i+1} + (Q_e)_{i,i-1}} = \frac{z_i - z_{i-1}}{z_{i+1} - z_{i-1}}  \\
 \frac{(Q_e)_{i,i-1}}{(Q_e)_{i,i+1} + (Q_e)_{i,i-1}} = \frac{z_{i+1} - z_i}{z_{i+1} - z_{i-1}} 
\end{aligned}
\end{equation}
Incidentally, these are the transition probabilities used in discrete-time optimal milestoning \cite{VaVeCiEl2008}, where the milestones correspond to the grid points
depicted in Figure~\ref{fig:milestones}.    Combining \eqref{eq:intensity} and \eqref{eq:transition_probabilities} yields: \begin{equation} \label{eq:Qe_1D}
\begin{dcases}
(Q_e)_{i,i+1} =   \frac{z_i - z_{i-1}}{z_{i+1} - z_{i-1}} \frac{1}{u(x_i)}   \\
(Q_e)_{i,i-1} =  \frac{z_{i+1} - z_i}{z_{i+1} - z_{i-1}} \frac{1}{u(x_i) }
\end{dcases}
\end{equation} which shows that a $Q$-matrix can be constructed in one-dimension based on the exact mean holding time and exit probability.   We emphasize that this discretization is still an approximation, and in particular, does not exactly reproduce other statistics of the SDE solution as pointed out in \cite{VaVeCiEl2008}.  We will return to this point in \S\ref{sec:mfpt_1D} where we analyze some properties of this discretization.    In \S\ref{sec:downhill_time}, we will use this discretization to validate an asymptotic analysis of the mean holding time of finite-difference approximations in one dimension.

\medskip

To recap, the construction of $Q_e$ requires the exit probability (or committor function) to compute the transition probabilities, and the MFPT to compute the mean holding time.  Obtaining these quantities entails solving the boundary value problems \eqref{eq:committor_bvp} and \eqref{eq:mfpt_bvp}.  Thus, it is nontrivial to generalize $Q_e$ to many dimensions and we stress our main purpose in constructing $Q_e$ is to benchmark other (more practical) approximations in one dimension.


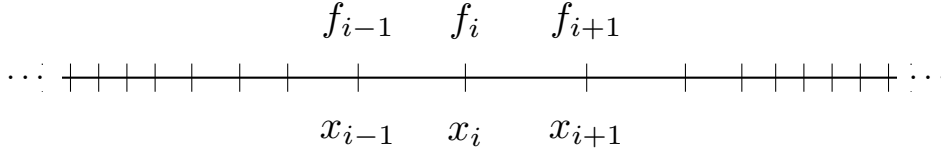
\begin{figure}
\centering
\begin{tikzpicture}[scale=1.5]
\draw[-, thin](0.125,-0.125) -- (0.125,0.125);
\draw[-, thin](0.375,-0.125) -- (0.375,0.125);	
\draw[-, thin](0.625,-0.125) -- (0.625,0.125);	
\draw[-, thin](0.875,-0.125) -- (0.875,0.125);				
\draw[-, thin](1.125,-0.125) -- (1.125,0.125);				
\draw[-, thin](1.45,-0.125) -- (1.45,0.125);				
\draw[-, thin](1.875,-0.125) -- (1.875,0.125);				
\draw[-, thin](2.3,-0.125) -- (2.3,0.125);				
\draw[-, thin](2.925,-0.125) -- (2.925,0.125);
\draw[-, thin](3.875,-0.125) -- (3.875,0.125);
\draw[-, thin](4.95,-0.125) -- (4.95,0.125);
\draw[-, thin](5.825,-0.125) -- (5.825,0.125);
\draw[-, thin](6.325,-0.125) -- (6.325,0.125);
\draw[-, thin](6.625,-0.125) -- (6.625,0.125);
\draw[-, thin](6.875,-0.125) -- (6.875,0.125);
\draw[-, thin](7.125,-0.125) -- (7.125,0.125);
\draw[-, thin](7.375,-0.125) -- (7.375,0.125);
\draw[-, thin](7.625,-0.125) -- (7.625,0.125);
\draw[-, thin](7.825,-0.125) -- (7.825,0.125);
\draw[-, thick](0,0.0) -- (8.0,0.0);
\node[black,scale=1.5] at (2.925,0.5) {$f_{i-1}$};
\node[black,scale=1.5] at (2.925,-0.5) {$x_{i-1}$};
\node[black,scale=1.5] at (3.875,0.5) {$f_i$};
\node[black,scale=1.5] at (3.875,-0.5) {$x_i$};
\node[black,scale=1.5] at (4.95,0.5) {$f_{i+1}$};
\node[black,scale=1.5] at (4.95,-0.5) {$x_{i+1}$};
\node[black, scale=1.25,fill=white] at (0.0,0.0) {$\dotsc$};
\node[black, scale=1.25,fill=white] at (8.0,0.0) {$\dotsc$};
\end{tikzpicture}
\caption{\small { \bf Milestones.}
Consider a non-uniform (bounded or unbounded) grid, a set of milestones
$\{ x_j \}$ on this grid, and corresponding function values $\{ f_j = f(x_j) \}$ as shown in this diagram.
Using this notation a realizable discretization which retains the exact transition probabilities and mean holding times
of the SDE is given in \eqref{eq:Qe_1D}.  
 }
  \label{fig:milestones}
\end{figure}


%
%

\section{Realizable Discretizations in 2D} \label{sec:realizable_methods_2D}

Here we derive realizable finite difference schemes for planar SDEs.  The main issue in this part is the discretization of the mixed partial derivatives appearing in \eqref{eq:generator}, subject to the constraint that the realizability condition (Q2) holds.  To be specific, consider an SDE whose domain $\Omega$ is in $\mathbb{R}^2$ and write its drift and noise fields as: \[
\mu(x,y) = \begin{bmatrix} \mu^1(x,y) \\ \mu^2(x,y) \end{bmatrix} \quad \text{and} \quad
 \sigma(x) \sigma(x)^T = M(x,y)  = \begin{bmatrix} M^{11}(x,y) & M^{12}(x,y) \\ M^{12}(x,y) & M^{22}(x,y) \end{bmatrix}  \;.
\]
Here $M$ is the diffusion field.   Let $S = \{ (x_i, y_j) \}$ be a collection of grid points on a rectangular grid in two-space as illustrated in Figure~\ref{fig:2d_grid}.  Given $f: \mathbb{R}^2 \to \mathbb{R}$, we adopt the notation:  

\medskip

\setlength\tabcolsep{5pt}
\setlength\extrarowheight{2pt}

\begin{center}
\begin{tabular}{cc}
shorthand & meaning \\
\hline
$f_{i,j} = f(x_i, y_j)$ &  grid point function evaluation \\ 
\hline
\begin{tabular}{c} $f_{i \pm 1, j} = f(x_{i \pm 1}, y_j)$ \\
 $f_{i, j \pm1} = f(x_i , y_{j \pm 1})$ \end{tabular} &  \begin{tabular}{c} nearest neighbor grid \\ point function evaluation \end{tabular} \\
\hline
\begin{tabular}{c} $f_{i \pm 1, j \pm 1} = f(x_{i \pm 1}, y_{j \pm 1})$ \\  
$f_{i \pm 1, j \mp 1} = f(x_{i \pm 1} , y_{j \mp 1})$ \end{tabular} & \begin{tabular}{c} next to nearest neighbor \\ grid point function evaluation  \end{tabular}
\end{tabular} 
\end{center}

\medskip

\noindent
The distance between neighboring vertical or horizontal grid points is permitted to be variable. Hence, it helps to define: \[
\begin{dcases}
\delta x^+_i = x_{i+1} - x_i \;, \quad \delta x^-_i = x_i - x_{i-1} \;, \quad \delta x_i = \frac{\delta x_i^+ + \delta x_i^-}{2}  \;, \\
\delta y^+_j = y_{j+1} - y_j \;, \quad \delta y^-_j = y_j - y_{j-1} \;, \quad \delta y_j = \frac{\delta y_j^+ + \delta y_j^-}{2}  
\end{dcases}
\]
at all grid points $(x_i, y_j) \in S$.  


\begin{figure}
\centering
\begin{tikzpicture}[scale=1.25]
\begin{scope}[very thick,decoration={
    markings,
    mark=at position 0.65 with {\arrow{>}}}
    ] 
\draw[step=1.0,gray,thin] (-0.1,0.9) grid (6.1,5.1);
\filldraw[color=black,fill=black] (2,2) circle (0.05);
\filldraw[color=black,fill=black] (2,3) circle (0.05);	
\filldraw[color=black,fill=black] (2,4) circle (0.05);
\filldraw[color=black,fill=black] (4,2) circle (0.05);
\filldraw[color=black,fill=black] (4,3) circle (0.05);		
\filldraw[color=black,fill=black] (4,4) circle (0.05);
\filldraw[color=black,fill=black] (3,2) circle (0.05);	
\filldraw[color=black,fill=black] (3,3) circle (0.05);	
\filldraw[color=black,fill=black] (3,4) circle (0.05);	
\node[black, scale=1.,fill=white] at (4.5,1.65) {$(x_{i+1}, y_{j-1})$};
\node[black, scale=1.,fill=white] at (4.75,3.0) {$(x_{i+1}, y_{j})$};
\node[black, scale=1.,fill=white] at (4.5,4.35) {$(x_{i+1}, y_{j+1})$};
\node[black, scale=1.,fill=white] at (1.5,4.35) {$(x_{i-1}, y_{j+1})$};
\node[black, scale=1.,fill=white] at (1.25,3.0) {$(x_{i-1}, y_{j})$};
\node[black, scale=1.,fill=white] at (1.5,1.65) {$(x_{i-1}, y_{j-1})$};
\node[black, scale=1.,fill=white] at (3,1.65) {$(x_{i}, y_{j-1})$};
\node[black, scale=1.,fill=white] at (3,4.35) {$(x_{i}, y_{j+1})$};
\draw[postaction={decorate}](3,3)--(3,4) ;
\draw[postaction={decorate}](3,3)--(4,3) ;
\draw[postaction={decorate}] (3,3)--(2,3) ;
\draw[postaction={decorate}](3,3)--(3,2) ;
\draw[postaction={decorate}](3,3)--(2,2) ;
\draw[postaction={decorate}](3,3)--(4,2) ;
\draw[postaction={decorate}](3,3)--(4,4) ;
\draw[postaction={decorate}](3,3)--(2,4) ;
\end{scope}
\end{tikzpicture}
\caption{\small { \bf Grid in two dimensions.}
This diagram illustrates the reaction channels from the point $(x_{i}, y_{j})$ on a gridded space in two dimensions.  These channels are used  to obtain a realizable discretization of the generator $L$ of the SDE.  Note that the channels in the diagonal and antidiagonal directions are needed to approximate mixed partial derivatives that may appear in the continuous generator $L$.   The resulting realizable discretizations can be found in \eqref{eq:tQu} and \eqref{eq:tQc_2d}.
 }
  \label{fig:2d_grid}
\end{figure}
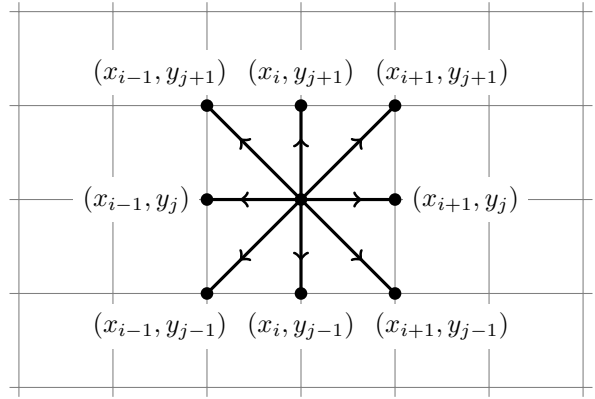


We first consider a generalization of the first-order scheme \eqref{eq:tQu_1d}.   The drift term in $L$ is approximated using an upwinded finite difference:
\begin{equation} \label{eq:upwind_b1}
\mu^1(x_i,y_j) \frac{\partial f}{\partial x}(x_i,y_j)  \approx \frac{ ( \mu^1_{i,j} \vee 0 ) }{\delta x^+_i} \left( f_{i+1,j} - f_{i,j} \right)  +  \frac{(\mu^1_{i,j} \wedge 0 )}{\delta x^-_i }  \left(  f_{i,j} -  f_{i-1,j} \right)
\end{equation}
and similarly, 
\begin{equation}\label{eq:upwind_b2}
\mu^2(x_i,y_j) \frac{\partial f}{\partial y}(x_i,y_j)  \approx \frac{( \mu^2_{i,j} \vee 0 )  }{\delta y^+_j } \left( f_{i,j+1} - f_{i,j} \right) + \frac{ ( \mu^2_{i,j} \wedge 0 )}{ \delta y^-_j  } \left(  f_{i,j} - f_{i, j-1} \right) \;.
\end{equation}
We recruit all of the channels that appear in Figure~\ref{fig:2d_grid} in order to discretize the second-order partial derivatives in $L$.  Specifically, \begin{equation} \label{eq:central_approx_2d}
\begin{aligned}
 \tr (  M(x_i,y_j) D^2f(x_i,y_j) )  &\approx  
A^{\pm}  \underset{\text{horizontal differences}}{\underbrace{(f_{i\pm1,j}-f_{i,j})}} + 
B^{\pm} \underset{\text{vertical differences}}{\underbrace{(f_{i,j\pm1}-f_{i,j})}}  \\
& + 
C^{\pm} \underset{\text{diagonal differences}}{\underbrace{(f_{i \pm 1,j \pm 1} - f_{i,j} )}} + 
D^{\pm} \underset{\text{antidiagonal differences}}{\underbrace{( f_{i \pm 1, j \mp 1} - f_{i,j} ) }}
\end{aligned}
\end{equation}
where we introduced the following labels for the rate functions: $A^{\pm}$, $B^{\pm}$, $C^{\pm}$ and $D^{\pm}$.   To be clear there are eight rate functions and eight finite differences in this approximation.  We determine the eight rate functions by expanding the horizontal/vertical finite differences in this approximation to obtain \begin{equation} \label{eq:expansions_of_channels_2d_hv}
\begin{aligned}
f_{i\pm1,j}-f_{i,j} &\approx \pm \left( \frac{\partial f}{\partial x}  \right)_{i,j}\delta x_i^{\pm} + \frac{1}{2}\left(  \frac{\partial^2 f}{\partial x^2} \right)_{i,j} (\delta x_i^{\pm})^2 \\
f_{i,j\pm1}-f_{i,j} &\approx \pm \left( \frac{\partial f}{\partial y} \right)_{i,j} \delta y_j^{\pm} + \frac{1}{2} \left( \frac{\partial^2 f}{\partial y^2}  \right)_{i,j} (\delta y_j^{\pm})^2  
\end{aligned}
\end{equation}
Similarly, we expand the diagonal/antidiagonal differences to obtain
\begin{equation} \label{eq:expansions_of_channels_2d_da}
\begin{aligned}
f_{i \pm 1,j \pm 1} - f_{i,j}  &\approx \pm  \left( \frac{\partial f}{\partial x}  \right)_{i,j}\delta x_i^{\pm}  \pm \left( \frac{\partial f}{\partial y} \right)_{i,j} \delta y_j^{\pm}  \\
& +  \frac{1}{2}\left(  \frac{\partial^2 f}{\partial x^2} \right)_{i,j} (\delta x_i^{\pm})^2 + \frac{1}{2} \left( \frac{\partial^2 f}{\partial y^2}  \right)_{i,j} (\delta y_j^{\pm})^2  \\
& + \left(  \frac{\partial^2 f}{\partial x \partial y} \right)_{i,j} \delta x_i^{\pm} \delta y_j^{\pm} \\
f_{i \pm 1, j \mp 1} - f_{i,j}  &\approx\pm  \left( \frac{\partial f}{\partial x}  \right)_{i,j}\delta x_i^{\pm}  \mp \left( \frac{\partial f}{\partial y} \right)_{i,j} \delta y_j^{\mp}  \\
& +  \frac{1}{2}\left(  \frac{\partial^2 f}{\partial x^2} \right)_{i,j} (\delta x_i^{\pm})^2 + \frac{1}{2} \left( \frac{\partial^2 f}{\partial y^2}  \right)_{i,j} (\delta y_j^{\mp})^2  \\
&- \left(  \frac{\partial^2 f}{\partial x \partial y} \right)_{i,j} \delta x_i^{\pm} \delta y_j^{\mp} 
\end{aligned}
\end{equation}
Note that the first order partial derivatives appearing in these expansions appear in plus and minus pairs, and ultimately cancel out.  This cancellation is necessary to obtain an approximation to the second-order pure and mixed partial derivatives in \eqref{eq:central_approx_2d}. Using these expansions it is straightforward to show that the following choice of rate functions \begin{align*}
A^{\pm} &= M^{11}_{i,j} (\delta x_i^{\pm} \delta x_i)^{-1} -  (M^{12}_{i,j} \vee 0) \; ( \delta x_i^{\pm} \delta y_j^{\pm})^{-1}  + (M^{12}_{i,j} \wedge 0) \; (\delta x_i^{\pm} \delta y_j^{\mp})^{-1}    \\
B^{\pm} &= M^{22}_{i,j} (\delta y_j^{\pm} \delta y_j)^{-1} -  (M^{12}_{i,j}  \vee 0) \; ( \delta x_i^{\pm} \delta y_j^{\pm})^{-1}  + (M^{12}_{i,j} \wedge 0) \; (\delta x_i^{\mp} \delta y_j^{\pm})^{-1}     \\
C^{\pm} &= (M^{12}_{i,j} \vee 0) \; ( \delta x_i^{\pm} \delta y_j^{\pm})^{-1}  \\
D^{\pm} &=- ( M^{12}_{i,j} \wedge 0 ) \; ( \delta x_i^{\pm} \delta y_j^{\mp})^{-1} 
\end{align*}
imply that the approximation in \eqref{eq:central_approx_2d} holds.   Note here that we have used an upwinded finite difference to approximate the mixed partial derivatives of $L$.  Combining the above approximations to the drift and diffusion terms yields: \begin{equation} \label{eq:tQu}
\begin{aligned}
\tilde  Q_u  f_{i,j}  & =    \left(  \frac{ \mu^1_{i,j} \vee 0 }{\delta x^+_i} + 
\frac{M^{11}_{i,j}}{ \delta x_i \delta x_i^+ }  -  \frac{M^{12}_{i,j} \vee 0}{ \delta x_i^{+} \delta y_j^{+}}  + \frac{M^{12}_{i,j} \wedge 0}{ \delta x_i^{+} \delta y_j^{-}}  \right) (f_{i+1,j}-f_{i,j}) \\
&   + \left(  - \frac{ \mu^1_{i,j} \wedge 0 }{\delta x^-_i} + 
\frac{M^{11}_{i,j}}{ \delta x_i \delta x_i^- }   -  \frac{M^{12}_{i,j} \vee 0}{ \delta x_i^{-} \delta y_j^{-}}  + \frac{M^{12}_{i,j} \wedge 0}{ \delta x_i^{-} \delta y_j^{+}}   \right)   (f_{i-1,j}-f_{i,j}) \\
&   +  \left(  \frac{ \mu^2_{i,j} \vee 0 }{\delta y^+_j} + 
\frac{M^{22}_{i,j}}{ \delta y_j \delta y_j^+ }   -  \frac{M^{12}_{i,j} \vee 0}{ \delta x_i^{+} \delta y_j^{+}}  + \frac{M^{12}_{i,j} \wedge 0}{ \delta x_i^{-} \delta y_j^{+}} \right) (f_{i,j+1}-f_{i,j}) \\
&   +  \left(  - \frac{ \mu^2_{i,j} \wedge 0 }{\delta y^-_j} + 
\frac{M^{22}_{i,j}}{ \delta y_j \delta y_j^- }   -  \frac{M^{12}_{i,j} \vee 0}{ \delta x_i^{-} \delta y_j^{-}}  + \frac{M^{12}_{i,j} \wedge 0}{ \delta x_i^{+} \delta y_j^{-}}   \right)  (f_{i,j-1}-f_{i,j}) \\
&   +  \frac{M^{12}_{i,j} \vee 0}{ \delta x_i^+ \delta y_j^+}   (f_{i+1,j+1}-f_{i,j}) +   \frac{M^{12}_{i,j} \vee 0}{ \delta x_i^- \delta y_j^-}   (f_{i-1,j-1}-f_{i,j}) \\
&   - \frac{M^{12}_{i,j} \wedge 0}{ \delta x_i^+ \delta y_j^-}  (f_{i+1,j-1}-f_{i,j})  -\frac{M^{12}_{i,j} \wedge 0}{ \delta x_i^- \delta y_j^+}  (f_{i-1,j+1}-f_{i,j}) \;.
  \end{aligned}
\end{equation}
This generator is a two-dimensional generalization of \eqref{eq:generator_for_upwind}.  However, unlike in 1D, to approximate mixed partial derivatives finite differences in the diagonal and antidiagonal channel directions are used.  To be accurate, the spurious pure derivatives produced by these differences, however, must be subtracted off the vertical and horizontal finite differences.  This subtraction affects the realizability condition (Q2), and in particular, the weights in the horizontal and vertical finite differences may become negative.   Note that this issue becomes a bit moot if the diffusion matrix $M(x,y)$ is uniformly diagonally dominant, since this property implies that \[
M^{11}_{i,j} \wedge M^{22}_{i,j} \ge | M^{12}_{i,j} |  \qquad \text{for all $(x_i, y_j) \in S$} 
\] 
and hence, $\tilde Q_u$ satisfies the realizability condition (Q2) on a gridded state space, with evenly spaced points in the $x$ and $y$ directions.   

Alternatively, if we approximate the first and second-order `pure' partial derivatives in $L$ using a weighted central difference, and use an upwinded finite difference to handle the mixed partials, we obtain the following generator:  \begin{equation} \label{eq:tQc_2d}
\begin{aligned}
 \tilde Q_c &  f_{i,j} =   \exp\left( + \frac{1}{2}  \tilde \mu^1_{i,j} \delta x_i^+ \right) 
 \left( \frac{M^{11}_{i,j}}{ \delta x_i \delta x_i^+ } -  \frac{M^{12}_{i,j} \vee 0}{ \delta x_i^{+} \delta y_j^{+}}  + \frac{M^{12}_{i,j} \wedge 0}{ \delta x_i^{+} \delta y_j^{-}} \right) (f_{i+1,j}-f_{i,j}) \\
& \qquad  + \exp\left( - \frac{1}{2}  \tilde \mu^1_{i,j} \delta x_i^-  \right) 
\left( \frac{M^{11}_{i,j}}{ \delta x_i \delta x_i^- }    -  \frac{M^{12}_{i,j} \vee 0}{ \delta x_i^{-} \delta y_j^{-}}  + \frac{M^{12}_{i,j} \wedge 0}{ \delta x_i^{-} \delta y_j^{+}}  \right)   (f_{i-1,j}-f_{i,j}) \\
& \qquad  + \exp\left( + \frac{1}{2}  \tilde \mu^2_{i,j} \delta y_j^+  \right)  
\left(  \frac{M^{22}_{i,j}}{ \delta y_j \delta y_j^+ }  -  \frac{M^{12}_{i,j} \vee 0}{ \delta x_i^{+} \delta y_j^{+}}  + \frac{M^{12}_{i,j} \wedge 0}{ \delta x_i^{-} \delta y_j^{+}} \right) (f_{i,j+1}-f_{i,j}) \\
& \qquad  + \exp\left( -\frac{1}{2}  \tilde \mu^2_{i,j} \delta y_j^-  \right) 
\left(  \frac{M^{22}_{i,j}}{ \delta y_j \delta y_j^- }   -  \frac{M^{12}_{i,j} \vee 0}{ \delta x_i^{-} \delta y_j^{-}}  + \frac{M^{12}_{i,j} \wedge 0}{ \delta x_i^{+} \delta y_j^{-}}  \right)  (f_{i,j-1}-f_{i,j}) \\
& \qquad  + \exp\left( +\frac{1}{2}  \tilde \mu^1_{i,j}  \delta x_i^+ +  \frac{1}{2}  \tilde \mu^2_{i,j} \delta y_j^+  \right)   \frac{M^{12}_{i,j} \vee 0}{ \delta x_i^+ \delta y_j^+}    (f_{i+1,j+1}-f_{i,j}) \\
 & \qquad+   \exp\left( - \frac{1}{2}  \tilde \mu^1_{i,j}  \delta x_i^- -  \frac{1}{2}  \tilde \mu^2_{i,j} \delta y_j^-  \right) \frac{M^{12}_{i,j} \vee 0}{ \delta x_i^- \delta y_j^-}    (f_{i-1,j-1}-f_{i,j}) \\
 & \qquad -\exp\left( +\frac{1}{2}  \tilde \mu^1_{i,j}  \delta x_i^+ -  \frac{1}{2}  \tilde \mu^2_{i,j} \delta y_j^-   \right)  \frac{M^{12}_{i,j} \wedge 0}{ \delta x_i^+ \delta y_j^-}  (f_{i+1,j-1}-f_{i,j}) \\
  & \qquad-\exp\left(  - \frac{1}{2}  \tilde \mu^1_{i,j}  \delta x_i^- +  \frac{1}{2}  \tilde \mu^2_{i,j} \delta y_j^+   \right)  \frac{M^{12}_{i,j} \wedge 0}{ \delta x_i^- \delta y_j^+} (f_{i-1,j+1}-f_{i,j}) \;.
\end{aligned}
\end{equation}
Here we have defined $ \tilde \mu = ( \tilde \mu^1,  \tilde \mu^2)^T$ as the solution to the linear system \[
M(x,y) \tilde \mu(x,y) = \mu(x,y) \;.
\]  Like  $\tilde Q_u$, if the diffusion matrix is diagonally dominant, then this generator is realizable on an evenly spaced grid.   In \S\ref{sec:lognormal_2d} and \S\ref{sec:lv_process}, we test  $ \tilde Q_c $  on planar SDE problems.  In these problems the diffusion matrices are not diagonally dominant, and we must use variable step sizes in order to obtain a realizable discretization. In \S\ref{sec:diagonally_dominant_case} we generalize the upwinded finite difference approximation $ \tilde Q_c $ to multiple dimensions and prove that if the diffusion matrix is uniformly diagonally dominant, then this generalized discretization is realizable.    

\medskip

To avoid assuming the diffusion matrix is diagonally dominant, we next consider approximations whose state space is gridless.    

%
%

\section{Realizable Discretizations in nD} \label{sec:realizable_methods_nD}

We now present a realization discretization on a gridless state space contained in $\Omega \subset \mathbb{R}^n$.  To introduce this approximation, set the number of channel directions equal to the dimension of the space and let $i$ be an index over these reaction channel directions $1 \le i \le n$.  Let $h_i^{\pm}(x)$ be forward/backward spatial step sizes at state $x$, let $h_i(x)$ be the average step size at state $x$ defined as \[
h_i(x) = (h_i^+(x) + h_i^-(x))/2 \;,
\] let $\{ \sigma_i(x) \}_{i=1}^n$ be the columns of the noise matrix $\sigma(x)$, and let $\tilde \mu(x)$ be a transformed drift field defined as: \begin{equation} \label{eq:transformed_drift}
M(x) \tilde \mu(x) = \mu(x) 
\end{equation}
where $M(x) = \sigma(x) \sigma(x)^T$ is the $n \times n$ diffusion matrix.  (If $M(x)$ is positive definite, then this linear system has a unique solution $\tilde \mu(x)$ for every $x \in \mathbb{R}^n$.  This condition holds in the interior of $\Omega$ for all of the SDE problems we consider in this paper.)   With this notation in hand, here is a second order accurate, realizable discretization in $n$-dimensions.

\medskip

\begin{equation} \label{eq:Qc}
\boxed{
\begin{aligned}
& Q_c f(x) =
 \sum_{i=1}^n
\frac{1}{h_i^+(x) h_i(x)}\exp\left(\quad \frac{h_i^+(x)}{2} \tilde \mu(x)^T \sigma_i(x)  \right) \left(  \vphantom{\sum} f\left( x+ h_i^+(x) \sigma_i(x)  \right) - f(x) \right)  \\
 & \qquad\qquad\quad +
 \frac{1}{h_i^-(x) h_i(x)}\exp\left(-  \frac{h_i^-(x)}{2} \tilde \mu(x)^T \sigma_i(x) \right)  \left( \vphantom{\sum} f\left( x- h_i^-(x)  \sigma_i(x) \right) - f(x) \right)  
\end{aligned}
}
\end{equation}

\medskip

\noindent
The generator $Q_c$ is a weighted central finite difference scheme in the directions of the columns of the noise coefficient matrix evaluated at the state $x$.  When the noise coefficient matrix is additive (state-independent) and isotropic (direction-independent), then the state space of this generator is gridded and the generator can be represented as matrix.   In general, we stress that the state space of this approximation is gridless, and thus the generator cannot be represented as a matrix.  Nevertheless it can be realized by using the SSA algorithm.   A realization of this process is a continuous-time random walk with a user-defined jump size, jump directions given by the columns of the noise coefficient matrix, and a mean holding time that is determined by the discretized generator.  This generator extends realizable spatial discretizations appearing in previous work \cite{elston1996numerical, WaPeEl2003,Ph2008,MeScVa2009,LaMeHaSc2011} to the SDE \eqref{eq:sde}, which generically has a non-symmetric infinitesimal generator $L$ and multiplicative noise.

The next proposition states that this generator satisfies (Q1) and (Q2).

\begin{prop} \label{prop:Qc_accuracy}
Assuming that: \[
|h_i^+(x) - h_i^-(x)| \vee (h_i^+(x))^2 \vee (h_i^-(x))^2 \le h^2 \quad \text{for all $x \in \Omega$}
\]
then the operator $Q_c$ in \eqref{eq:Qc} satisfies (Q2) and (Q1) with $p=2$.  
\end{prop}

\begin{proof}
This proof is somewhat informal.  A formal proof requires hypotheses on the SDE and the test function $f(x)$.  For clarity of presentation, we defer these technical issues to Chapter~\ref{chap:analysis}.  Here we concentrate on showing in simple terms that the approximation meets the requirements (Q1) and (Q2).  

Let $D^k f(x)(u_1, \dotsc, u_k)$ denote the $k$th-derivative of $f(x)$ applied to the vectors $u_1,\dotsc,u_k$ for any $k \ge 1$.  Note that the rate functions appearing in $Q_c$ are all positive, and thus, the realizability condition (Q2) is satisfied. To check second-order accuracy or (Q1) with $p=2$, Taylor expand the forward and backward finite differences appearing in the function $Q_cf(x)$ to obtain: \begin{align*}
  Q_c f(x) = \sum_{i=1}^n \frac{1}{h_i(x)}  Df(x)(  \sigma_i(x)) & \left( \exp\left( \tilde \mu(x)^T \sigma_i(x) \frac{h_i^+(x)}{2} \right)  \right. \\
&  \left.  - \exp\left( - \tilde \mu(x)^T \sigma_i(x) \frac{h_i^-(x)}{2} \right)  \right)  \\
 + \frac{1}{2 h_i(x)} D^2 f(x)  ( \sigma_i(x),  \sigma_i(x)) &  \left(h_i^+(x) \exp\left( \tilde \mu(x)^T \sigma_i(x) \frac{h_i^+(x)}{2} \right)  \right. \\
&  \left. + h_i^-(x) \exp\left( - \tilde \mu(x)^T \sigma_i(x) \frac{h_i^-(x)}{2} \right)  \right)  \\
 + \frac{1}{6 h_i(x)} D^3 f(x)  ( \sigma_i(x),  \sigma_i(x), \sigma_i(x)) &  \left( h_i^+(x)^2 \exp\left( \tilde \mu(x)^T \sigma_i(x) \frac{h_i^+(x)}{2} \right) \right. \\
& \left. - h_i^-(x)^2 \exp\left( - \tilde \mu(x)^T \sigma_i(x) \frac{h_i^-(x)}{2} \right)  \right) + O(h^2) \;.
\end{align*}
In order to obtain a crude estimate of the remainder term in this expansion,  we used the identity $a^3 + b^3 = (a+b) (a^2 + b^2 - ab)$ with $a=h^+_i(x)$ and $b=h^-_i(x)$, and the relation $h_i(x) = (h_i^+(x) + h_i^-(x))/2$.  A similar Taylor expansion of the exponentials yields, \begin{align*}
  Q_c f(x) &= \sum_{i=1}^n Df(x)^T  \sigma_i(x)  \tilde \mu(x)^T \sigma_i(x) + D^2 f(x)  ( \sigma_i(x),  \sigma_i(x)) +O(h^2) \\
&=   \tr\left( \left( \vphantom{\sum} Df(x) \tilde \mu(x)^T + D^2 f(x) \right)  \underbrace{\sum_{i=1}^n \sigma_i \sigma_i^T}_{=M(x)}  \right)  +O(h^2)  \\
&= \tr\left( Df(x) \mu(x)^T + D^2 f(x) M(x) \right)  +O(h^2) \\
&= L f(x) + O(h^2) 
 \end{align*}
where we used the cyclic property of the trace operator, the basic linear algebra fact that the product $\sigma \sigma^T$ is the sum of the outer products of the columns of $\sigma$, and elementary identities to estimate the remainder terms, like $a^2 - b^2 = (a+b) (a-b)$ with $a=h^+_i(x)$ and $b=h^-_i(x)$. Thus, $Q_c$ is second-order accurate as claimed.
\end{proof}


The analysis in Chapter~\ref{chap:analysis} focusses on the stability, accuracy and complexity of the generator $Q_c$, mainly because it is second-order accurate.   However, let us briefly consider two simpler alternatives, which are first-order accurate.  To avoid evaluating the exponential function appearing in the transition rates of $Q_c$ in \eqref{eq:Qc}, an upwinded generator is useful 
\begin{equation} \label{eq:Qu} \begin{aligned}
 & Q_u f(x) =
 \sum_{i=1}^n
 \frac{( \tilde \mu(x)^T \sigma_i(x)) \vee 0}{h_i^+(x)}  \left( f\left( x+ h_i^+(x) \sigma_i(x)  \right) - f(x) \right)  \\
 & \qquad\quad \quad-
 \frac{(\tilde \mu(x)^T \sigma_i(x) ) \wedge 0}{h_i^-(x)}  \left( f\left( x- h_i^-(x)  \sigma_i(x) \right) - f(x) \right) \\
 & \qquad\quad\quad+ \frac{1}{h_i^+(x) h_i(x)}  \left( f\left( x+ h_i^+(x) \sigma_i(x)  \right) - f(x) \right) \\
 & \qquad\quad\quad+  \frac{1}{h_i^-(x) h_i(x)}  \left( f\left( x- h_i^-(x)  \sigma_i(x) \right) - f(x) \right) \;.
\end{aligned}
\end{equation}
Let $\{ e_i \}_{i=1}^n$ be the standard basis of $\mathbb{R}^n$.  To avoid solving the linear system in \eqref{eq:transformed_drift}, an upwinded generator that uses the standard basis to discretize the drift term in $L$ is useful
\begin{equation} \label{eq:Quu}
\begin{aligned}
 & Q_{uu} f(x) =
 \sum_{i=1}^n
 \frac{( \mu(x)^T e_i) \vee 0}{h_i^+(x)} \left( f\left( x+ h_i^+(x) e_i  \right) - f(x) \right)  \\
 & \qquad\quad \quad-
 \frac{( \mu(x)^T e_i ) \wedge 0}{h_i^-(x)}  \left( f\left( x- h_i^-(x)  e_i \right) - f(x) \right) \\
 & \qquad \quad\quad+ \frac{1}{h_i^+(x) h_i(x)}  \left( f\left( x+ h_i^+(x) \sigma_i(x)  \right) - f(x) \right) \\
 & \qquad \quad\quad+  \frac{1}{h_i^-(x) h_i(x)}  \left( f\left( x- h_i^-(x)  \sigma_i(x) \right) - f(x) \right) \;.
\end{aligned}
\end{equation}

\begin{prop} \label{accuracy_of_Qu}
Assuming that: \[
|h_i^+(x) - h_i^-(x)| \vee h_i^+(x) \vee h_i^-(x) \le h \quad \text{for all $x \in \Omega$}
\]
then the operators $Q_u$ and $Q_{uu}$ satisfy (Q2) and (Q1) with $p=1$.  
\end{prop}

The proof of Proposition~\ref{accuracy_of_Qu} is similar to the proof of Proposition~\ref{prop:Qc_accuracy} and therefore omitted.

\section{Scaling of Approximation with System Size} \label{sec:scaling}

Here we briefly discuss the complexity of the SSA induced by the generator in \eqref{eq:Qc}.  We will return to this point in Chapter~\ref{chap:analysis} where we quantify the complexity of sample paths as a function of the spatial step size parameter.  

First we compare the complexity of the SSA to a standard numerical PDE solver for the Kolmogorov equation.  When the grid is bounded, even though the matrix associated to standard PDE discretizations (like finite difference, volume and element methods) is often sparse, the computational effort required to construct its entries in order to time integrate the resulting semi-discrete ODE (either exactly or approximately) scales exponentially with the number of dimensions.   On an unbounded grid, this ODE is infinite-dimensional and a standard numerical PDE approach does not apply. 

In contrast, note from Algorithm~\ref{algo:ssa} that an SSA simulation of the Markov process induced by the generator in \eqref{eq:Qc} takes as input the reaction rates and channels, which can be peeled off of  \eqref{eq:Qc}.  More to the point, both the spatial and temporal updates in each SSA step use local information.  (Put another way, in the case of a gridded state space, a step of the SSA method only requires computing a row of a typically sparse matrix associated to the generator, not the entire matrix.)    Thus, the cost of each SSA step in high dimension is much less severe than standard numerical PDE solvers and remains practical.   

Next consider the average number of SSA steps required to reach a finite time.  Referring to Algorithm~\ref{algo:ssa}, the random time increment $t_1 - t_0$ gives the amount of time the process spends in state $X(t_0)$ before jumping to a new state $X(t_1)$.  In each step of the SSA method, only a single degree of freedom is updated because the process moves from the current position to a new position in the direction of only one of the $n$ columns of the noise matrix.  In other words, the method involves local moves, and in this way, it is related to splitting methods for SDEs and ODEs based on local moves, for example the per-particle splitting integrator proposed for Metropolis-corrected Langevin integrators in \cite{BoVa2012} or J-splitting schemes for Hamiltonian systems \cite{FeWa1994}.  However, the physical clock in splitting algorithms does not tick with each local move as it does in (Step 1) of the SSA method.  Instead the physical time in a splitting scheme does not pass until all of the local moves have happened.   In contrast, the physical clock ticks with each local move of the SSA method. This difference implies that the mean holding time of the SSA method is inversely proportional to dimension.   

The easiest way to confirm this point is to consider a concrete example, like evolving $n$ decoupled oscillators using the SSA induced by \eqref{eq:Qc}.  In this example, the mean time step for a global move of the system -- where each oscillator has taken a step --  cannot depend on dimension because the oscillators are decoupled.     For simplicity assume the SDE for each oscillator is $d Y = \mu(Y) dt + \sqrt{2} dW$ and the spatial step size in every reaction channel direction is uniform and equal to $h$, and let $x = (x_1, \dotsc, x_n)^T$.  Then, from \eqref{eq:Qc} the total jump rate is:   \[
\lambda(x) = \sum_{i=1}^n \lambda_i(x)
\]
where $\lambda_i(x)$ is the jump rate of the $i$th oscillator: \[
\lambda_i(x) = \frac{2}{h^2} \cosh(\mu(x_i) \frac{h}{2} ) \;.
\]
Expanding the total jump rate about $h=0$ yields \[
\lambda(x) = \frac{2n}{h^2} + \frac{1}{4} \sum_{i=1}^n \mu(x_i)^2 + O(h^2) \;.
\] Similarly, expanding the mean holding time about $h=0$ yields \[
\lambda(x)^{-1} = \frac{h^2}{2 n} - \frac{h^4}{16 n^2} \sum_{i=1}^n \mu(x_i)^2 + O(h^6) \;.
\]
Thus, the total jump rate is roughly $O( (2 n) / h^2)$ and the mean holding time scales like $O(h^2 / (2 n))$ with dimension $n$, as expected.  

To summarize, we see that given the nature of the SSA method, accuracy necessitates a smaller mean time step for each local move and the overall complexity of the SSA method is no different from a splitting temporal integrator that uses local moves.   

%
%

\section{Generalization of Realizable Discretizations in nD} \label{sec:generalization}

Suppose the diffusion matrix can be written as a sum of rank-$1$ matrices: \begin{equation} \label{eq:assumption_on_M}
M(x) = \sum_{i=1}^s w_i(x) \eta_i(x) \eta_i(x)^T
\end{equation} where $\eta_i(x) \in \mathbb{R}^n$ is an $n$-vector and $w_i(x)\ge 0$ is a non-negative weight function for all $x \in \Omega$, and for $1 \le i \le s$.  Since the diffusion matrix $M(x)$ is positive definite, it must be that $s \ge n$ and that the set of vectors $\{ \eta_i(x) \}_{i=1}^s$ spans $\mathbb{R}^n$ for every $x \in \mathbb{R}^n$.   For example, since $M(x) = \sigma(x) \sigma(x)^T$, the diffusion matrix can always be put in the form of \eqref{eq:assumption_on_M} with $s=n$, unit weights $w_i(x) = 1$, and $\eta_i(x) = \sigma_i(x)$ for $i=1,\dotsc,n$, or equivalently, \[
M(x) = \sum_{i=1}^n \sigma_i(x) \sigma_i(x)^T \;.
\] 
The following generator generalizes \eqref{eq:Qc} to diffusion matrices of the form \eqref{eq:assumption_on_M}.

\medskip

\begin{equation} \label{eq:generalizedQc}
\boxed{
\begin{aligned}
 \mathcal{Q}_c f(x) =
 \sum_{i=1}^n
\frac{w_i(x)}{h_i^+(x) h_i(x)}\exp\left(\quad \frac{h_i^+(x)}{2} \tilde \mu(x)^T \eta_i(x)  \right) &\left( f\left( x+ h_i^+(x) \eta_i(x)  \right) - f(x) \right)  \\
  \qquad\qquad +
 \frac{w_i(x)}{h_i^-(x) h_i(x)}\exp\left(-  \frac{h_i^-(x)}{2} \tilde \mu(x)^T \eta_i(x) \right) & \left( f\left( x- h_i^-(x)  \eta_i(x) \right) - f(x) \right)  \;.
\end{aligned}
}
\end{equation}

\medskip

\noindent
This generator is a weighted, central finite difference approximation along the directions given by the vectors $\{\eta_i(x)\}_{i=1}^s$. To be sure $\tilde \mu$ in \eqref{eq:generalizedQc} is the transformed drift field introduced earlier in \eqref{eq:transformed_drift}.  The next proposition states that this generalization of $Q_c$ meets the requirements (Q1) and (Q2).    

\begin{prop} \label{prop:generalized_Qc_accuracy}
Suppose that the diffusion matrix satisfies \eqref{eq:assumption_on_M}.  Assuming that: \[
|h_i^+(x) - h_i^-(x)| \vee (h_i^+(x))^2 \vee (h_i^-(x))^2 \le h^2 \quad \text{for all $x \in \Omega$}
\]
then the operator $\mathcal{Q}_c$ in \eqref{eq:generalizedQc} satisfies (Q2) and (Q1) with $p=2$.  
\end{prop}

\begin{proof} The operator $\mathcal{Q}_c$ satisfies the realizability condition (Q2) since $w_i(x) \ge 0$ by hypothesis on the diffusion field $M(x)$.  
To check accuracy, Taylor expand the forward/backward differences appearing in $\mathcal{Q}_cf(x)$ to obtain: \begin{gather*}
  \mathcal{Q}_c f(x) = \sum_{i=1}^s \frac{w_i}{h_i}  Df^T  \eta_i  \left( \exp\left( \tilde \mu^T \eta_i \frac{h_i^+}{2} \right)  - \exp\left( - \tilde \mu^T \eta_i \frac{h_i^-}{2} \right)  \right)  \\
 + \frac{w_i}{2 h_i} D^2 f  ( \eta_i,  \eta_i)   \left(h_i^+ \exp\left( \tilde \mu^T \eta_i \frac{h_i^+}{2} \right) + h_i^- \exp\left( - \tilde \mu^T \eta_i \frac{h_i^-}{2} \right)  \right)  \\
+ \frac{w_i}{6 h_i} D^3 f  ( \eta_i,  \eta_i, \eta_i)   \left( (h_i^+)^2 \exp\left( \tilde \mu^T \eta_i \frac{h_i^+}{2} \right) - (h_i^-)^2 \exp\left( - \tilde \mu^T \eta_i \frac{h_i^-}{2} \right)  \right) + O(h^2) \;.
\end{gather*}
A similar Taylor expansion of the exponentials yields, \begin{align*}
  \mathcal{Q}_c f(x) &= \sum_{i=1}^s w_i Df^T  \eta_i  \tilde \mu^T \eta_i + w_i D^2 f  ( \eta_i,  \eta_i) +O(h^2) \\
&=   \tr\left( ( Df \tilde \mu^T + D^2 f ) \left( \underbrace{\sum_{i=1}^s w_i \eta_i \eta_i^T}_{=M(x)} \right) \right)  +O(h^2)  \\
&= \tr\left( Df(x) \; \mu(x)^T + D^2 f(x) \; M(x) \right)  +O(h^2) \\
&= Lf(x) + O(h^2) 
 \end{align*}
where we used \eqref{eq:assumption_on_M}. Thus, $\mathcal{Q}_c$ is second-order accurate as claimed. 
\end{proof}

As an application of Proposition~\ref{prop:generalized_Qc_accuracy} we next show  that realizable discretizations on gridded state spaces are possible to construct if the diffusion matrix of the SDE is diagonally dominant.

%
%

\section{Weakly Diagonally Dominant Case} \label{sec:diagonally_dominant_case}

 Here we extend the upwinded finite-difference approximation $\tilde Q_c$ in \eqref{eq:tQc_2d} from two to many dimensional SDEs assuming that the diffusion matrix is weakly diagonally dominant.   We prove that this extension is realizable and second-order accurate under this assumption.  Let us quickly recall what it means for a matrix to be weakly diagonally dominant.

\begin{defn}
An $n \times n$ matrix $A$ is {\em weakly diagonally dominant} if there exists an $n$-vector $y = (y_1, \dotsc, y_n)^T$ with positive entries such that \begin{equation} \label{eq:weakly_diagonally_dominant}
|A_{ii}| y_i \ge \sum_{j \ne i} |A_{ij}| y_j 
\end{equation} for all $i \in \{1,\dotsc,n \}$.    
\end{defn}

The following lemma is straightforward to derive from this definition.

\begin{lemma} \label{lem:pap}
Let $A$ be a square matrix satisfying \eqref{eq:weakly_diagonally_dominant}.  Then there exists a positive diagonal matrix $P$ such that the square matrix $\tilde A = PAP$  is diagonally dominant. 
\end{lemma}

Let $\Delta(x) = (\Delta_1(x), \dotsc, \Delta_n(x))^T$ be a field of spatial step sizes, which permit a different spatial step size for each coordinate axis.  
Let $\tilde \mu$ be the transformed drift field defined in \eqref{eq:transformed_drift}.  Here is a multi-dimensional generalization of $\tilde Q_c$
\begin{equation}  \label{eq:tQc}
\tilde{\mathcal{Q}}_c f(x) = \lambda(x) \;  \Ex ( f(x+ \xi_x) - f(x) ) 
\end{equation}
where the expectation is over the random vector $\xi_x$ which has a discrete probability distribution given by: \begin{equation}
\label{eq:xi_x}
\begin{aligned} 
 \xi_x \sim   \lambda(x)^{-1} & \left( \vphantom{\sum_{\substack{
            1\le i \le n\\
             j>i}}}  \sum_{i=1}^n \;  \delta( y_i^+) \exp\left( \frac{1}{2} \tilde \mu^T (y_i^+ - x) \right)  \; \left( \frac{M_{ii}}{\Delta_i^2} - \sum_{j \ne i} \frac{|M_{ij}|}{\Delta_i \Delta_j}  \right) \right. \\
& + \vphantom{\sum_{\substack{
            1\le i \le n\\
             j>i}}}  \sum_{i=1}^n \;  \delta( y_i^-) \exp\left( \frac{1}{2} \tilde \mu^T (y_i^- - x) \right)  \; \left( \frac{M_{ii}}{\Delta_i^2} - \sum_{j \ne i} \frac{|M_{ij}|}{\Delta_i \Delta_j}  \right)  \\
&  +  \sum_{\substack{ 
            1\le i \le n\\
             j>i}} \delta( y_{ij}^{+,d}) \; \exp\left( \frac{1}{2} \tilde \mu^T (y_{ij}^{+,d} - x) \right)  \; \left( \frac{M_{ij} \vee 0}{\Delta_i \Delta_j} \right) \\
&  +  \sum_{\substack{ 
            1\le i \le n\\
             j>i}} \delta( y_{ij}^{-,d}) \; \exp\left( \frac{1}{2} \tilde \mu^T (y_{ij}^{-,d} - x) \right)  \; \left( \frac{M_{ij} \vee 0}{\Delta_i \Delta_j} \right) \\
&   -  \sum_{\substack{
            1\le i \le n\\
             j>i}} \delta( y_{ij}^{+,a} ) \; \exp\left( \frac{1}{2} \tilde \mu^T (y_{ij}^{+,a} - x) \right)  \; \left( \frac{M_{ij} \wedge 0}{\Delta_i \Delta_j} \right)   \\
&   - \left.  \sum_{\substack{
            1\le i \le n\\
             j>i}} \delta( y_{ij}^{-,a} ) \; \exp\left( \frac{1}{2} \tilde \mu^T (y_{ij}^{-,a} - x) \right)  \; \left( \frac{M_{ij} \wedge 0}{\Delta_i \Delta_j} \right)  \right)
\end{aligned}
\end{equation} 
where $\delta(y)$ is a point mass concentrated at $y$, $\lambda(x)$ is a normalization constant, $\lambda(x)^{-1}$ is the mean holding time, and we have introduced the following reaction channels:
\begin{equation} \label{eq:2nsqrd_channels}
\begin{cases}
y_i^{\pm} =x \pm e_i \Delta_i \;,  ~~  1 \le i \le n  \\
y_{ij}^{\pm,d} = x \pm e_i \Delta_i \pm e_j \Delta_j \;, ~~  1 \le i \le n \;, ~~ j > i \\
y_{ij}^{\pm,a} = x \pm e_i \Delta_i \mp e_j \Delta_j \;, ~~  1 \le i \le n \;, ~~ j> i 
\end{cases}
\end{equation}
The superscripts $d$ and $a$ in the reaction channels above refer to the diagonal and anti-diagonal directions in the $ij$-plane, respectively.  

The following proposition states that $\tilde{\mathcal{Q}}_c$ satisfies (Q1) and (Q2).

\begin{prop}
Consider the generator $\tilde{\mathcal{Q}}_c$ given in \eqref{eq:tQc}. If for all $x \in \mathbb{R}^n$, the Cholesky factorization of the diffusion matrix $M(x)$ has at most two nonzero entries in every column, then there exists a field $\Delta$ such that $\tilde{\mathcal{Q}}_c$ in \eqref{eq:tQc} is realizable and second-order accurate. 
\end{prop}

\begin{proof}
This proof is an application of Proposition~\ref{prop:generalized_Qc_accuracy}.   To apply this proposition we show that if the Cholesky factorization of $M(x)$ has at most two nonzero entries per column, then $M(x)$ can be put in the form of \eqref{eq:assumption_on_M}.   According to Theorem 9 of  \cite{BoChPaTo2005}, the Cholesky factorization of a symmetric matrix with positive diagonals has at most two nonzero entries in every column if and only if it is weakly diagonally dominant.  Since $M(x)$ is symmetric positive definite by hypothesis, and since the Cholesky factorization of $M(x)$ has at most two nonzero entries per column also by hypothesis, $M(x)$ is weakly diagonally dominant for all $x \in \mathbb{R}^n$.  Hence, by Lemma~\ref{lem:pap} there exists a positive diagonal matrix $P(x)$ such that: \begin{equation} \label{eq:PMP}
\tilde M(x) = P(x) M(x) P(x)
\end{equation} is diagonally dominant for each $x \in \mathbb{R}^n$.  Define a field of step sizes $\Delta$ by setting $\Delta_i = h/P_{ii}$.  Hence, we have that: \begin{align*}
 M =& \sum_{i=1}^n M_{ii} e_i e_i^T + \sum_{j>i} M_{ij} (e_i e_j^T + e_j e_i^T ) \\ 
= & \sum_{i=1}^n ( \tilde M_{ii} - \sum_{j \ne i} | \tilde M_{ij} | ) \frac{(y_i^+ - x)}{h} \frac{(y_i^+ - x)}{h}^T \\
  + &\sum_{\substack{
            1\le i \le n\\
             j>i}}   (  \tilde M_{ij} \vee 0) \; \frac{(y_{ij}^{+,d} - x)}{h} \frac{(y_{ij}^{+,d} - x)}{h}^T \\
 - &\sum_{\substack{
            1\le i \le n\\
             j>i}}  ( \tilde M_{ij} \wedge 0) \;  \frac{(y_{ij}^{+,a} - x)}{h} \frac{(y_{ij}^{+,a} - x)}{h}^T            
\end{align*}
where we used the forward direction channels defined in \eqref{eq:2nsqrd_channels}.  Note that the coefficient of each rank one matrix in this last expression is non-negative because $\tilde M(x)$ is diagonally dominant.   Thus, $M$ is in the form of \eqref{eq:assumption_on_M} with $s=n^2$ and note that $\tilde{\mathcal{Q}}_c$ is a special case of $\mathcal{Q}_c$ in \eqref{eq:generalizedQc} with $h_i^{\pm}(x) = h$.  The desired result follows from Proposition~\ref{prop:generalized_Qc_accuracy}.
\end{proof}

We have the following corollary.

\begin{cor}
If the $n \times n$ diffusion matrix $M(x)$ is diagonally dominant for all $x \in \mathbb{R}^n$, then $\tilde{\mathcal{Q}}_c$ in \eqref{eq:tQc} is realizable on a gridded state space in $\R^n$.
\end{cor}

As far as we can tell,  realizable, finite difference discretizations of this type seem to have been first developed by H.~Kushner \cite{Ku1976A, Ku1976B, Ku1977}.

\begin{proof}
In this case the diffusion matrix can be written as:  \begin{align*}
 M  =& \sum_{i=1}^n (  M_{ii} - \sum_{j \ne i} |  M_{ij} | ) e_i e_i^T \\
   + &\sum_{\substack{
            1\le i \le n\\
             j>i}}   (   M_{ij} \vee 0) \;  (e_i + e_j) (e_i + e_j)^T   \\
 - &\sum_{\substack{
            1\le i \le n\\
             j>i}}  (  M_{ij} \wedge 0) \;  (e_i - e_j) (e_i - e_j)^T           
\end{align*}
and the coefficient of each rank one matrix in this last expression is non-negative because $M(x)$ is diagonally dominant by hypothesis.  Thus, $M$ is in the form of \eqref{eq:assumption_on_M} with $s=n^2$ and note that $\tilde{\mathcal{Q}}_c$ is a special case of $\mathcal{Q}_c$ in \eqref{eq:generalizedQc} with $h_i^{\pm}(x) = h$.  Note also that the state space is gridded since the jumps are either to nearest neighbor grid points: $x \pm h e_i$ for all $1 \le i \le n$, or to next to nearest neighbor grid points: $x \pm h e_i \pm h e_j$ and $x \pm h e_i \mp h e_j$ for all $1 \le i \le n$ and $j>i$.   
\end{proof}

\chapter{Examples \& Applications} \label{chap:numerics}

\section{Introduction} \label{sec:intro_to_examples}

In this chapter we apply realizable discretizations to \begin{enumerate}
\item a cubic oscillator in 1D with additive noise as described in \S\ref{sec:cubic_oscillator};
\item a log-normal process in 1D with multiplicative noise as described in \S\ref{sec:lognormal_1d};
\item the Cox-Ingersoll-Ross process in 1D with multiplicative noise as described in \S\ref{sec:cir_process};
\item a non-symmetric Ornstein-Uhlenbeck process as described in \S\ref{sec:planar_ou};
\item the Maier-Stein SDE in 2D with additive noise as described in \S\ref{sec:doublewell};
\item simulation of a particle in a planar square-well potential well in 2D with additive noise as described in \S\ref{sec:squarewell};
\item simulation of a particle in a worm-like chain potential well in 2D with additive noise as described in \S\ref{sec:wlc};
\item a log-normal process in 2D with multiplicative noise as described in \S\ref{sec:lognormal_2d};
\item a Lotka-Volterra process in 2D with multiplicative noise as described in \S\ref{sec:lv_process}; and,
\item a Brownian dynamics simulation of the collapse of a colloidal cluster in 39D with and without hydrodynamic interactions as described in \S\ref{sec:colloidal_cluster}.
\end{enumerate}
\noindent
For the 1D SDE problems, we numerically test two realizable spatial discretizations with a gridded state space.  The first uses an upwinded (resp.~central) finite difference method to approximate the first (resp.~second) order derivative in \eqref{eq:generator}.
\begin{equation} \label{eq:Qu_1d}
\begin{dcases}
(\tilde Q_u)_{i,i+1} =  \frac{1 }{ \delta x^+_i }  \left(  ( \mu_i \vee 0) +  \frac{ M_i }{\; \delta x_i}    \right)   \\
(\tilde Q_u)_{i,i-1} =   \frac{1 }{ \delta x^-_i } \left( - ( \mu_i\wedge 0 ) +  \frac{ M_i}{\; \delta x_i}    \right)  
\end{dcases}
\end{equation}
The second generator uses a weighted central finite difference scheme to approximate the derivatives in \eqref{eq:generator}.
\begin{equation} \label{eq:Qc_1d}
\begin{dcases}
(\tilde Q_c)_{i,i+1} =  \frac{1 }{ \delta x_i \; \delta x^+_i } M_i \exp\left(  \frac{\mu_i}{M_i} \; \frac{\delta x_i^+}{2} \right)    \\
(\tilde Q_c)_{i,i-1} =   \frac{1 }{ \delta x_i \; \delta x^-_i }  M_i \exp\left( -\frac{\mu_i}{M_i} \; \frac{\delta x_i^-}{2} \right) 
\end{dcases}
\end{equation}
These discretizations are slight modifications to the discretizations given in \eqref{eq:tQu_1d} and \eqref{eq:tQc_1d}, respectively.  For the planar SDE problems, we apply a 2D version of \eqref{eq:Qc_1d} given in \eqref{eq:tQc_2d}.  When the noise is additive, we use an evenly spaced grid.   When the noise is multiplicative, we use adaptive mesh refinement as described in \S\ref{sec:amr_1d} in 1D and \S\ref{sec:amr_2d} in 2D.  For the 39D SDE problem, we apply the generator on a gridless state space given in \eqref{eq:Qc}.   

%
%

\section{Cubic Oscillator in 1D with Additive Noise} \label{sec:cubic_oscillator}

This SDE problem is a concrete example of exploding numerical trajectories.  Consider \eqref{eq:sde} with $\Omega=\mathbb{R}$, $\mu(x) = -x^3$, and $\sigma(x) = 1$ i.e.~\begin{equation} \label{eq:co_sde}
dY=-Y^3 dt + \sqrt{2} dW \;, \quad Y(0) \in \mathbb{R} \;.
\end{equation}
The solution to \eqref{eq:co_sde} is geometrically ergodic with respect to a stationary distribution with density \begin{equation} \label{eq:co_nu}
\nu(x) =  Z^{-1} \exp( - x^4/4 ) \;, \quad \text{where $Z = \int_{\mathbb{R}} \exp(-x^4/4) dx$} \;.
\end{equation}
However, explicit methods are transient in this example since the drift entering \eqref{eq:co_sde} is only locally Lipschitz continuous. More concretely, let $\{ \tilde X_n \}$ denote a (discrete-time) Markov chain produced by forward Euler with time step size $h_t$.  Then, for any $h_t>0$ this Markov chain satisfies \[
\Ex_x \tilde X^2 _{ \lfloor t/h_t \rfloor} \to \infty \quad \text{as $t \to \infty$} \;,
\]
where $\Ex_x$ denotes expectation conditional on $\tilde X_0=x$.  (See Lemma 6.3 of \cite{HiMaSt2002} for a simple proof, and \cite{HuJeKl2012} for a generalization.)  Metropolizing explicit integrators, like forward Euler, mitigates this divergence.  Indeed, the (discrete-time) Markov chain produced by a Metropolis integrator $\{ \bar X_n \}$ is ergodic with respect to the stationary distribution of the SDE, and hence, \[
\Ex_x \bar X^2_{\lfloor t/h_t \rfloor} \to \int_{\mathbb{R}} x^2 \nu(x) dx \quad \text{as $t \to \infty$} \;.
\]
However, the rate of convergence of this chain to equilibrium is not geometric, even though the SDE solution is geometrically ergodic \cite{RoTw1996B,BoVa2010}.  The severity of this lack of geometric ergodicity was recently studied in \cite{BoHa2013}, where it was shown that the deviation from geometric ergodicity is exponentially small in the time step size.  This result, however, does not prevent the chain from ``getting stuck'' (rarely accepting proposal moves) in the tails of the stationary density in \eqref{eq:co_nu}.  

In this context we test the generators $\tilde Q_u$ in \eqref{eq:Qu_1d} and $\tilde Q_c$ in \eqref{eq:Qc_1d} on an infinite, evenly spaced grid on $\mathbb{R}$.  Figure~\ref{fig:cubic_oscillator_geometric_ergodicity} plots sample paths produced by the SSA induced by these generators.  The initial condition is large and positive.  For comparison, we plot a sample path (in light grey) of a Metropolis integrator with the same initial condition.  At this initial condition, the Metropolis integrator rarely accepts proposal moves and gets stuck for the duration of the simulation.  In contrast, the figure shows that the proposed approximations do not get stuck, which is consistent with Theorem~\ref{thm:geometric_ergodicity_Qc}. The inset in the figure illustrates that the time lag between jumps in the SSA method is small (resp.~moderate) at the tails (resp.~middle) of the stationary density.  This numerical result manifests that the mean holding time adapts to the size of the drift.  Note that the time lag for $\tilde Q_c$ is smaller than for $\tilde Q_u$.  In the next section, this difference in time lags is theoretically accounted for by an asymptotic analysis of the mean holding time.    Figure~\ref{fig:cubic_oscillator_accuracy} provides evidence that the approximations are accurate in representing the stationary distribution, mean first passage time, and committor function.  To produce this figure, we use results described in Chapter~\ref{chap:tridiagonal}.  In particular, in the scalar case, the numerical stationary density, mean first passage time and committor satisfy a three-term recurrence relation, which can be exactly solved as detailed in Sections~\ref{sec:nu_1D}, \ref{sec:mfpt_1D} and \ref{sec:committor_1D}, respectively.  We also checked that these solutions agree -- up to statistical error -- with empirical estimates obtained by an SSA simulation.  

\begin{rem}
It is known that the rate of convergence (and hence accuracy) of the solution of standard time integrators to first exit problems drops because they fail to account for the probability that the system exits in between each discrete time step~\cite{gobet2004exact,gobet2007discrete,gobet2010stopped}.   Building this capability into time integrators may require solving a boundary value problem per step, which is prohibitive to do in many dimensions \cite{Ma1999}.  By keeping time continuous, and provided that the grid is adjusted to fit the boundaries of the region of the first exit problem, the proposed approximations take these probabilities into account.  As a consequence, in the middle panel of Figure~\ref{fig:cubic_oscillator_accuracy}  we observe that the order of accuracy of the approximations with respect to first exit times is identical to the accuracy of the generator, that is, there is no drop of accuracy as happens with time discretizations.  To be precise, this figure illustrates the accuracy in these approximations with respect to the mean first passage time of the cubic oscillator to $(0,2)^c$ and shows that the generators $\tilde Q_u$ and $\tilde Q_c$ are $O(\delta x)$ and $O(\delta x^2)$, respectively.  
\end{rem}


\begin{figure}[ht!]
\begin{center}
\includegraphics[width=0.8\textwidth]{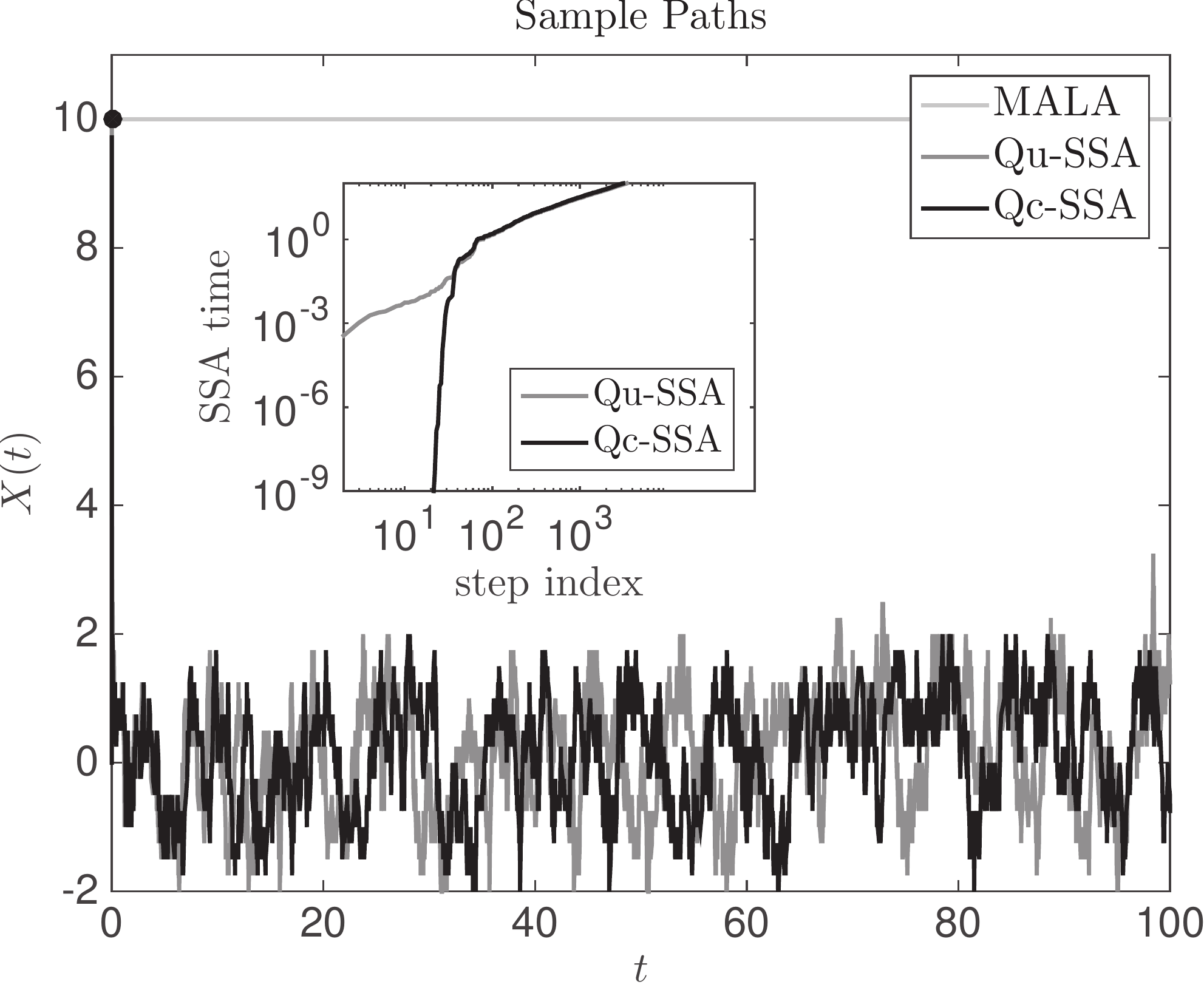}  
\end{center}
\caption{ \small {\bf Cubic Oscillator.}    
                                                The main figure plots sample paths produced by a Metropolis integrator (MALA) and the SSA 
                                                induced by $\tilde Q_u$ in \eqref{eq:tQu_1d} and $\tilde Q_c$ in \eqref{eq:tQc_1d} 
                                                over the time interval $[0,100]$.
                                                The temporal step size of MALA and the spatial step size of the SSA integrators are 
                                                set at $\delta t =0.015$ and $\delta x = 0.25$, respectively.  
                                                The mean holding time for the SSA integrators operated at this spatial step size is approximately $\langle \delta t \rangle = 0.03$.
                                                At the initial condition marked by the black dot at $t=0$, the true dynamics is dominated by the deterministic drift, 
                                                 which causes the true solution to drop toward the origin.   Notice that MALA gets stuck at high energy, 
                                                 which illustrates that it is not geometrically ergodic.
                                                In contrast, SSA is geometrically ergodic (on an infinite grid), and hence, does not get stuck.
                                                The inset shows that the SSA time adapts to the size of the drift: when the drift is large (resp.~small) the lag times are
                                                smaller (resp.~larger).  
}
\label{fig:cubic_oscillator_geometric_ergodicity}
\end{figure}

\begin{figure}[ht!]
\begin{center}
\includegraphics[width=0.52\textwidth]{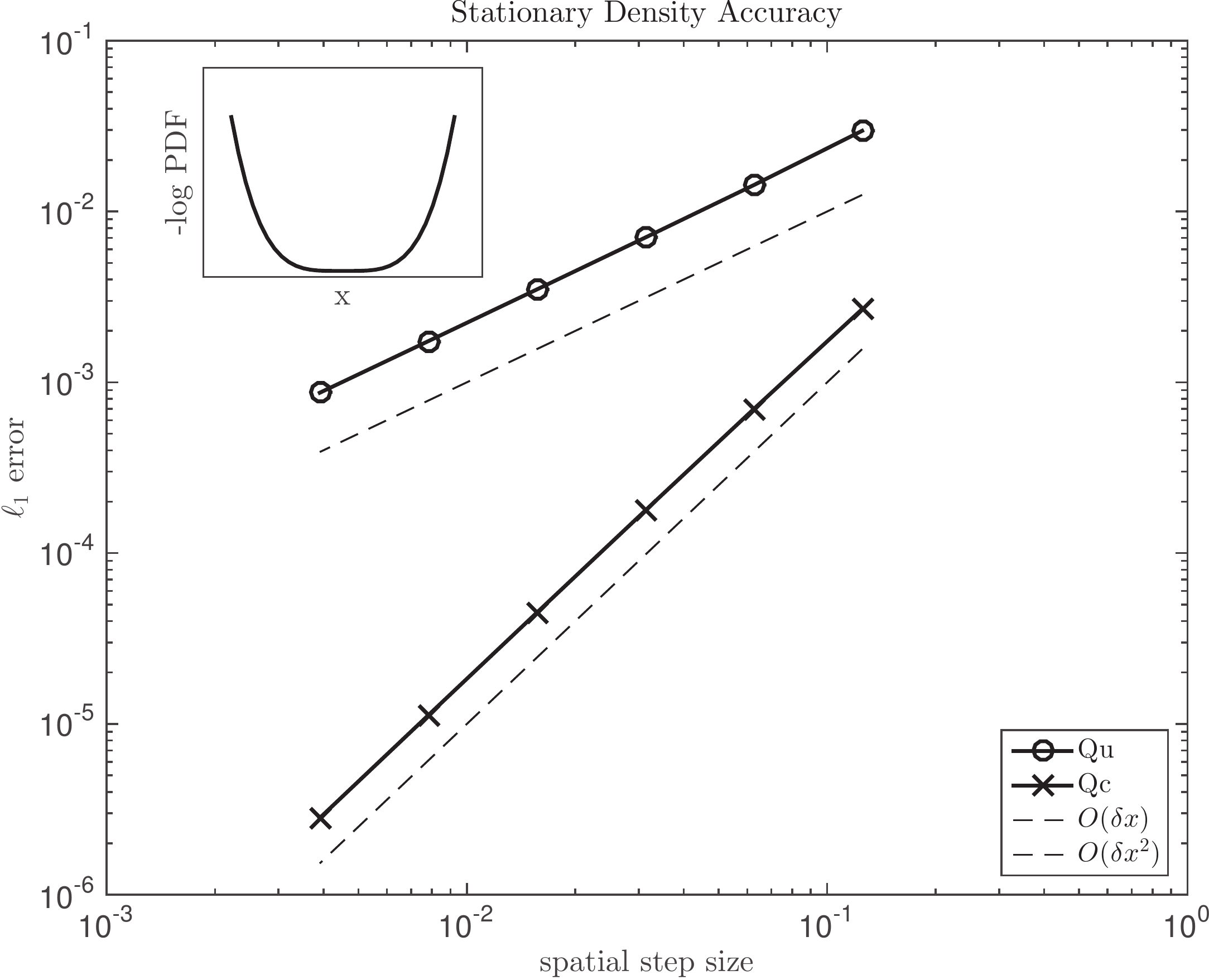}  \\
  \vspace{1em}
\includegraphics[width=0.52\textwidth]{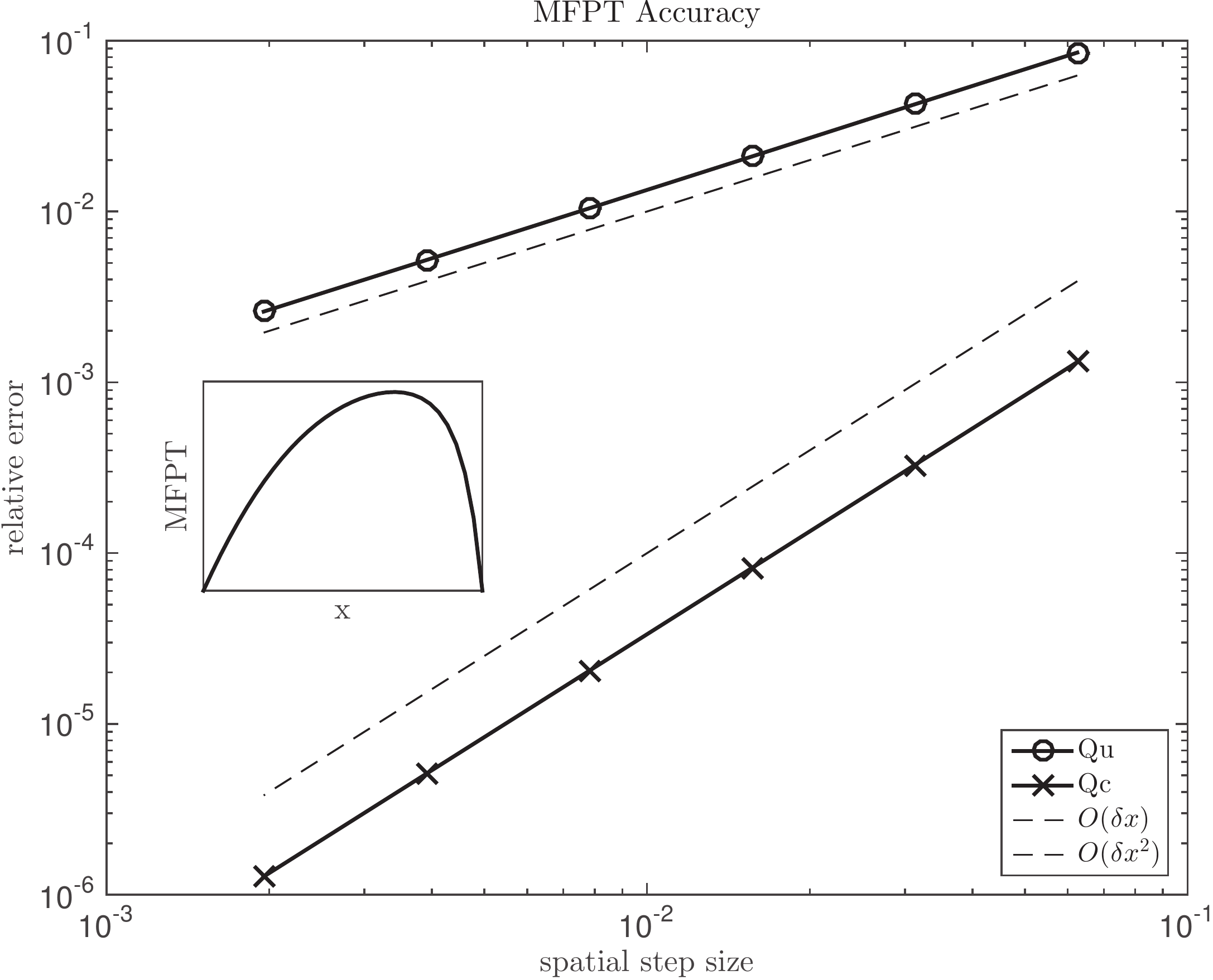} \\
  \vspace{1em}
\includegraphics[width=0.52\textwidth]{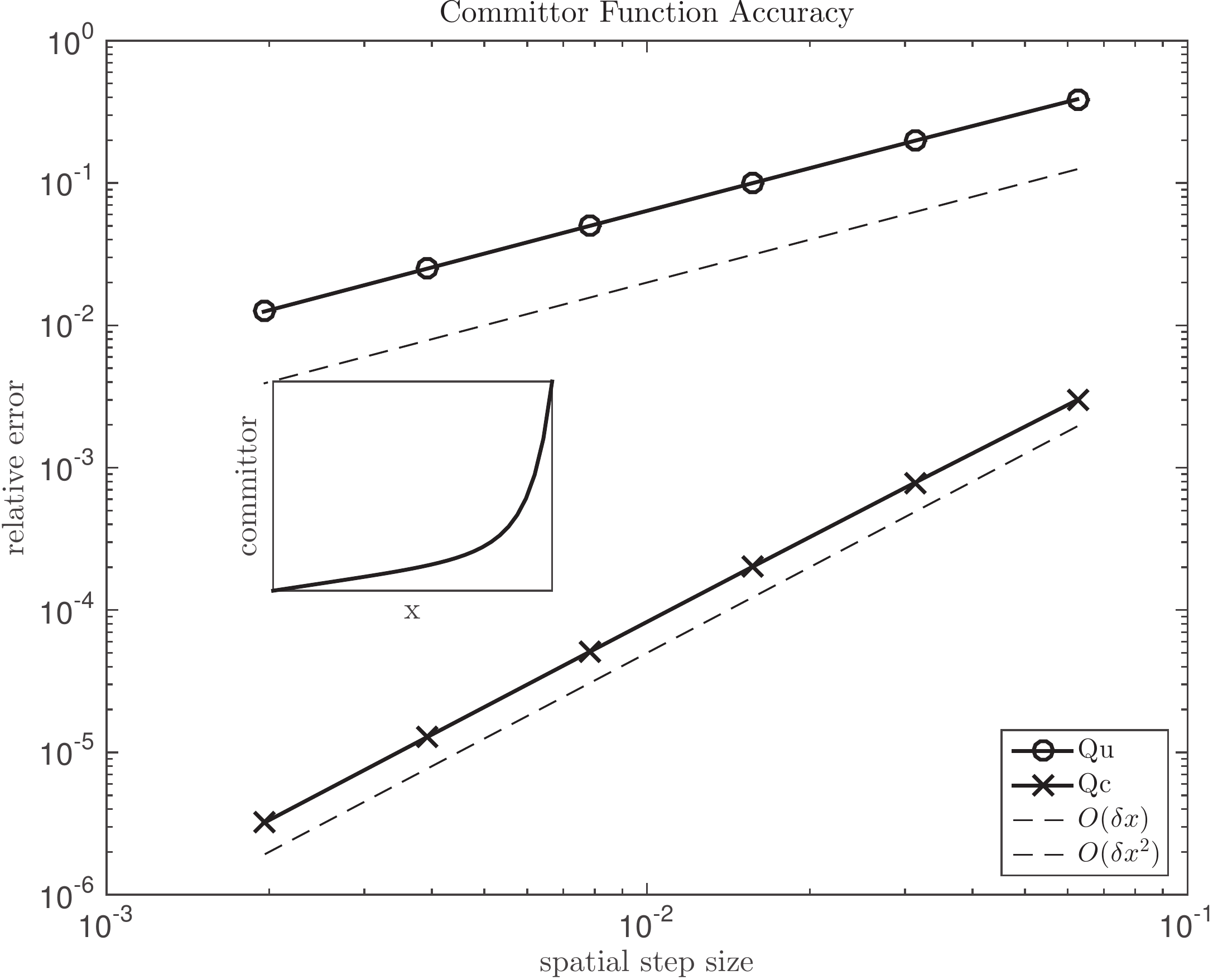}  
\end{center}
\caption{ \small {\bf Cubic Oscillator.}    
	From top: accuracy of the generators $\tilde Q_u$ in \eqref{eq:tQu_1d} and $\tilde Q_c$ in \eqref{eq:tQc_1d}  
	with respect to the stationary density, mean first passage time, and committor function of the SDE.  A formula for the true solutions
	is explicitly available in this example and serves as a benchmark for comparison.  For the numerical solution, we use the formulas given in 
	Sections~\ref{sec:nu_1D}, \ref{sec:mfpt_1D} and \ref{sec:committor_1D}.    In all cases we see that $\tilde Q_u$ (resp.~$\tilde Q_c$)
	is first (resp.~second) order accurate in representing these quantities.  The reference solutions are graphed in the insets.
}
\label{fig:cubic_oscillator_accuracy}
\end{figure}


%
%

\section{Asymptotic Analysis of Mean Holding Time} \label{sec:downhill_time}

Consider the following instance of \eqref{eq:sde} \begin{equation} \label{eq:scale_sde}
d Y = \mu(Y) dt + \sqrt{2} dW \;, \quad Y(0) \in \mathbb{R} \;.
\end{equation}   Assume the drift entering \eqref{eq:scale_sde} is differentiable and satisfies a dissipativity condition, like \[
\sign(x) \mu(x) \to -\infty  \qquad \text{as $|x| \to \infty$} \;.
\] Hence, for sufficiently large $|x|$ the SDE dynamics is dominated by the drift, and asymptotically, the time it takes for the SDE solution $Y(t)$ to move a fixed distance in space can be derived from analyzing the ODE \begin{equation} \label{eq:scalar_ode}
\dot Y = \mu( Y) \;, \quad Y(0) = x \;.
\end{equation}   Equation~\eqref{eq:scalar_ode} implies that the time lapse between $Y(0)=x_i$ and $Y(t^*) = x_{i-1}$ where $x_i > x_{i-1} \gg 0$ satisfies\begin{equation} \label{time_lapse}
t^e = \int_{x_{i-1}}^{x_i} \frac{d x}{ | \mu(x) | } \;.
\end{equation}
By using integration by parts and the mean value theorem, observe that \eqref{time_lapse} can be written as: \begin{align}
t^e &=   \frac{\delta x}{| \mu(x_i) |} - \int_{x_{i-1}}^{x_i} (x-x_{i-1}) \mu(x)^{-2} \mu'(x) dx \nonumber \\
&=t^* - \frac{1}{2} \frac{\mu'(\xi)}{\mu(\xi)^2 } \delta x^2   \label{te_asymptote}
\end{align}
for some $\xi \in [x_{i-1}, x_i]$, and where we have introduced: $t^* = \delta x / |\mu_i|$.   Let us compare $t^e$ to the mean holding times predicted by the generators: $\tilde Q_u$ in \eqref{eq:Qu_1d} and $\tilde Q_c$ in \eqref{eq:Qc_1d}. For clarity, we assume that $\delta x = \delta x_i^+ = \delta x_i^-$.  

By hypothesis, $\mu_i = \mu(x_i)$ is less than zero if $x_i$ is large and positive, and hence, from \eqref{eq:Qu_1d} the mean holding time of $\tilde Q_u$ can be written as: \begin{align*}
t^u = ( (\tilde Q_u)_{i,i+1} + (\tilde Q_u)_{i,i-1} ) ^{-1} = \frac{\delta x^2 }{ 2 + |\mu_i| \delta x } 
\end{align*} This expression can be rewritten as \begin{equation} \label{tu_asymptote}
t^u = t^* - t^* \frac{ 2 }{ 2 + |\mu_i| \delta x}\;.
\end{equation}
 From \eqref{tu_asymptote}, we see that $t^u$ approaches $t^*$, as the next Proposition states. 

\begin{prop}
For any $\delta x>0$, the mean holding time of $\tilde Q_u$ satisfies: \[
\frac{|t^u - t^*|}{t^*} \to 0  \quad \text{as $|x_i| \to \infty$} \;. 
\] 
\end{prop}

\noindent
Likewise, if the second term in \eqref{te_asymptote} decays faster than the first term, then the relative error between $t^e$ and $t^*$ also tends to zero, and thus, the estimate predicted by $\tilde Q_u$ for the mean holding time asymptotically agrees with the exact mean holding time. This convergence happens if, e.g., the leading order term in $\mu(x)$ is of the form $-a \; x^{2p+1}$ for $p\ge0$ and $a>0$. 

Repeating these steps for the mean holding time predicted by $\tilde Q_c$ yields: \begin{equation} \label{tbar_asymptote}
t^c = ( (\tilde Q_{c})_{i,i+1} + (\tilde Q_{c})_{i,i-1} ) ^{-1} =  \frac{\delta x^2}{2} \sech(\mu_i \delta x)  \;.
\end{equation}
It follows from this expression that even though $\tilde Q_c$ is a second-order accurate approximation to $L$, it does not capture the right asymptotic mean holding time, as the next Proposition states.

\begin{prop}
For any $\delta x>0$, the mean holding time of $\tilde Q_c$ satisfies: \[
\frac{|t^c - t^*|}{t^*} \to 1  \quad \text{as $|x_i| \to \infty$} \;. 
\] 
\end{prop}

\noindent
Simply put, the mean holding time of $\tilde Q_c$ converges to zero too fast.  This analysis is confirmed in Figure~\ref{fig:cubic_oscillator_asymptotic_time_rate}, which shows that the mean holding time for $\tilde Q_u$ agrees with the mean holding time of the generator $Q_e$ that uses the exact mean holding time.  This generator was introduced in \S\ref{sec:scalar_miletoning}, see \eqref{eq:Qe_1D} for its definition.  This numerical result agrees with the preceding asymptotic analysis.  In contrast, note that the mean holding time of $Q_c$ is asymptotically an underestimate of the exact mean holding time, as predicted by \eqref{tbar_asymptote}.


\begin{figure}[ht!]
\begin{center}
\includegraphics[width=0.8\textwidth]{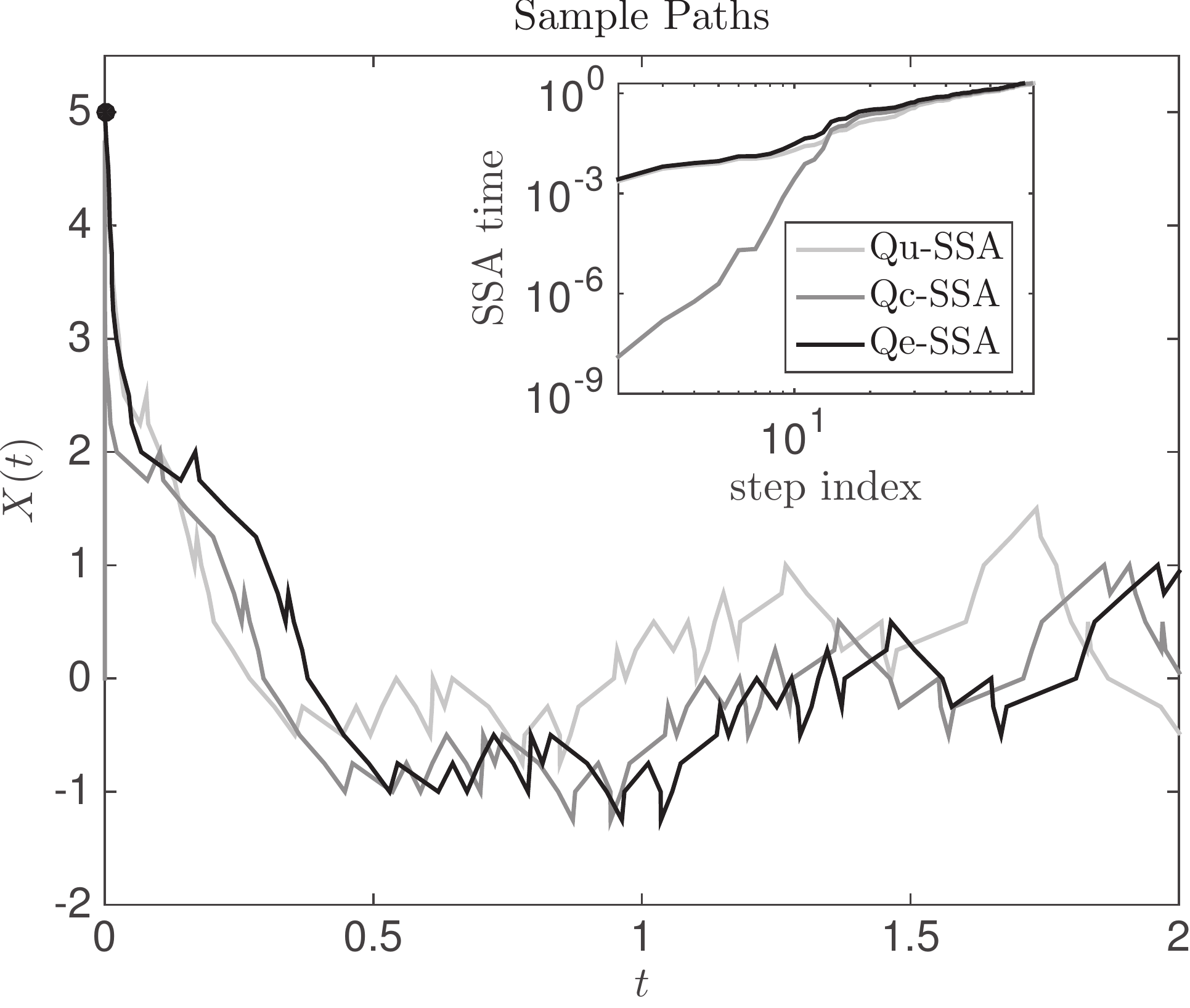}  
\end{center}
\caption{ \small {\bf Cubic Oscillator Simulation.}    
                                                This figure plots paths produced by the SSA induced by $\tilde Q_u$ (light grey), $\tilde Q_c$ (grey) and $Q_e$ (black).  
                                                In all cases the meshes are equidistant with $\delta x = 0.25$ and the initial condition is marked by the black dot at $t=0$.  
                                                The paths are driven by the same sequence of uniform random variables.   
                                                 We stress that the SSA produced by $Q_e$ uses the exact mean holding time of the SDE solution, which far enough 
                                                 from the origin is well approximated by \eqref{te_asymptote}. 
                                                The figure in the inset shows how the SSA time evolves with the SSA step index.  
                                                Far away from the origin, the dynamics is dominated by the deterministic drift, which causes the walkers to move toward the origin and 
                                                then become localized around the origin.  The SSA associated to $\tilde Q_c$ moves to the origin too fast because 
                                                its mean holding time tends to zero too fast, as \eqref{tbar_asymptote} predicts.   
                                                In contrast, the mean holding times for $Q_u$ and $Q_e$ are in close agreement far from the origin
                                                (where the asymptotic analysis in \S\ref{sec:downhill_time} is valid), but closer to the origin, this agreement is not apparent, 
                                                as expected.
}
\label{fig:cubic_oscillator_asymptotic_time_rate}
\end{figure}


%
%

\section{Adaptive Mesh Refinement in 1D} \label{sec:amr_1d}

The subsequent scalar SDE problems have a domain $\Omega = \mathbb{R}^+$, and may have a singularity at the origin for certain parameter values.  To efficiently resolve this singularity, we will use adaptive mesh refinement, and specifically, a variable step size infinite grid $S = \{ x_i \} \in \mathbb{R}^+$ that contains more points near the singularity at $x=0$.  We will construct this grid by mapping the grid $S$ to logspace to obtain a transformed grid $\hat S$ on $\mathbb{R}$ defined as \[
\hat S = \{ \xi_i \} \;, \quad \xi_i = \log(x_i)\;, \quad \forall~x_i \in S \;.
\]  Note that as $\xi_i \to -\infty$ (resp.~$\xi_i \to \infty$) we have that $x_i \to 0$ (resp.~$x_i \to \infty$).  For the transformed grid in logspace, we assume that the distance between neighboring grid points is fixed and given by $\delta \xi$, i.e., \[
\delta \xi = \xi_{i+1} - \xi_i \quad \text{for all $i \in \mathbb{Z}$} \;.
\]  
Since \[
x_{i+1} = \exp(\xi_{i+1}) = \exp(\delta \xi) \exp(\xi_i) =  \exp( \delta \xi) x_i
\] it follows from the definitions introduced in \eqref{eq:amr} that: \begin{equation} \label{eq:amr_exp}
\delta x_i^+ = (\exp(\delta \xi) - 1) x_i \;, \quad \delta x_i^- = (1- \exp(-\delta \xi) ) x_i \;, \quad \delta x_i = \sinh(\delta \xi) x_i \;. 
\end{equation}

%
%

\section{Log-normal Process in 1D with Multiplicative Noise} \label{sec:lognormal_1d}

Consider \eqref{eq:sde} with  \[
\Omega=\mathbb{R}^+ \;, \quad \mu(x) = - x \log x + x \;, \quad M(x) = (\sigma(x))^2 = x^2  \;,
\]
and initial condition $Y(0) \in \Omega$.  This process has a lognormal stationary distribution with probability density: \begin{equation} \label{eq:ln_nu_1d}
\nu(x) = \frac{1}{\sqrt{2 \pi}} \exp\left( - \frac{1}{2} (\log(x))^2 - \log(x) \right) \;.
\end{equation}
In fact,  $\log(Y(t))$ satisfies an Ornstein-Uhlenbeck equation with initial condition $\log(Y(0))$.  Exponentiating the solution to this Ornstein-Uhlenbeck equation yields \begin{equation*}
Y(t) = \exp\left( e^{-t} \log{Y(0)} + \sqrt{2} \int_0^t e^{-(t-s)} dW(s) \right) \;.
\end{equation*}
It follows that: \begin{equation} \label{eq:ln_moment_1d}
\Ex_x ( Y(t)^2 ) = \exp\left( 2 e^{-t} \log(x) + 2 (1 - e^{-2 t }) \right) \;.
\end{equation}
The formulas \eqref{eq:ln_nu_1d} and \eqref{eq:ln_moment_1d} are useful to numerically validate the approximations.

Figure~\ref{fig:ln_bvp} show numerical results of applying the generators in  \eqref{eq:Qu_1d} and \eqref{eq:Qc_1d} to solve the boundary value problems associated to the mean first passage time and exit probability (or committor function) on the interval $[1/2,5]$ and using an evenly spaced mesh.   Figure~\ref{fig:ln_nu_accuracy} illustrates that these generators are accurate with respect to the stationary density \eqref{eq:ln_nu_1d} using the adaptive mesh refinement described  in \S\ref{sec:amr_1d}.   Figure~\ref{fig:ln_ssa} plots sample paths produced by SSA for these generators and the SSA time as a function of the step index.  Figure~\ref{fig:ln_weak_accuracy} illustrates the accuracy of the proposed approximations with respect to the mean-squared displacement $\Ex_x(Y(1)^2)$ with an initial condition of $Y(0)=x=2$.    The relative error is plotted as both a function of the spatial step size (top panel) and the mean lag time (bottom panel).  Finally, Figure~\ref{fig:ln_complexity} plots the average number of computational steps (top panel) and the mean lag time (bottom panel) as a function of the spatial step size.  The statistics are obtained by averaging this data over $100$ SSA realizations with initial condition $X(0)=1$ and time interval of simulation of $[0,10]$.

%
\begin{figure}[ht!]
\begin{center} 
\includegraphics[width=0.65\textwidth]{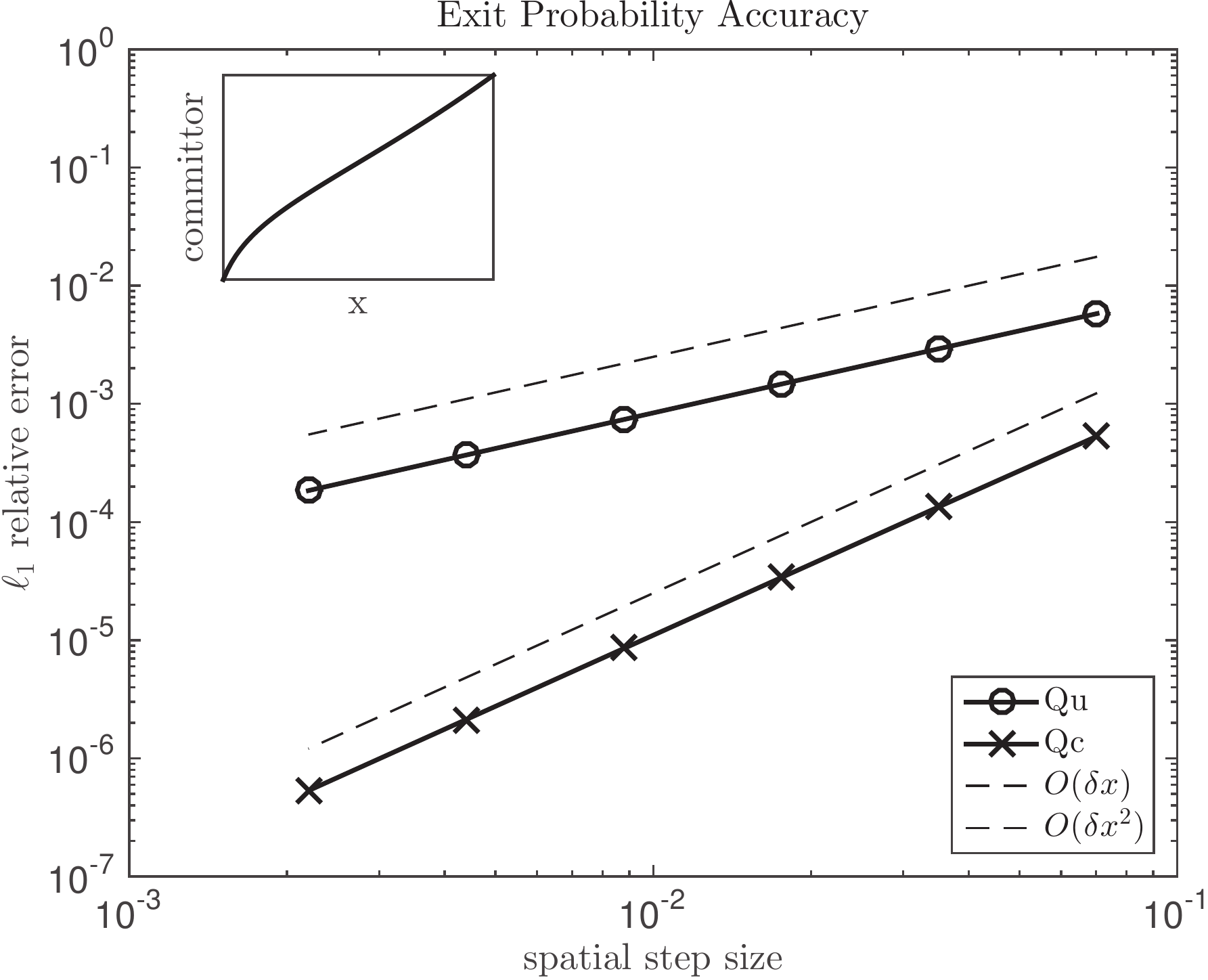}  \\
\vspace{0.25in}
\includegraphics[width=0.65\textwidth]{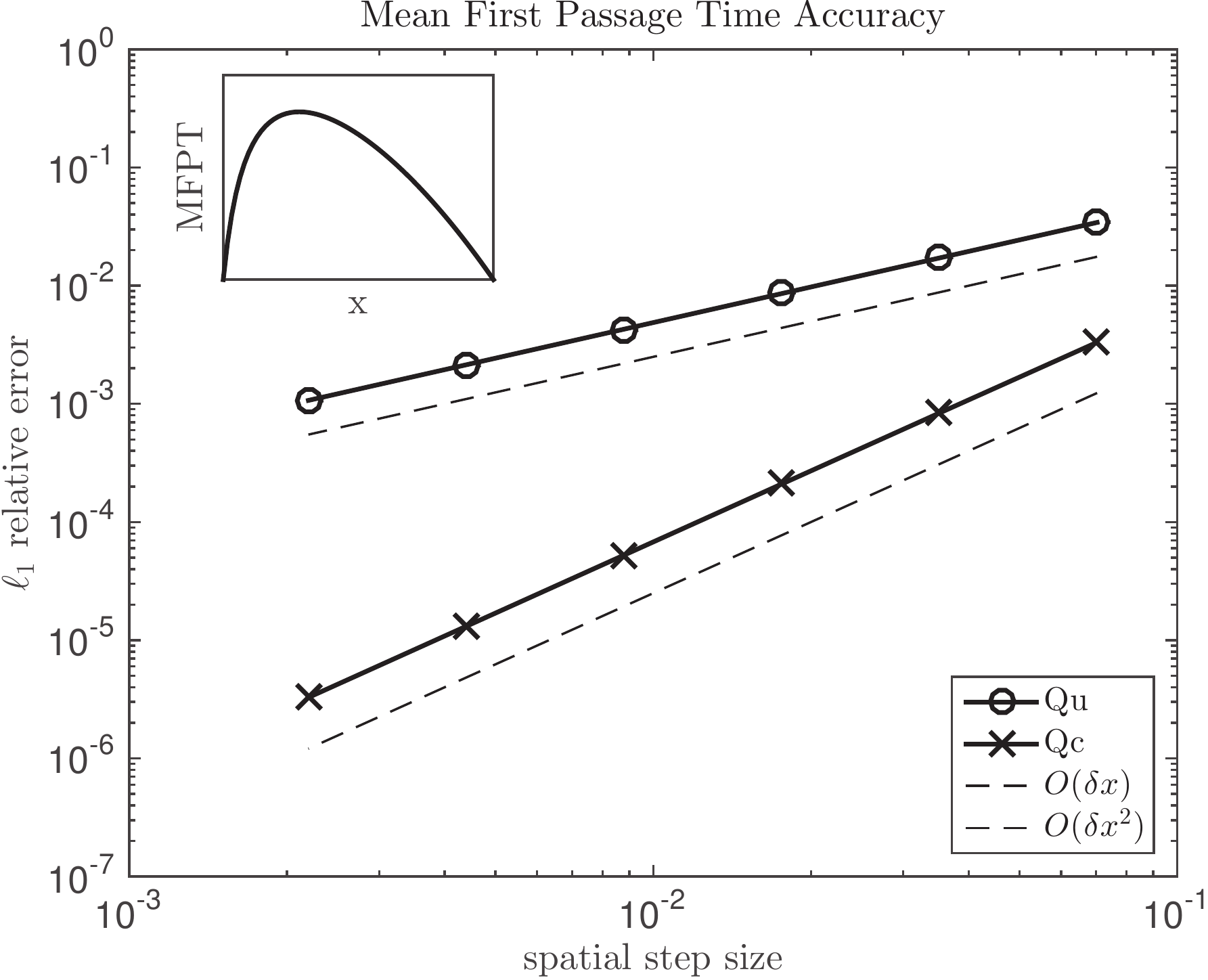} 
\caption{ \small {\bf Log-normal Process in 1D.}   This figure confirms that the realizable discretizations given in \eqref{eq:Qu_1d} and \eqref{eq:Qc_1d} are accurate with respect to the mean first passage time (bottom panel) and exit probability or committor function (top panel) computed on the interval $[1/2,5]$.  The reference solution is shown in the inset.  An evenly spaced mesh is used to produce this numerical result, with spatial step sizes as indicated in the figures.   The figure confirms that $\tilde Q_u$ (resp.~$\tilde Q_c$) is first  (resp.~second) order accurate.  
} \label{fig:ln_bvp} \end{center}
\end{figure}

\begin{figure}[ht!]
\begin{center} 
\includegraphics[width=0.65\textwidth]{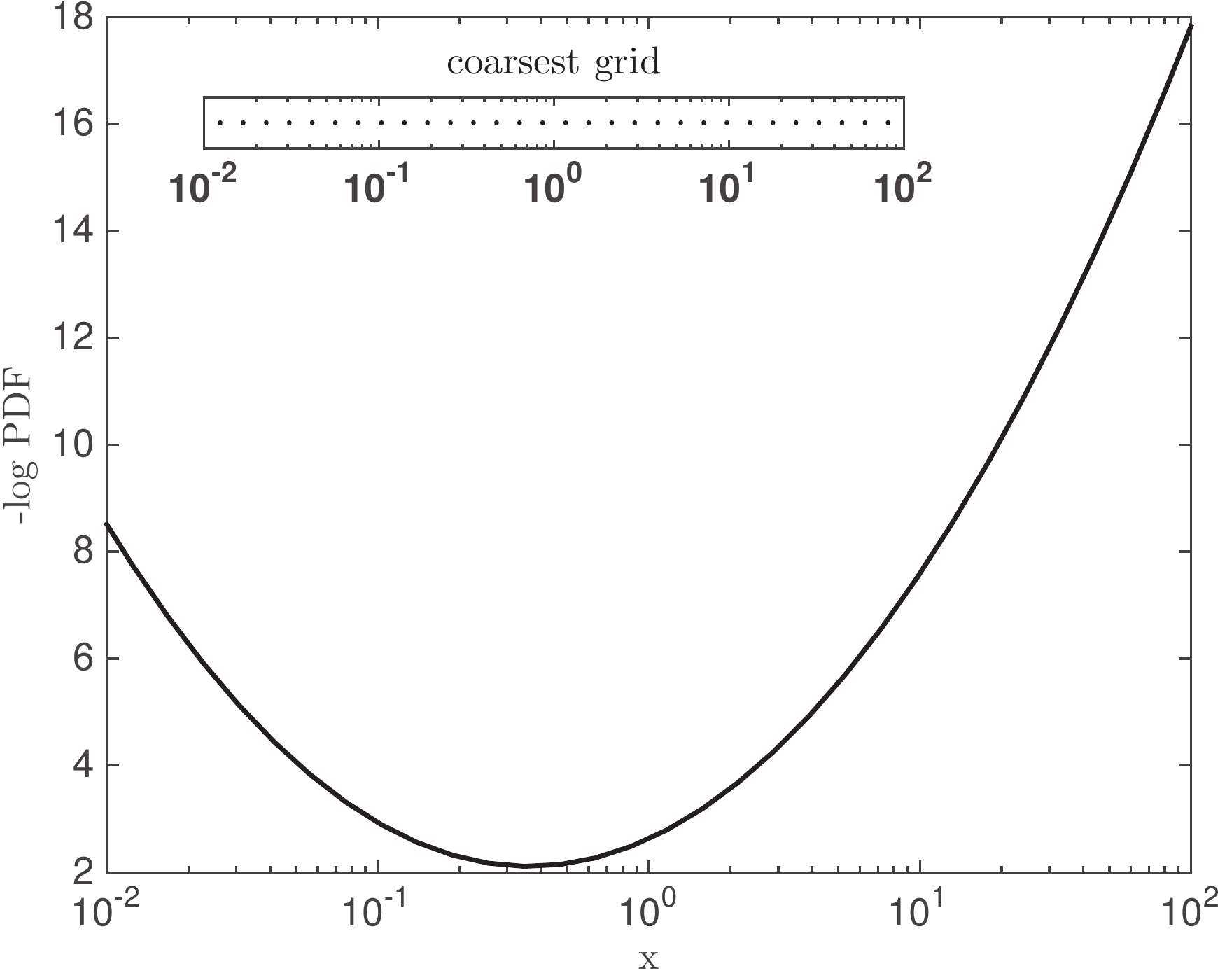} \\
\vspace{0.25in}
\includegraphics[width=0.70\textwidth]{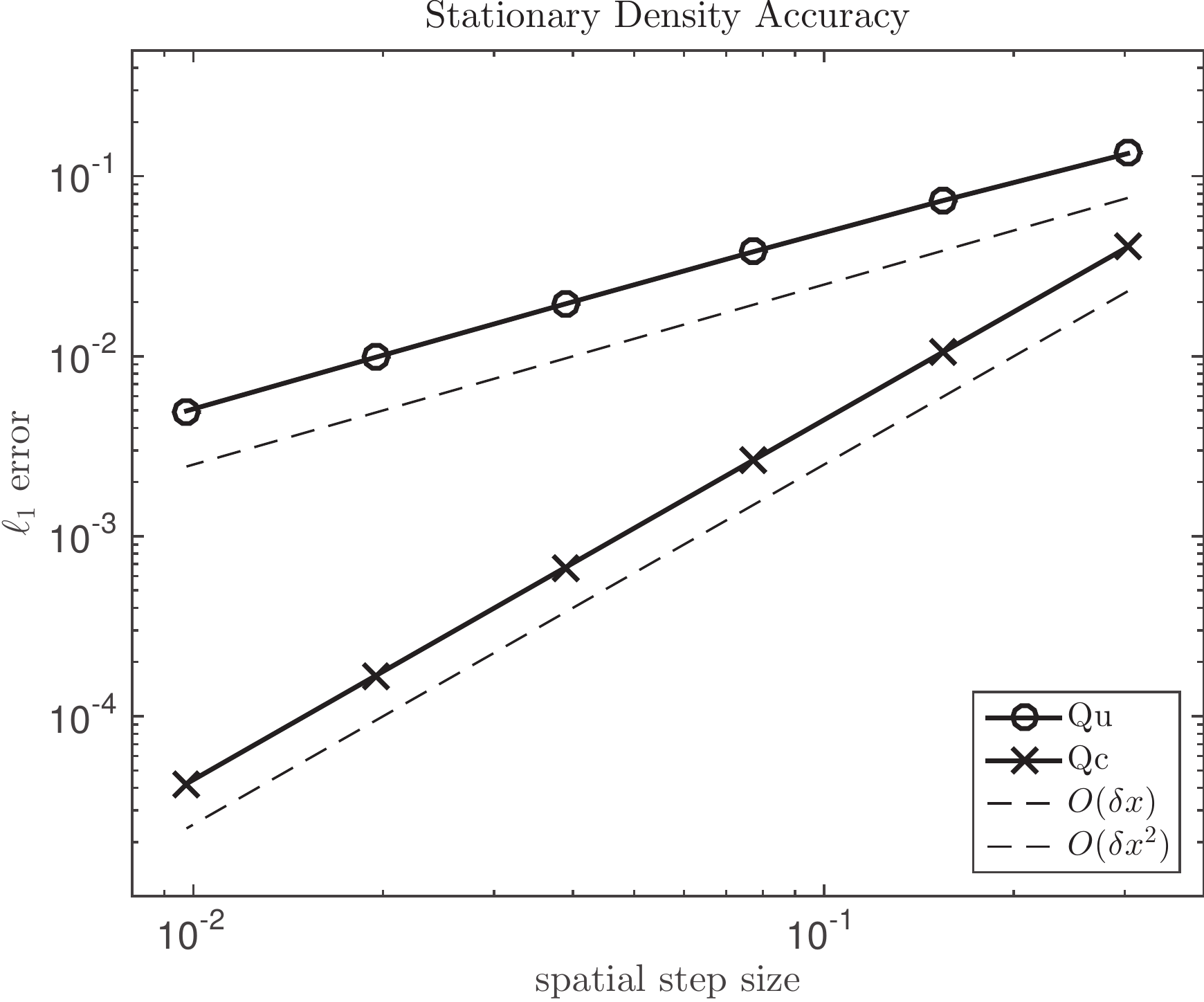}  
\caption{ \small {\bf Log-normal Process in 1D.}  The top panel plots the free energy of the stationary density of the SDE. The bottom panel plots the accuracy of $\tilde Q_u$ and $\tilde Q_c$ in representing this stationary density as a function of the spatial step size $\delta \xi$ in log-space.  The benchmark solution is the stationary density \eqref{eq:ln_nu_1d} evaluated on the variable step size grid described in \S\ref{sec:amr_1d}.   Note that the central scheme $\tilde Q_c$ appears to be second-order accurate in representing the stationary  probability distribution of the SDE,  which numerically supports Theorem~\ref{thm:nu_accuracy}.
} \label{fig:ln_nu_accuracy} \end{center}
\end{figure}

\begin{figure}[ht!]
\begin{center} 
\includegraphics[width=0.70\textwidth]{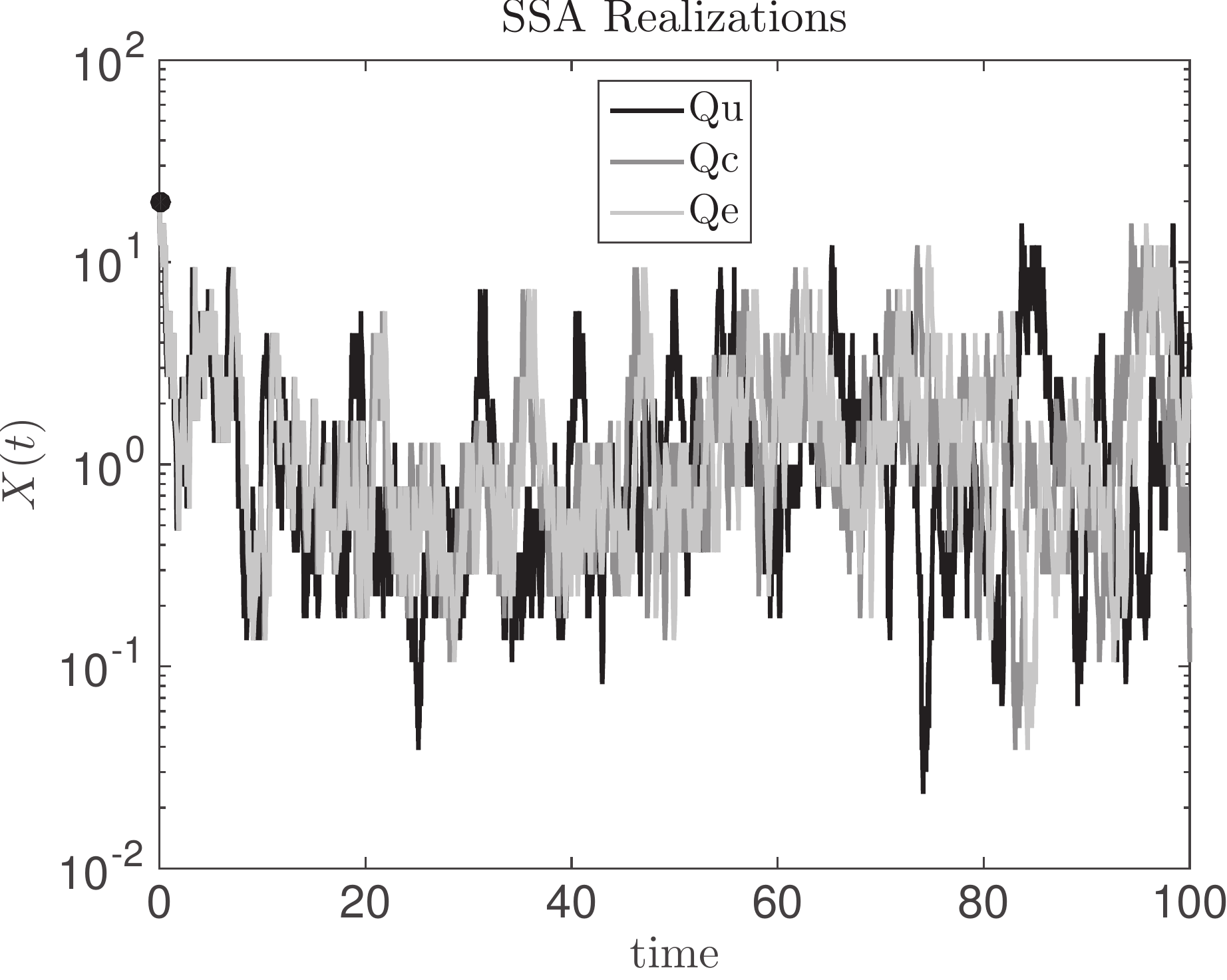}  \\
\vspace{0.25in}
\includegraphics[width=0.60\textwidth]{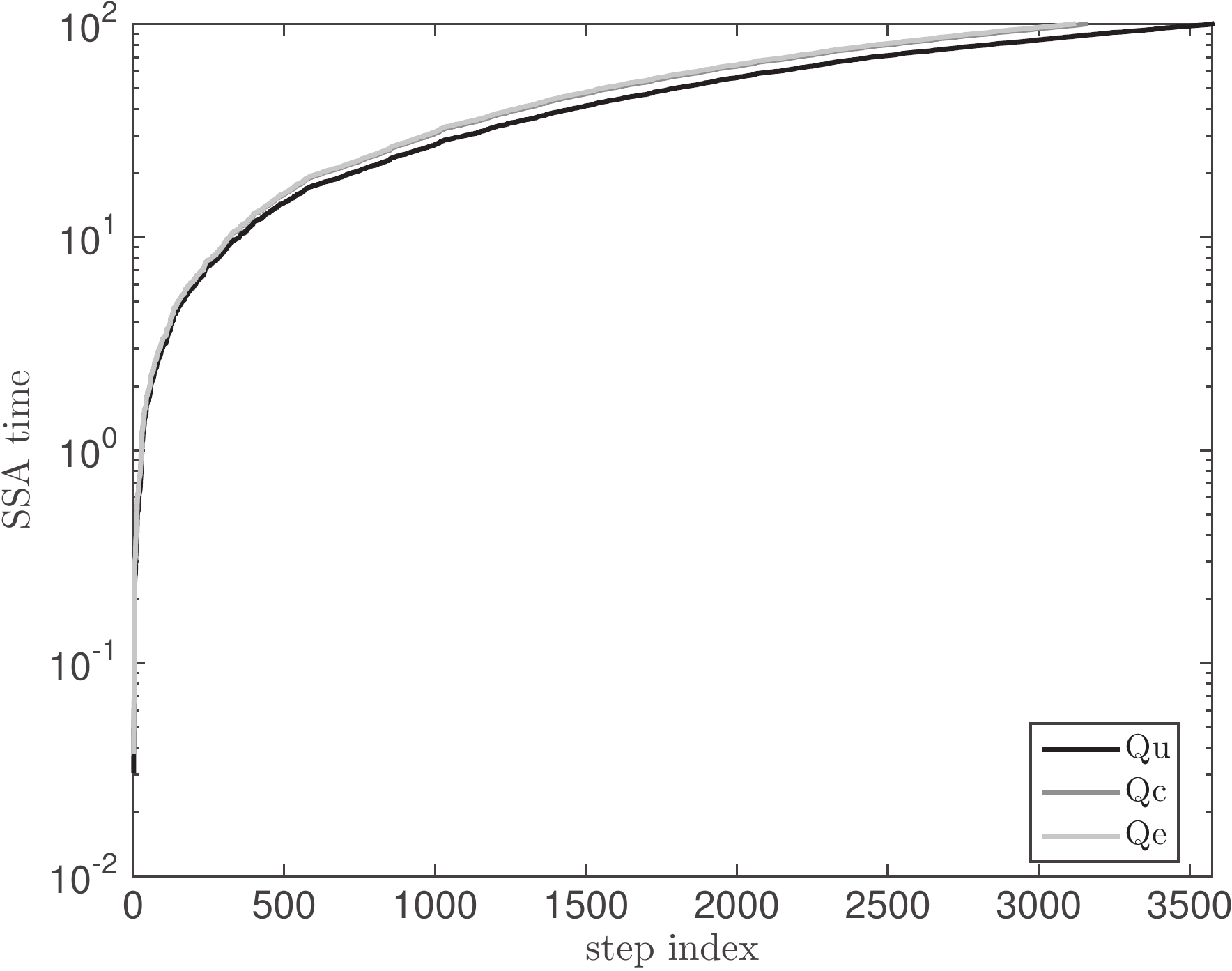} 
\caption{ \small {\bf Log-Normal Process in 1D.}   The top panel plots realizations of the generators $\tilde Q_u$, $\tilde Q_c$, and $Q_e$ produced using the SSA integrator with the initial condition at $X(0)=20$ marked by the black dot at $t=0$.  Recall that $Q_e$ reproduces the true mean holding time of the process, see \eqref{eq:Qe_1D} for its definition.  The state space of the processes is the variable step size grid described in \S\ref{sec:amr_1d} with a grid size in log-space of $\delta \xi=0.25$.  The time interval of simulation is $[0,100]$.  The mean SSA time for all three methods is approximately $0.03$. The bottom panel plots the SSA time as a function of the computational step index.  This figure illustrates that the SSA integrators never exit the domain of definition of the coefficients of the SDE.
} \label{fig:ln_ssa} \end{center}
\end{figure}

\begin{figure}[ht!]
\begin{center} 
\includegraphics[width=0.65\textwidth]{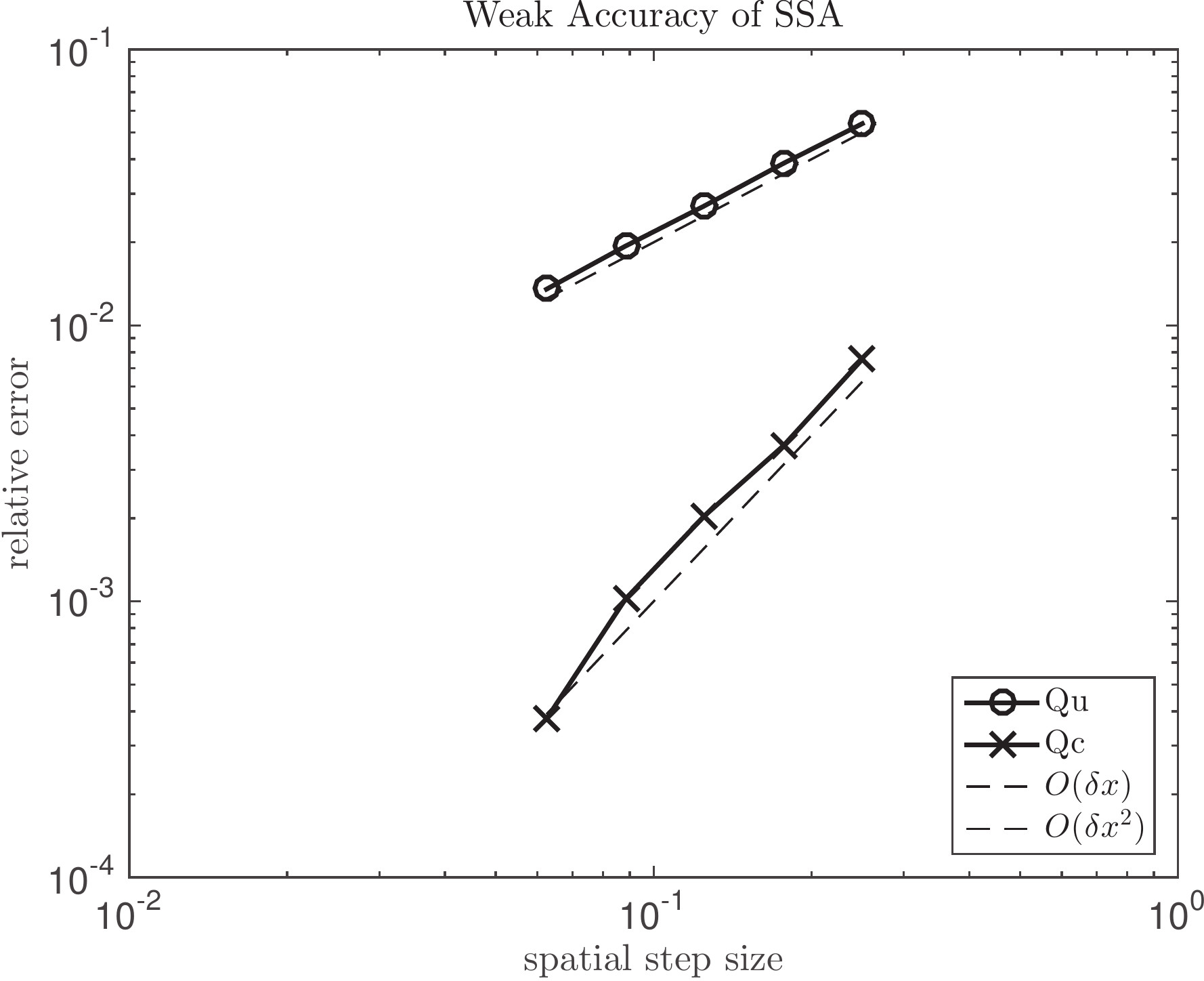}  \\
\vspace{0.25in}
\includegraphics[width=0.65\textwidth]{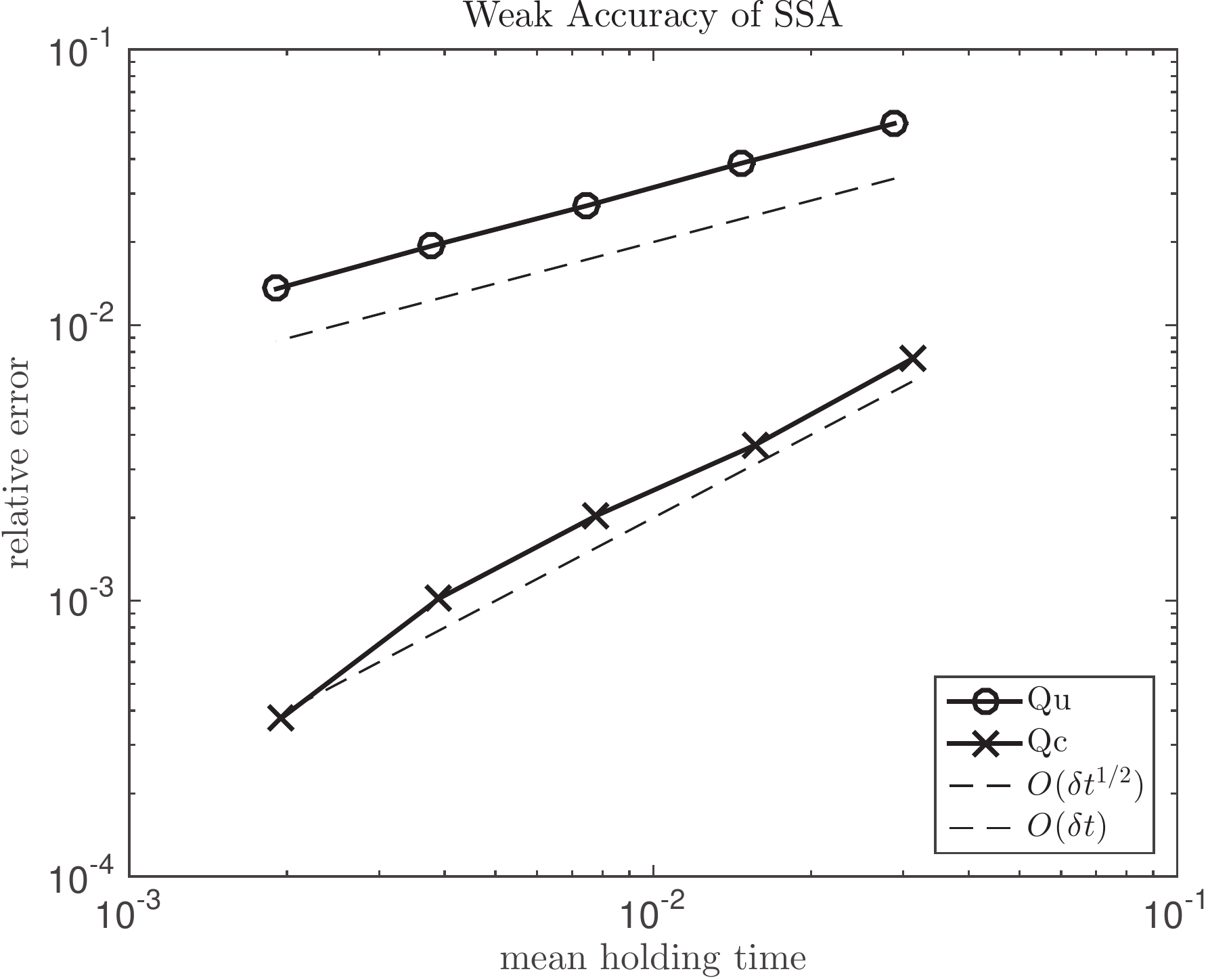} 
\caption{ \small {\bf Log-Normal Process in 1D.}  
The top (resp.~bottom) panel plots the accuracy of the  generators $\tilde Q_u$ and $\tilde Q_c$ with respect to $\Ex_x(Y(t)^2)$ as a function of spatial step size (resp.~mean holding time) with initial condition $Y(0)=2$ and terminal time $t=1$.  Note that the central scheme $\tilde Q_c$ appears to be second-order accurate,  which numerically supports Theorem~\ref{thm:finite_time_accuracy}.  This figure was produced using $10^8$ SSA realizations.
} \label{fig:ln_weak_accuracy} \end{center}
\end{figure}

\begin{figure}[ht!]
\begin{center} 
\includegraphics[width=0.65\textwidth]{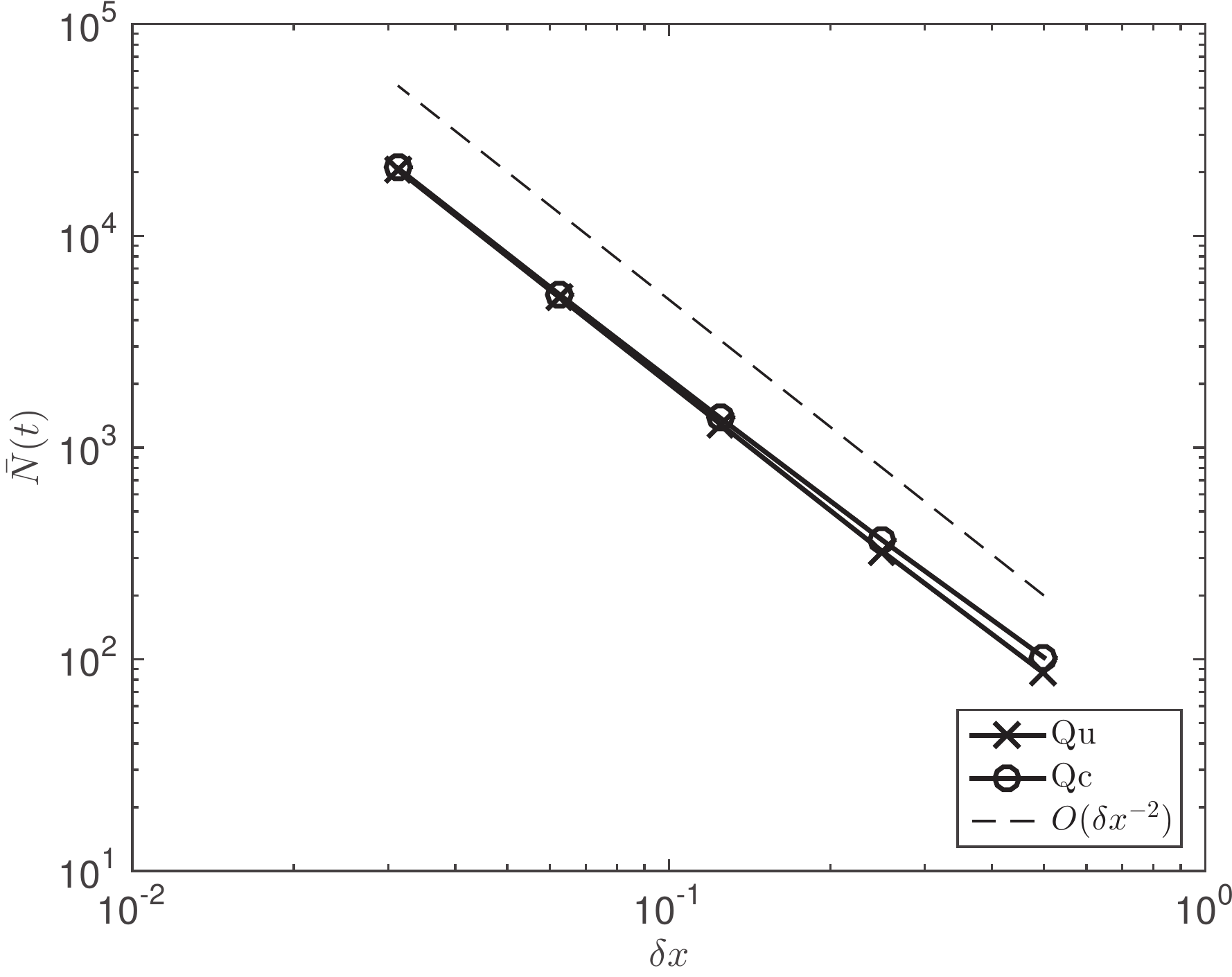}  \\
\vspace{0.25in}
\includegraphics[width=0.65\textwidth]{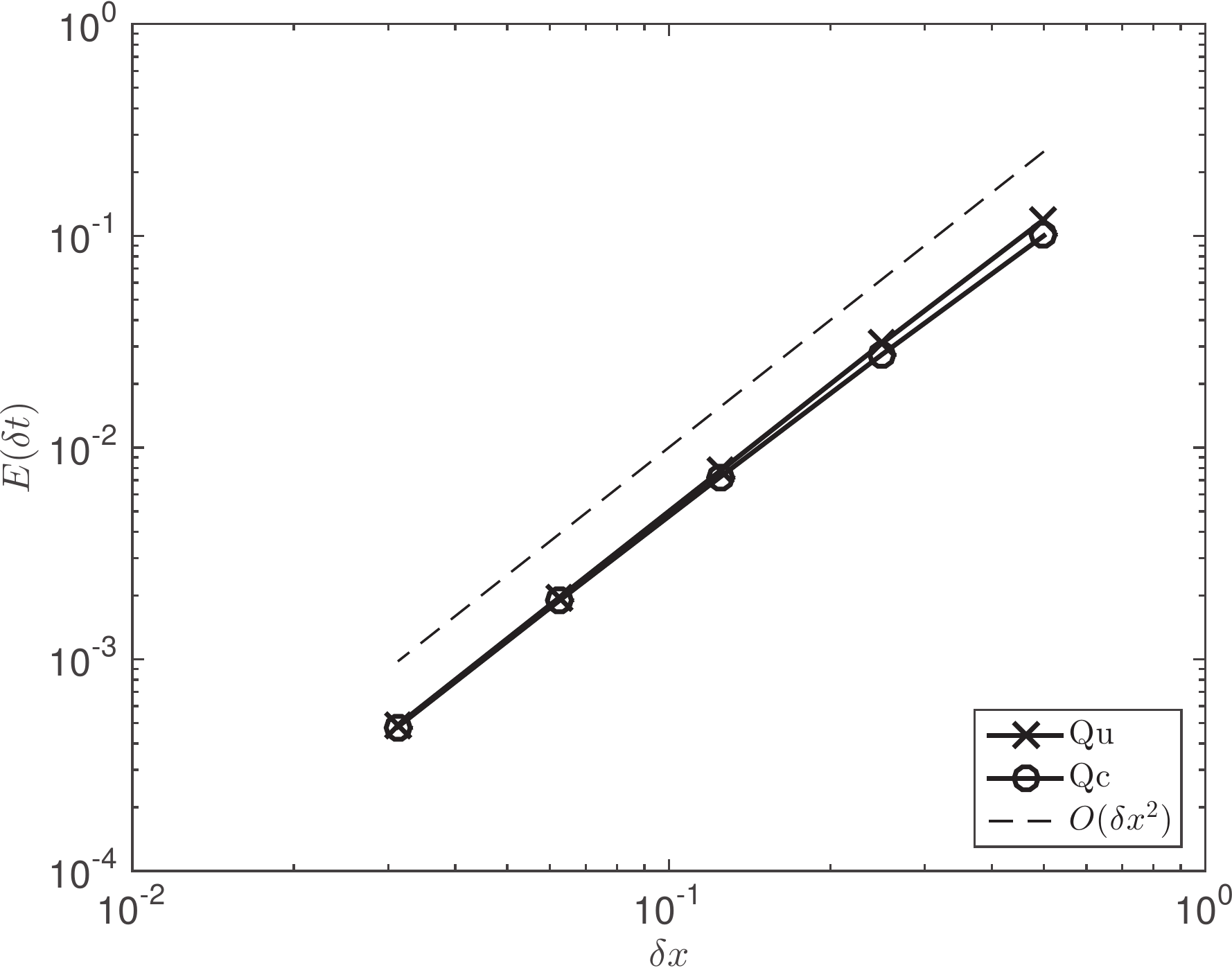} 
\caption{ \small {\bf Log-Normal Process in 1D.}   
These statistics are produced using the SSA induced by the generators $\tilde Q_u$ in  \eqref{eq:Qu_1d} and $\tilde Q_c$ in \eqref{eq:Qc_1d} as indicated
in the figure Legend.    The top panel of this figure plots the mean number of computational steps of the SSA integrator as a function of the spatial step size.
The bottom panel of this figure plots the mean holding time of the SSA integrator as a function of the spatial step size.
The statistics are obtained by taking a sample average over $100$ realizations with initial condition at $X(0)=1$ and a time interval of simulation of $[0,10]$.  
This figure validates Proposition~\ref{lem:barN}.  
} \label{fig:ln_complexity} \end{center}
\end{figure}


\section{Cox-Ingersoll-Ross Process in 1D with Multiplicative Noise} \label{sec:cir_process}

Consider a scalar SDE problem with boundary \begin{equation} \label{eq:cir_sde}
d Y= \beta ( \alpha - Y) dt + \sigma \sqrt{Y} dW \;, \quad Y(0) \in \mathbb{R}^+ \;.
\end{equation}
In quantitative finance, this example is a simple model for short-term interest rates \cite{CoInRo1985}.  It is known that discrete-time integrators may not satisfy the non-negativity condition: $Y(t) \ge 0$ for all $t > 0$ \cite{LoKoKi2010, andersen2010simulation}.  Corrections to these integrators that enforce this condition tend to be involved, specific to this example, or lower the accuracy of the approximation.   This issue motivates applying realizable discretizations to this problem, which can preserve non-negativity by construction.

According to the Feller criteria, the boundary at zero is classified as: \begin{itemize}
\item natural if $ 2 \beta \alpha / \sigma^2 \ge 1$;
\item regular if $1 > 2 \beta \alpha / \sigma^2 > 0$; and, 
\item absorbing if $\beta \alpha < 0$.
\end{itemize}
When the boundary at zero is natural (resp.~regular), the process is positive (resp.~non-negative).  When the origin is attainable, the boundary condition at the origin is treated as reflective.  If $\beta>0$ and $\alpha>0$, the SDE solution is ergodic with respect to a stationary distribution with density proportional to: \begin{equation} \label{eq:nu_cir}
\nu(x) = x^{ \frac{2 \beta \alpha}{ \sigma^2} - 1 } \exp\left( - \frac{2 \beta}{\sigma^2} x \right) \;.
\end{equation}
We emphasize that $\nu(x)$ is integrable over $\mathbb{R}^+$ if and only if $\beta>0$ and $\alpha>0$.   Note that when the boundary at zero is regular, the stationary density in \eqref{eq:nu_cir} is unbounded at the origin.   

The numerical tests assess how well the approximation captures the true long-time behavior of the SDE.  To carry out this assessment, we estimate the difference between the numerical stationary densities and \begin{equation} \label{eq:bar_nu}
\bar \nu_i = Z^{-1} \int_{ \frac{x_i+x_{i-1}}{2} }^{ \frac{x_i + x_{i+1}}{2} } \nu(x) dx  \quad \text{for all $x_i \in S$}
\end{equation} where the constant $Z$ is chosen so that $\sum_i \bar \nu_i = 1$.   Figures~\ref{fig:cir_process_natural} and \ref{fig:cir_process_regular} illustrate the results of using  adaptive mesh refinement.  The top panels of Figures~\ref{fig:cir_process_natural} and \ref{fig:cir_process_regular} plot the free energy of the stationary density when the boundary at zero is natural and regular, respectively. The benchmark solution is the cell-averaged stationary density \eqref{eq:nu_cir} on the adaptive mesh shown in the inset.  Note that when the boundary at zero is natural, the minimum of the free energy of the stationary density is centered away from zero.   In contrast, when the boundary is regular, the minimum of the free energy is attained at the grid point closest to zero.   The bottom panels of Figures~\ref{fig:cir_process_natural} and \ref{fig:cir_process_regular} plot the $\ell_1$ error in the numerical stationary densities as a function of the mesh size  in logspace $\delta \xi$.   These numerical results confirm the theoretically expected rates of convergence when the mesh size is variable.  


\begin{figure}[ht!]
\begin{center}
\includegraphics[width=0.65\textwidth]{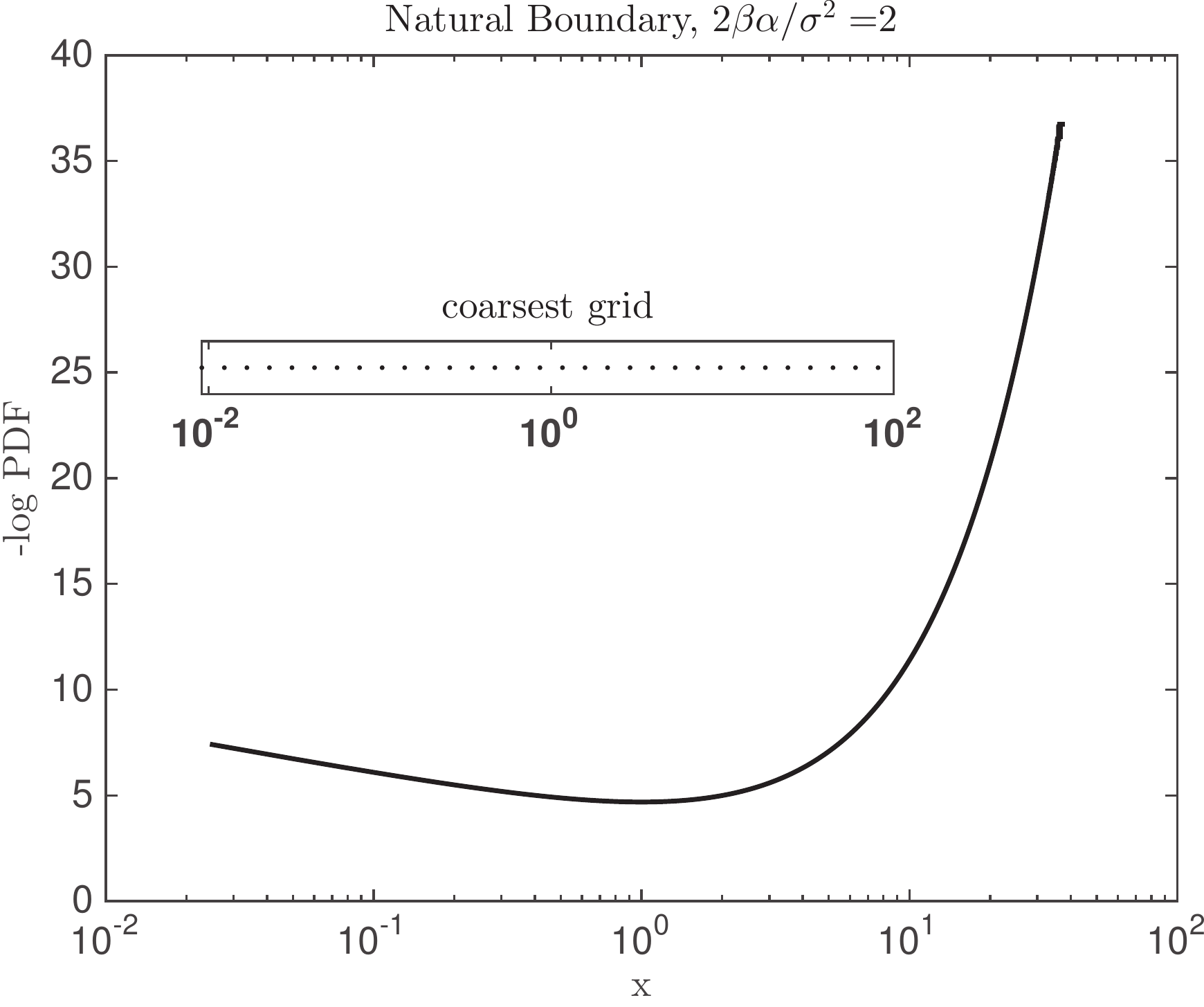}  \\ \vspace{0.25in}
\includegraphics[width=0.65\textwidth]{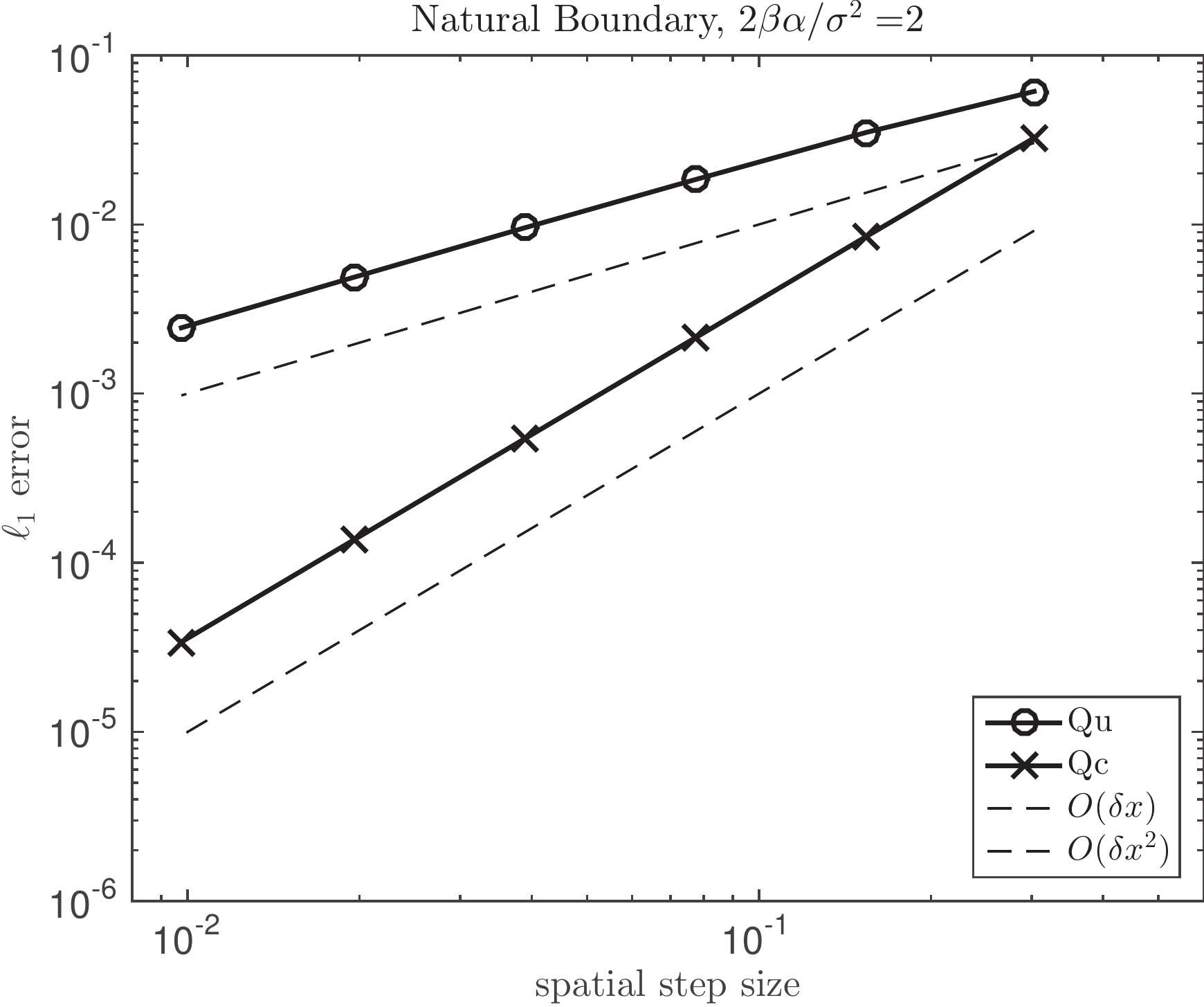} 
\end{center}
\caption{ \small {\bf Cox-Ingersoll-Ross Simulation.}   
The top panel plots the free energy of the stationary density of the SDE, when the boundary is natural. The bottom panel plots the accuracy of $\tilde Q_u$ and $\tilde Q_c$ in representing this stationary density as a function of the spatial step size $\delta \xi$ in log-space.  The benchmark solution is the cell averaged stationary density \eqref{eq:bar_nu} using the variable step size grid described in \S\ref{sec:amr_1d}.   Note that the central scheme $\tilde Q_c$ is second-order accurate in representing the stationary  probability distribution of the SDE,  which numerically supports Theorem~\ref{thm:nu_accuracy}.
}
\label{fig:cir_process_natural}
\end{figure}

\begin{figure}[ht!]
\begin{center}
\includegraphics[width=0.65\textwidth]{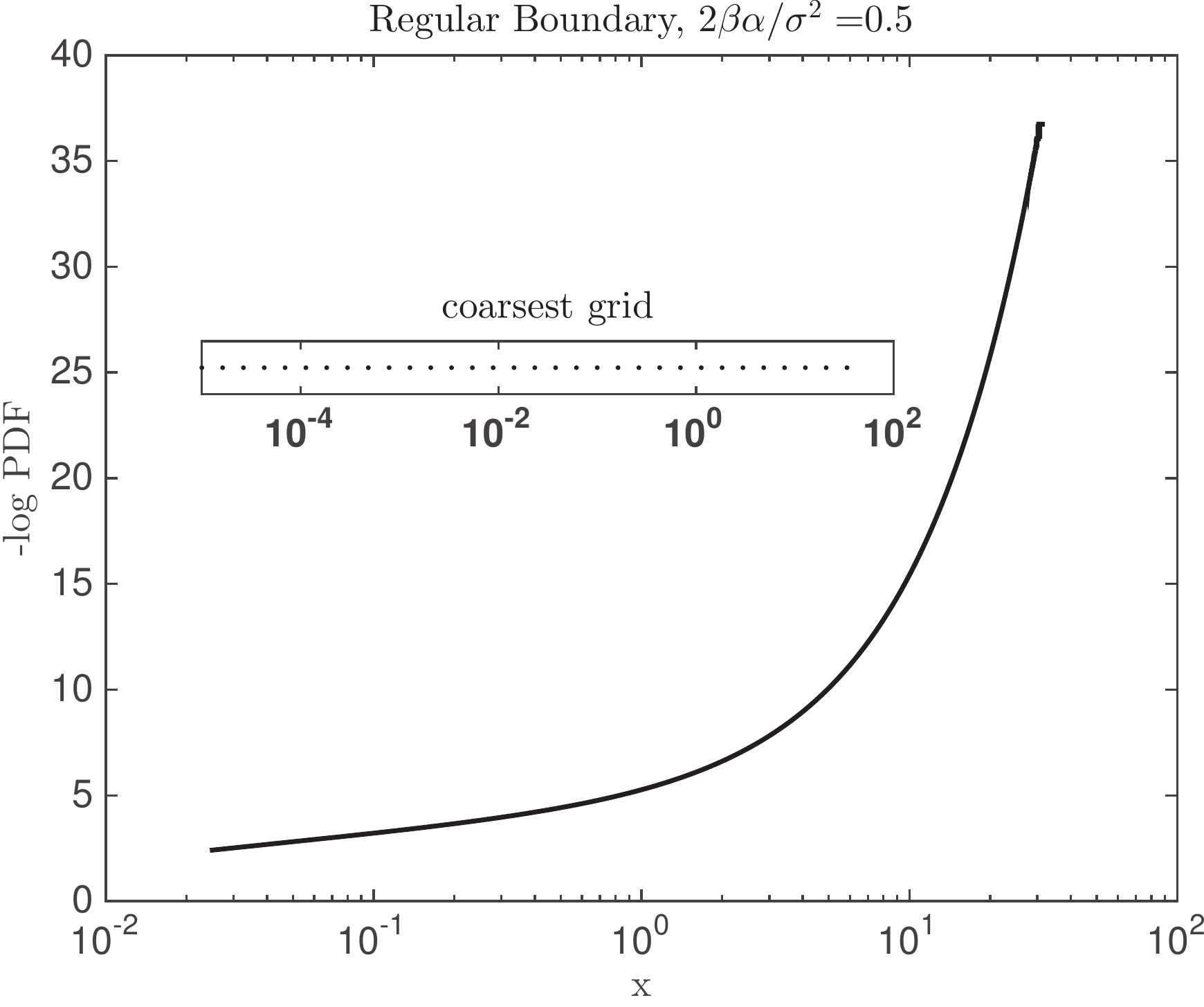}  \\ \vspace{0.25in}
\includegraphics[width=0.65\textwidth]{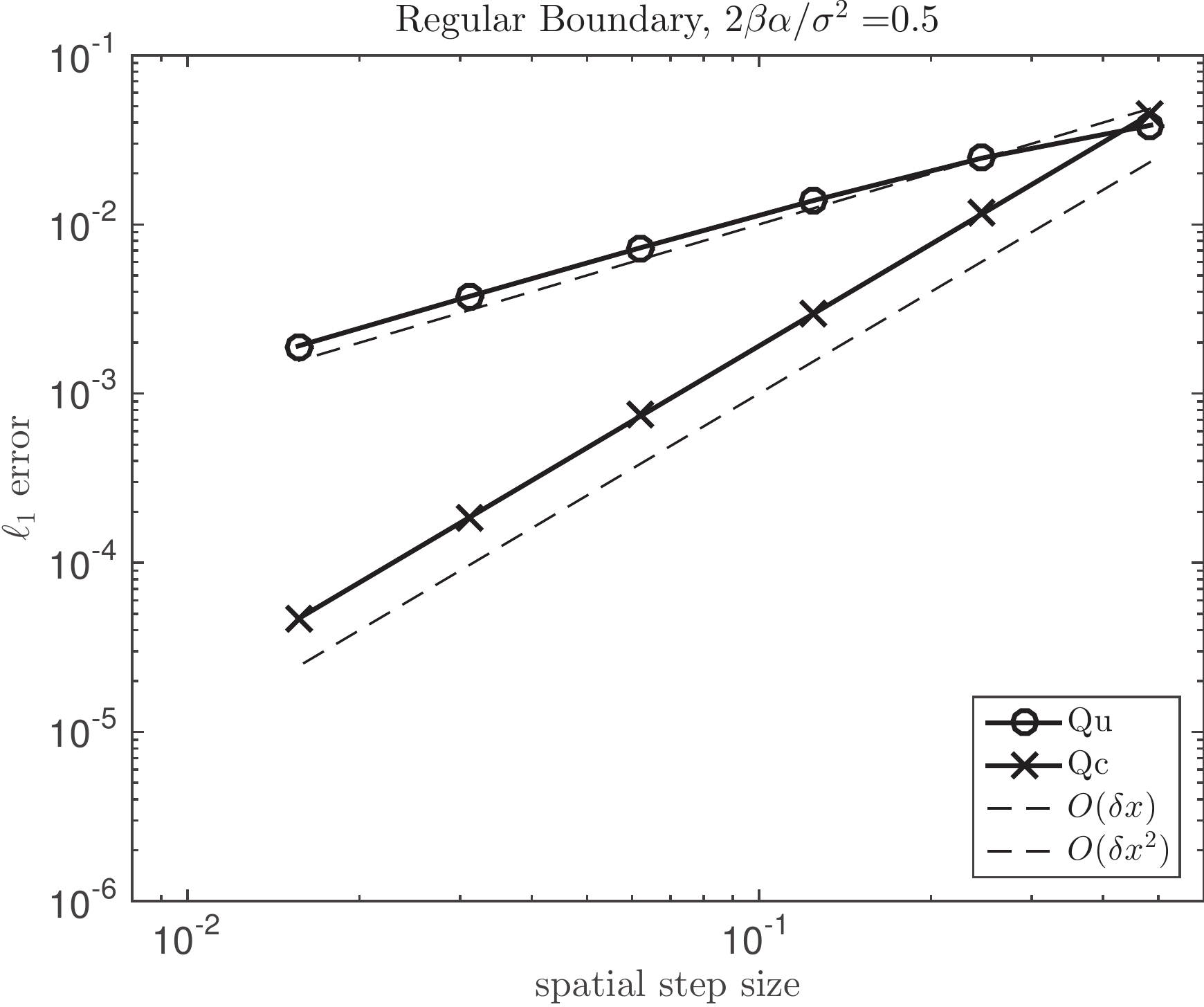}  
\end{center}
\caption{ \small {\bf Cox-Ingersoll-Ross Simulation.}   
The top panel  plots the free energy of the stationary density of the SDE, when the boundary is regular.  The bottom panel plots the accuracy of $\tilde Q_u$ and $\tilde Q_c$ in representing this stationary density as a function of the spatial step size parameter $\delta \xi$ in log-space.   The benchmark solution is the cell averaged stationary density \eqref{eq:bar_nu} using the variable step size grid described in  \S\ref{sec:amr_1d}.  Note that the central scheme $\tilde Q_c$ is second-order accurate in representing the stationary  probability distribution of the SDE,  which numerically supports Theorem~\ref{thm:nu_accuracy}.  
}
\label{fig:cir_process_regular}
\end{figure}


\section{SDEs in 2D with Additive Noise} \label{sec:planar_sdes}

The examples in this part are two-dimensional SDEs of the form  \begin{equation} \label{eq:sde_planar_additive}
d Y = b(Y) dt - DU(Y) dt + \sqrt{2 \beta^{-1}} d W \;, \quad Y(0) \in \Omega \subset \mathbb{R}^2 \;,
\end{equation}  
where $U: \mathbb{R}^2 \to \mathbb{R}$ is a potential energy function (which will vary for each example), $\beta$ is the inverse `temperature', and $b(x)$ is one of the following flow fields:
 \begin{equation} \label{flows}
b(x) \in \left\{  \underset{\text{rotational}}{\underbrace{\begin{pmatrix} \gamma x_2 \\ -\gamma x_1 \end{pmatrix}}}, 
 \underset{\text{extensional}}{\underbrace{\begin{pmatrix} \gamma x_2 \\ \gamma x_1 \end{pmatrix}}},
 \underset{\text{shear}}{\underbrace{\begin{pmatrix} \gamma x_2 \\ 0 \end{pmatrix}}},
 \underset{\text{nonlinear}}{\underbrace{\begin{pmatrix}  - \gamma x_1 x_2^2 \\ 0 \end{pmatrix}}} \right\}   \;, \quad  x = (x_1, x_2)^T 
\end{equation}
where $\gamma$ is a flow strength parameter.   In the flow-free case ($\gamma=0$) and in the examples we consider, the generator of the SDE in \eqref{eq:sde_planar_additive} is self-adjoint with respect to the probability density: \[
\nu(x) = Z^{-1} \exp(-\beta U(x))  \;, \quad Z = \int_{\Omega} \exp(-\beta U(x)) dx \;.
\]
Moreover, the process $Y$ is reversible \cite{HaPa1986}.  Realizable discretizations have been developed for such diffusions with additive noise in arbitrary dimensions, see \cite{MeScVa2009,LaMeHaSc2011}. With flow ($\gamma \ne 0$), however, the generator is no longer self-adjoint, the equilibrium probability current is nonzero, and the stationary density of this SDE is no longer $\nu(x)$, and in fact, is not explicitly known. In this context we will assess how well the second-order accurate generator $\tilde Q_c$ in \eqref{eq:tQc_2d} reproduces the spectrum of the infinitesimal generator of the SDE $L$.  Since the noise is additive, the state space of the approximation given by $\tilde Q_c$ is a grid over $\Omega \subset \mathbb{R}^2$.  Moreover, this generator can be represented as an infinite-dimensional matrix.  In order to use sparse matrix eigensolvers to compute eigenvalues of largest real part and the leading left eigenfunction of $\tilde Q_c$, we will approximate this infinite-dimensional matrix using a finite-dimensional matrix.   For this purpose, we truncate the grid as discussed next.

\begin{rem}[Pruned Grid] \label{rem:pruning}
Let $\{ x_{ij}  \in \mathbb{R}^2 \mid (i,j) \in \mathbb{Z}^2 \}$ be the set of all grid points and consider the magnitude of the drift field evaluated at these grid points: $\{ | \mu(x_{ij}) | \}$.  In order to truncate the infinite-dimensional matrix associated to $\tilde Q_c$, grid points at which the magnitude of the drift field is above a critical value $E^{\star}$ are pruned from the grid, and the jump rates to these grid points are set equal to zero.  The price of this type of truncation is an unstructured grid that requires using an additional data structure to store a list of neighbors to each unpruned grid point.  To control errors introduced in this pruning, we verify that computed quantities do not depend on the parameter $E^{\star}$.   
\end{rem}

\begin{rem}[Error Estimates] \label{rem:error}
In the linear context, we compare to the exact analytical solution.  In the nonlinear context, and in the absence of an analytical solution, we measure error relative to a converged numerical solution.  
\end{rem}

%
%

\subsection{Planar, Asymmetric Ornstein-Uhlenbeck Process}  \label{sec:planar_ou}

Consider \eqref{eq:sde_planar_additive} with the following two dimensional potential energy: \begin{equation} \label{linear}
U(x) = \frac{1}{2} (x_1^2 + x_2^2 ) \;, \qquad x = (x_1,x_2)^T  \in \mathbb{R}^2 \;,
\end{equation}
and the linear flows defined in \eqref{flows}.  We will use this test problem to assess the accuracy of the generator $\tilde Q_c$ in representing the spectrum of $L$.

In the linear context \eqref{eq:sde_planar_additive} with $\beta=2$,  the SDE \eqref{eq:sde_planar_additive} can be written as: \begin{equation} \label{ou_sde}
d Y =  C Y dt +   d W 
\end{equation}
where $C$ is a $2 \times 2$ matrix.  The associated generator $L$ is a two-dimensional, non-symmetric Ornstein-Uhlenbeck operator.  Let $\lambda_1$ and $\lambda_2$ denote the eigenvalues of the matrix $C$.  If $\Re(\lambda_1)<0$ and  $\Re(\lambda_2) < 0$, then $L$ admits a unique Gaussian stationary density $\nu(x)$ \cite[See Section 4.8]{IkWa1989}, and in addition, the spectrum of $L$ in $L_{\nu}^p$ (for $1 < p < \infty$) is given by: \begin{equation} \label{sigma_L}
\sigma(L) = \{ n_1 \lambda_1 + n_2 \lambda_2 ~:~ n_1, n_2 \in \mathbb{N}_0 \} \;.
\end{equation}
The original statement and proof of \eqref{sigma_L} in arbitrary, but finite, dimension can be found in \cite{MePaPr2002}.  This result serves as a benchmark for the approximation.   In the numerical tests, we select the flow rate to be $\gamma=1/2$, which ensures that the eigenvalues of the matrix $C$ have strictly negative real parts for each of the linear flow fields considered.  

Figure~\ref{fig:ou_stationary_densities} plots contours of the free energy of the numerical stationary density at $\delta x=0.3$ for the linear flow fields considered.   At this spatial step, the $\ell_1$ error in the stationary densities shown is about $1 \%$.  Figure~\ref{fig:ou_stationary_density_accuracy} plots the $\ell_1$ difference between the numerical and exact stationary densities. Figure~\ref{fig:ou_sigma_L} plots the first twenty eigenvalues of $L$ as predicted by \eqref{sigma_L}.  Figure~\ref{fig:ou_sigma_L_accuracy} plots the relative error between the twenty eigenvalues of largest real part of $L$ and $Q_c$.  The eigenvalues of $L$ are obtained from \eqref{sigma_L}, and the eigenvalues of $\tilde Q_c$ are obtained from using a sparse matrix eigensolver and the truncation of $\tilde Q_c$ described in Remark~\ref{rem:pruning}. This last figure suggests that the accuracy of $\tilde Q_c$ in representing the spectrum of $L$ is, in general, first-order.  


\begin{figure}[ht!]
\begin{center}
\includegraphics[width=\textwidth]{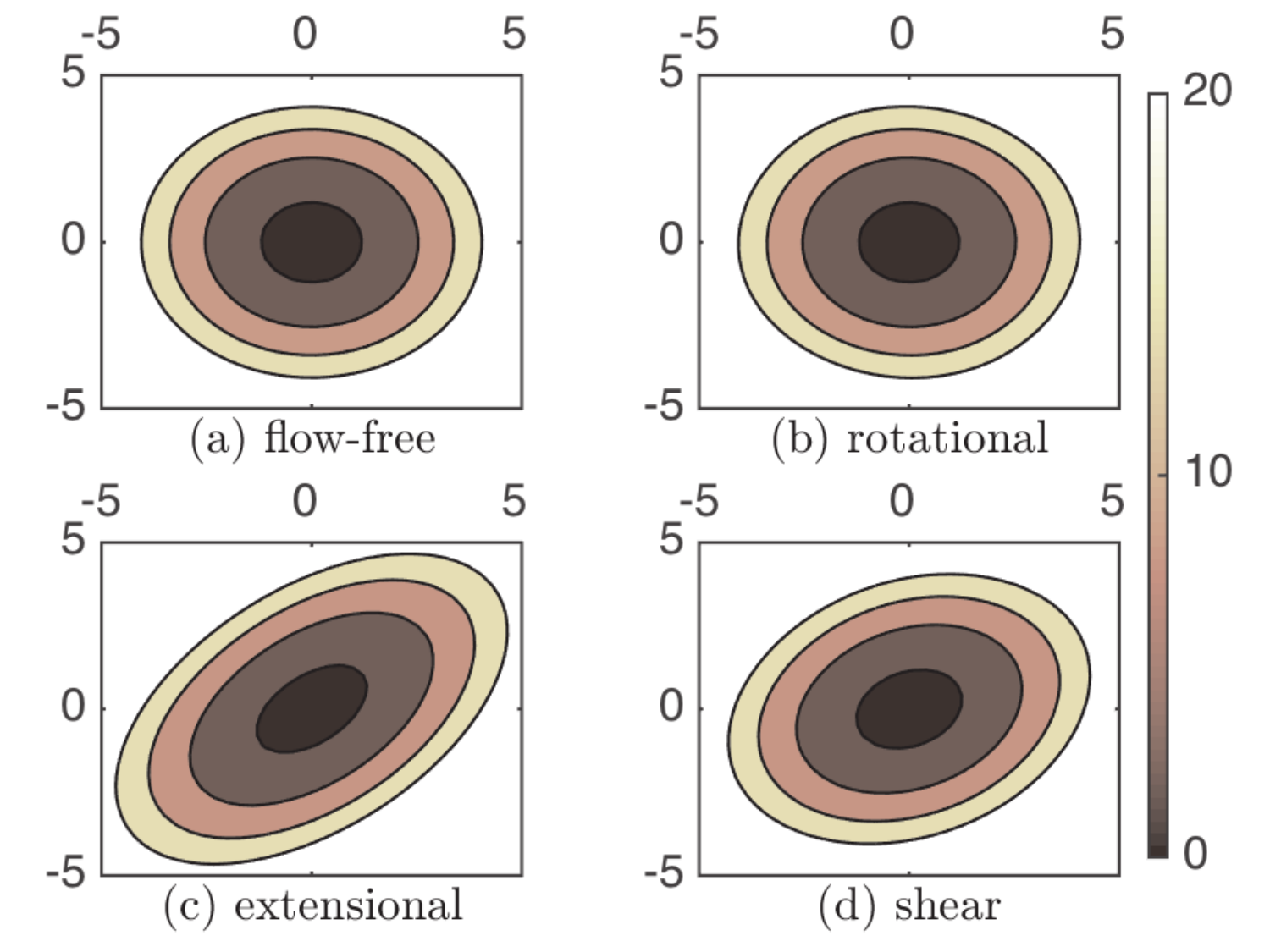} 
\end{center}
                \caption{\small {\bf Planar, Asymmetric Ornstein-Uhlenbeck Process.}
                                                These figures plot contours of the free energy of the numerical stationary density associated to $\tilde Q_c$
                                                 in the flow-free (a), rotational (b), extensional (c), and shear flow (d) cases.  The grid size is selected to be $h=0.3$, which
                                                yields an approximation to the true stationary density that is within $1 \%$ error in the $\ell_1$ norm.
                                                The circular contour lines of the stationary densities in the (a) flow-free 
                                                and (b) rotational cases are identical.  In the (c) extensional and (d) shear flow cases, the contours are elliptical.  
                                                The flow strength parameter is set at: $\gamma = 1/2$.    
                }
                \label{fig:ou_stationary_densities}
\end{figure}

\begin{figure}[ht!]
\begin{center}
\includegraphics[width=0.8\textwidth]{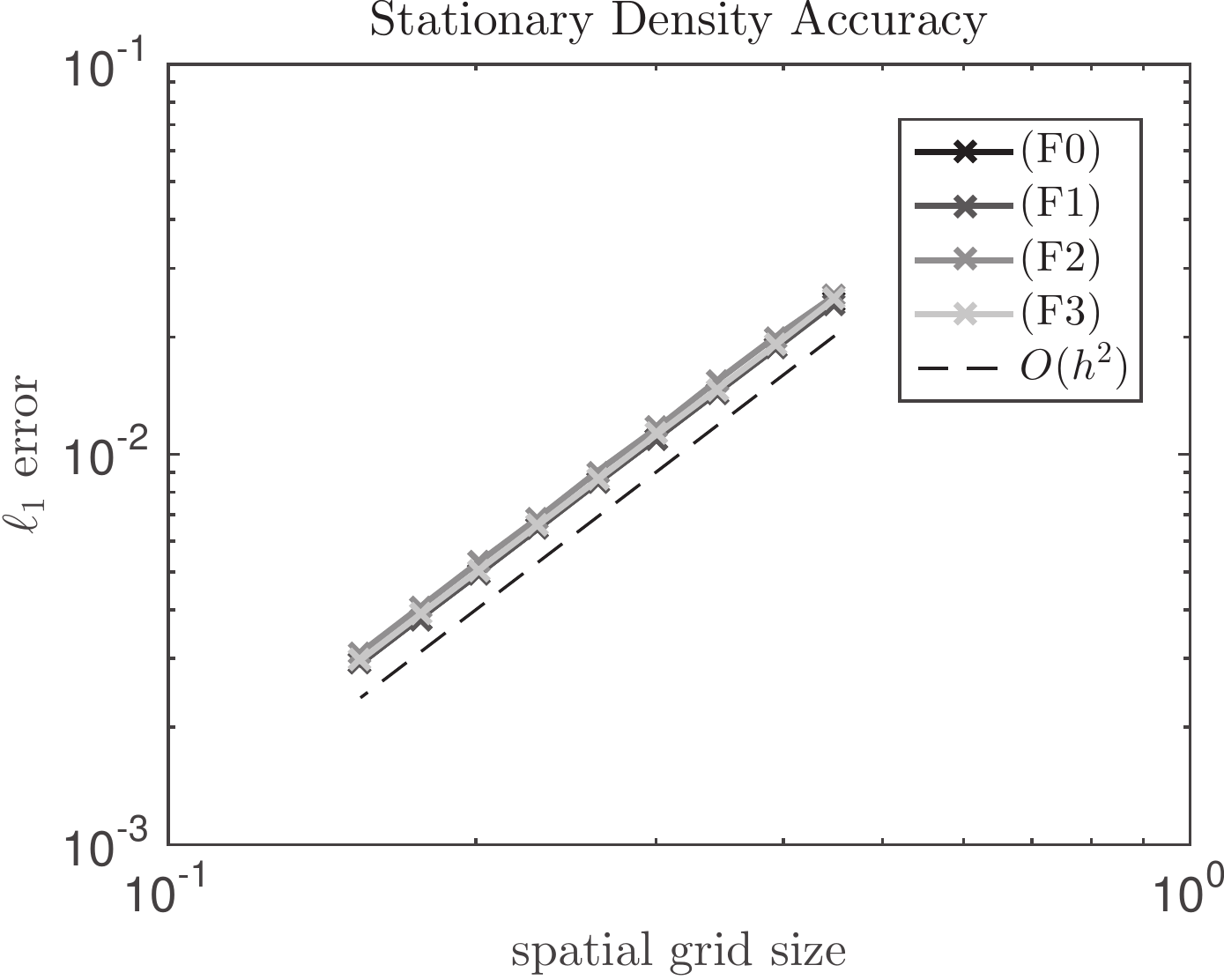}
\end{center}
                \caption{\small {\bf Planar, Asymmetric Ornstein-Uhlenbeck Process.}
                               		 These figures plot the $\ell_1$ error in the estimate of the true stationary density as a function of grid size for $\tilde Q_c$ in the 
                                                flow-free (F0), rotational (F1), extensional (F2), and shear (F3) flow cases, as indicated in the figure legend.
                                                The figure shows that this approximation is second-order accurate in estimating the stationary density.
                                                 The flow strength parameter is set at: $\gamma = 1/2$.   
                                                 This numerical result is consistent with Theorem~\ref{thm:nu_accuracy}.
                                                }
                \label{fig:ou_stationary_density_accuracy}
\end{figure}

\begin{figure}[ht!]
\begin{center}
\includegraphics[width=0.48\textwidth]{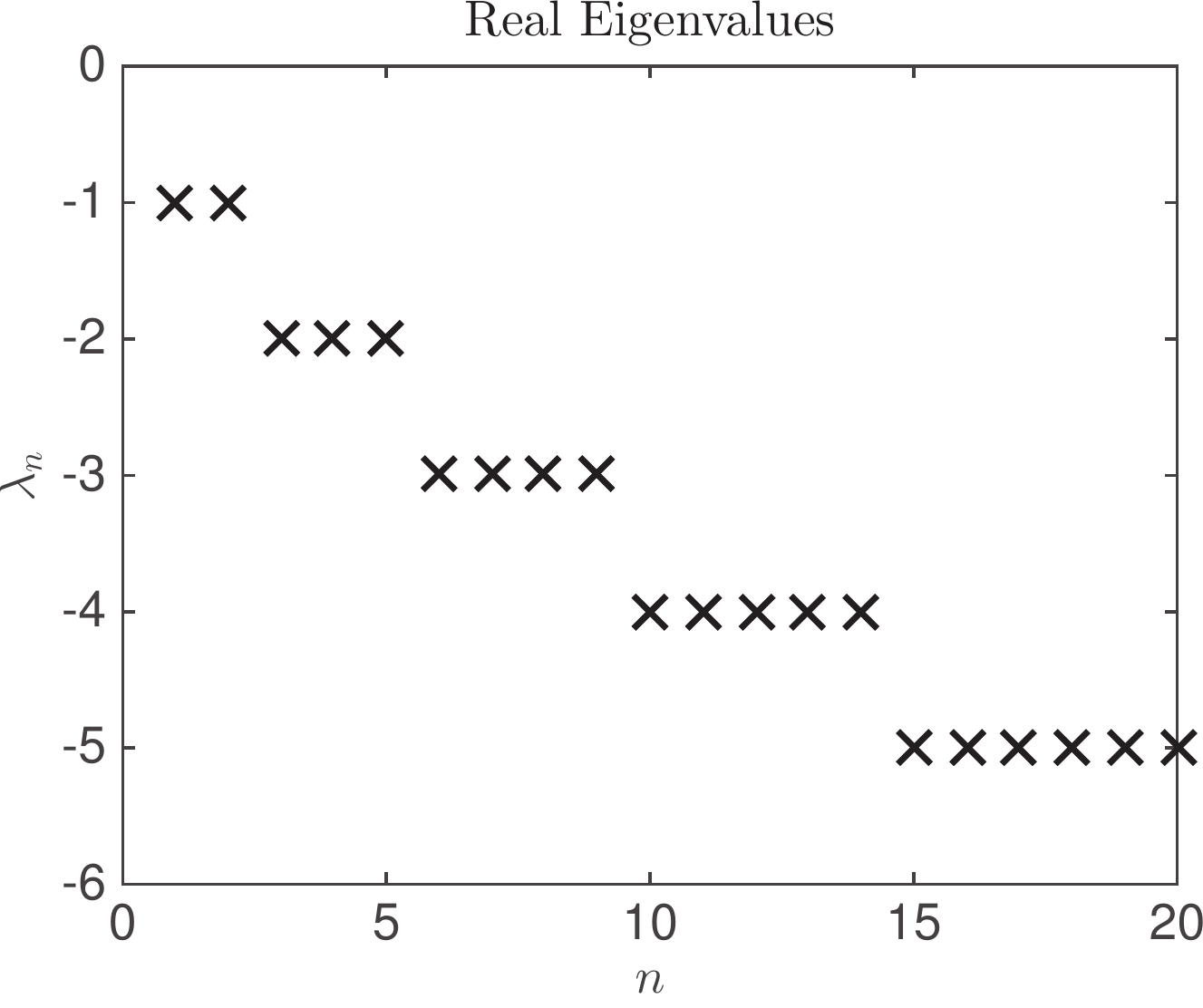}
\includegraphics[width=0.48\textwidth]{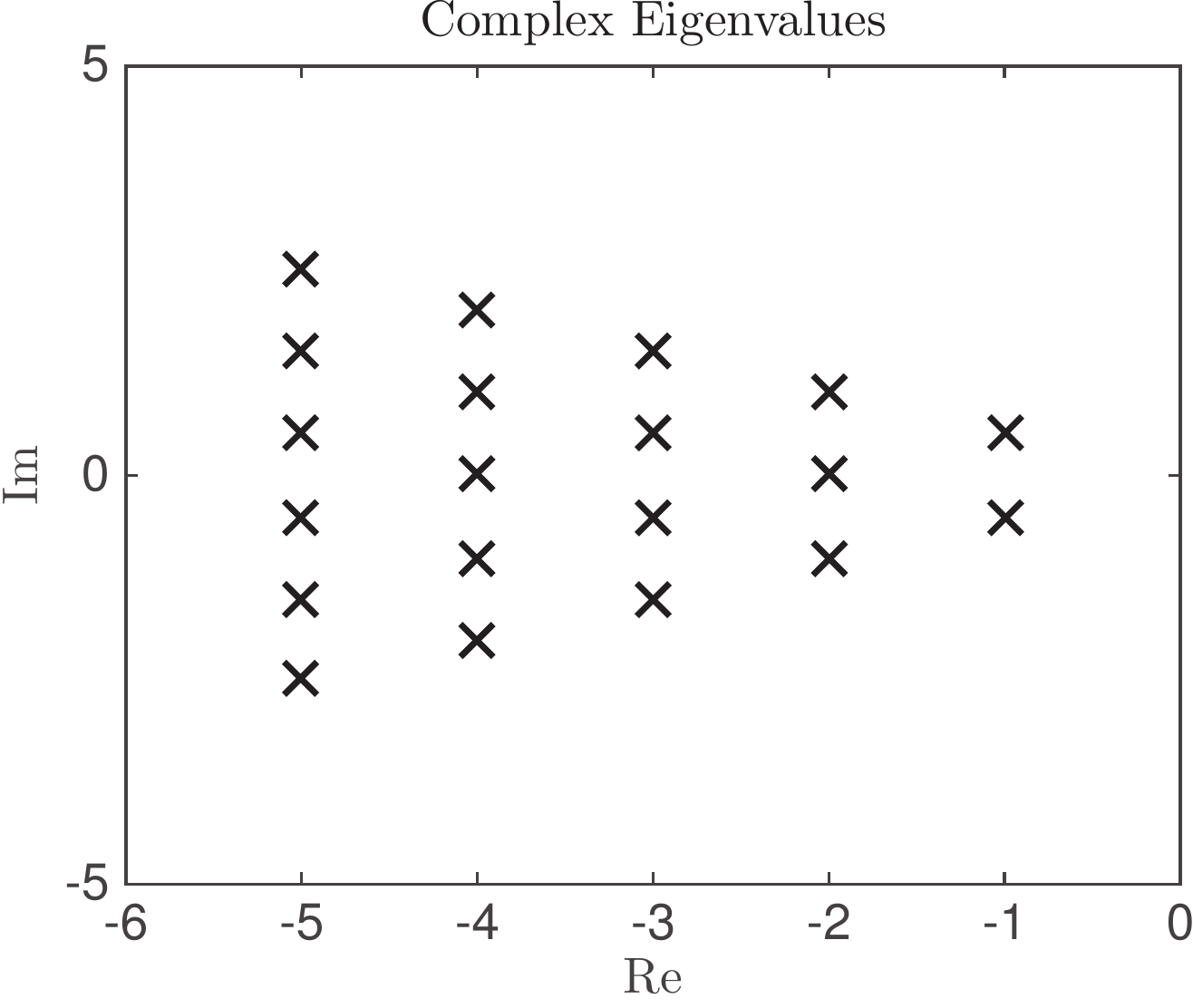} \\
\hbox{
                \hspace{1in} (a) flow-free 
               \hspace{1.5in} (b) rotational 
        }  
\vspace{1em}
\includegraphics[width=0.48\textwidth]{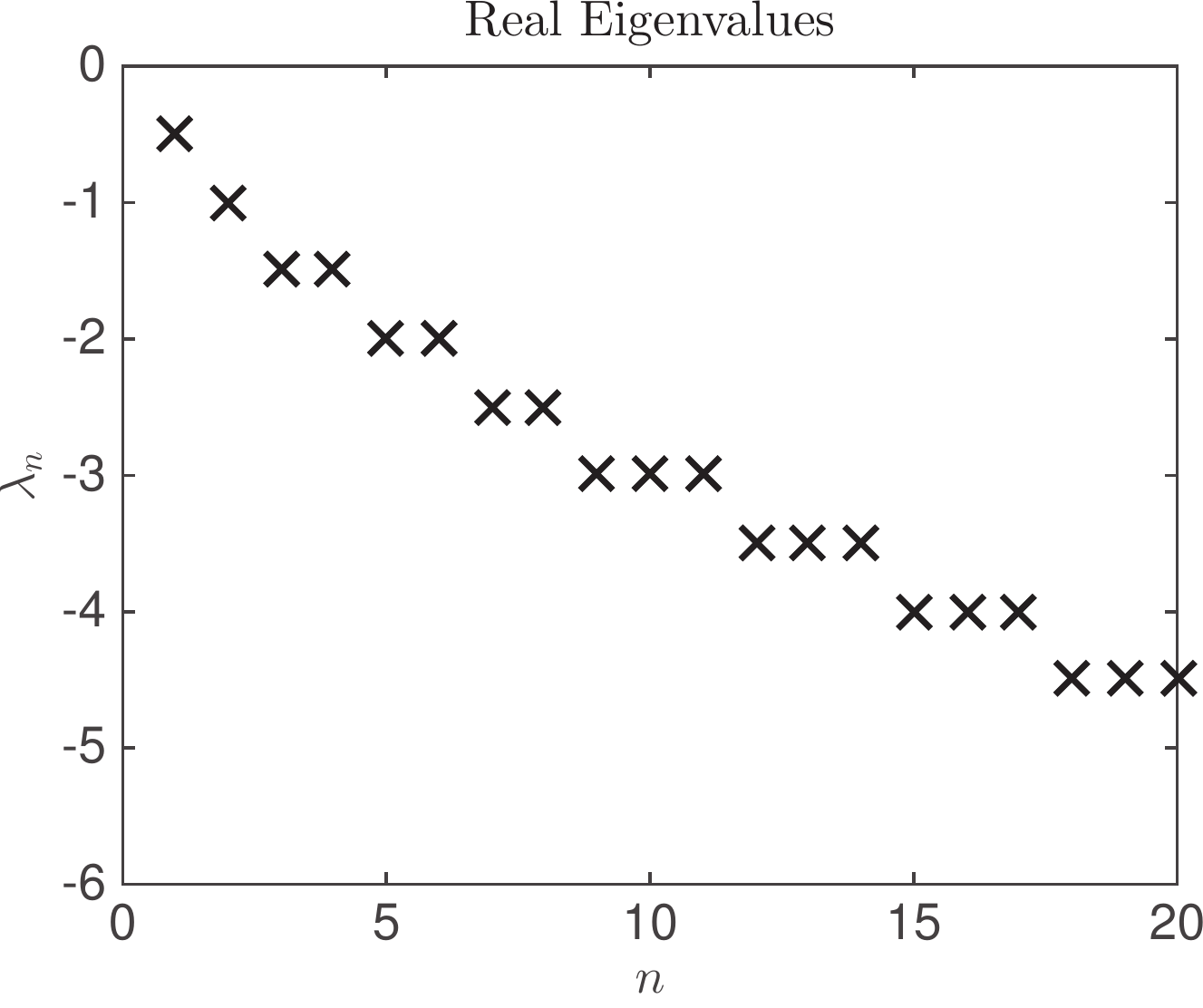}
\includegraphics[width=0.48\textwidth]{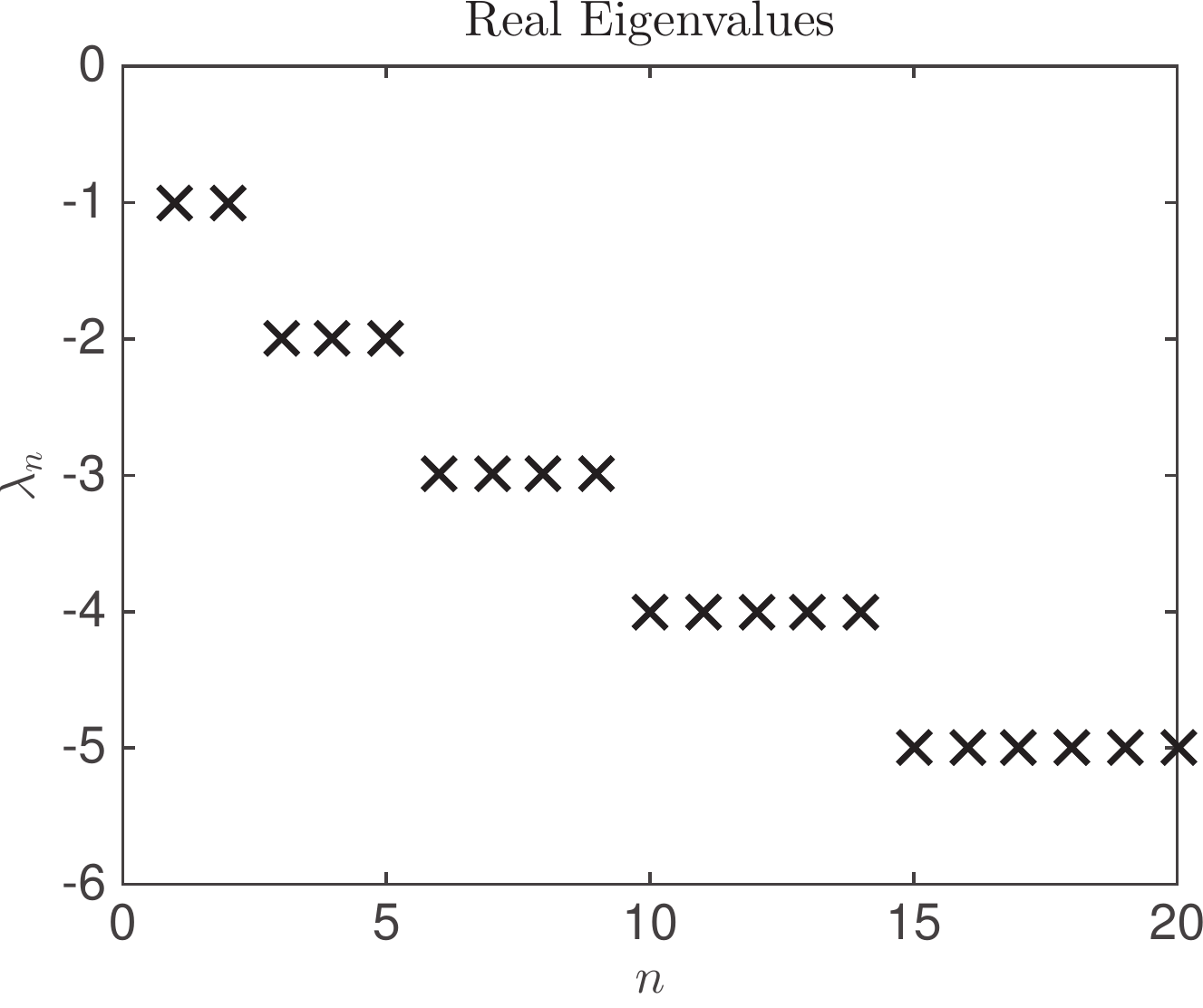}
\hbox{
                \hspace{0.0in} (c) extensional
               \hspace{1.75in} (d) shear 
        }           
        \end{center}
        
         \caption{\small {\bf Planar, Asymmetric Ornstein-Uhlenbeck Process.}
                                                These figures plot the twenty  eigenvalues of largest nontrivial real part of the operator $L$
                                                as predicted by the formula given in \eqref{sigma_L}.  
                                                In (a), (c) and (d) the formula implies that the eigenvalues are strictly real, while in the rotational flow case shown in 
                                                (b) the eigenvalues are complex.  The flow strength parameter is set at: $\gamma = 1/2$.   These eigenvalues serve
                                                as a benchmark solution to validate the accuracy of $\tilde Q_c$ in representing the spectrum of $L$.
                }
                \label{fig:ou_sigma_L}
\end{figure}

\begin{figure}[ht!]
\begin{center}
\includegraphics[width=0.8\textwidth]{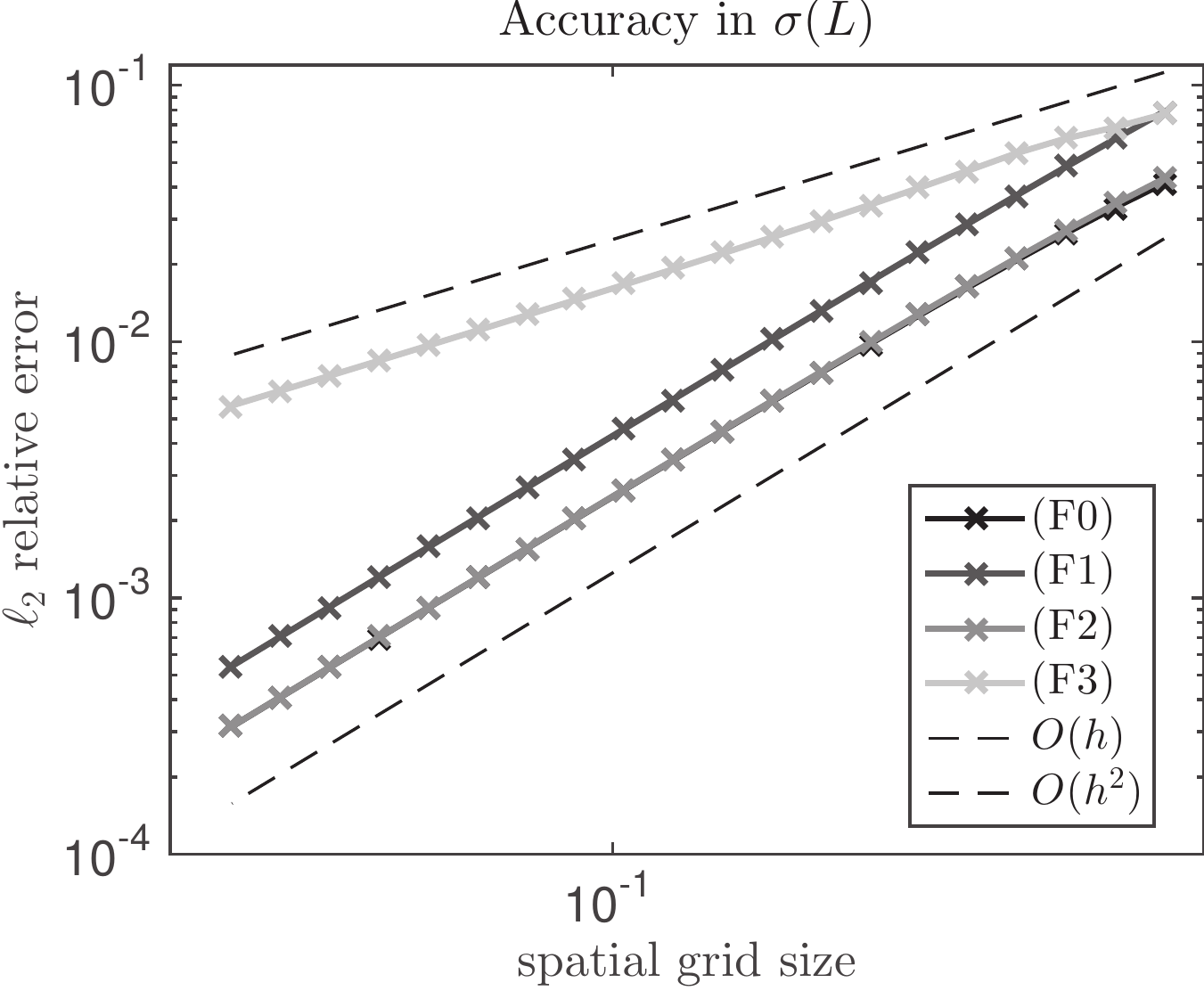}
\end{center}
                \caption{\small {\bf Planar, Asymmetric Ornstein-Uhlenbeck Process.}
                                                These figures plot the relative error of the numerical estimate produced by $\tilde Q_c$ of the twenty eigenvalues of largest nontrivial real part 
                                                of the operator $L$  in the flow-free (F0), rotational (F1), extensional (F2), and shear (F3) flow cases, as indicated in the figure legend.
                                                The figures suggest that the accuracy of $\tilde Q_c$ in representing
                                                these eigenvalues is $\mathcal{O}(h)$ in general, even though it is $\mathcal{O}(h^2)$ in the flow-free (F0), rotational (F1), 
                                                and extensional (F2) flow cases. The flow strength parameter is set at: $\gamma = 1/2$.  
                }
                \label{fig:ou_sigma_L_accuracy}
\end{figure}


\subsection{Maier-Stein SDE: Locally Lipschitz Drift \& Nonlinear Flow} \label{sec:doublewell}

The Maier-Stein system is governed by an SDE of the form \eqref{eq:sde_planar_additive} with $\Omega=\mathbb{R}^2$, the nonlinear flow field in \eqref{flows}, $\beta=1$, and \begin{equation} \label{maierstein}
U(x) =  \frac{x_1^4}{4} - \frac{x_1^2}{2} + \mu \frac{(1 + x_1^2) x_2^2}{2} \;,  \qquad x = (x_1,x_2)^T  \in \mathbb{R}^2 \;.
\end{equation}
In the flow-free case ($\gamma=0$), the Maier-Stein SDE describes the evolution of a self-adjoint diffusion.  In this case, the governing SDE is an overdamped Langevin equation with a drift that is the gradient of a double well potential with two minima located at the points $(\pm 1, 0)^T$.  When $\gamma > 0$, this process is no longer self-adjoint and the drift cannot be written as the gradient of a potential \cite{maier1992transition,maier1993escape,maier1996scaling}.  Our interest in this  system is to verify that \eqref{eq:tQc_2d} is convergent for SDEs of the type \eqref{eq:sde_planar_additive} with a nonlinear flow field.  Figure~\ref{fig:E2_Stationary_Density_Verification} reproduces slices of the stationary density, which is consistent with Figure~5 of \cite{valeriani2007computing}.  Contours of the free energy of the stationary density at $\gamma=0$ and $\gamma=4$ are shown in Figure~\ref{fig:E2_stationary_densities}. Finally, Figure~\ref{fig:E2_sigma_L_accuracy} shows that the spectrum of $L$ is accurately represented by $\tilde Q_c$.


\begin{figure}[ht!]
\includegraphics[width=0.65\textwidth]{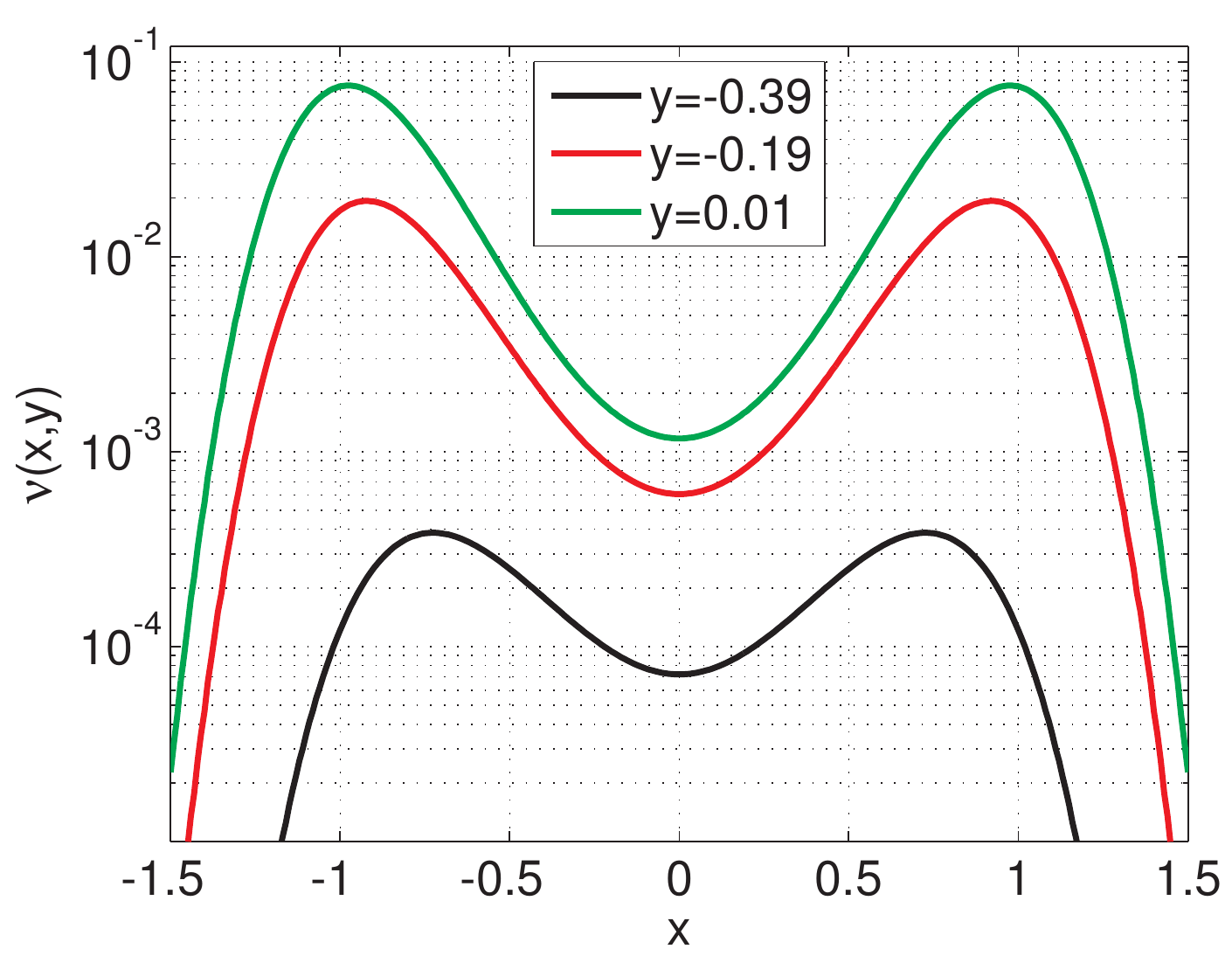} \\
\includegraphics[width=0.65\textwidth]{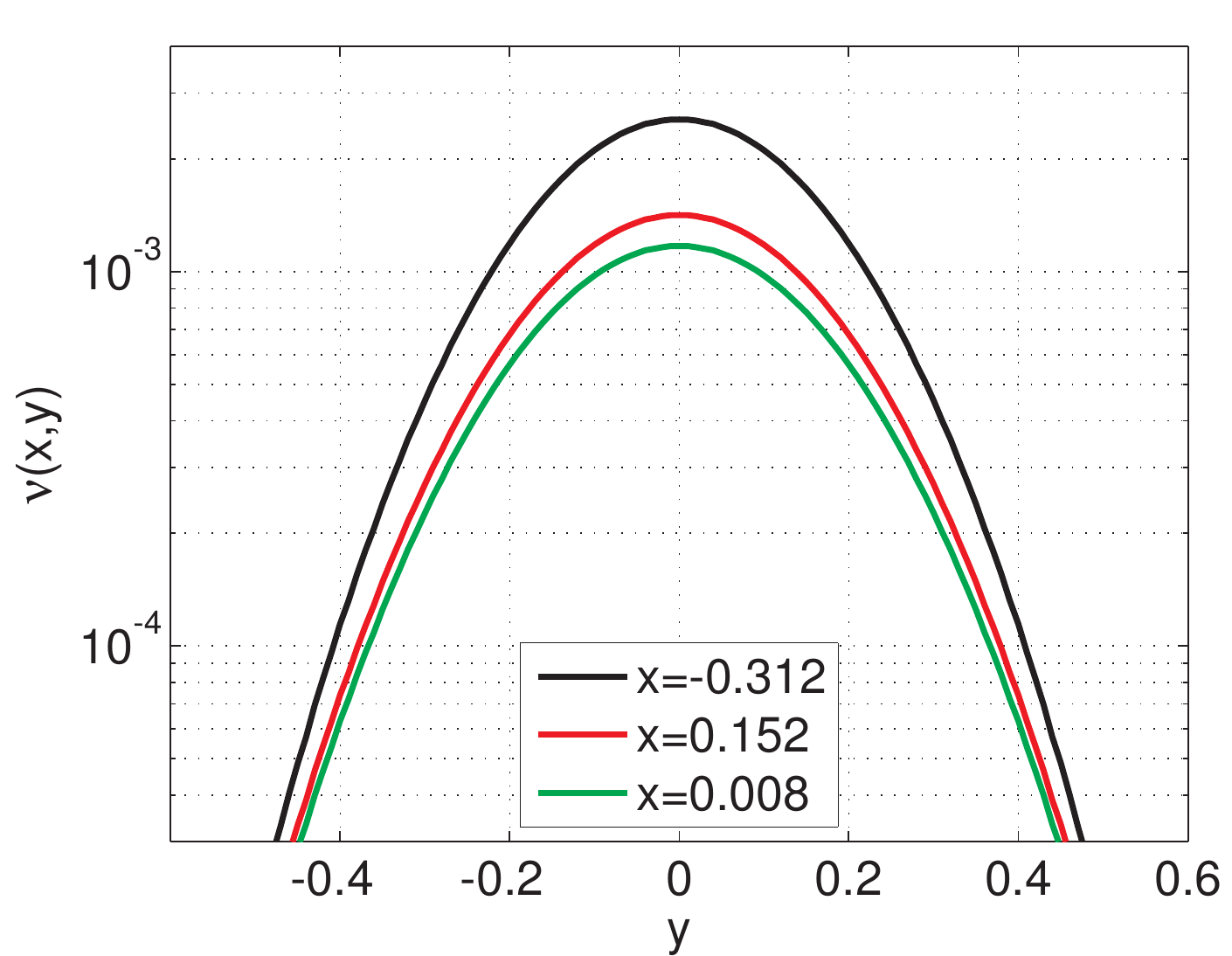} 
                \caption{\small {\bf Maier-Stein: Slices of Stationary Density.} 
                                                These figures plot the intersection of the stationary density of the Maier-Stein system
                                                with the planes indicated in the legends with parameters: $\mu=2$ and $\gamma=4.67$.  These graphs agree
                                                with those given in Figure~5 of \cite{valeriani2007computing}.
                }
                \label{fig:E2_Stationary_Density_Verification}
\end{figure}

\begin{figure}[ht!]
\begin{center}
\includegraphics[width=\textwidth]{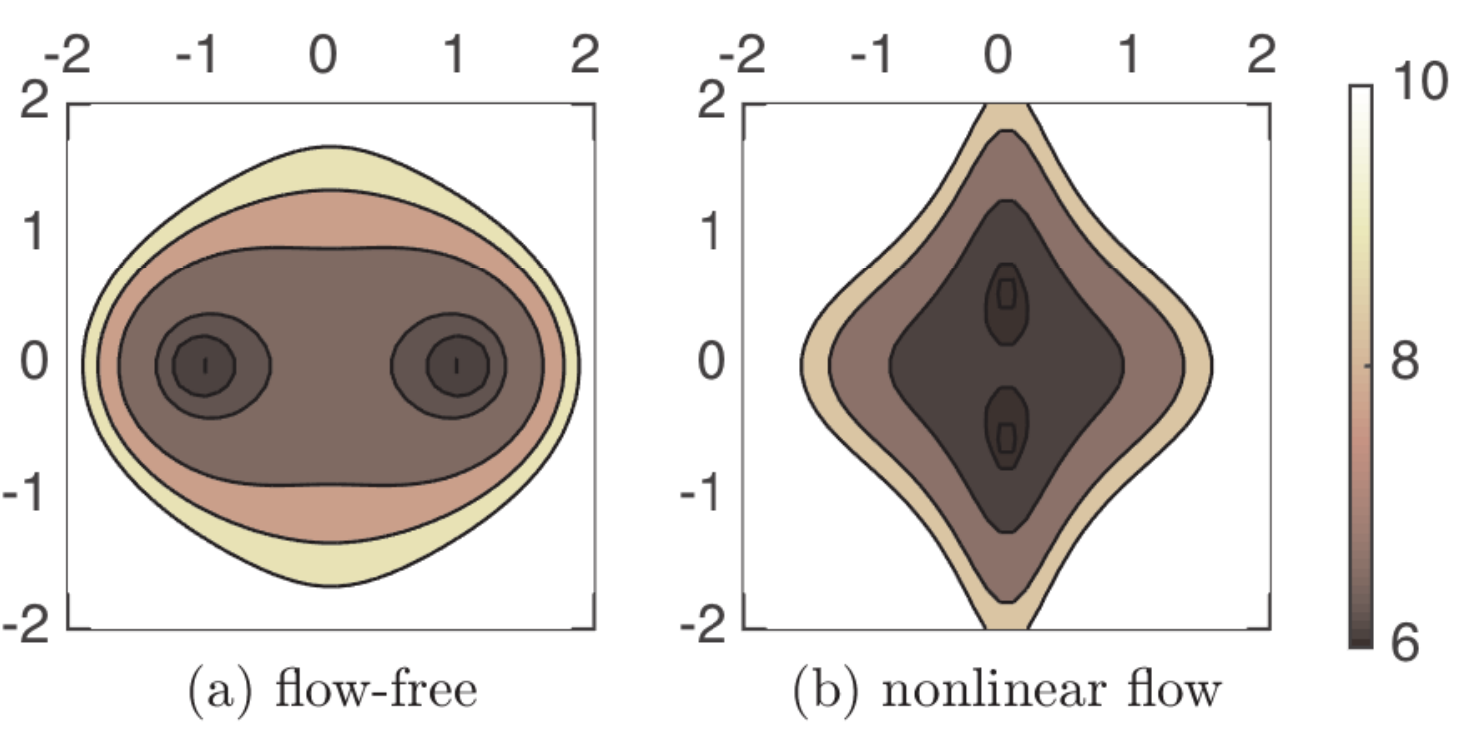}
\end{center}
                \caption{\small {\bf Maier-Stein: Stationary Densities.} 
                				These figures plot contours of the free energy of the 
					stationary density for the generator $\tilde Q_c$ with grid size $h=0.05$.
					The test problem is a particle in the double-well potential given in \eqref{maierstein} and the nonlinear flow field given in \eqref{flows}.
					The flow strength parameters are set at: $\gamma=0$ (left panel) and $\gamma=4$ (right panel).
                }
                \label{fig:E2_stationary_densities}
\end{figure}

\begin{figure}[ht!]
\begin{center}
\includegraphics[width=0.8\textwidth]{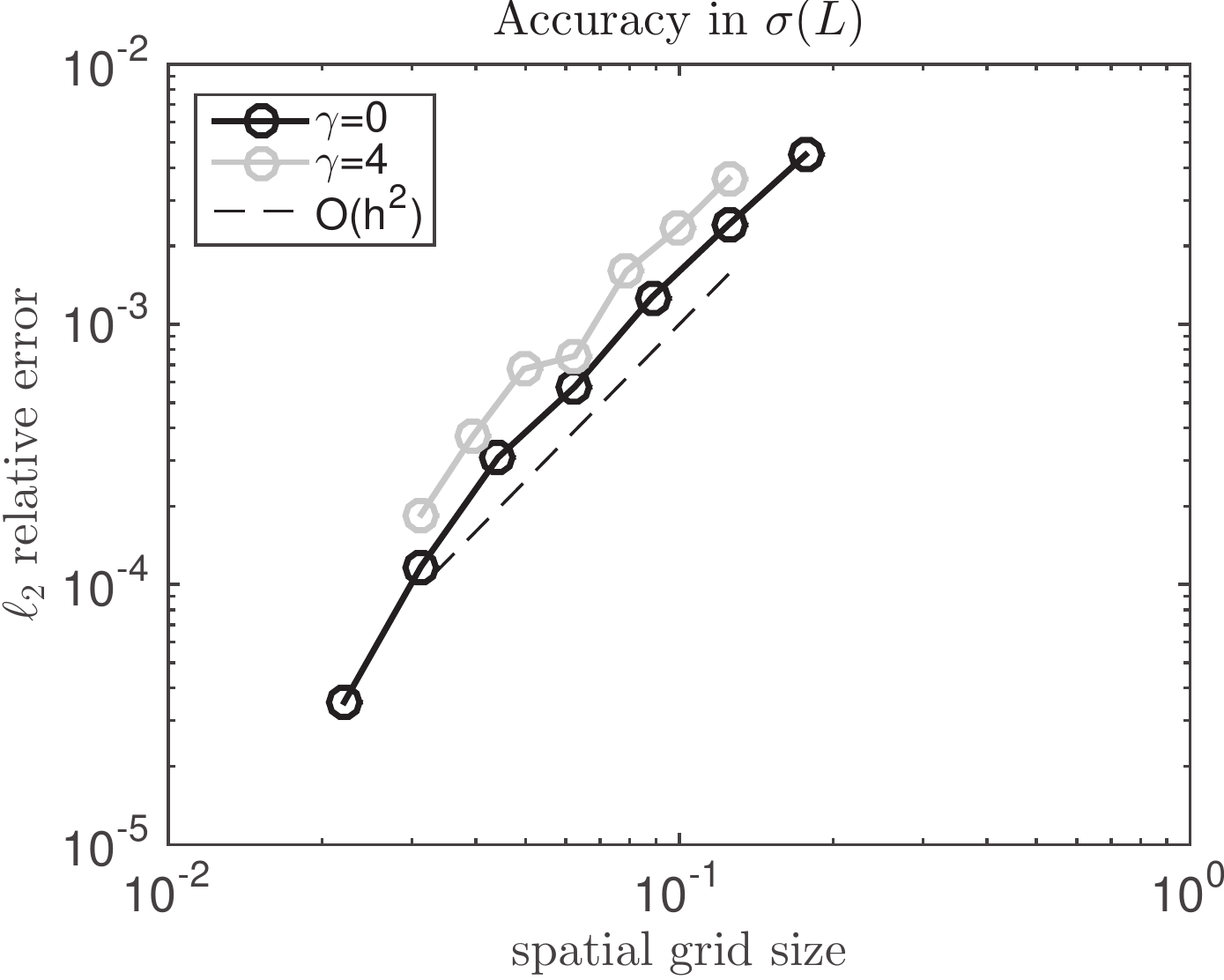}
\end{center}
                \caption{\small {\bf Maier-Stein: Accuracy in $\sigma(L)$.}
                                                This figure illustrates the convergence in the $\ell_2$ norm of the numerical estimate produced by $\tilde Q_c$ of the twenty eigenvalues of largest nontrivial real part 
                                                of the operator $L$ for the values of the flow strength parameter $\gamma$ indicated in the legend.  
                }
                \label{fig:E2_sigma_L_accuracy}
\end{figure}


\subsection{Square Well SDE: Discontinuous Potential} \label{sec:squarewell}

Here, we test $\tilde Q_c$ in \eqref{eq:tQc_2d} on a non-self-adjoint diffusion whose coefficients have internal discontinuities.  Specifically, consider \eqref{eq:sde_planar_additive} with the linear flow fields in \eqref{flows}, $\beta=1$, and a regularized square well potential: \begin{equation} \label{squarewell}
\begin{aligned}
U(x) &= \tanh\left(\frac{\max(x_1,x_2) - d_1}{\epsilon} \right) -  \tanh\left(\frac{\max(x_1,x_2) - d_2}{\epsilon} \right)   \\
 x &= (x_1,x_2)^T  \in \Omega = [-1,1] \times [-1,1]
  \end{aligned}
\end{equation}
where $d_1$ and $d_2$ specify the location of the discontinuity and $\epsilon$ is a smoothness parameter.  We assume that the solution is periodic at the boundaries of $\Omega$ and we discretize the domain so that all of the grid points are in the interior of $\Omega$.  Since the potential energy is nearly discontinuous, we replace the jump rates in \eqref{eq:Qc} with the following rates which avoid differentiating the potential energy \begin{equation} \label{sw_rates}
\begin{aligned}
q(x,y) = \frac{\beta^{-1}}{h^2}  \exp\left(-\frac{\beta}{2} (U(y)-U(x) - b(x)^T (y-x) ) \right)  \;.
\end{aligned}
\end{equation} 
To be clear, $x,y$ in \eqref{sw_rates} are $2$-vectors.  Contours of the numerical stationary density are plotted in Figure~\ref{fig:E3_stationary_densities_known} when the stationary density is known.   The spatial grid size used is $h=0.1$, which is selected so that the numerical densities represent the true stationary densities to within $1 \%$ accuracy in the $\ell_1$ norm.  Figure~\ref{fig:E3_stationary_densities_unknown} shows a similar plot in the flow cases for which the stationary densities are unknown.  In this case we checked that at the spatial grid size of $h=0.1$ the stationary density is well-resolved.  To accurately represent the probability flux through the square well, we construct the coarsest cell-centered mesh to have borderlines that coincide with the discontinuities of the square well.   Subsequent refinements involve halving the size of each cell, so that the discontinuities always lie on a cell edge.   Figure~\ref{fig:E3_sigma_L_accuracy} illustrates the $\ell_2$ convergence of the numerical estimate to the first twenty eigenvalues of the continuous operator $L$.  As a side note, one may be able to construct a more accurate generator in this context by extending the spatial discretization introduced in \cite{Wa2008} to non-self-adjoint diffusions.


\begin{figure}[ht!]
\begin{center}
\includegraphics[width=\textwidth]{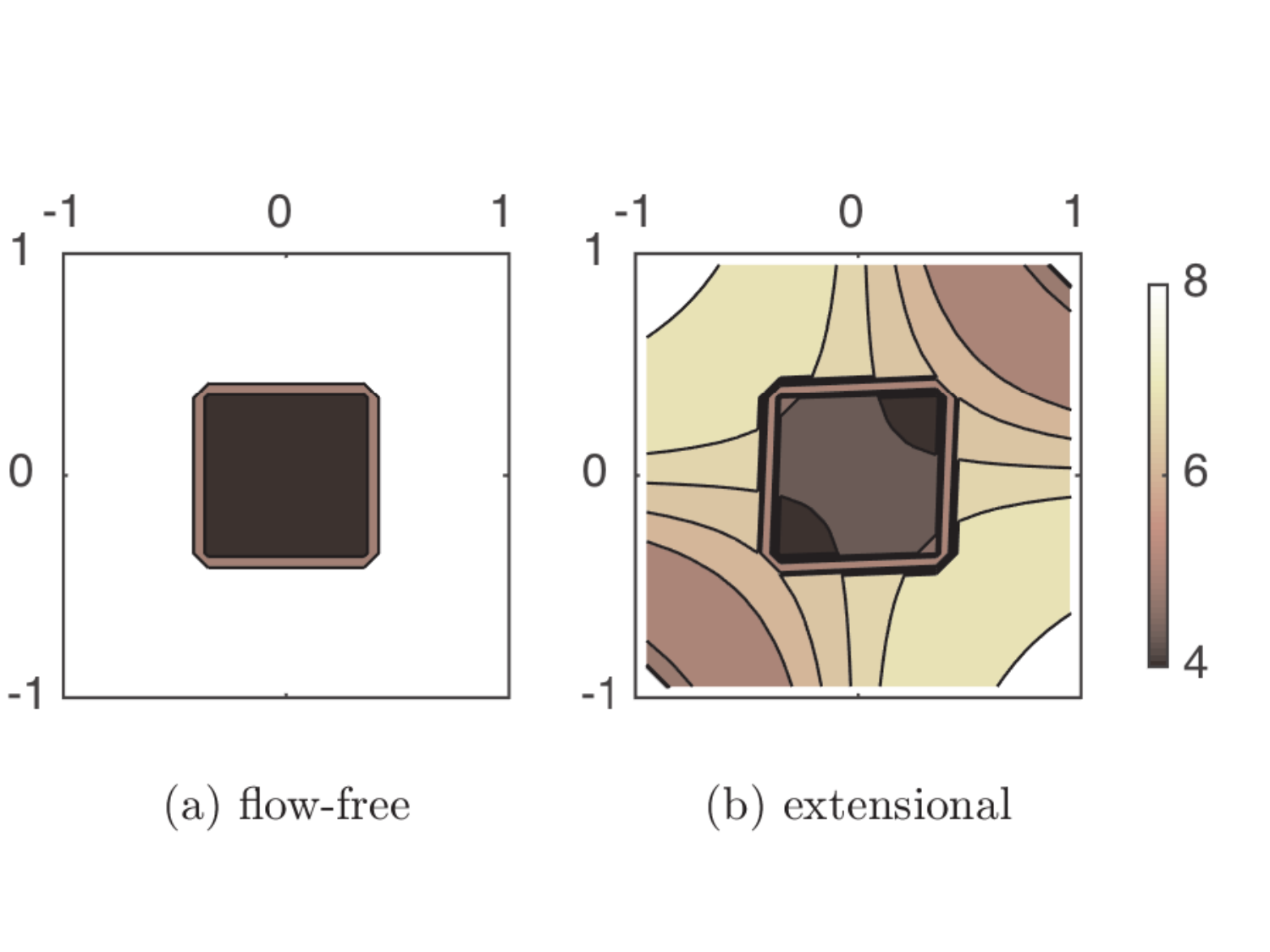} 
\end{center}
                \caption{\small {\bf Square Well: Stationary Densities (known).} 
                				These figures plot contours of the free energy of the numerical stationary density 
					for a particle in the potential given by \eqref{squarewell} with $\beta=1$, $d_1=0.4$, $d_2=-0.4$, $\epsilon=0.001$, 
					and the flow-free (left panel) and extensional flow (right panel) cases with $\gamma=2$.  For these flows a formula for the stationary
					density $\nu(x)$ is available, and the numerical densities shown are within $1 \%$ accurate in the $\ell_1$norm.}
                \label{fig:E3_stationary_densities_known}
\end{figure}

\begin{figure}[ht!]
\begin{center}
\includegraphics[width=\textwidth]{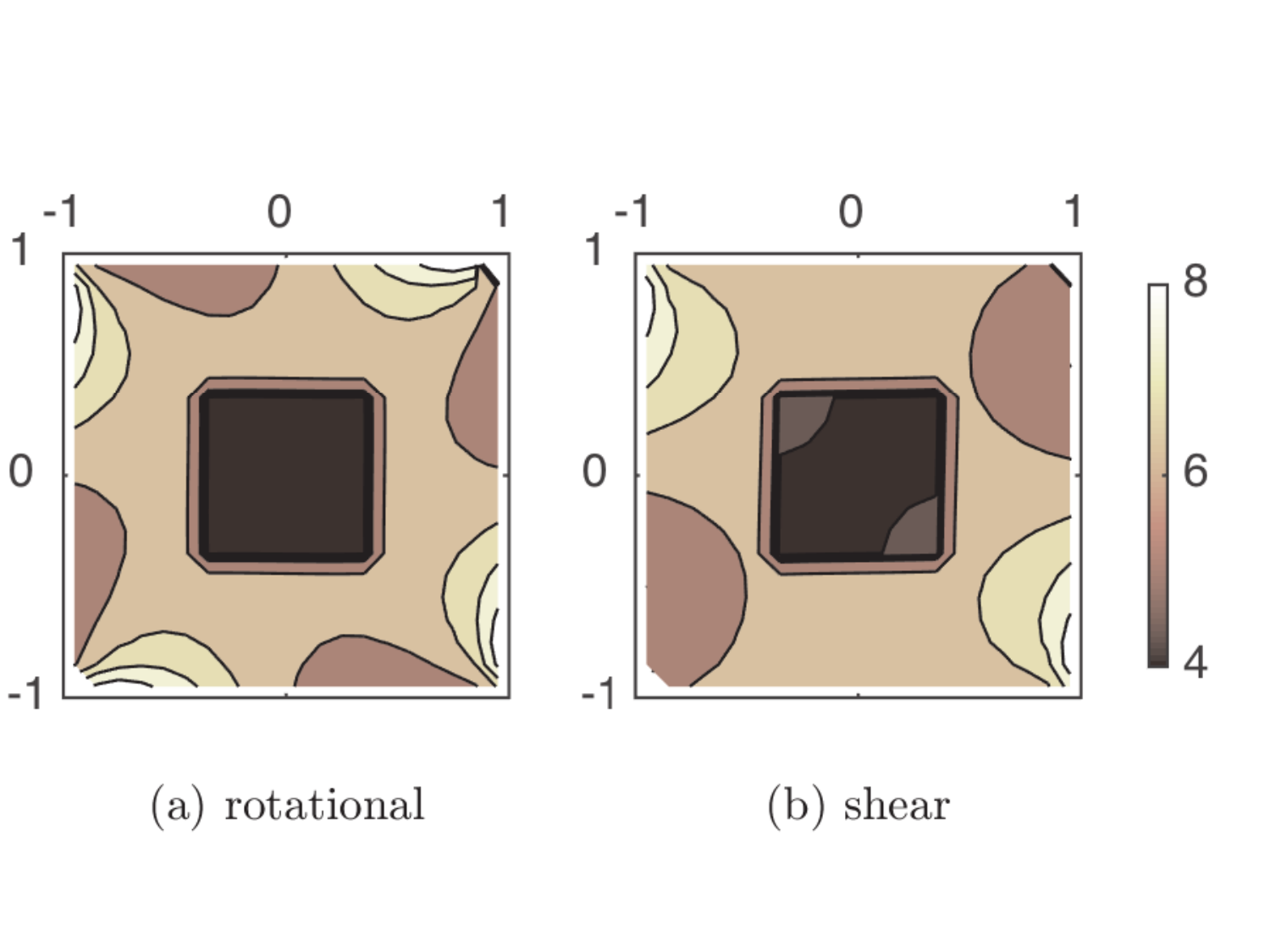}
\end{center}
                \caption{\small {\bf Square Well: Stationary Densities (unknown).} 
                				These figures plot contours of the free energy of the numerical stationary density 
					for a particle in the potential given by \eqref{squarewell} with $\beta=1$, $d_1=0.4$, $d_2=-0.4$, $\epsilon=0.001$, 
					and the shear (left panel) and rotational flow (right panel)  cases with $\gamma=2$.  For these flows a formula for the stationary
					density $\nu(x)$ is not available, but we verified that these numerical solutions are well-resolved.
					}
                \label{fig:E3_stationary_densities_unknown}
\end{figure}

\begin{figure}[ht!]
\begin{center}
\includegraphics[width=0.8\textwidth]{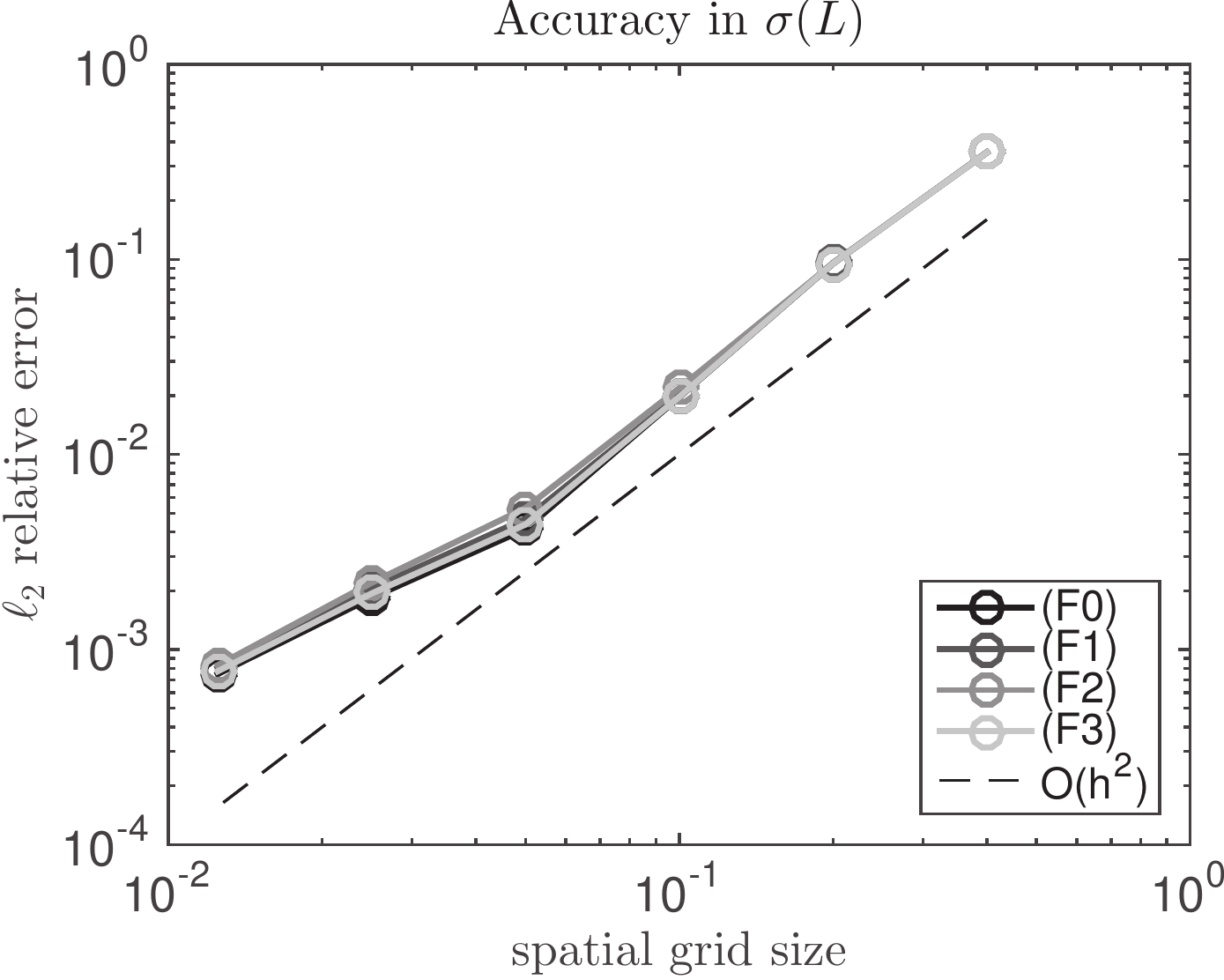}
\end{center}
                \caption{\small {\bf Square Well: Accuracy in $\sigma(L)$.}
                                                This figure plots the relative error in the $\ell_2$ norm of the numerical estimate produced by $\tilde Q_c$ of the twenty eigenvalues of largest nontrivial real part 
                                                of the operator $L$.   The parameters here are set at: $\gamma=1$ and $\beta=1$.  The figure suggests that the relative error
                                                of $\tilde Q_c$ in representing the spectrum of $L$ is about second-order in the spatial step size.  The slight divergence at small spatial step size is attributed to errors in the sparse matrix eigensolver used.
                }
                \label{fig:E3_sigma_L_accuracy}
\end{figure}


\subsection{Worm-Like Chain SDE: Singular Drift} \label{sec:wlc}

Consider \eqref{eq:sde_planar_additive} with the linear flow fields in \eqref{flows}, $\beta=1$, and the potential energy given by: \begin{equation} \label{wlc}
U(x) = \frac{x_2^2}{2} + \frac{1}{1-|x_1|}  - | x_1 | + 2 x_1^2  \;,  \qquad x = (x_1,x_2)^T  \in \Omega = (-1,1) \times \mathbb{R} \;,
\end{equation}
This system is an example of an SDE with a boundary and with a drift vector field that has singularities at the boundaries of $\Omega$: 
$(\pm 1,x_2)^T$ for all $x_2 \in \mathbb{R}$. 
Because of this singularity, standard integrators for \eqref{eq:sde_planar_additive} may diverge in finite time even in the flow-free case.  
Contour lines of the numerical stationary density $\nu(x)$ are plotted in 
Figure~\ref{fig:E4_stationary_densities_known} and \ref{fig:E4_stationary_densities_unknown} when
the stationary density is known and unknown, respectively, and for a flow rate of $\gamma=5$.    We choose a high flow rate to make the numerical test more rigorous.
The accuracy of the spectrum of the generator $\tilde Q_c$ is assessed in Figure~\ref{fig:E4_sigma_L_accuracy}.  This figure suggests that $\tilde Q_c$ is second-order accurate in representing the spectrum of the generator $L$.  


\begin{figure}[ht!]
\begin{center}
\includegraphics[width=\textwidth]{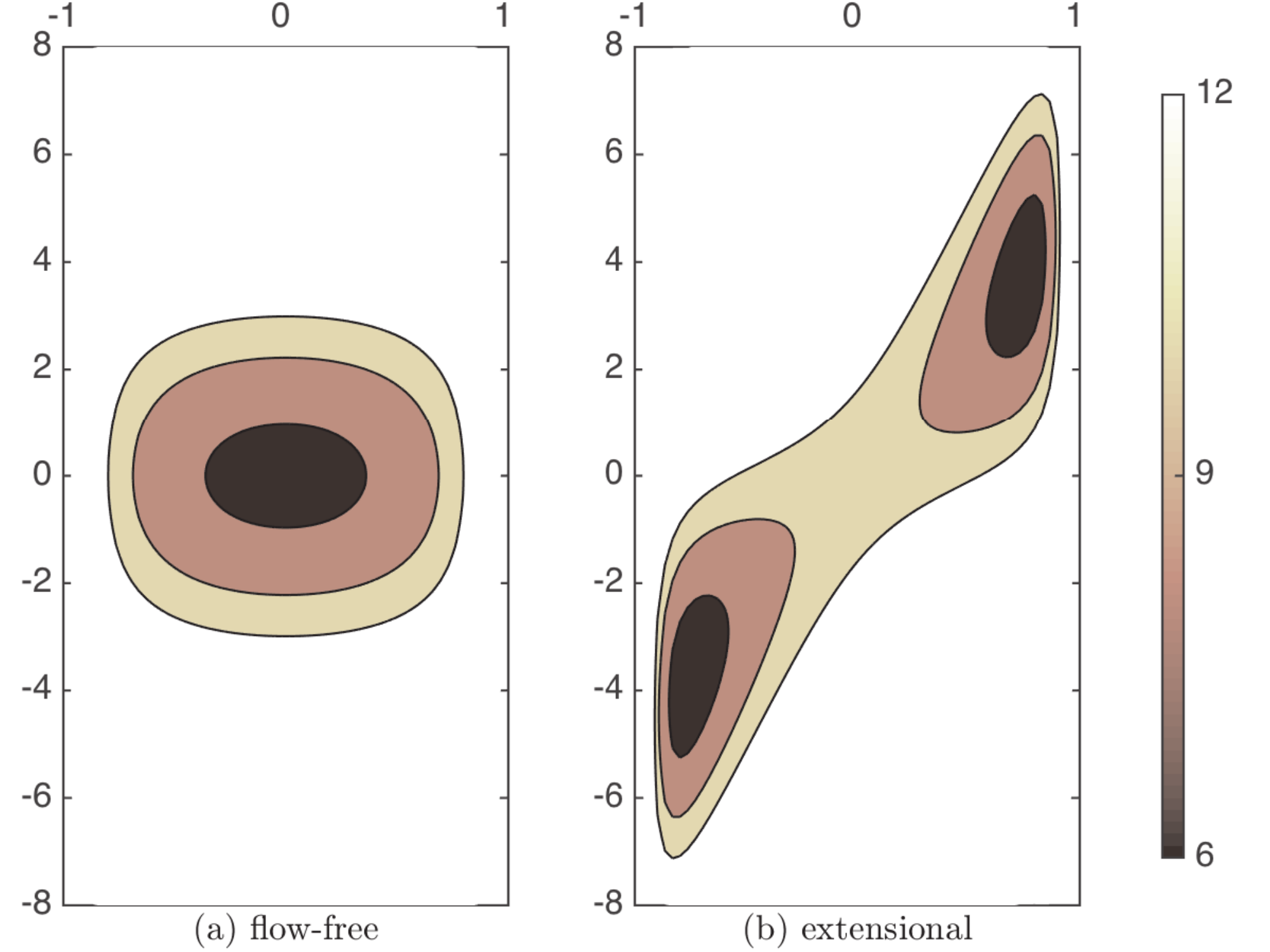} 
\end{center}
                \caption{\small {\bf Worm-Like Chain:  Stationary Densities (Known).} 
                				These figures plot contours of the stationary density 
					for a particle in the potential given by \eqref{wlc},
					and from left, flow-free and the extensional flow
					given in \eqref{flows} with $\gamma=5$ and $\beta=1$.  For these flows a formula for the stationary
					density $\nu(x)$ is available, and the numerical densities shown are within $1 \%$ accurate in the $\ell_1$norm.
                }
                \label{fig:E4_stationary_densities_known}
\end{figure}

\begin{figure}[ht!]
\begin{center}
\includegraphics[width=\textwidth]{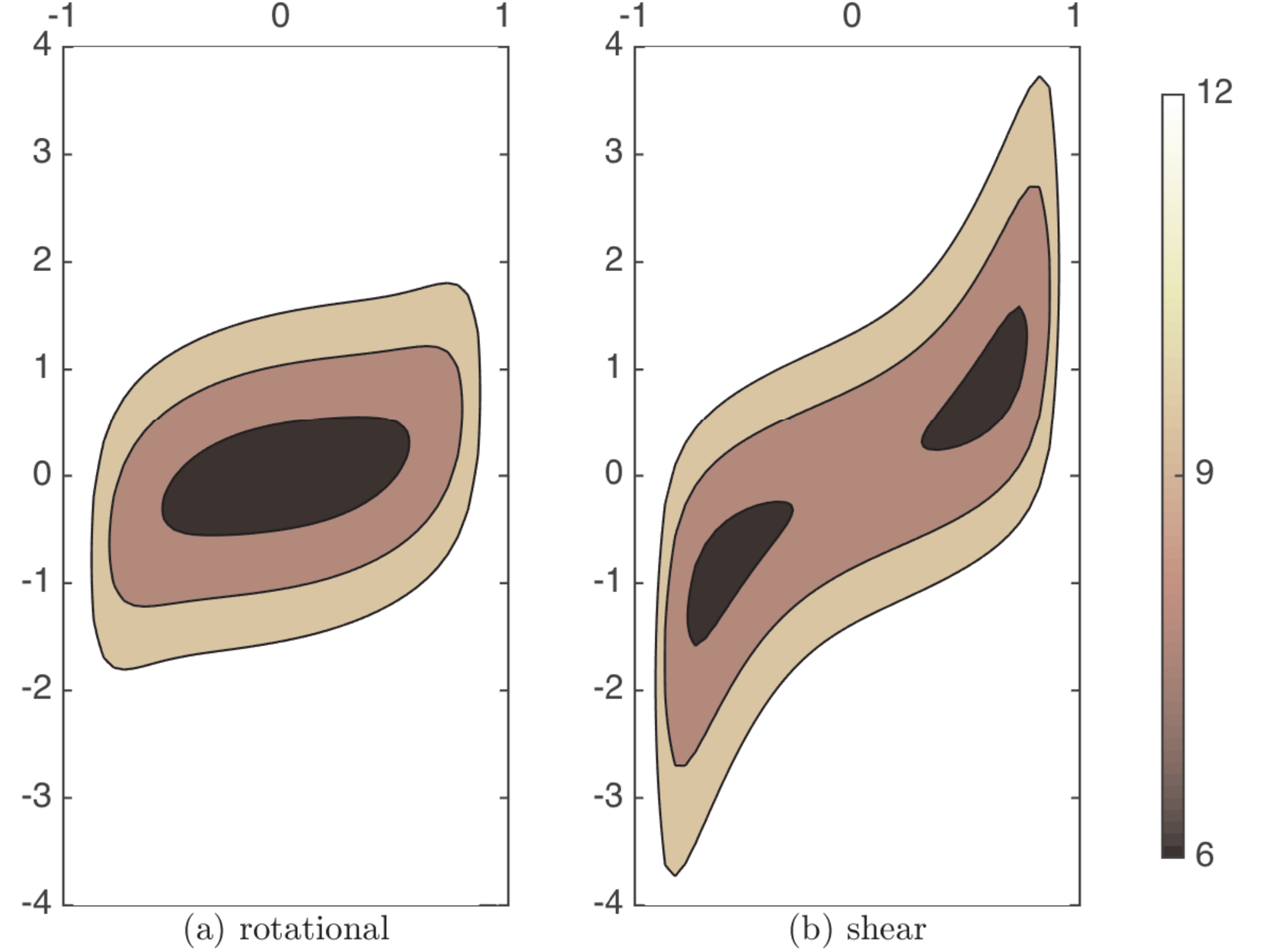}
\end{center}
                \caption{\small {\bf Worm-Like Chain:  Stationary Densities (Unknown).} 
                				These figures plot contours of the free energy of the numerical stationary density 
					for a particle in the potential given by \eqref{wlc},
					and from left, rotational and shear flows
					given in \eqref{flows} with $\gamma=5$ and $\beta=1$. 
                }
                \label{fig:E4_stationary_densities_unknown}
\end{figure}

\begin{figure}[ht!]
\begin{center}
\includegraphics[width=0.8\textwidth]{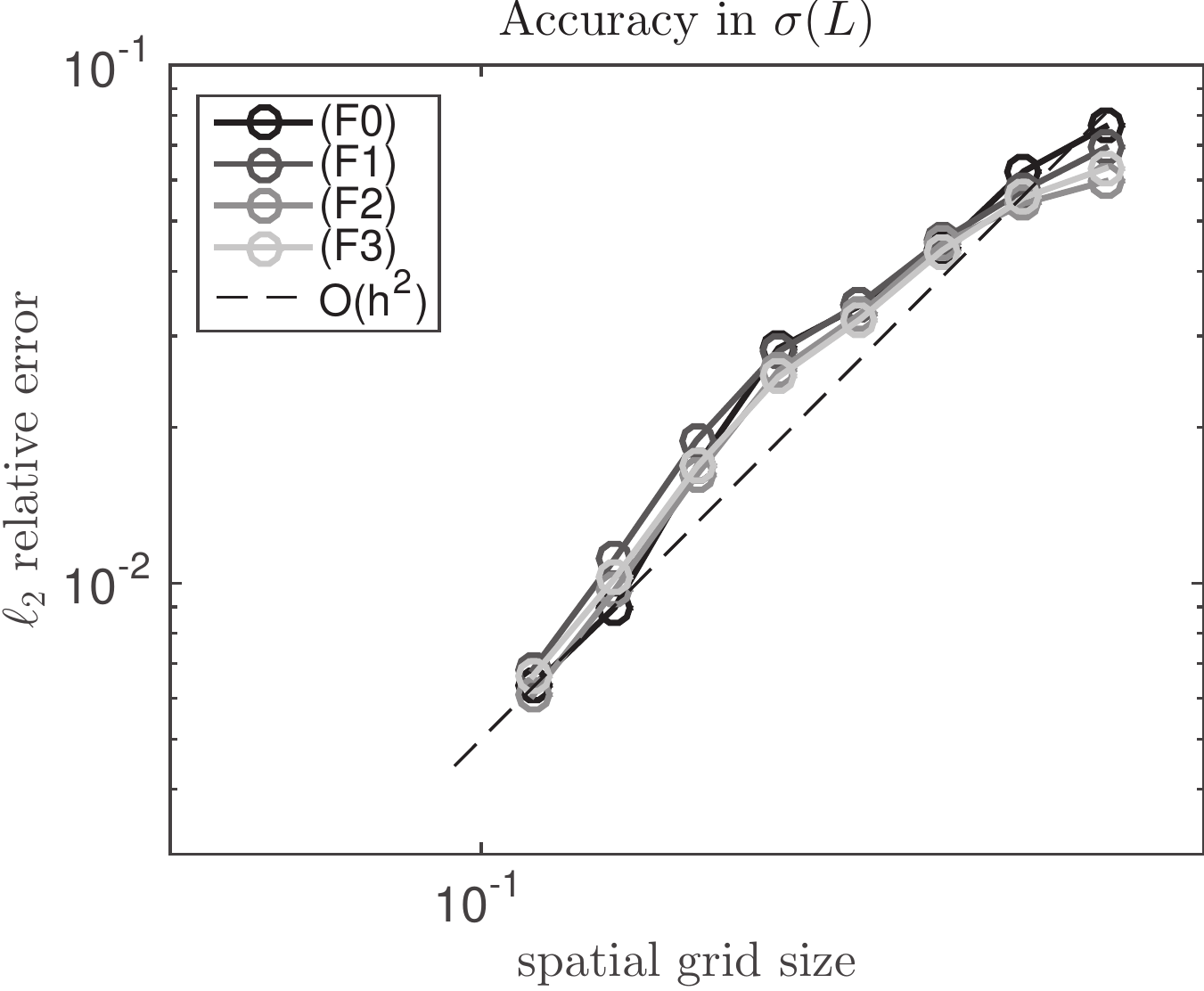}
\end{center}
                \caption{\small {\bf Worm-Like Chain: Accuracy in $\sigma(L)$.}
                                               This figure plots the relative error in the $\ell_2$ norm of the numerical estimate produced by $\tilde Q_c$ of the twenty eigenvalues of largest nontrivial real part 
                                               of the operator $L$.   Parameters are set at $\gamma=1$ and $\beta=1$.   This figure suggests that $\tilde Q_c$ is about second-order accurate
					in representing the spectrum of the generator $L$.  
                }
                \label{fig:E4_sigma_L_accuracy}
\end{figure}


\section{Adaptive Mesh Refinement in 2D} \label{sec:amr_2d}

For the subsequent planar SDE problems with multiplicative noise, we use a variable step size grid with the property that the generator $\tilde Q_c$ in \eqref{eq:tQc_2d} evaluated on this grid is realizable.  Recall, if the diffusion matrix is diagonally dominant, a grid with equal spacing in the horizontal and vertical directions suffices.  If the diffusion matrix is constant, but not diagonally dominant,  the ratio of the spacing in the horizontal and vertical directions can be chosen so that the grid has the desired property, since a $2 \times 2$, symmetric, and positive-definite matrix is always weakly diagonally dominant \cite{BoChPaTo2005}.  In general, if the diffusion matrix is neither diagonally dominant nor constant, it is not straightforward to construct a grid such that $\tilde Q_c$ is realizable on this grid.  However, in the specific context of the log-normal and Lotka-Volterra examples considered in this paper, the following lemma states that a grid with the desired property exists.  The key idea in this construction is to make the grid evenly spaced in logspace.    

To state this lemma, define: \[
\diag(x) = \begin{bmatrix} x_1 & 0 \\ 0 & x_2 \end{bmatrix} \;, \quad \text{for any $x = (x_1,x_2)^T \in \mathbb{R}^2$} \;.
\]

\begin{lemma} \label{lem:log_grid_2d}
Consider the generator $\tilde Q_c$ in \eqref{eq:tQc_2d} with domain $\mathbb{R}^2_+$, an arbitrary drift field, and a diffusion matrix of the form: \begin{equation} \label{eq:diffusion_matrix_planar}
M(x) = \diag(x) \begin{bmatrix} M^{11} & M^{12} \\ M^{12} & M^{22} \end{bmatrix} \diag(x) \;.
\end{equation}
Assume that \begin{equation} \label{eq:detM_condition}
M^{12} \ne 0 \;, \quad \det \begin{bmatrix} M^{11} & M^{12} \\ M^{12} & M^{22} \end{bmatrix}  > 0 \;, \quad M^{11} + M^{22} > 0 \;.
\end{equation}
In this context, $\tilde Q_c$ in \eqref{eq:tQc_2d} is realizable on a gridded state space.
\end{lemma}

\begin{proof}
To prove this lemma we construct a grid such that $\tilde Q_c$ in \eqref{eq:tQc_2d} is realizable on this grid.  Let $S = \{ (x_i, y_j) \} \subset \mathbb{R}^2_+$ be the desired grid.  We define $S$ to be the image of an evenly spaced grid in logspace: \[
\hat S = \{ (\xi_i, \eta_j) \in \mathbb{R}^2  \mid \quad \delta \xi = \xi_{i+1} - \xi_i \;, \quad  \delta \eta = \eta_{j+1} - \eta_j \}  
\] where $\delta \xi $ and $\delta \eta$ set the grid size in the horizontal and vertical directions in logspace, respectively.  We choose the spacing in logspace so that $S$ has the desired property in physical space.  Specifically, map every point $(\xi_i, \eta_j) \in \hat S$ to a point in $(x_i,y_j) \in S$ using the transformation: $x_i = \exp(\xi_i)$ and $y_j = \exp(\eta_j)$.    Suppose that \[
\delta \xi = \alpha \epsilon \quad \text{and} \quad \delta \eta = \epsilon
\] where $\epsilon$ and $\alpha$ are positive parameters.  We choose these parameters so that $\tilde Q_c$ evaluated on $S$ is realizable.  Note that the spacing between neighboring horizontal and vertical grid points in physical space satisfy: \begin{equation}
\label{eq:horizontal_grid_spacing}
\begin{cases}
\delta x_i^+ = x_{i+1} - x_i = ( \exp( \alpha \epsilon) - 1) x_i \\
\delta x_i^- = x_{i} - x_{i-1} = ( 1 - \exp(-  \alpha \epsilon)) x_i  \\
\delta x_i = \dfrac{\delta x_i^+ + \delta x_i^-}{2} = \sinh( \alpha \epsilon) x_i
\end{cases}
\end{equation} and  \begin{equation} \label{eq:vertical_grid_spacing}
\begin{cases}
\delta y_j^+ = y_{j+1} - y_j = ( \exp( \epsilon) - 1) y_j \\
\delta y_j^- = y_{j} - y_{j-1} = ( 1 - \exp(-\epsilon)) y_j  \\
\delta y_j = \dfrac{\delta y_j^+ + \delta y_j^-}{2} = \sinh(\epsilon) y_j
\end{cases}
\end{equation}

Referring to \eqref{eq:tQc_2d}, a necessary and sufficient condition for realizability of $\tilde Q_c$ is that all of the following inequalities hold: \begin{equation} \label{eq:tQc_conditions}
\begin{aligned}
& \frac{M^{11}_{i,j}}{ \delta x_i \delta x_i^+ }  - \frac{M^{12}_{i,j} \vee 0}{ \delta x_i^+ \delta y_j^+ } + \frac{M^{12}_{i,j} \wedge 0}{ \delta x_i^+ \delta y_j^- } \ge 0 \\
& \frac{M^{11}_{i,j}}{ \delta x_i \delta x_i^- }  - \frac{M^{12}_{i,j} \vee 0}{ \delta x_i^- \delta y_j^- } + \frac{M^{12}_{i,j} \wedge 0}{ \delta x_i^- \delta y_j^+ } \ge 0 \\
& \frac{M^{22}_{i,j}}{ \delta y_j \delta y_j^+ }  - \frac{M^{12}_{i,j} \vee 0}{ \delta x_i^+ \delta y_j^+ } + \frac{M^{12}_{i,j} \wedge 0}{ \delta x_i^- \delta y_j^+ } \ge 0 \\
&\frac{M^{22}_{i,j}}{ \delta y_j \delta y_j^- } - \frac{M^{12}_{i,j} \vee 0}{ \delta x_i^- \delta y_j^- } + \frac{M^{12}_{i,j} \wedge 0}{ \delta x_i^+ \delta y_j^- } \ge 0 
\end{aligned}
\end{equation}
Let us express these inequalities in terms of $\alpha$ and $\epsilon$.  Substitute \eqref{eq:diffusion_matrix_planar}  and the grid spacings \eqref{eq:horizontal_grid_spacing} and \eqref{eq:vertical_grid_spacing} into \eqref{eq:tQc_conditions} and simplify to obtain: \begin{equation} \label{eq:tQc_conditions_simplified}
\begin{aligned}
& \frac{M^{11}}{ \sinh(\alpha \epsilon) }  - \frac{M^{12} \vee 0}{ \exp(\epsilon) - 1 } + \frac{M^{12} \wedge 0}{ 1 - \exp(-\epsilon) } \ge 0 \\
& \frac{M^{11}}{ \sinh(\alpha \epsilon) }  - \frac{M^{12} \vee 0}{ 1 - \exp(-\epsilon) } + \frac{M^{12} \wedge 0}{ \exp(\epsilon) - 1} \ge 0 \\
& \frac{M^{22}}{\sinh( \epsilon) }  - \frac{M^{12} \vee 0}{  \exp(\alpha \epsilon) - 1 } + \frac{M^{12} \wedge 0}{ 1 - \exp(-\alpha \epsilon) } \ge 0 \\
&\frac{M^{22}}{ \sinh( \epsilon) } - \frac{M^{12} \vee 0}{ 1 - \exp(-\alpha \epsilon) } + \frac{M^{12} \wedge 0}{ \exp(\alpha \epsilon) - 1 } \ge 0 
\end{aligned}
\end{equation}
Applying the elementary inequality \[
\exp(a) - 1  > 1 - \exp(-a) \quad \forall ~a \in \mathbb{R}
\] to \eqref{eq:tQc_conditions_simplified} yields the following sufficient condition for realizability \begin{equation} \label{eq:tQc_conditions_simplified_2}
\begin{aligned}
 \frac{M^{11}}{ |M^{12}|  }  \ge \frac{\sinh(\alpha \epsilon)}{1 - \exp(-\epsilon) }  \quad \text{and} \quad \frac{ 1 - \exp(-\alpha \epsilon)}{\sinh( \epsilon) }  \ge \frac{| M^{12} |}{M^{22} }  
\end{aligned}
\end{equation}
The hypothesis in \eqref{eq:detM_condition} implies that \[
M^{11}>0 \;, \quad M^{22} > 0 \quad \text{and} \quad \frac{M^{11}}{|M^{12}|}  - \frac{|M^{12}|}{M^{22}} > 0 \;.
\]  Hence, we can choose $\alpha$ to satisfy: \[
\frac{M^{11}}{|M^{12}|} > \alpha > \frac{|M^{12}|}{M^{22}} >0 \;.
\]  This choice of $\alpha$ implies \eqref{eq:tQc_conditions_simplified_2} holds if $\epsilon$ is small enough.  Indeed, a straightforward application of l'H\^{o}pital's Monotonicity Rule implies that \begin{equation} \label{eq:amr_2d_inequality}
\overline \phi(\epsilon) = \frac{\sinh(\alpha \epsilon)}{1 - \exp(-\epsilon) }  \ge \alpha \ge \frac{ 1 - \exp(-\alpha \epsilon)}{\sinh( \epsilon) } = \underline \phi(\epsilon) 
\end{equation} for any $\epsilon > 0$.  Specifically, to prove the upper bound, note that: \[
\lim_{s \to 0} \overline \phi(s) = \alpha \;.
\]  Moreover, l'H\^{o}pital's Monotonicity Rule implies that the function $\overline \phi(s)$ is increasing since the function: \[
s \mapsto \alpha \cosh(\alpha s) \exp(s)
\] is increasing with $s$.  Thus, $\overline \phi(s)$ is increasing.  Similarly, $\underline \phi(s)$ is decreasing since the function: \[
s \mapsto \alpha \exp(-\alpha s) \sech(s)
\] is decreasing with $s$, and $\lim_{s\to0} \underline \phi(s) = \alpha$.
\end{proof} 

\section{Log-normal Process in 2D with Multiplicative Noise} \label{sec:lognormal_2d}

We derive a log-normal process by transforming a two-dimensional Ornstein-Uhlenbeck process.    Thus, we obtain a log-normal process whose transition and equilibrium probability distributions are explicitly known and can be used as benchmarks to validate the approximation given by $\tilde Q_c$ in \eqref{eq:tQc_2d}.  First, let us recall some basic facts about Ornstein-Uhlenbeck processes.  Consider a two-dimensional, non-symmetric Ornstein-Uhlenbeck process \begin{equation} \label{eq:asymmetric_ou_2d}
d Y = A Y dt + \sqrt{2} B dW \;, ~~ Y(0) = x \in \mathbb{R}^2 \;.
\end{equation}   Let $M$ be the diffusion matrix defined as: \[
M  = B B^T = \begin{bmatrix} M^{11} & M^{12} \\ M^{12} & M^{22} \end{bmatrix} \;.
\] Recall that the transition probability of $Y(t)$ is a two-dimensional Gaussian distribution with mean vector $\xi(x,t)$ and covariance matrix $\Sigma(t)$ given by: \begin{align*}
& \xi(x,t) =  \exp(t A) x \;, \\
& \Sigma(t) =2 \int_0^t \exp(s A)  M \exp(s A^T) ds  
\end{align*}  
See, e.g., \cite[Section 4.8]{IkWa1989}.
To numerically compute $\Sigma(t)$ for any $t>0$, observe that the covariance matrix $\Sigma(t)$ satisfies the Lyapunov equation: \begin{align*}
\Sigma(t) A^T + A \Sigma(t) &= 2 \left( \int_0^t \frac{d}{ds} \exp(s A)  M \exp(s A^T) ds \right) \;, \\
& = 2 \left( \exp(t A) M (\exp(t A))^T  - M  \right) \;.
\end{align*}
The MATLAB command {\tt lyap} can be used to solve equations of this type.  Specifically, given the matrices $\mathcal{A}$ and $\mathcal{B}$ as inputs, {\tt lyap} solves for the unknown matrix $\mathcal{X}$ which satisfies the equation:  \[
\mathcal{A} \mathcal{X} + \mathcal{X}  \mathcal{A}^T + \mathcal{B} = 0 \;.
\]
Defining the inputs to {\tt lyap} as: \[
 \mathcal{A} = A \;,~~\mathcal{B} = -2 \left( \exp(t A) M (\exp(t A))^T  - M  \right) \;,
\] the function will output the desired covariance matrix.   (A sufficient condition for a unique solution to this Lyapunov equation is that all of the eigenvalues of the matrix $\mathcal{A}$ have negative real parts.)  
If the eigenvalues of the matrix $A$ in \eqref{eq:asymmetric_ou_2d} all have strictly negative real part, passing to the limit in the above expressions as $t\to \infty$ implies that the equilibrium distribution is Gaussian with zero mean vector and a covariance matrix $\Sigma_{\infty}$ that satisfies the Lyapunov equation: \begin{equation} \label{lyapunov}
\Sigma_{\infty} A^T + A \Sigma_{\infty}  = - 2 M \;.
\end{equation}

Consider the transformation of $Y(t) = (Y_1(t), Y_2(t))^T$ given by \[
X = ( \exp(Y_1), \exp( Y_2) )^T \;.
\] Under this transformation, the linear SDE \eqref{eq:asymmetric_ou_2d} becomes: \begin{equation} \label{eq:lognormal_2d}
d X = \mu(X) dt +  \sqrt{2}  \sigma(X) dW \;, ~~ X(0) \in \mathbb{R}_+^2 
\end{equation}
where the drift and diffusion fields are: \begin{align*}
\mu(x) &= \begin{bmatrix}  
M^{11} x_1 \\
M^{22} x_2
\end{bmatrix} + \diag(x,y) A \begin{bmatrix} \log(x) \\ \log(x_2) \end{bmatrix} \\
M(x)  &= ( \sigma  \sigma^T)(x) \\
&= \diag(x_1,x_2) \begin{bmatrix} M^{11}  & M^{12}  \\  M^{12}  & M^{22} \end{bmatrix} \diag(x_1,x_2) = \begin{bmatrix} M^{11} x_1^2 & M^{12} x_1 x_2 \\  M^{12} x_1 x_2 & M^{22} x_2^2 \end{bmatrix}
\end{align*}
By change of variables, the (log-normal) probability and stationary densities of the transformed process $X(t)$ are given by: \begin{equation} \label{eq:lognormal_densities_2d}
\left.\begin{aligned}
\tilde p_t(x) &= p_t(\log(x_1), \log(x_2)) (x_1 x_2 )^{-1}  \\
 \tilde \nu(x) &= \nu(\log(x_1), \log(x_2)) (x_1 x_2 )^{-1}
 \end{aligned} \right\}
 \qquad x = (x_1, x_2)^T
\end{equation} where $p_t(x)$ and $\nu(x)$ are the (normal) probability and stationary densities of the Ornstein-Uhlenbeck process $Y(t)$, respectively.  

\medskip

Figure~\ref{fig:ln_process} illustrates stationary density accuracy of $\tilde Q_c$ in \eqref{eq:tQc_2d} using the variable step size grid derived in the preceding Lemma~\ref{lem:log_grid_2d}.  

\begin{figure}[ht!]
\begin{center}
\includegraphics[width=\textwidth]{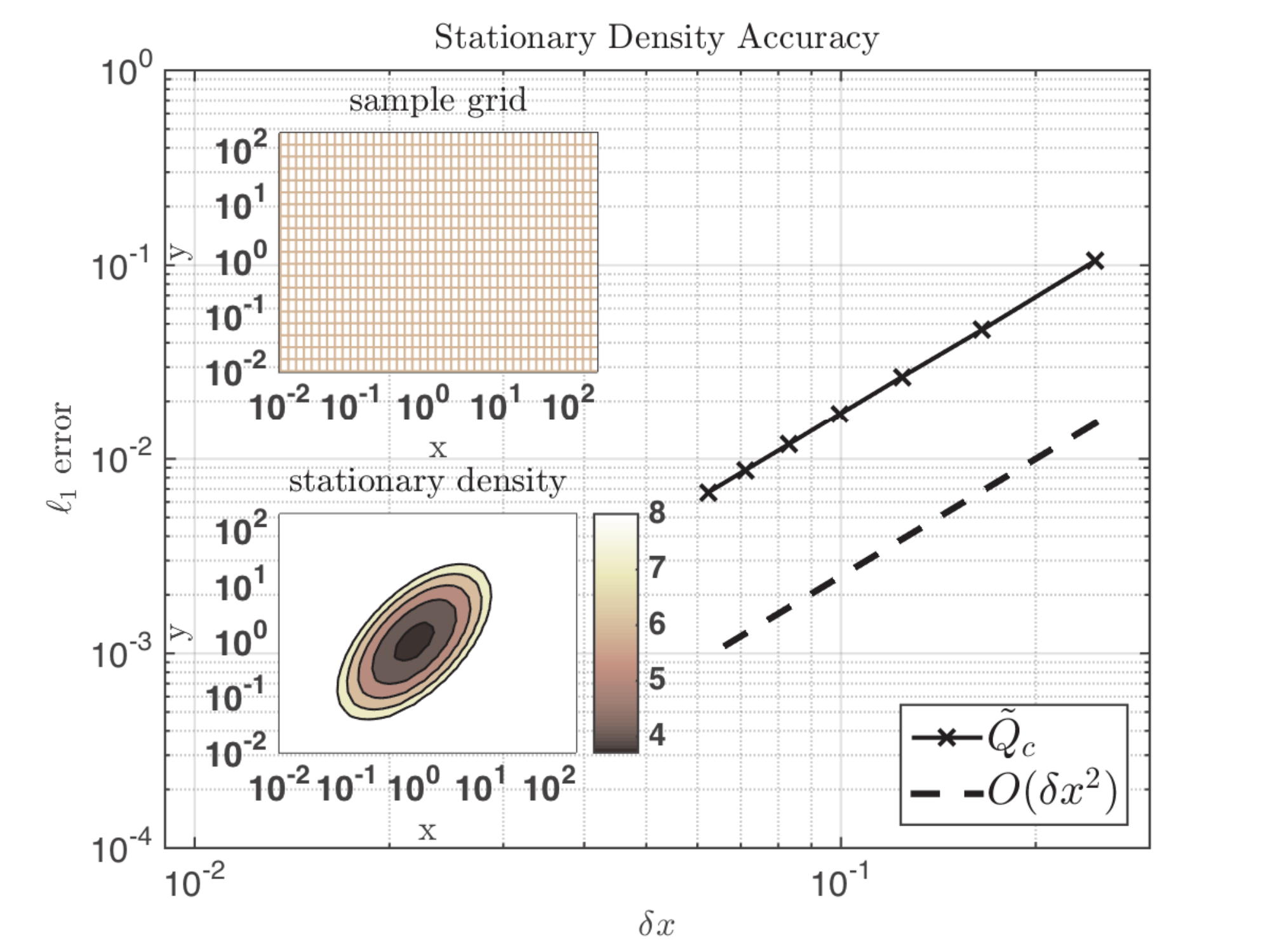}  
\end{center}
\caption{ \small {\bf Log-normal Process in 2D.}    
	This figure illustrates stationary density accuracy of $\tilde Q_c$ in \eqref{eq:tQc_2d}.  
	The grid used is evenly spaced in logspace as shown in the upper left inset.  Contours of the stationary density are plotted in the lower left inset.   
	The spatial step size shown is with respect to the spacing of the grid in logspace.   
	In this example the benchmark solution is the true stationary density given by the formula in \eqref{eq:lognormal_densities_2d}.
	See Lemma~\ref{lem:log_grid_2d} for more details about the grid.
}
\label{fig:ln_process}
\end{figure}

\section{Lotka-Volterra Process in 2D with Multiplicative Noise} \label{sec:lv_process}

Consider \eqref{eq:sde} with $\Omega=\mathbb{R}^2_+$ and set: \begin{align*}
\mu(x) = \begin{bmatrix} 
  k_1 x_1 - x_1 x_2 - \gamma_1 x_1^2 \\
 -k_2 x_2 + x_1 x_2 - \gamma_2 x_2^2 
 \end{bmatrix} \;, \quad \text{and} \quad
 (  \sigma   \sigma^T)(x) = \diag(x) \begin{bmatrix} M^{11} & M^{12} \\ M^{12} & M^{22} \end{bmatrix} \diag(x) \;.
\end{align*}  
This SDE is a model of interaction between a population of $x_1$ prey and $x_2$ predators that includes environmental fluctuations.   Recall that when $\gamma_1=\gamma_2=0$, the noise-free system reduces to an ODE with a fixed point at $(k_2,k_1)$ and a first integral given by the following function: $V(x) = k_2 \log(x_1)+k_1 \log(x_2)- (x_1+x_2)$. The diffusion field of this SDE satisfies the hypotheses in Lemma~\ref{lem:log_grid_2d}.  Thus, we will be able to construct a grid with the property that $\tilde Q_c$  in \eqref{eq:tQc_2d} evaluated on this grid is realizable.  

We use some known theoretical results on this SDE problem to benchmark the generator $\tilde Q_c$ \cite{RuPi2006}.  For this purpose, it helps to define the parameters: $c_1 = k_1 - \frac{1}{2} M^{11}$ and $c_2 = k_2 + \frac{1}{2} M^{22}$. According to Theorem 1 of \cite{RuPi2006}, for every $t>0$ the probability distribution of the SDE is absolutely continuous with respect to Lebesgue measure on $\Omega$.  This property implies that the solution is strictly positive for all $t>0$.  This positivity, however, does not prevent the probability distribution of the solution from converging in the long time limit to a measure supported only on the boundaries of $\Omega$, in which case one or both of the populations may become extinct.  In fact, the finite-time probability distribution of the SDE solution converges to an invariant probability measure (IM) with the following support properties: 

\medskip

\begin{description}
\item[If $c_1 > \gamma_1 c_2 $] then the IM is supported on $\mathbb{R}^2_+$.
\item[If $ \gamma_1 c_2 > c_1 > 0$] then the IM is supported on $\mathbb{R}_+ \times \{ 0 \}$ with marginal density of $x_1$ given by: \begin{equation} \label{eq:marginal_in_x}
f_{\star}(x_1) = Z^{-1} x_1^{ \frac{2 c_1}{M^{11}} - 1} \exp \left( - \frac{2 x_1}{M^{11}} \right) 
\end{equation} where $Z$ is the following normalization constant: \[
Z = \left( \frac{M^{11}}{2} \right)^{ \frac{2 c_1}{M^{11}}}  \Gamma\left( \frac{2 c_1}{M^{11}} \right)  \;.
\]
\item[If $c_1 < 0$] then the IM is atomic at the origin.
\end{description}

\medskip

Note that if $2 c_1 / M^{11} < 1$ (or $k_1 < M^{11}$), then the marginal stationary density in \eqref{eq:marginal_in_x} is unbounded at $x=0$.   Figures~\ref{fig:lv_process_c1gc2} and~\ref{fig:lv_process_c1lc2} use $\tilde Q_c$ in \eqref{eq:tQc_2d} with the variable step size grid described in Lemma~\ref{lem:log_grid_2d} to numerically verify these theoretical results. 


\begin{figure}[ht!]
\begin{center}
\includegraphics[width=0.9\textwidth]{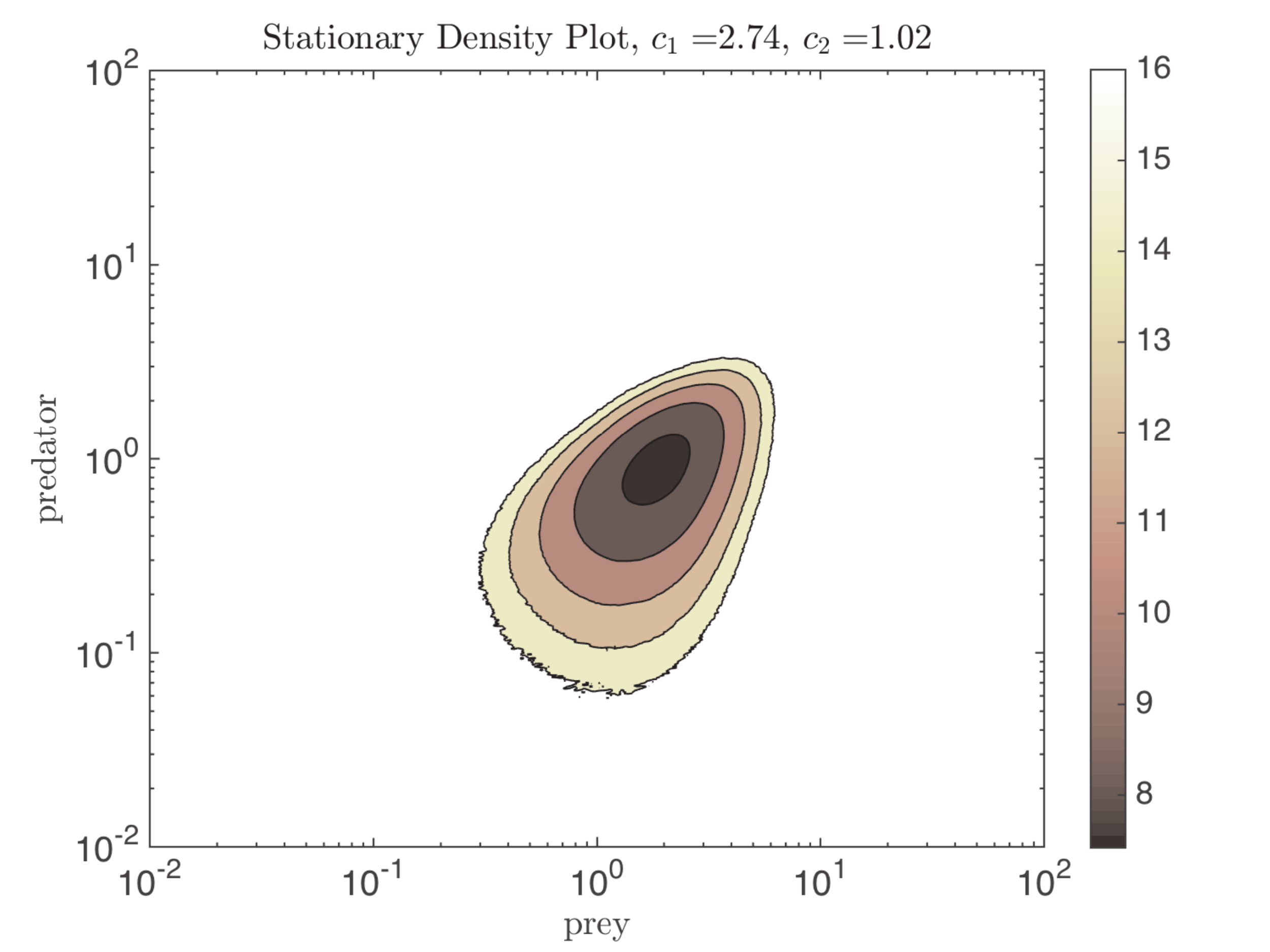}  
\end{center}
\caption{ \small {\bf Lotka-Volterra (LV) Process.}    
	      This figure plots contours of the numerical stationary density for the LV process produced by the generator $\tilde Q_c$ in \eqref{eq:tQc_2d} 
	      with the variable step size grid described in Lemma~\ref{lem:log_grid_2d}. The parameters are selected so that the true stationary distribution is 
	      supported on $\mathbb{R}^2_+$. This figure is in agreement with Theorem 1 of \cite{RuPi2006}.  
}
\label{fig:lv_process_c1gc2}
\end{figure}

\begin{figure}[ht!]
\begin{center}
\includegraphics[width=0.9\textwidth]{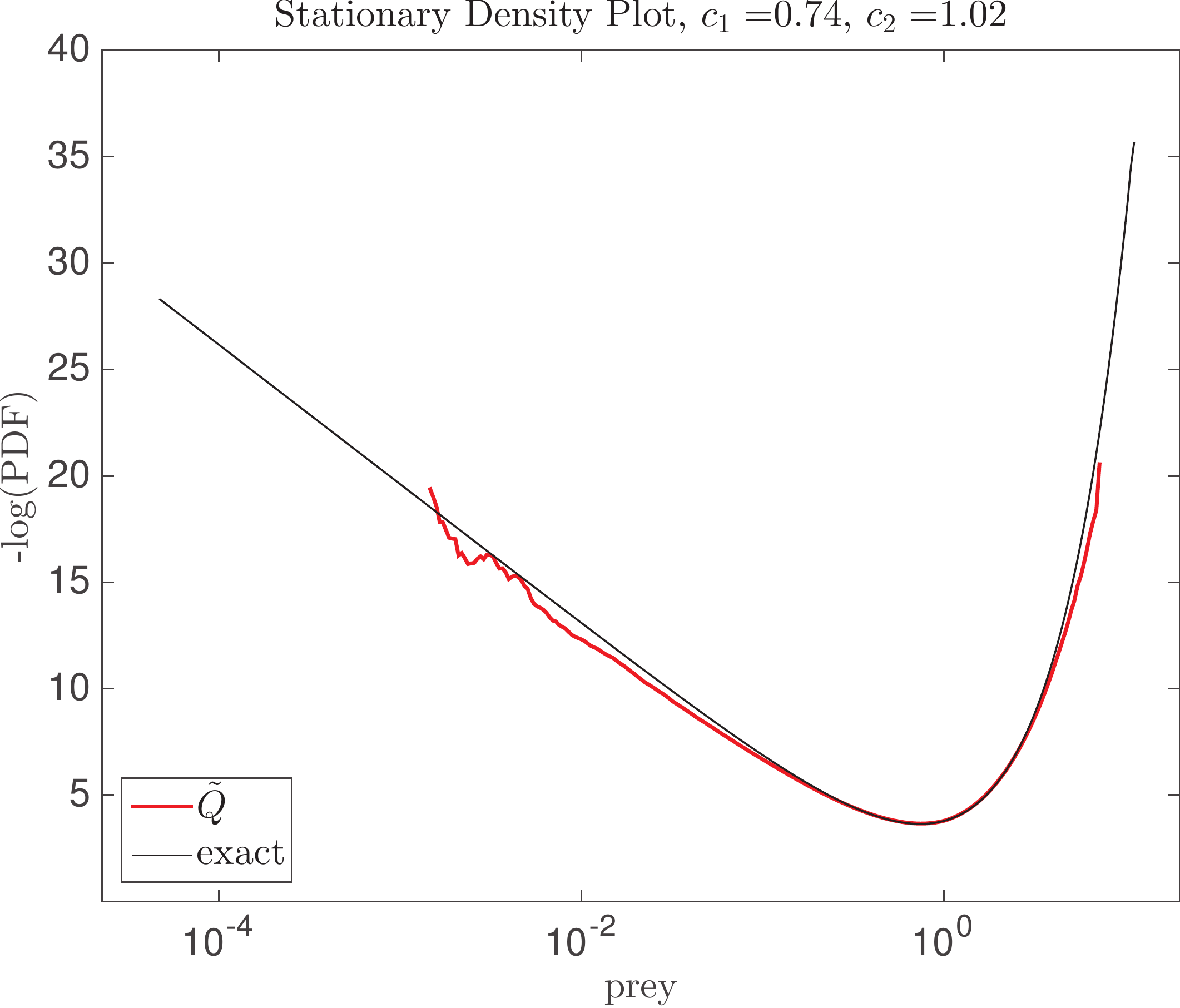}  
\end{center}
\caption{ \small {\bf Lotka-Volterra (LV) Process.}    
	      Here we numerically verify Theorem 1 of \cite{RuPi2006} using  the generator $\tilde Q_c$ in \eqref{eq:tQc_2d} 
	      with the variable step size grid described in Lemma~\ref{lem:log_grid_2d}.
              The figure plots the density of the stationary distribution for the LV process in a parameter regime where it is supported on $\mathbb{R}_+ \times \{ 0 \}$.  
	      In this case in the long time limit, the `predator population' (or $x_2$ degree of freedom) becomes extinct,
                and an explicit formula for the stationary density in the `prey population' (or $x_1$ degree of freedom) is given in \eqref{eq:marginal_in_x}, which we use
                to benchmark the numerically produced solution.
}
\label{fig:lv_process_c1lc2}
\end{figure}

\section{Colloidal Cluster in 39D with Multiplicative Noise}
 \label{sec:colloidal_cluster}

Consider a Brownian Dynamics (BD) simulation of a cluster consisting of $13$ particles governed by the following SDE: \begin{equation} \label{sde_bd}
\begin{aligned}
& d Y = - M(Y) D\mathcal{E}(Y) dt + \beta^{-1} \divergence M(Y) dt + \sqrt{2 \beta^{-1}} \sigma(Y) dW \;, ~~ Y(0) \in \mathbb{R}^n \;, \\
& \sigma(x) \sigma(x)^T = M(x) \;,
\end{aligned}
\end{equation}
where $\mathcal{E}(x)$ is the potential energy of the system, $M(x)$ is the mobility matrix (or what we have been calling the diffusion matrix), and $W$ is a $n$-dimensional Brownian motion.  To be clear, the total dimension of this system is $n=3 \times 13 = 39$.  This BD model of a colloidal gel demonstrates that  hydrodynamic interactions can influence the collapse dynamics of the cluster \cite{FuTa2010} as verified below.

The potential energy $\mathcal{E}(x)$ involves short-range, pair-wise interactions between particles.   It features a very steep hard core with a short-range attractive tail.    When hydrodynamic interactions are accounted for, the mobility matrix depends on the state of the system and its entries involve long-range interactions that decay like $1/r$ with distance $r$ between particle pairs.  The precise form of $\mathcal{E}$ and $M$ are derived below.  In the situations we consider, the divergence of the mobility matrix appearing in \eqref{sde_bd} vanishes.  Additionally, this diffusion is self-adjoint with respect to the density: \begin{equation} \label{eq:nu}
\nu(x) = \exp(- \beta \mathcal{E}(x) ) \;.
\end{equation}
For background on the numerical treatment of self-adjoint diffusions with multiplicative noise see \cite{BoDoVa2014}.  

For this BD problem, numerical stability and accuracy impose demanding time step size requirements for time integrators due to the presence of singular coefficients in the drift of the governing SDE.  Another issue with this SDE problem is that the hydrodynamic interactions between particles induce long-range interactions, which leads to an SDE with multiplicative noise.  

\paragraph{Potential Energy}

Consider a collection of $N$ identical spherical particles with radius $a$.
Let $q = (q_1, \dotsc, q_{N}) \in \mathbb{R}^{3 N}$ denote the position of all $N$ particles where $q_i \in \mathbb{R}^3$ for $i=1,\dotsc,N$.  
The distance between the ith and jth particle is the usual Euclidean one:
\begin{equation} 
d(q_i, q_j) = | q_i - q_j | \;.
\end{equation}
In terms of this distance, the total energy is a sum of pair-wise potentials over all particle pairs:
\begin{equation} \label{total_energy}
\mathcal{E}(q) =  \sum_{\substack{
            1\le i \le n\\
             j>i}}  U(d(q_i, q_j) ) \;,
\end{equation}
Following the description in \cite{FuTa2010}, we assume that the pair-wise potential energy $U(r)$ is given by:
\begin{equation} \label{LJpair}
U(r) = U_{\text{SC}}(r) + U_{\text{AO}}(r) \;,
\end{equation}
which is a sum of a soft-core potential: \begin{equation} \label{eq:soft_core}
U_{\text{SC}}(r)  = \epsilon_{\text{SC}} \left( \frac{2 a}{r} \right)^{24} \;,
\end{equation}
and an attractive potential that is defined in terms of the inter-particle force $F_{\text{AO}}(r)$ it exerts as follows: \begin{equation}
U_{\text{AO}}(r) = \int_r^{\infty} F_{\text{AO}}(s) ds \;, \quad 
F_{\text{AO}}(r) = \begin{cases} 
c_{\text{AO}} (D_2^2 - D_1^2) & \text{if  $r < D_1$} \\
c_{\text{AO}} (D_2^2 - r^2) & \text{if  $D_1 \le r < D_2$} \\
0 & \text{otherwise} 
\end{cases}
\end{equation}
Here, we have introduced the parameters $\epsilon_{\text{SC}}$, $c_{\text{AO}}$, $D_1$, and $D_2$.   Figure~\ref{fig:graph_of_U} plots graphs of $U(r)$ and $F(r)$ for the parameter values given in Table~\ref{tab:colloidal_gellation_simulation_parameters}.  Note that $U(r)$ induces a short-range force. Further, because of the $1/r^{24}$-singularity in $U_{\text{SC}}(r)$ in \eqref{eq:soft_core} at $r=0$, this force makes the resulting SDE stiff.   Since the energy is short-range, it is not uniformly coercive.  Thus, the invariant density in \eqref{eq:nu} is not integrable on $\mathbb{R}^n$ and the system is not ergodic.  In the numerical tests, we observe that -- despite this non-ergodicity -- for most sample paths, the cluster does not evaporate over the time-span of simulation.

\paragraph{Mobility Matrix}

Consider again a collection of $N$ identical spherical particles with radius $a$.
Here we describe a commonly used, implicit model for the effect of the solvent known as the Rotne-Pragner-Yamakawa (RPY) approximation.  This model results in solvent-mediated interactions between the $N$ spherical particles through the mobility matrix $M(q)$.  Let $I_{3 \times 3} $ be the $3 \times 3$  identity matrix, $R_{\text{hydro}}$ be the hydrodynamic radius of each bead, and $\eta_s$ be the solvent viscosity.  The RPY mobility matrix is given by: \begin{equation} \label{RPYmobility}
M( q ) = \begin{bmatrix} \Omega_{1,1} & \dotsc & \Omega_{1,N} \\
\vdots & \ddots & \vdots \\
\Omega_{N,1} & \dotsc & \Omega_{N, N} 
\end{bmatrix} \;,~~
\Omega_{i,j} = \begin{cases} \frac{1}{\zeta} I_{3 \times 3} \;,  & \text{if} ~~ i=j  \\
\Omega_{RPY}(q_i -q_j) \;, & \text{otherwise} 
\end{cases}
\end{equation}
for all $ q \in \mathbb{R}^{3 N}$.  Here, we have introduced the $3 \times 3$ matrix $\Omega_{RPY}( x )$ defined as:
\begin{equation} \label{rpy}
\Omega_{RPY}( x ) =  \frac{1}{\zeta} \left(  C_1(x) I_{3 \times 3} + C_2(x) \frac{x}{|x|} \otimes \frac{x}{|x|}  \right) 
\end{equation} where $C_1(x)$ and $C_2(x)$ are the following scalar-valued functions: \[
C_1(x) = \begin{cases} \frac{3}{4} \left( \frac{R_{\text{hydro}}}{|x|} \right) + \frac{1}{2} \left( \frac{R_{\text{hydro}}}{|x|} \right)^3 \;, & \text{if} ~~|x|>2 R_{\text{hydro}}   \\
 1 - \frac{9}{32} \left( \frac{|x|}{R_{\text{hydro}}} \right) \;, & \text{otherwise}   \end{cases} 
 \]  and, \[
 C_2(x) = \begin{cases}  \frac{3}{4} \left( \frac{R_{\text{hydro}}}{|x|} \right) - \frac{3}{2} \left( \frac{R_{\text{hydro}}}{|x|} \right)^3 \;, & \text{if} ~~|x|>2 R_{\text{hydro}}   \\ ~
  \frac{3}{32} \left( \frac{|x|}{R_{\text{hydro}}} \right) \;,  & \text{otherwise}  \end{cases} 
\]
The quantity $1/\zeta$ is the mobility constant produced by a single bead translating in an unbounded solvent at a constant velocity: $\zeta = 6 \pi \eta_s R_{\text{hydro}}$.     The approximation \eqref{RPYmobility} preserves the physical property that the mobility matrix is positive definite, satisfies $(\divergence M)(q) =0$ for all $q \in \mathbb{R}^{3 N}$,  and is exact up to $\mathcal{O}((R_{\text{hydro}}/r_{ij})^4)$ where $r_{ij}$ is the distance between distinct particles $i$ and $j$.  To read more about the RPY approximation see \cite{RoPr1969}. 

\paragraph{Description of Numerical Test}

We are now in position to repeat the simulation in \cite{FuTa2010,DeUsDeGrDo2014} using the generator $Q_c$ in \eqref{eq:Qc}. Consider $13$ particles  initially placed on the vertices and center of an icosahedron of edge length $8.08$, which corresponds to a radius of gyration of $7.68$.  Note that this edge length is slightly greater than the equilibrium bond length, which is $6.4$.   Thus, without noise, the cluster would simply relax to an icosahedron with an edge length equal to the equilibrium bond length.  To prevent exploding trajectories that can occur if any two particles get too close, we set the spatial grid size to be $10 \%$ of the equilibrium radius or $0.32$.    This choice ensures that the numerical approximation to the repulsive force cannot cause any particle to take a single spatial step that is larger than the cutoff radius of the short-range interaction.  Thus, the approximation to the repulsive force between every particle pair is sufficiently resolved in space to prevent explosion.  Figure~\ref{fig:bd_simulation_sample_path} and \ref{fig:bd_simulation_mean_Rg} illustrate the results produced by an SSA integrator induced by the generator $Q_c$ in \eqref{eq:Qc} operated at this spatial step size.  Other simulation parameters are given in Table~\ref{tab:colloidal_gellation_simulation_parameters}.   An important physical parameter is the Brownian time-scale defined as:  $t_B=\beta \eta_s a^3$.   In order to compare to the data from \cite{DeUsDeGrDo2014}, the time span of simulation is long and set equal to $T= 400 t_B=1065$.   We emphasize that the spatial grid size with and without hydrodynamic interactions is set equal to $10 \% a$ or $\delta x = 0.32$.    Finally, Figure~\ref{fig:bd_simulation_Nbar} validates Proposition~\ref{lem:barN}, which quantifies the average number of computational steps of the SSA integrator as a function of the spatial step size.  


\noindent
\begin{table}
\centering
\begin{tabular}{|c|c|c|}
\hline
\multicolumn{1}{|c|}{\bf Parameter} & \multicolumn{1}{|c|}{\bf Description} &  \multicolumn{1}{|c|}{\bf Value(s)}   \\
\hline
\hline
\multicolumn{3}{|c|}{ {\em Physical Parameters} } \\
\hline
  \hline
$N$ & \# of particles & $13$  \\
\hline
$\eta_s$ & solvent viscosity & $1$ \\
  \hline
$\beta^{-1}$ & temperature factor & $12.3$  \\
  \hline
  $a$ & {\bf particle radius} & $3.2$  \\
\hline
  $R_{\text{hydro}}$ & hydrodynamic radius & \{$a$, $0$\}  \\
\hline
$\epsilon_{\text{SC}}$ & parameter in $U_{\text{SC}}$ & $10$ \\
\hline
$D_1$ & parameter in $U_{\text{AO}}$ & $2.245 a$ \\
\hline
$D_2$ & parameter in $U_{\text{AO}}$ & $2.694 a$ \\
\hline
$c_{\text{AO}}$ & parameter in $U_{\text{AO}}$ & $58.5/a^3$ \\
\hline
\hline
\multicolumn{3}{|c|}{{\em  Numerical Parameters }} \\
\hline
\hline
$\delta x$ & {\bf spatial step size} &  $10 \% \; a$ \\
\hline
$T$ & {\bf time-span of simulation} & $1065$ \\
\hline
\end{tabular}
\caption{ \small {\bf Colloidal Collapse Simulation Parameters.}  Physical parameter values are taken from: \cite{FuTa2010,DeUsDeGrDo2014}.  
}
\label{tab:colloidal_gellation_simulation_parameters}
\end{table}


\begin{figure}[ht!]
\begin{center}
\includegraphics[width=0.8\textwidth]{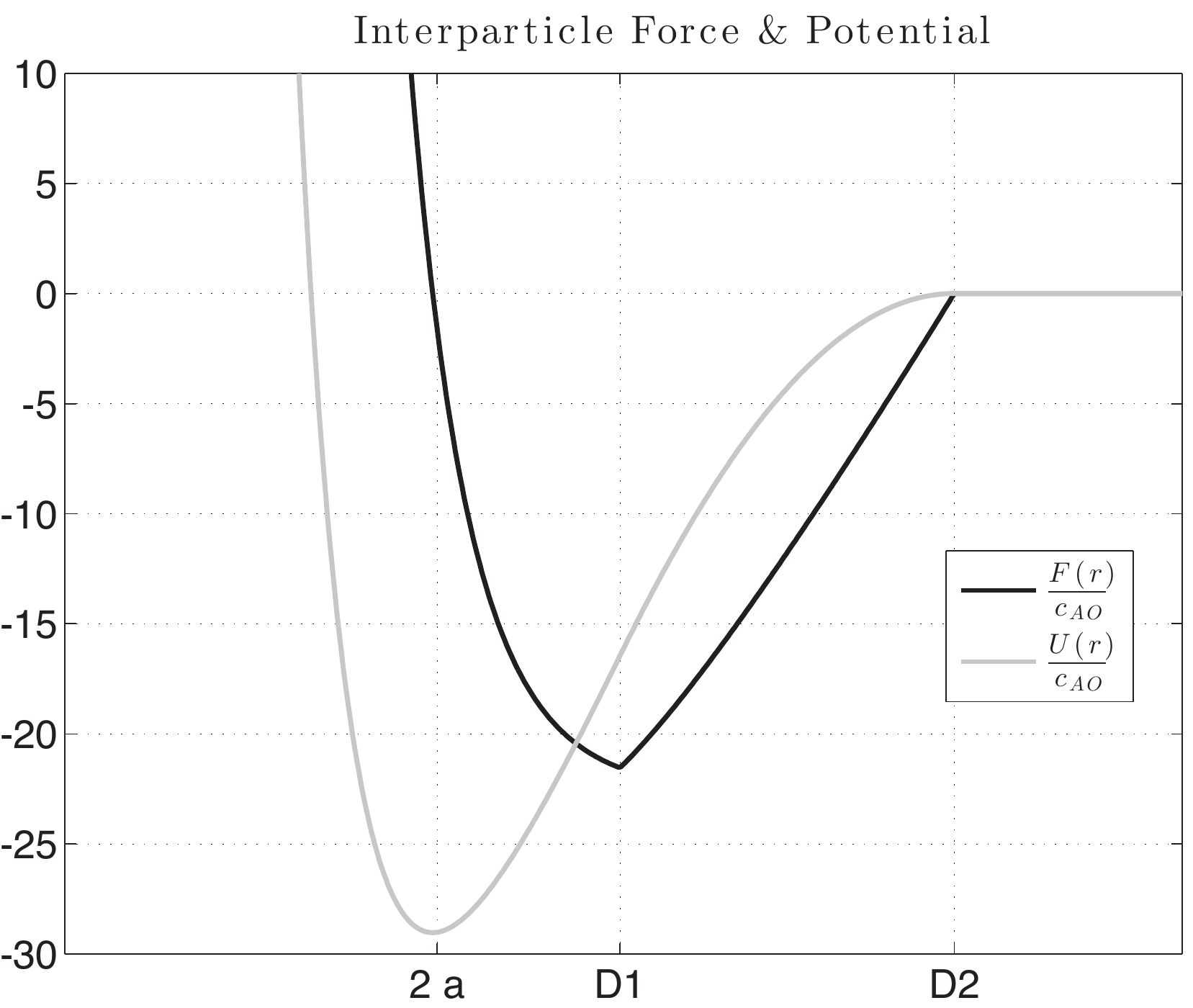}
\end{center}
                \caption{\small  {\bf Graph of pair-wise potential energy and force.}
                			Graph of the inter-particle energy $U(r)/c_{\text{AO}}$ given in \eqref{LJpair} and corresponding inter-particle force $F(x)/c_{\text{AO}} = -U'(r)/c_{\text{AO}}$ 
				as a function of the inter-particle distance $r$ with $c_{\text{AO}} =  1.7853$ and $a=3.2$.  The range of interaction is specified by the parameter $D_2$,
				which here is set equal to $D_2 = 2.694 a \approx 8.6$.  The energy is minimized at twice the particle radius $a$
				with minimum value: $U(2 a) \approx -51.82$.  Up to a factor $c_{\text{AO}} $, this graph agrees with Fig.~S1 in the 
				Supplementary Material accompanying \cite{FuTa2010}.
                }
                \label{fig:graph_of_U}
\end{figure}

\begin{figure}[ht!]
\begin{center}
\includegraphics[width=0.8\textwidth]{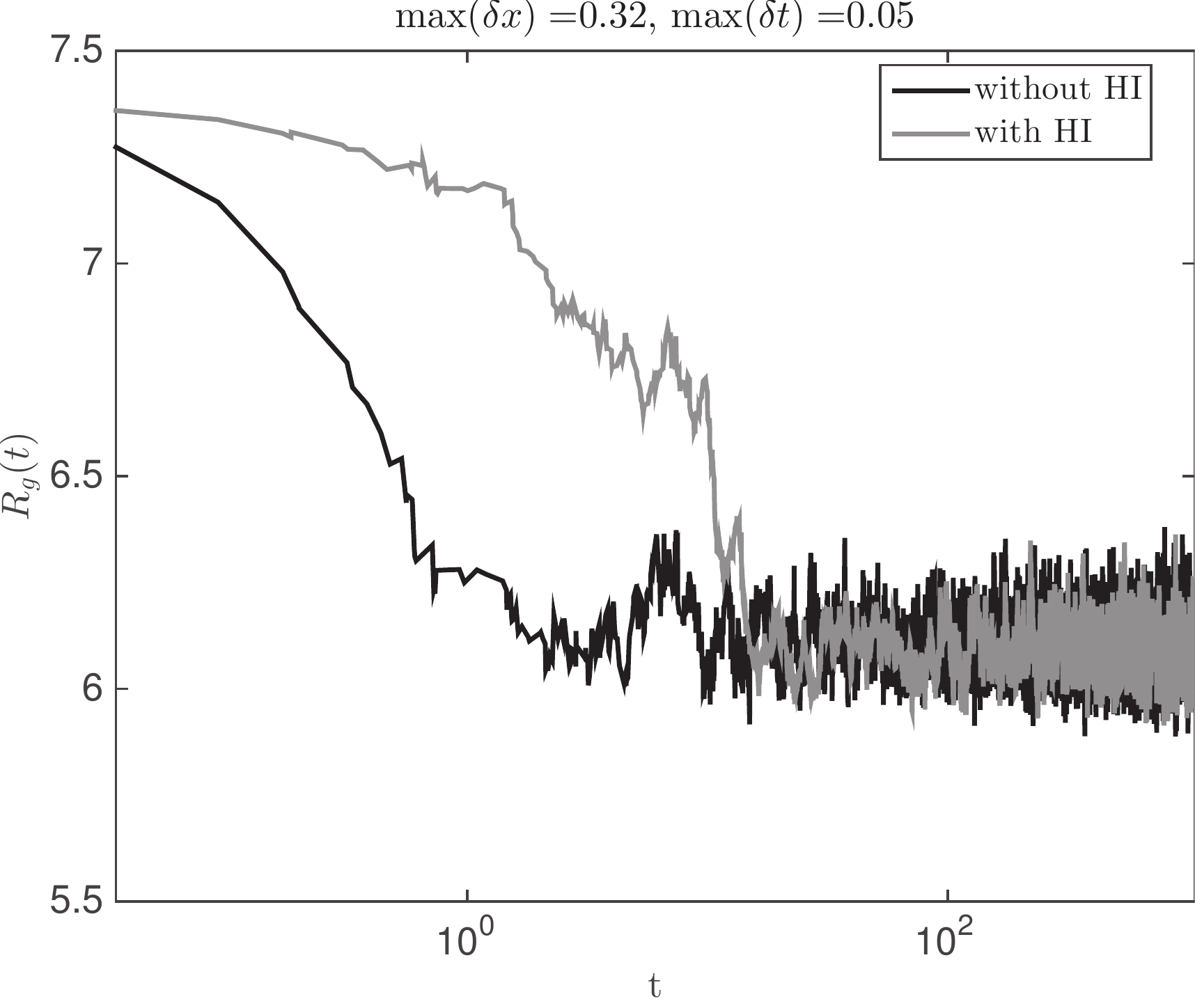} \\
\includegraphics[width=0.8\textwidth]{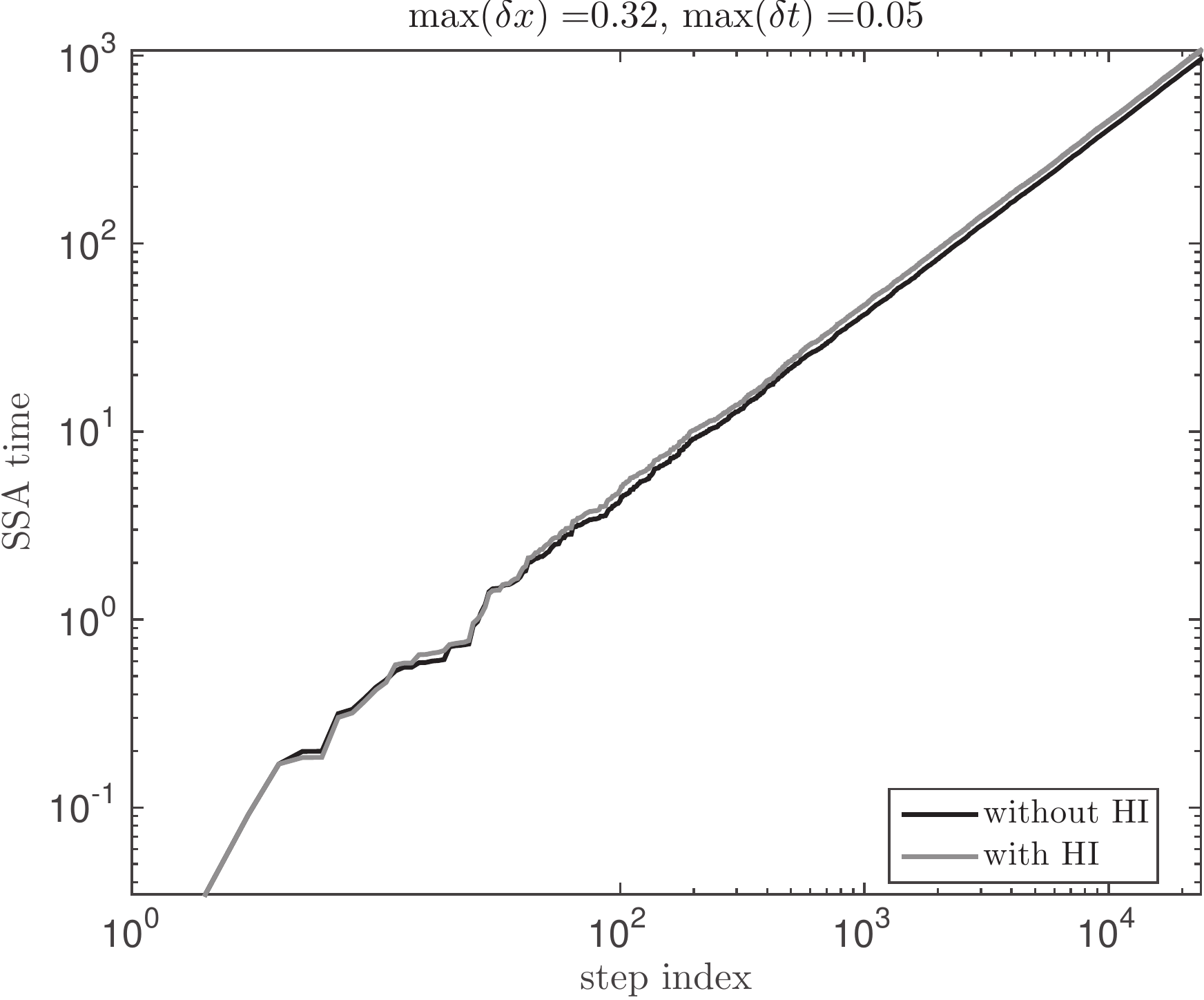}   
\end{center}
\caption{ \small {\bf Brownian Dynamics Simulation in 39D.}      
  This figure shows sample paths generated by an SSA integrator of the collapse of a colloidal cluster.
 The top panel shows a sample trajectory with and without hydrodynamic interactions (HI).
  The bottom panel shows the SSA time vs step index for these sample trajectories.
 These sample paths are produced using the SSA induced by the generator $Q_c$ in \eqref{eq:Qc} operated at a spatial step size
 set equal to $10 \% a$, where $2 a$ is the equilibrium bond length from Figure~\ref{fig:graph_of_U}.
}
\label{fig:bd_simulation_sample_path}
\end{figure}

\begin{figure}[ht!]
\begin{center}
\includegraphics[width=0.8\textwidth]{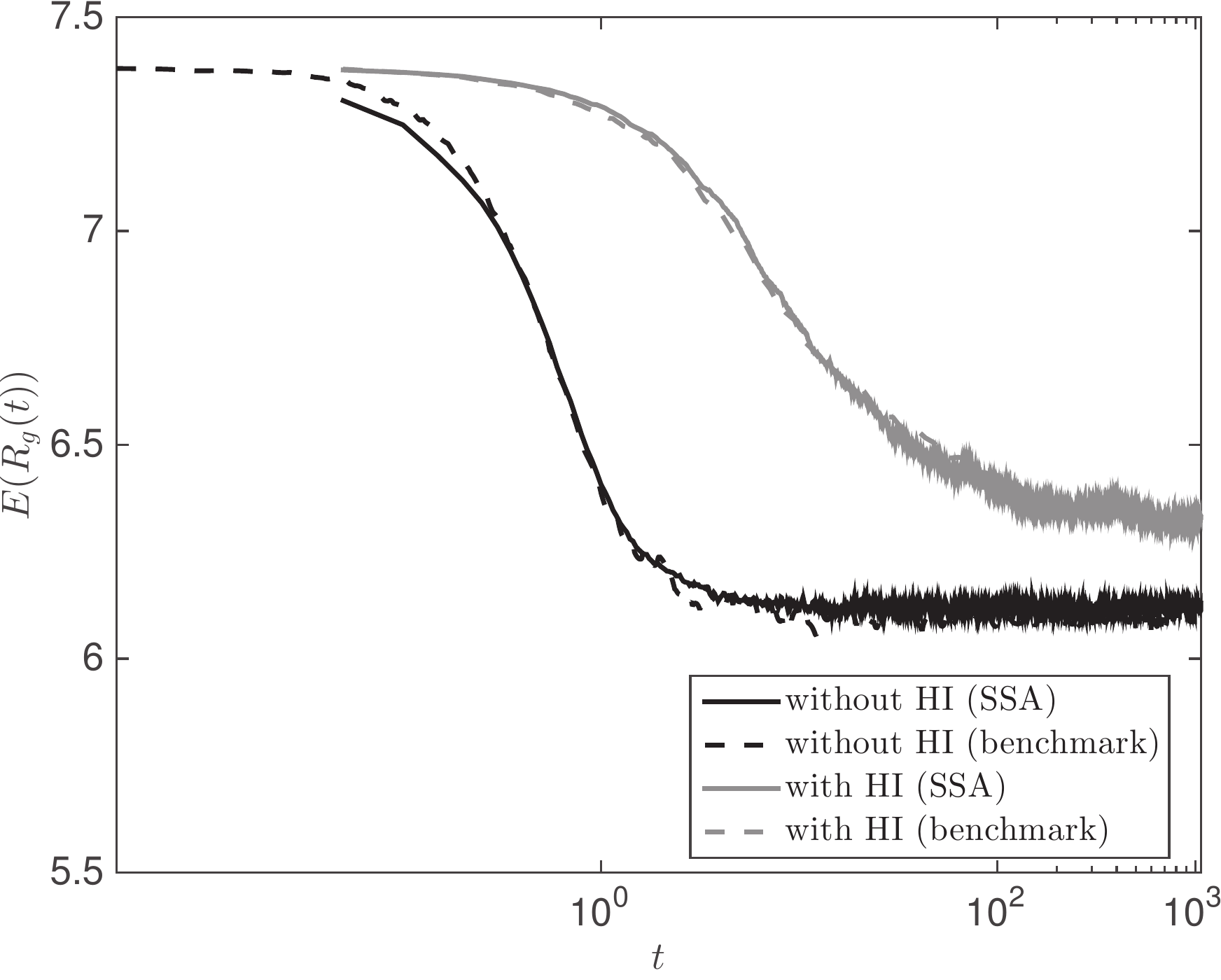}  
\end{center}
\caption{ \small {\bf Brownian Dynamics Simulation in 39D.}    
 This figure shows the mean radius of gyration
  of the colloidal cluster as a function of time over the interval $[0,T]$ computed using an SSA integrator (solid)
  and a benchmark (dashed) from \cite{FuTa2010} with and without hydrodynamic interactions (HI).    
   This numerical result is produced using the SSA induced by the generator $Q_c$ in \eqref{eq:Qc} operated at a spatial step size
 set equal to $10 \% a$, where $2 a$ is the equilibrium bond length  from Figure~\ref{fig:graph_of_U}.
  }
\label{fig:bd_simulation_mean_Rg}
\end{figure}

\begin{figure}[ht!]
\begin{center}
\includegraphics[width=0.8\textwidth]{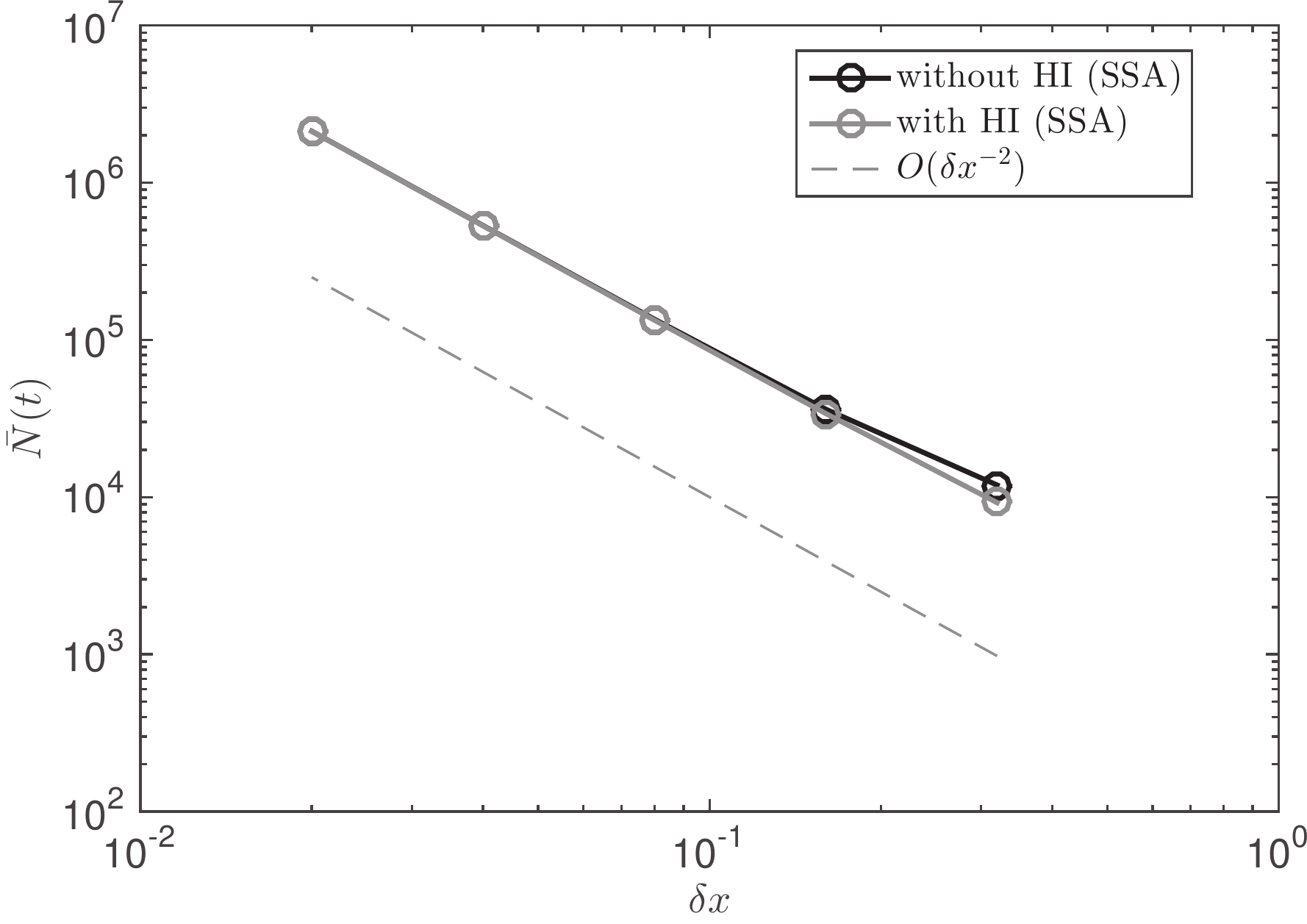}   \\
\includegraphics[width=0.8\textwidth]{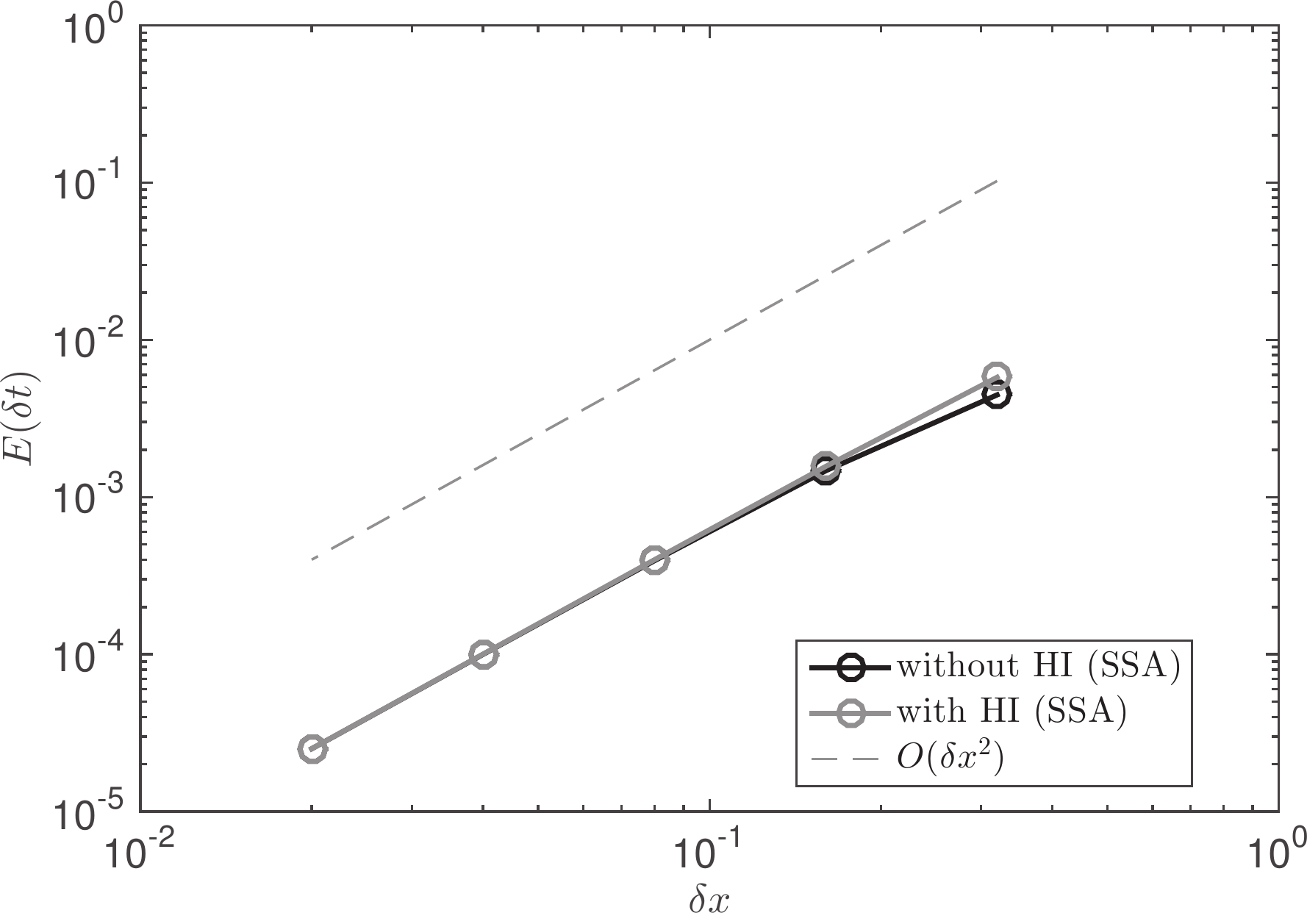}   
\end{center}
\caption{ \small {\bf Complexity of SSA in Brownian Dynamics Simulation in 39D.}  
  The top panel of this figure plots the mean number of computational steps of the SSA integrator as a function of the spatial step size.
    The bottom panel of this figure plots the mean holding time of the SSA integrator as a function of the spatial step size.
  The figure validates Proposition~\ref{lem:barN}.  
  This numerical result is produced using the SSA induced by the generator $Q_c$ in \eqref{eq:Qc}.
  }
\label{fig:bd_simulation_Nbar}
\end{figure}


\chapter{Analysis on Gridded State Spaces} \label{chap:analysis}

In this chapter we analyze the stability, complexity, and accuracy of the realizable discretization $Q_c$ in \eqref{eq:Qc} on a gridded state space.  (Refer to \S\ref{sec:gridded_vs_gridless} for the definition of a gridded state space.)  For the stability analysis, we assume that the drift field is locally Lipschitz continuous and weakly dissipative.  To show that the approximation is second-order accurate, we further assume that the drift field is of class $C^4$. In order to ensure that the state space is gridded, we assume that the noise is additive and isotropic.  These assumptions are sufficient to prove that the approximation has a stochastic Lyapunov function, which is then used to quantify the complexity and accuracy of the approximation.

\section{Assumptions}  \label{sec:preliminaries}

 Here we describe the main objects of the theory in this chapter:
 \begin{itemize}
 \item  generators of the processes, given in \eqref{eq:model_generator} and \eqref{eq:Qc2};
 \item  stochastic equations governing the dynamics of sample paths, given in \eqref{eq:model_sde} and \eqref{eq:Xsde}; and,
 \item  Kolmogorov equations governing the evolution of conditional expectations of an observable, given in \eqref{eq:model_kolmogorov} and \eqref{eq:Qc_backward_kolmogorov}.  
\end{itemize}

 \subsection{Some Notation}

 We shall use the following notation.  Let $| \cdot |$ denote the $2$-norm of an $n$-vector, let $| \cdot |_p$ denote the $p$-norm of an $n$-vector where $p \ge 1$ and  let $\| \cdot \|$  denote the standard operator norm.  For a real-valued function $f \in C^r(\mathbb{R}^n)$ where $r \in \mathbb{N}$, let $D^r f$ denote the $r$th derivative of $f$.   Let $B_b(\mathbb{R}^n)$ denote the Banach space of all Borel measurable and bounded functions from $\mathbb{R}^n$ to $\mathbb{R}$ endowed with the supremum norm: 
 \begin{equation}
\| f \|_{\infty} = \sup_{x \in \mathbb{R}^n} | f(x) | \;, \quad f \in B_b(\mathbb{R}^n) \;.
\end{equation} We denote by $C_b(\mathbb{R}^n)$ the subspace of $B_b(\mathbb{R}^n)$ consisting of all uniformly continuous and bounded functions.  This subspace with the supremum norm:
\begin{equation}
\| f \|_{\infty} = \sup_{x \in \mathbb{R}^n} | f(x) | \;, \quad f \in C_b(\mathbb{R}^n) \;,
\end{equation}  is  a Banach space. For $k \in \mathbb{N}$, we denote by $C_b^k(\mathbb{R}^n)$ the subspace of all $k$-times differentiable functions  whose derivatives up to $k$th order are uniformly continuous and bounded in the supremum norm.  This subspace is also a Banach space when endowed with the norm: \begin{equation} \label{eq:ksupnorm}
 \| f \|_k = \| f \|_{\infty}  + \sum_{i=1}^k  \sup_{x \in \mathbb{R}^n} \| D^i f(x) \| \;, \quad f \in C_b^k(\mathbb{R}^n) \;.
 \end{equation}

%
%

\subsection{SDE}

%
%

Consider an infinitesimal generator $L$, which maps a twice differentiable test function $f: \mathbb{R}^n \to \mathbb{R} $ to another function $Lf : \mathbb{R}^n \to \mathbb{R}$ defined as:  \begin{equation}  \label{eq:model_generator}
\boxed{
L f( x ) = \tr( Df(x) \mu(x)^T   + D^2 f(x) ) 
}
 \end{equation}
where $\mu: \mathbb{R}^n \to \mathbb{R}^n$ is the drift field. The SDE associated to $L$ has an unbounded domain $\mathbb{R}^n$ and isotropic (or direction-independent), additive (or state-independent) noise: 
\begin{equation} \label{eq:model_sde}
\boxed{
dY = \mu(Y) dt + \sqrt{2} \; d W \;, \quad Y(0) = x \in \mathbb{R}^n 
}
\end{equation} 
Under certain assumptions on the drift field (e.g.~those stated in Assumption~\ref{assumptions_on_drift}), this SDE uniquely defines in a pathwise sense a Markov process $Y(t)$ for any $t \ge 0$.   Basic properties of solutions to \eqref{eq:model_sde} are reviewed in \S\ref{sec:properties}.

For any $t\ge0$ and $x \in \mathbb{R}^n$, let $P_t$  and $\Pi_{t,x}$ resp.~denote the Markov semigroup and Markov transition probabilities of $Y(t)$ related by:\[
P_t f(x) = 
\int_{\mathbb{R}^n}  f(y) \Pi_{t,x}(dy)  = \Ex_x f( Y(t) ) \quad t \ge 0 \;, ~ x \in \mathbb{R}^n \;,~ f \in B_b(\mathbb{R}^n) \;.
\] 
Here  $\Ex_x$ denotes expectation conditional on $Y(0)=x \in \mathbb{R}^n$. Define the domain of $L$ to be the set: \[
 \dom(L) = \{ f \in B_b(\mathbb{R}^n) ~\mid~ \lim_{t \to 0^+} \frac{1}{t} (P_t f - f ) ~~\text{exists in $B_b(\mathbb{R}^n)$} \} \;.
 \]
 
 %
%
 
The Kolmogorov equation associated to the SDE problem \eqref{eq:model_sde} is given by: \begin{equation} \label{eq:model_kolmogorov}
\boxed{
\frac{\partial u}{\partial t}(x,t) = L u ( x, t) \quad \text{for all $x \in \mathbb{R}^n$ and $t \ge 0$}
}
\end{equation} with initial condition: \[
u(x,0) = f(x)  \quad \text{for all $x \in \mathbb{R}^n$ and for any $f \in \dom(L)$} \;.
\]  
The solution $u(x,t)$ to \eqref{eq:model_kolmogorov} has the following stochastic representation: \[
u(x,t) = \Ex_x f(Y(t) ) = P_t f(x) \;.
\]   

Under suitable conditions on the drift field (e.g.~those stated in Assumption~\ref{assumptions_on_drift}), the process $Y(t)$ admits a unique invariant probability measure $\Pi$ that is absolutely continuous with respect to Lebesgue measure on $\mathbb{R}^n$.  Let $\nu(x)$ be the stationary density associated to $\Pi$ that satisfies \[
(L^* \nu) (x) = 0 \quad \text{for all $x \in \mathbb{R}^n$}
\] where $L^*$ is the (formal) adjoint of the infinitesimal generator $L$ of the SDE solution.   Hence, \[
\Pi( L f ) = \int_{\mathbb{R}^n} L f(x) \nu(x) dx = 0 \;, \quad \text{for all $f \in \dom(L)$} \;.
\]
For the associated semigroup this invariance implies: \begin{align*}
\Pi( P_t f) &= \int_{\mathbb{R}^n} P_t f(x) \nu(x) dx = \int_{\mathbb{R}^n} f(x) \nu(x) dx = \Pi(f)  
\end{align*}
for any $t \ge 0$.  

\medskip

Next we introduce (space) discrete analogs of \eqref{eq:model_generator},  \eqref{eq:model_sde}, and \eqref{eq:model_kolmogorov}.

%
%

 \subsection{Approximation}  
 
 To introduce the generator of the approximation, let $\{ e_i \}_{i=1}^n$ be the standard basis on $\mathbb{R}^n$ and let $S_h = \{ x_i \}$ be an infinite Cartesian grid over $\mathbb{R}^n$ with edge length $h$, which we refer to as the spatial step size.  The precise location of this grid is typically set by the choice of initial condition.  Let $\ell^{\infty}(S_h)$ (or $\ell^{\infty}$ for short) be the Banach space of all grid functions $\varphi: S_h \to \mathbb{R}$ that are bounded in the supremum norm: \begin{equation}
\threebars \varphi \threebars_{\infty} = \sup_{x \in S_h} | \varphi(x) |  \;.
\end{equation}
 In the context of the SDE problem \eqref{eq:model_sde}, the approximate generator given in~\eqref{eq:Qc} maps a grid function $\varphi: S_h \to \mathbb{R}$ to another grid function $Q \varphi: S_h \to \mathbb{R}$ defined as: 
 
 %
 %
 
 %
 \begin{equation} \label{eq:Qc2}
 \boxed{
\begin{aligned}
 Q \varphi(x) =
 \sum_{i=1}^n
\frac{1}{h^2}\exp\left( \frac{h}{2}  \mu_i \right) &\left( \varphi\left( x+ h e_i  \right) - \varphi(x) \right)  \\
  \qquad\qquad +
 \frac{1}{h^2}\exp\left(-  \frac{h}{2}  \mu_i  \right) & \left( \varphi\left( x- h  e_i \right) - \varphi(x) \right)  
\end{aligned}
}
\end{equation}

Here we used the shorthand notation $\mu_i = \mu(x)^T e_i$.  We also dropped the subscript $c$ appearing in~\eqref{eq:Qc}, since the theory in this chapter focuses on proving properties of this generator only.  For any $t \ge 0$, the generator $Q$ induces a Markov jump process $X(t)$ on $S_h$, which can be exactly simulated using the SSA specified in Algorithm~\ref{algo:ssa}.  For basic facts about Markov jump processes, we refer the reader to \cite{ReYo1999,protter2004stochastic,EtKu2009,Kl2012}.   
  
 %
 %
 
Since our assumptions (given below) permit the drift field entering the reaction rates in \eqref{eq:Qc2} to be unbounded, $Q$ may not be a bounded operator with respect to grid functions in $\ell^{\infty}$.  Nevertheless under a weak dissipativity condition on the drift field of \eqref{eq:model_sde} given in Assumption~\ref{assumptions_on_drift}, we show that the approximation $X(t)$ has a stochastic Lyapunov function.  This property is then used to show that the local martingale \begin{equation} \label{eq:semimartingale_phi}
M^{\varphi}(t) = \varphi(X(t)) - \varphi(X(0)) -\int\limits_0^t  Q \varphi(X(s) ) ds \;, \quad t \ge 0
\end{equation} is a global martingale for all $\varphi: S_h \to \mathbb{R}$ satisfying a growth condition, which is developed in \S\ref{sec:properties}.  For functions satisfying this condition, Dynkin's formula holds: \begin{equation} \label{eq:dynkins_formula}
\Ex_x \varphi(X(t)) = \varphi(x) + \int\limits_0^t \Ex_x Q \varphi (X(s)) ds  \;, \quad t \ge 0\;.
\end{equation}
This class of functions includes locally Lipschitz continuous functions that have a local Lipschitz constant that does not grow faster than a sharp stochastic Lyapunov function for the SDE solution.  Note that the quadratic variation of $M^{\varphi}$ is given by \begin{equation} \label{eq:quadratic_variation_Mphi}
\langle M^{\varphi}, M^{\varphi} \rangle(t) = \int\limits_0^t ( Q \varphi^2(X(s)) - 2 \varphi(X(s)) Q \varphi(X(s)) ) \; ds \;. 
\end{equation}

 %
 %
 
By choosing $\varphi(X(t))$ in \eqref{eq:semimartingale_phi} to be the $i$th-component of $X(t)$, it can be shown that the process $X(t)$ induced by the generator $Q$ satisfies the following stochastic equation: \begin{equation}
\label{eq:Xsde}
\boxed{
d X = \mu^h(X) dt + \sqrt{2} \; d \mathcal{M} \;, \quad X(0) = x \in S_h 
}
\end{equation}
where we have introduced the drift field $\mu^h$ with $i$th component \[
\mu^h_i(x) =  \frac{2}{h} \sinh( \frac{h}{2} \mu_i(x) ) \quad \text{for $1 \le i \le n$}
\]
and a local (zero-mean) martingale \[
\mathcal{M}(t) = (\mathcal{M}_1(t), \dotsc, \mathcal{M}_n(t))^T  \in \mathbb{R}^n \;.
\]  The quadratic covariation of the components of $\mathcal{M}$ satisfies: \begin{equation}
\langle \mathcal{M}_i, \mathcal{M}_j \rangle(t) =
\begin{dcases*} 2 \int_0^t \cosh(\frac{h}{2} \mu_i(X(s))  ) ds  &  if $i=j$ \\
\vphantom{\int_0^t} 0 & otherwise 
\end{dcases*}
\end{equation}
for all $t \ge 0$.  (Since jumps only occur in each coordinate direction, the covariation of two distinct components of $\mathcal{M}$ is zero.) We stress that \eqref{eq:Xsde} is shorthand notation for a semimartingale representation of the Markov jump process $X(t)$.   The noise entering \eqref{eq:Xsde} is purely discontinuous and the drift field is less regular than in the original SDE.   However, this stochastic equation has the nice feature that realizations of $X$ can be produced using the SSA specified in Algorithm~\ref{algo:ssa}.   This stochastic equation resembles an ODE with random impulses \cite{SaSt1999}.  Properties of strong solutions to \eqref{eq:Xsde} are analyzed in \S\ref{sec:properties}.   We also remark that although there is a $1/h$ prefactor in $\mu^h$, to leading order $\mu^h \approx \mu$, as expected.

Note that $Q$ evaluated on the grid $S_h$ can be represented as an infinite matrix $\mathsf{Q}$ with the following nonzero off-diagonal transition rates from state $x \in S_h$ to state $y \in S_h$: \[
\mathsf{Q}(x, y) =  \frac{1}{h^2}\exp\left( \frac{1}{2}  \mu(x)^T (y - x)  \right)  \quad \text{if $|y-x| = h$} 
\]
or in words, if $y$ is a nearest neighbor of $x$.  The $Q$-matrix property given in \eqref{Qmatrix_property} implies that: \[
\mathsf{Q}(x,x) = -\sum_{y \in S_h \setminus \{x\}} \mathsf{Q} (x, y) \;.
\] Recall from Algorithm~\ref{algo:ssa}, that the diagonal term of $\mathsf{Q}$ is the reciprocal of the mean holding time of the approximation $X(t)$ at the state $x$.    

 %
 %

Consider the semi-discrete (discrete in space and continuous in time) Kolmogorov equation: 
\begin{equation} \label{eq:Qc_backward_kolmogorov}
\boxed{
\frac{d u^h}{dt}(x,t) = Q u^h(x,t) \quad \text{for all $x \in S_h$ and $t\ge0$} 
}
\end{equation} 
with initial condition: \[
u^h(x,0) =\varphi(x) \quad \text{for all $x \in S_h$}  \;.
\]  Analogous to the continuous case, this solution can be explicitly represented in terms of $X(t)$: \[
u^h(x,t) = \Ex_x \varphi (X(t)) = P^h_t \varphi (x) 
\]  where we have introduced  $P^h_t$, which is the Markov semigroup associated to $Q$.  In terms of which, the domain of $Q$ is defined as: \[
 \dom(Q) = \{ \varphi \in \ell^{\infty} ~\mid~ \lim_{t \to 0^+} \frac{1}{t} (P^h_t \varphi - \varphi ) ~~\text{exists in $ \ell^{\infty}$} \} \;.
 \]
Let $\Pi^h$ denote an invariant probability measure of $Q$, which we will prove exists under certain conditions on the drift field stated below in Assumption~\ref{assumptions_on_drift}.  In the particular case of a gridded state space, $\Pi^h$ is also unique.   In contrast, if the state space is gridless, uniqueness is nontrivial to prove because the process may lack irreducibility.   We address this point in Chapter~\ref{chap:gridless}.

Let $\nu^h(x)$ be the probability density function associated to $\Pi^h$, which satisfies the identity \[
(Q^* \nu^h)(x) = 0 \quad \text{for all $x \in S_h$} 
\] where $Q^*$ is the adjoint of $Q$.  This identity implies that \[
 \Pi^h( Q \varphi) = 0 \quad \text{ for all $\varphi \in \dom(Q)$} \;.
 \]  Moreover, an invariant measure $\Pi^h$ satisfies: \[
\Pi^h( P^h_t \varphi) =  \sum_{x \in S_h} P^h_t \varphi(x) \nu^h(x) = \sum_{x \in S_h} \varphi(x) \nu^h(x) = \Pi^h( \varphi) 
\] for any $t \ge 0$.   
 

Table~\ref{notation} summarizes this notation.

 \setlength\tabcolsep{1.1pt}

\noindent
\begin{table}[ht!]
\centering
\begin{tabular}{|c|c|c|c|c|c|c|c|}
\hline
 & process & generator & \begin{tabular}{c} transition \\ probability \end{tabular}  & semigroup & \begin{tabular}{c}  invariant  \\ measure \& density \end{tabular}  \\
\hline
\begin{tabular}{c} SDE \\ Solution \end{tabular} & $Y(t) \in \mathbb{R}^n$ & $L$ &  $\Pi_{t,x}$    & $P_t$ & \begin{tabular}{c} $\Pi(A) = \int_A \nu(x) dx$  \\ $A \subset \mathbb{R}^n$ \end{tabular} \\
  \hline
  \begin{tabular}{c} Numerical \\
Approximation \end{tabular} & $X(t) \in S_h$ &  $Q$  & $\Pi^h_{t,x}$     & $P^h_t$ & \begin{tabular}{c} $\Pi^h(A) = \sum_{x \in A} \nu^h(x)$  \\ $A \subset S_h$ \end{tabular} \\
  \hline
\end{tabular}
\caption{ \small {\bf Tabular summary of (space) continuous \& discrete objects.}  
}
\label{notation}
\end{table}

\subsection{Drift Field Assumptions} Here we state and illustrate our assumptions on the drift field.

%
%

\medskip

\begin{assumption} \label{assumptions_on_drift} 
The drift field $\mu: \mathbb{R}^n \to \mathbb{R}^n$ of the SDE in \eqref{eq:model_sde} is of class $C^4$ and satisfies conditions (A1) - (A3).

\medskip

\begin{description}
\item[(A1) One-Sided Lipschitz Continuity]  There exists $C^{\mu}_L \in \mathbb{R}$ such that: \[
D\mu(x) (\xi,\xi) \le C_L^{\mu} | \xi |^2
\]
for all $x, \xi \in \mathbb{R}^n$.

\medskip

\item[(A2) Polynomial Growth] There exist $m \in \mathbb{N}_0$ and $C^{\mu}_P>0$ such that: \[
\sup_{x \in \mathbb{R}^n} \frac{ | \mu(x) | }{ 1 + |x|^{2 m+1} } \vee \sup_{x \in \mathbb{R}^n} \frac{ \| D^k \mu(x) \| }{ 1 + |x|^{2 m+1-k} } \le C^{\mu}_P \quad 
\]
for any $1 \le k \le 4$.
 
\medskip

\item[(A3) Weak Dissipativity]  There exist $\beta_0>0$ and $\alpha_0 \ge 0$ such that: \begin{equation} \label{eq:weak_dissipativity_mui}
\mu_i(x) \sign(x_i)  \le \alpha_{0} \; ( 1 +   |x|^{2m} )  - \beta_0 \; | x_i |^{2m + 1}
\end{equation}
for all $x \in \mathbb{R}^n$, for every $1 \le i \le n$, and where $m$ is the natural number appearing in Assumption (A2).

\medskip

\item[(A4) Minimal Growth]  For any $\ell \in \mathbb{N}$, there exist  $\beta_{\ell}>0$ and $\alpha_{\ell} \ge 0$ whose growth in $\ell$ is subfactorial (i.e.~is less than the growth of $[\ell !]$) which satisfy: 
\[
| \mu_i(x) |^{2 \ell} \ge - \alpha_{\ell} \; ( 1 +   |x|^{2 m} )  | x_i |^{(2m + 1) (2 \ell - 1)}  + \beta_{\ell} \; | x_i |^{(2m + 1) (2 \ell)} 
\]
for all $x \in \mathbb{R}^n$, for every $1 \le i \le n$, and where $m$ is the natural number appearing in Assumption (A2).
\end{description}

\end{assumption}


%
%

We chose these assumptions with the following drift fields in mind.

\begin{example} Let $g: \mathbb{R}^n \to \mathbb{R}^n$ be of class $C^4$ with component form \[
g(x) = (g_1(x), \dotsc, g_n(x))^T
\] that satisfies the polynomial growth condition: \begin{equation} \label{eq:g_polynomial_growth}
 \sup_{x \in \mathbb{R}^n} \frac{ | g(x) | }{ 1 + |x|^{2 m} }  \vee \sup_{x \in \mathbb{R}^n} \frac{ \| D^k g(x) \| }{ 1 + |x|^{2 m-k} } < C^{g}_P   \quad \text{for any $1 \le k \le 4$}  
\end{equation}
and the one-sided Lipschitz property: \[
D g(x) (\xi, \xi) \le C^{g}_L | \xi |^2 \quad \text{for all $x, \xi \in \mathbb{R}^n$} 
\]
for some constants  $C^g_P > 0$ and $C^g_L \in \mathbb{R}$.  The drift field whose $i$th-component can be decomposed as: \[ \boxed{
\mu_i(x) = - a_i \; x_i^{2 m +1} + g_i(x) \;, \quad a_i > 0 \;, \quad 1 \le i \le n
}
\] satisfies Assumption~\ref{assumptions_on_drift}.  Indeed, \begin{align*}
D \mu(x)(\xi,\xi) &= \sum_{1 \le i,j \le n} \frac{\partial \mu_i}{\partial x_j} \xi_i \xi_j \\
 &= 
\sum_{1 \le i,j \le n}  -a_i \; (2 m+1) \; x_i^{2m} \; \xi_i^2 \; \delta_{ij} + \frac{\partial g_i}{\partial x_j} \xi_i \xi_j \le C^{g}_L | \xi |^2
\end{align*}
where $\delta_{ij}$ is the Kronecker delta.  This inequality implies that Assumption (A1) holds with $C^{\mu}_L = C^g_L$.   According to the hypotheses made on $g$, it is apparent that Assumption (A2) holds.  We also have that  \begin{align*}
\mu_i(x) \sign(x_i) &= - a_i \; |x_i|^{2 m +1}  + g_i(x) \sign(x_i) \\
&\le - a_i \; |x_i|^{2 m +1} + C^g_P (1+|x|^{2m})  
\end{align*} for all $x \in \mathbb{R}^n$, which implies  Assumption (A3) holds with $\alpha_0 = C^g_P$ and $\beta_0 = \min_{1 \le i \le n} a_i $.  
By Bernoulli's inequality and \eqref{eq:g_polynomial_growth}, there exists $\alpha_{\ell} \ge 0$ such that: \begin{align*}
| \mu_i(x)|^{2 \ell} &=  a_i^{2 \ell} \; x_i^{(2 m +1) (2\ell)} \left( 1  -  \frac{g_i(x)}{a_i x_i^{2m+1}} \right)^{2 \ell} \\
&\ge a_i^{2 \ell} \; x_i^{(2 m +1) (2\ell)} - 2 \ell a_i^{2 \ell -1} C^g_P (1+ |x|^{2m} )   x_i^{(2m+1) (2 \ell-1)} 
\end{align*}
for all $x \in \mathbb{R}^n$, which implies Assumption (A4) holds with constants $\beta_{\ell} = \beta_0^{2 \ell}$ and $\alpha_{\ell} =  2 \ell C^g_P \max_{1 \le i \le n} a_i^{2 \ell -1} $ whose dependence on $\ell$ is subfactorial.
\end{example}

\begin{lemma}
Assumption~\ref{assumptions_on_drift} (A1) implies that $\mu$ satisfies 
\[
\left\langle \mu(y) - \mu(x) , 
y - x \right\rangle  
\le  C_L^{\mu} | x - y |^2 
\]
for all $x, y \in \mathbb{R}^n$.  
\end{lemma}

\begin{proof}
Use Taylor's formula and Assumption~\ref{assumptions_on_drift} (A1) to obtain: \begin{align*}
\left\langle \mu(y) - \mu(x), y-x \right\rangle &= \int_0^1  D \mu(x + s (y-x) ) (y-x, y-x) ds \\
&\le C_L^{\mu} \; |y-x|^2
\end{align*}
as required.
\end{proof}

\begin{rem}[On Usage of Assumption~\ref{assumptions_on_drift}]  
Local Lipschitz continuity of the drift field is sufficient to show that for any initial condition $x \in \mathbb{R}^n$ and for any Brownian motion $W$ there a.s.~exists a unique process $Y$ adapted to $W$ that satisfies \eqref{eq:model_sde} up to a random explosion time.  The weak dissipativity condition (A3) is used in Lemma~\ref{SDEdrift} to prove that $Y$ possesses a stochastic Lyapunov function, which implies that this pathwise unique solution can be extended to any $t \ge 0$ \cite{Kh2012}.   Additionally, Lemma~\ref{SDEdrift} is used in Lemma~\ref{lem:moments_of_process} to obtain bounds on every moment of the solution that are uniform in time.  


Assumption (A1) and (A2) are used to prove Lemma~\ref{lem:derivatives_of_semigroup}, which implies that the semigroup $P_t$ is strongly Feller.  Local Lipschitz continuity of the drift field is also used in Lemma~\ref{SDEminorization} to prove that the process is irreducible and its transition probabilities are absolutely continuous with respect to Lebesgue measure.  The strong Feller property and irreducibility imply uniqueness of an invariant probability measure.  In fact, a stochastic Lyapunov function, irreducibility, and the strong Feller property imply that the process $Y$ is geometrically ergodic with respect to this unique invariant probability measure, as stated in Theorem~\ref{SDEgeometricallyergodic}.    

Regarding the stability of the approximation,  Assumptions (A2), (A3) and (A4) are used  in Lemmas~\ref{Qdrift} and \ref{Qq_drift_gridless} to prove that the approximation has a stochastic Lyapunov function.   In the particular case of gridded state space, the approximation is also irreducible, and it then follows that the approximation is geometrically ergodic with respect to a unique invariant probability measure as stated in Theorem~\ref{thm:geometric_ergodicity_Qc}.  Assumption (A2) is used to bound the growth of $\mu$ by this stochastic Lyapunov function in Lemmas~\ref{lem:martingale_solution} and \ref{lem:barN}. 
 In the gridless state space context, Assumption (A2) is also used in Lemma~\ref{lem:Qq_is_bounded} to show that a mollified generator is a bounded linear operator in the standard operator norm.  This property implies that the semigroup of the  approximation is Feller, and together with a stochastic Lyapunov function, that the approximation has an invariant probability measure.   

Regarding the accuracy of the approximation, we require the additional assumption that the drift field has four derivatives.  This assumption is used in Lemmas~\ref{lem:derivatives_of_solution} and~\ref{lem:derivatives_of_semigroup}.  The latter Lemma implies that if $f \in C^4_b(\mathbb{R}^n)$ then $P_t f \in C^4_b(\mathbb{R}^n)$ for any $t \ge 0$. This result is then used to estimate the global error of the approximation over finite (resp.~infinite) time intervals in Theorem~\ref{thm:finite_time_accuracy} (resp.~Theorem~\ref{thm:nu_accuracy}).  Four derivatives of the drift field are needed for proving second-order accuracy because the remainder term in $(Qf(x) - Lf(x))$ involves four derivatives of $f(x)$, see Lemma~\ref{lem:Qc_accuracy_model}.
\end{rem}

The following Lemma relates our weak dissipativity assumption to more standard ones.  Note that the converse implication does not hold, in general.

\begin{lemma} \label{lem:weak_dissipativity_mu}
Assumption (A3) implies that there exist constants $\alpha \ge 0$ and $\beta >0$ such that \[
\mu(x)^T x \le \alpha \; ( 1 + |x|^{2 m+1} ) - \beta \; |x|^{2 m + 2} \quad  \text{for all $x \in \mathbb{R}^n$} \;.
\] 
\end{lemma}

\begin{proof}
Multiply both sides of \eqref{eq:weak_dissipativity_mui} by $|x_i|$ and sum over $i$ to obtain: \[
\mu(x)^T x \le \alpha_0 \; ( 1 + |x|^{2 m} ) |x|_1 - \beta_0 |x|^{2m+2}_{2m+2} \;.
\]  Since norms on $\mathbb{R}^n$ are equivalent, there exist $C_0, D_0>0$ such that:\[
|x|_1 \le C_0 |x| \quad \text{and} \quad |x|^{2m+2}_{2m+2} \ge D_0 |x|^{2m+2}  \quad \text{for all $x \in \mathbb{R}^n$} \;.
\] These inequalities combined with the preceding one imply the desired inequality.  
\end{proof}

%
%

\section{Stability by Stochastic Lyapunov Function} \label{sec:geometric_ergodicity} 

In this section we show that the approximation and the SDE solution are long-time stable under Assumption~\ref{assumptions_on_drift} on the drift field.  Specifically, we show that the Markov processes induced by the generators $L$ in \eqref{eq:model_generator} and $Q$ in \eqref{eq:Qc2} are geometrically ergodic.    The model problem considered for this analysis is \eqref{eq:model_sde}, which is an SDE of elliptic type with an unbounded state space, locally Lipschitz drift coefficients, and additive noise.   The weak dissipativity condition Assumption~\ref{assumptions_on_drift} (A3) is a crucial ingredient to prove that both the SDE solution and numerical approximation are geometrically ergodic.  We emphasize that for this model problem the state space of the approximation is gridded and the generator $Q$ can be represented by an infinite matrix.   Our main tool to prove geometric ergodicity is Harris Theorem, which we recall below.  We use the total variation (TV) metric to quantify the distance between probability measures.  

\medskip

\begin{defn}
If  $\mathsf{\Pi}_1$ and $\mathsf{\Pi}_2$ are two probability measures on a measurable space $\Omega$, their total variation is defined as:
\begin{equation}
\|\mathsf{\Pi}_1 - \mathsf{\Pi}_2\|_\TV = 2\sup_{A} |\mathsf{\Pi}_1(A) - \mathsf{\Pi}_2(A)|\;,
\end{equation}
where the supremum runs over all measurable sets in $\Omega$.  In particular, the total variation distance between two probability measures is two if and only if they are mutually singular.  
\end{defn}

\medskip

The factor two in this definition is convenient to include because this metric can be equivalently written as \[
\| \mathsf{\Pi}_1 - \mathsf{\Pi}_2\|_\TV = \sup\left\{ \left|  \int_{\Omega} \varphi  \; d \mathsf{\Pi}_1 -  \int_{\Omega} \varphi \; d \mathsf{\Pi}_2 \right| ~~\text{s.t.}~~
\begin{array}{c}
\varphi: \Omega \to \mathbb{R} \;, \\
|\varphi(x)| \le 1 ~\text{for all $x \in \Omega$}
\end{array} \right\} \;. 
\]  

\subsection{Harris Theorem} 

%
%

Harris Theorem states that if a Markov process admits a stochastic Lyapunov function such that its sublevel sets are `small', then it is geometrically ergodic.   More precisely, Harris Theorem applies to any Markov process $\mathsf{X}(t)$ on a Polish space $\Omega$, with generator $\mathsf{L}$, Markov semigroup $\mathsf{P}_t$, and transition probabilities \[ 
\mathsf{\Pi}_{t,x}( A) = \Pr( \mathsf{X}(t) \in A \mid \mathsf{X}(0) = x ) \quad t>0 \;, \quad  x \in \Omega \;, 
\]
that satisfies Assumptions~\ref{driftcondition} and ~\ref{minorization} given below.

\begin{assumption}[Infinitesimal Drift Condition] \label{driftcondition}
There exist a function $\Phi: \Omega \to \mathbb{R}^+$ and strictly positive constants $\mathsf{w}$ and $\mathsf{k}$, such that  \begin{equation} \label{infinitesimal_drift_condition}
\mathsf{L} \Phi ( x) \le \mathsf{k} - \mathsf{w} \; \Phi(x) \;,
\end{equation}
for all $x \in \Omega$.
\end{assumption} 

Assumption~\ref{driftcondition} implies that the associated semigroup $\mathsf{P}_t$ satisfies: \[
\mathsf{P}_t \Phi(x) \le e^{-\mathsf{w} t} \Phi(x) + \frac{\mathsf{k}}{\mathsf{w}} (1-e^{-\mathsf{w} t} ) 
\]
for every $t\ge0$ and for every $x \in \mathbb{R}^n$.  

\begin{assumption}[Associated `Small Set' Condition]  \label{minorization}
There exists a constant $\mathsf{a} \in (0,1)$ such that the sublevel set $\{ z \in \Omega \mid  \Phi(z) \le 2 \mathsf{k}/\mathsf{w}\}$
is `small' i.e.
\begin{equation} \label{weakerminor}
\| \mathsf{\Pi}_{t,x}- \mathsf{\Pi}_{t,y} \|_\TV \le 2(1 - \mathsf{a})
\end{equation}
for all $x, y$ satisfying $\Phi(x) \vee \Phi(y) \le 2 \mathsf{k}/\mathsf{w}$, where the constants $\mathsf{k}$ and $\mathsf{w}$ are taken from Assumption~\ref{driftcondition}.
\end{assumption}

\begin{rem} \label{rem:minorization}
Assumption~\ref{minorization} is typically verified by proving that:
\begin{enumerate}

\item
There exists an {\em accessible point}: that is a $y \in \Omega$ such that \begin{equation}
\mathsf{\Pi}_{t,x}(\mathcal{O}_y)>0
\end{equation} for all $x \in \Omega$, for every neighborhood $\mathcal{O}_y$ of $y$, and for some $t \ge 0$.  

\item
For every $t>0$ the associated semigroup is {\em strongly Feller}: \begin{equation} \label{eq:strongly_feller}
 \varphi \in B_b(\Omega) \implies \mathsf{P}_t \varphi \in C_b(\Omega) \;.
\end{equation}

\end{enumerate}
If a Markov process has an accessible point and its semigroup is strongly Feller, then every compact set is small (including the sublevel set of the Lyapunov function mentioned in Assumption~\ref{minorization}) and the process can have at most one invariant probability measure \cite{MeTw2009,DaZa1996,hairer2008ergodic}.    
\end{rem}

%
%

\begin{theorem}[Harris Theorem] \label{uncountable_harris}
Consider a Markov process $\mathsf{X}(t)$ on $\Omega$ with generator $\mathsf{L}$ and transition probabilities $\mathsf{\Pi}_{t,x}$,  which satisfies  Assumptions~\ref{driftcondition}  and~\ref{minorization}.   Then $\mathsf{X}(t)$ possesses a unique invariant probability measure $\mathsf{\Pi}$, and there exist positive constants $\mathsf{C}$ and $\mathsf{r}$ (both depending only on the constants $\mathsf{w}$, $\mathsf{k}$ and $\mathsf{a}$ appearing in the assumptions) such that
\[
\| \mathsf{\Pi}_{t,x} - \mathsf{\Pi} \|_\TV  \le \mathsf{C} \; \exp(-\mathsf{r} \; t)  \; \Phi(x)\;,
\]
for all $x \in \Omega$ and for any $t \ge 0$.  
\end{theorem}

\begin{rem}
It follows from \eqref{infinitesimal_drift_condition} that: \[
\mathsf{\Pi}( \Phi ) \le \frac{ \mathsf{k} }{ \mathsf{w} } \;.
\]
\end{rem}

For further exposition and a proof of Harris Theorem in a general context, see the monograph \cite{MeTw2009} (or \cite{HaMa2011} for an alternative proof), and in the context of SDEs, see Section 2 and Appendix A of \cite{HiMaSt2002}.  On a countable state space, the conditions in Theorem~\ref{uncountable_harris} become less onerous, and in fact, if the process is irreducible then all finite sets are small, in the sense of Assumption~\ref{minorization}.  In fact, it is sufficient to check irreducibility and an infinitesimal drift condition \cite[see Theorem 7.1]{meyn1993stability}.  

%
%

\begin{theorem}[Harris Theorem on a Countable State Space]  \label{countable_harris}
Consider  an irreducible Markov process $\mathsf{X}$  on a countable state space $\Omega$ with generator $\mathsf{L}$ and transition probabilities $\mathsf{\Pi}_{t,x}$, which satisfies  the infinitesimal drift condition Assumption~\ref{driftcondition}.  Then $\mathsf{X}$ possesses a unique invariant probability measure $\mathsf{\Pi}$, and there exist positive constants $\mathsf{C}$ and $\mathsf{r}$ (both depending only on the constants $\mathsf{w}$ and $\mathsf{k}$ appearing in the assumption) such that
\[
\| \mathsf{\Pi}_{t,x} - \mathsf{\Pi} \|_\TV  \le \mathsf{C} \; \exp(-\mathsf{r} \; t) \; \Phi(x)\;,
\]
for all $x \in \Omega$ and for any $t \ge 0$.  
\end{theorem}

\subsection{Geometric Ergodicity of SDE Solution}

The material here is an extension of Theorem 2.12 in \cite{BoHa2013} to non-self-adjoint SDEs.  We first show that a weak dissipativity condition on the drift implies that the generator $L$ of the SDE  \eqref{eq:model_sde} satisfies an infinitesimal drift condition.

%
%

\begin{lemma}\label{SDEdrift}
Suppose the drift field of the SDE~\eqref{eq:model_sde}  satisfies Assumption~\ref{assumptions_on_drift}.  Consider the function $V: \mathbb{R}^n \to \mathbb{R}_+$ defined as: \begin{equation} \label{eq:VforL}
V(x) = \exp\left( a |x|^{2m+2}  \right) \quad \text{for any $0 < a < \frac{\beta}{2 (m+1)} $}
\end{equation}
where $m$ and $\beta$ are the constants appearing in Assumption~\ref{assumptions_on_drift} (A2) and Lemma~\ref{lem:weak_dissipativity_mu}, respectively.  Then there exist positive constants $K$ and $\gamma$ such that 
\begin{equation} \label{SDE_LV}
L V (x)  \le K - \gamma \; V(x)
\end{equation}
holds for all $x \in \mathbb{R}^n$.
\end{lemma}

\begin{rem}
Note that the Lyapunov function in \eqref{eq:VforL} is defined in terms of the natural number $m$ appearing in Assumption~\ref{assumptions_on_drift} (A2), which sets the leading order polynomial growth of the drift field $\mu$.  This Lyapunov function is sharp in the sense that it implies integrability of any observable that can be dominated by the exponential of the magnitude of the drift field.  A discrete analog of this lemma is given in Lemma~\ref{Qdrift}, and  is a crucial ingredient to the subsequent analysis of the approximation.
\end{rem}

\begin{proof}
Set $\mathcal{E}(x) = x^T x$ and $\Theta(u) = \exp(a u^p)$, so that $V = \Theta \circ \mathcal{E}$.  We assume that $p$ is a free parameter at first.  The proof then determines $p$ such that the desired infinitesimal drift condition holds. From these definitions it follows that: \begin{gather*}
 D\mathcal{E}(x)^T \eta = 2 x^T \eta \;, \quad D^2 \mathcal{E}(x)(\eta, \eta) = 2 \eta \eta^T \;,  \\
 \Theta'(u) = a  p u^{p-1} \Theta(u) \;, \quad \text{and} \quad \Theta''(u) = ( a^2  p^2 u^{2 p - 2} + a  p  (p-1) u^{p-2} ) \Theta(u) \;.
\end{gather*}
Hence, \begin{align*}
L (\Theta \circ \mathcal{E}) \over  \Theta \circ \mathcal{E} &= \frac{1}{\Theta \circ \mathcal{E}} \left( ( \Theta' \circ \mathcal{E}) L \mathcal{E} +   ( \Theta'' \circ \mathcal{E} ) |D \mathcal{E}|^2 \right)  \\
&= \frac{1}{\Theta \circ \mathcal{E}} \left( ( \Theta' \circ \mathcal{E}) L \mathcal{E} + 4  ( \Theta'' \circ \mathcal{E} ) \mathcal{E} \right) \\
&= a  p  \mathcal{E}^{p-1}  L \mathcal{E} + 4  a^2  p^2  \mathcal{E}^{2 p - 1} + 4  a  p  (p-1)  \mathcal{E}^{p-1} 
\end{align*}
Lemma~\ref{lem:weak_dissipativity_mu} implies that \[
L \mathcal{E} \le - 2  \beta  \mathcal{E}^{m+1} + 2  \alpha  \mathcal{E}^{m+1/2} + 2  (\alpha + n)  \;.
\]
Applying this inequality to the preceding equation yields,\begin{align*}
& {L (\Theta \circ \mathcal{E}) \over  \Theta \circ \mathcal{E}}  \le 2  a  p  ( 2  a  p  \mathcal{E}^{2 p - 1}  -   \beta  \mathcal{E}^{p+m}  ) \\
& \qquad\qquad + ( 4 a  p (p-1) + 2  a  p    (\alpha +  n) + 2  \alpha  \mathcal{E}^{m+1/2} ) \mathcal{E}^{p-1}   \;. 
\end{align*}
Given the hypothesis on $a$, the negative term $- \beta \; \mathcal{E}^{p+m}$ in this upper bound dominates over the leading order positive term $ 2 \; a \; p \; \mathcal{E}^{2 p - 1}$ provided that $p \le m+1$.  Hence, the largest $p$ can be is $m+1$.  Substituting this value of $p$ yields:
\[
{L (\Theta \circ \mathcal{E}) \over  \Theta \circ \mathcal{E}}  \le 2  a  (m+1) \; \underbrace{( 2  a  (m+1) - \beta)}_{ \text{$< 0 $ by hypothesis}} \; \mathcal{E}^{2 m + 1} + \text{lower order in $\mathcal{E}$ terms} \;.
\]
Hence, for any $\gamma >0$, there exists $E^{\star}$ such that: \[
{L (\Theta \circ \mathcal{E}) \over  \Theta \circ \mathcal{E}} \le - \gamma \quad \text{for all $|x|^2 > E^{\star}$} \;.
\]  Accordingly, a suitable constant $K$ in \eqref{SDE_LV} is: \[
K = \sup_{|x|^2 \le E^{\star} } | L (\Theta \circ \mathcal{E}) + \gamma (\Theta \circ \mathcal{E}) | 
\] which is bounded since the set $\{|x|^2 \le E^{\star}\}$ is compact and $\Theta \circ \mathcal{E}$ is smooth.
\end{proof}

\begin{lemma}\label{SDEminorization}
Suppose the drift field of the SDE~\eqref{eq:model_sde}  satisfies Assumption~\ref{assumptions_on_drift}.   For every $t>0$ and $E>0$, there exists $\epsilon \in (0,1)$ such that 
 \begin{equation} \label{eq:minor}
\| \Pi_{t,x} - \Pi_{t,y} \|_\TV \le 2(1 - \epsilon)\;,
\end{equation}
for all $x, y$ satisfying 
$V(x) \vee V(y) \le E$.  
\end{lemma}

\begin{proof}
It follows from the assumed ellipticity of the equations and the assumed regularity of the drift field, that there exists a function $p(t,x,y)$ continuous in all of its arguments (for $t > 0$) such that the transition probabilities are given by  $\Pi_{t,x}(dy) = p(t,x,y)\,dy$.  Furthermore, $p$ is strictly positive (see, e.g., Prop.~2.4.1 of \cite{cerrai2001second}). Hence, by the compactness of the set $\{ x \,:\, V(x) \le E\}$, one can find a probability measure $\eta$ on $\mathbb{R}^n$ and a constant $\epsilon>0$ such that, 
\[
\Pi_{t,x} > \epsilon \eta
\]
for any $x$ satisfying $V(x) \le E$.   This condition implies the following transition probability $\tilde{P}_t$ on $\mathbb{R}^n$ is well-defined: 
\[
\tilde{\Pi}_{t,x} 
= \frac{1}{1-\epsilon} \Pi_{t,x} - \frac{\epsilon}{1-\epsilon} \eta
\]
for any $x$ satisfying $V(x) \le E$.   Therefore, 
\[
\| \Pi_{t,x} - \Pi_{t,y} \|_\TV  =
(1-\epsilon) \|   \tilde{\Pi}_{t,x} - \tilde{\Pi}_{t,y} \|_\TV
\]
for all $x, y$ satisfying $V(x) \vee V(y) \le E$.  Since the TV norm is always bounded by $2$, one obtains the desired result.
\end{proof}

With Lemmas~\ref{SDEdrift} and~\ref{SDEminorization}, we can now invoke Harris Theorem~\ref{uncountable_harris} to obtain the following geometric ergodicity result for the SDE solution.

%
%

\begin{theorem} \label{SDEgeometricallyergodic}
Suppose the drift field of the SDE~\eqref{eq:model_sde}  satisfies Assumption~\ref{assumptions_on_drift}.    Let  $V(x)$ be the Lyapunov function from Lemma~\ref{SDEdrift}. Then, the SDE admits a unique invariant probability measure $\Pi$.  Moreover,  there exist positive constants $\lambda$ and $C$ such that 
\begin{equation}\label{eq:boundTV}
\| \Pi_{t,x} - \Pi \|_\TV \le C \exp(-\lambda t) V(x)
 \end{equation}
for all $t \ge 0$ and for all $x \in \R^n$.
\end{theorem}

\begin{proof}
According to Lemma~\ref{SDEdrift}, for any $a \in (0, \beta (2 (m+1))^{-1})$, $V(x)$ is a Lyapunov function for the SDE.  Moreover, by Lemma~\ref{SDEminorization} it satisfies a small set condition on every sublevel set of $V$. Hence, Harris Theorem implies that the SDE solution has a unique invariant probability measure $\Pi$ and that \eqref{eq:boundTV} holds.
\end{proof}

\subsection{Geometric Ergodicity of Approximation} 

Here we show that the Markov process induced by the generator $Q$ in \eqref{eq:Qc2} is geometrically ergodic.  The main ingredient in this proof is a discrete analog of Lemma~\ref{SDEdrift}, which shows that the approximation has a stochastic Lyapunov function.

%
%

\begin{lemma}\label{Qdrift}
Suppose the drift field of the SDE satisfies Assumption~\ref{assumptions_on_drift}.  Let \begin{equation} \label{eq:VforQc}
V(x) = \exp\left( a |x|^{2m+2}_{2m+2}  \right)  \quad \text{for any $0 < a < \frac{\beta_0}{2 (m+1)} $}
\end{equation}
where $m$ and $\beta_0$ are the constants appearing in Assumption~\ref{assumptions_on_drift} (A2).    Then, there exist positive constants $h_c$, $K$ and $\gamma$ (uniform in $h$) such that 
\begin{equation} \label{Qc_infinitesimal_drift_condition}
Q V(x)  \le K - \gamma V(x)
\end{equation}
holds for all $x \in \mathbb{R}^n$ and $h<h_c$.
\end{lemma}

\begin{rem}
Note that the stochastic Lyapunov function in \eqref{eq:VforQc} differs from the one defined in equation \eqref{eq:VforL} of Theorem~\ref{SDEdrift}.  In particular, the Lyapunov function in \eqref{eq:VforL} uses the $2$-norm whereas the one in \eqref{eq:VforQc} uses the $(2m+2)$-norm.   Using the $(2 m+2)$-norm in \eqref{eq:VforQc} is convenient for the analysis and makes the proof more transparent, since the reaction channels in $Q$ are the coordinate axes.  However, since all norms in $\mathbb{R}^n$ are equivalent, this difference is minor, and in particular, does not affect the type of observables we can prove are integrable.  
\end{rem}

The following lemma is used in the proof of Lemma~\ref{Qdrift}. 

\begin{lemma} \label{lem:bernoulli}
 Suppose that $\mathsf{a},\mathsf{b} \ge 0$ and $\mathsf{k} \in \mathbb{N}$.  Then we have the bound: 
 \begin{equation}  \label{eq:bernoulli2}
\mathsf{a}^{\mathsf{k}} - \mathsf{b}^{\mathsf{k}} \le \mathsf{k} \; (\mathsf{a} - \mathsf{b}) \; \mathsf{a}^{\mathsf{k}-1} 
\end{equation}
\end{lemma}

\begin{proof}
Note that the bound holds if $\mathsf{a}=0$ for any $\mathsf{b} \ge 0$ and for any $\mathsf{k} \in \mathbb{N}$. 
For $\mathsf{a}>0$, the bound comes from Bernoulli's inequality \begin{equation} \label{eq:bernoulli1}
\left( \frac{\mathsf{b}}{\mathsf{a}} \right)^{\mathsf{k}} \ge 1 + \mathsf{k} \; ( \frac{\mathsf{b}}{\mathsf{a}} - 1 ) \;.
\end{equation} Indeed, recall that Bernoulli's inequality holds for all $\mathsf{b}/\mathsf{a} \ge 0$ and $\mathsf{k} \in \mathbb{N}$, with equality if and only if $\mathsf{a}=\mathsf{b}$ or  $\mathsf{k}=1$.  Since  $\mathsf{a}>0$, \eqref{eq:bernoulli1} implies \eqref{eq:bernoulli2}.
\end{proof}

Here is the proof of Lemma~\ref{Qdrift}.

\begin{proof}

Apply the generator $Q$ in \eqref{eq:Qc2} to $V(x)$ to obtain: \begin{equation} \label{eq:Qc_on_V_over_V}
\begin{aligned}
 { Q V(x) \over V(x) }  &=  \sum_{i=1}^n ( A_i^{+}(x) + A_i^{-}(x) )
 \end{aligned}
\end{equation} where we have defined: \begin{align*}
A_i^{\pm}(x) = \frac{1}{h^2} \exp\left( \pm   \frac{h}{2} \mu_i(x)   \right)  \left( { V^{\pm}_i(x) \over V(x)} -1 \right)  \;, \quad V_i^{\pm}(x) = V(x \pm e_i h) 
\end{align*}  for any $1 \le i \le n$.

\medskip

\paragraph{\em Step 1: Bound $A_i^{\pm}$}
 Lemma~\ref{lem:bernoulli} with \[
 \mathsf{a}=(x_i \pm h)^{2}\;, \quad \mathsf{b}=x_i^{2}\;, \quad \text{and} \quad \mathsf{k}=m+1  
 \] implies that:\begin{align*}
 {V^{\pm}_i \over V}   
 &= \exp\left( \vphantom{\frac{h}{2}} a ((x_i \pm h)^{2})^{m+1}  - a (x_i^{2})^{m+1} \right) \\
 &\le  \exp\left( \vphantom{\frac{h}{2}} a (m+1) ( \pm 2 x_i h + h^2)  (x_i \pm h )^{2 m} \right)  \\
 &\le  \exp\left( \vphantom{\frac{h}{2}} \pm 2 a (m+1)  x_i^{2 m +1}  h + R_i^{\pm} \right)  
\end{align*}
where we have introduced: \[
R_i^{\pm} = a (m+1)  x_i^{2 m} h^2 + a (m+1) ( \pm 2 x_i h + h^2) \left( ((x_i \pm h )^{2})^{m} - (x_i^{2} )^{m} \right) \;.
\]
A second application of  Lemma~\ref{lem:bernoulli} shows that \[
|  ((x_i \pm h )^{2})^{m} - (x_i^{2} )^{m}  | \le m | (x_i \pm h)^2 - x_i^2 | ( x_i^2 \vee (x_i \pm h)^2 )^{m-1}
 \] from which it follows that \begin{align*}
R_i^{\pm} &\le   a (m+1)  x_i^{2 m} h^2 + a m (m+1) ( \pm 2 x_i h + h^2)^2 (2 x_i^2 + 2 h^2)^{m-1} \\
&  \le  a (m+1)  x_i^{2 m} h^2 + a m (m+1) (8 x_i^2 h^2 + 2 h^4 ) (2 x_i^2 + 2 h^2)^{m-1} \\
& \le ( C_1 x_i^{2 m} +  C_0 ) h^2  
\end{align*}
for some $C_0, C_1 \ge 0$, for all $x \in \mathbb{R}^n$ and $h<1$.  Thus, we obtain the following bound:\begin{equation} \label{eq:Aipm}
\begin{aligned}
 A_i^{\pm}  & \le \frac{1}{h^2} \exp\left(  \pm  \frac{h}{2} \mu_i  \right)\\
& \qquad  \times \left( \vphantom{\frac{h}{2}}  \exp\left(\pm 2 a (m+1)  x_i^{2 m+1}  h + ( C_1 x_i^{2 m} +  C_0 ) h^2 \right)  - 1 \right) 
\end{aligned}
\end{equation}

\medskip

\paragraph{\em Step 2: Merge bounds}
Substitute \eqref{eq:Aipm} into \eqref{eq:Qc_on_V_over_V} to obtain: \begin{equation} \label{eq:QVoV}
{Q V(x) \over V(x)} \le \phi(h,x) =  \sum_{i=1}^n \phi_i(h,x) 
\end{equation} where we have introduced: \begin{equation}
\phi_i(s,x) = \frac{2}{s^2} \left( \exp( c_i(x) s^2) \cosh( a_i(x) s) - \cosh( b_i(x) s) \right)
\end{equation} and the following functions of position only: \begin{equation}
\begin{cases}
 c_i(x) =  C_1 x_i^{2m} +  C_0 \\
 b_i(x) =  \frac{1}{2}  \mu_i   \\
 a_i(x) = 2  a  (m+1)  x_i^{2 m+1} + b_i(x)
\end{cases}
\end{equation}

\medskip

\paragraph{\em Step 3: Leverage monotonicity} 
We propose to upper bound the right-hand-side of \eqref{eq:QVoV} by the limit of $ \phi(h,x)$ as $h \to 0$.  Indeed, by l'H\^{o}pital's Rule this limit is given by:  \begin{align} 
\lim_{s \to 0} \phi(s,x) &= \sum_{i=1}^n a_i^2 -  b_i^2 + 2  c_i  \nonumber \\
&= \sum_{i=1}^n  2 a (m+1) x_i^{2 m}  (2 a (m+1) x_i^{2 m+2} + \mu_i x_i  ) + 2 c_i \nonumber \\
&= \sum_{i=1}^n 2 a (m+1)(2 a (m+1) - \beta_0) x_i^{4 m+2} + \cdots \nonumber \\
&\le 2 a (m+1) \underbrace{(2 a (m+1) - \beta_0)}_{\text{$<0$ by hypothesis on $a$}} | x |^{4 m+2}_{4 m+2}   + \text{lower order in $|x|$ terms}   \label{eq:limit_of_phi}
\end{align}  
where in the third step we used Assumption~\ref{assumptions_on_drift} (A3) to bound $\mu_i x_i = \mu_i \sign(x_i) |x_i|$.  Thus, this upper bound can be made arbitrarily large and negative if $|x|$ is large enough.  Another application of l'H\^{o}pital's Rule shows that \begin{equation} \label{eq:limit_of_phiprime}
\lim_{s \to 0} \frac{\partial \phi}{\partial s}(s,x) = 0 
\end{equation} and hence, the point $s=0$ is a critical point of $\phi(s,x)$.    To determine the properties of this critical point, we apply  l'H\^{o}pital's Rule and Assumption~\ref{assumptions_on_drift} (A4) to obtain: \begin{align}
\lim_{s \to 0} \frac{\partial^2 \phi}{\partial s^2}(s,x) &= \sum_{i=1}^n \frac{1}{12} (a_i^4 - b_i^4 + 12 a_i^2 c_i + 12 c_i^2 ) \;, \nonumber \\
&= \sum_{i=1}^n \frac{1}{12} (a_i^2 - b_i^2) (a_i^2 + b_i^2) + a_i^2 c_i + c_i^2 \nonumber \\
&\le \sum_{i=1}^n \frac{1}{6} a (m+1)  (2 a (m+1) - \beta_0)  x_i^{4 m+2}   (a_i^2 + b_i^2) + \cdots  \nonumber \\
&\le \sum_{i=1}^n \frac{1}{24} a (m+1)(2 a (m+1) - \beta_0) x_i^{4 m+2}   \mu_i^2  + \cdots   \nonumber \\
&\le \sum_{i=1}^n \frac{\beta_1}{24} a (m+1)(2 a (m+1) - \beta_0) x_i^{8 m+4}  + \cdots  \nonumber \\
&\le \frac{\beta_1}{24} a (m+1) \underbrace{(2 a (m+1) - \beta_0)}_{\text{$<0$ by hypothesis on $a$}}   |x|_{8 m+4}^{8 m+4} + \text{lower order in $|x|$ terms}   \label{eq:limit_of_phiprimeprime} 
\end{align}
which is negative if $|x|$ is large enough.  By the point-wise second derivative rule, $\lim_{s \to 0} \phi(s,x) $ is a local maximum of the function $\phi(s,x)$, and therefore, the function $\phi(s,x)$ is strictly decreasing in a neighborhood of $s=0$.   Thus, we may use $\lim_{s \to 0} \phi(s,x)$ in \eqref{eq:limit_of_phi} to bound $\phi(s,x)$ as long as $s$ is sufficiently small (in order to also justify the bound on $A_i^{\pm}$ in Step 1) and $|x|$ is large enough.   In the same way, we can bound a Maclaurin series expansion of $\phi(s,x)$ as follows \begin{align*}
& \phi (s,x) - \lim_{s \to 0} \phi(s,x) = \sum_{\substack{1 \le i \le n \\ 2 \le k < \infty}} \frac{2 s^{2 (k-1)}}{(2 k)!} \left[ a_i^{2 k} - b_i^{2 k } + \cdots \right] \\
&= \sum_{\substack{1 \le i \le n \\ 2 \le k < \infty}}  \frac{2 s^{2 (k-1)}}{(2 k)!} \left[ (a_i^{2} - b_i^{2 }) \left( \sum_{j=1}^k a_i^{2 (k-j)} b_i^{2 (j-1)} \right) + \cdots \right] \\
&\le\sum_{\substack{1 \le i \le n \\ 2 \le k < \infty}}  \frac{2 s^{2 (k-1)}}{(2 k)!} \left[ 2 a (m+1)  (2 a (m+1) - \beta_0)  x_i^{4 m+2} \left( \sum_{j=1}^k a_i^{2 (k-j)} b_i^{2 (j-1)} \right) + \cdots \right. \\
&\le\sum_{\substack{1 \le i \le n \\ 2 \le k < \infty}}  \frac{2 (s/2)^{2 (k-1)}}{(2 k)!} \left[ 2 a (m+1)  (2 a (m+1) - \beta_0)  x_i^{4 m+2}  \mu_i^{2 (k-1)}  + \cdots \right. \\
&\le\sum_{\substack{1 \le i \le n \\ 2 \le k < \infty}}  \frac{2 (s/2)^{2 (k-1)}}{(2 k)!} \left[ 2  \beta_{k-1} a (m+1)  (2 a (m+1) - \beta_0)  x_i^{(4 m+2) k}   + \cdots \right. \\
&\le   \sum_{k=2}^{\infty} \frac{2 (s/2)^{2 (k-1)}}{(2 k)!}    2 \beta_{k-1} a (m+1)\underbrace{(2 a (m+1) - \beta_0)}_{\text{$<0$ by hypothesis on $a$}}  |x|^{(4 m+2) k}_{(4 m+2) k}  \\
&\qquad +   \text{lower order in $|x|$ terms} 
\end{align*}
Since Assumption~\ref{assumptions_on_drift} (A4) states that the constants $\beta_{\ell}$ and $\alpha_{\ell}$ are subfactorial in $\ell \in \mathbb{N}$, the series appearing in this bound are convergent.  This bound shows that the leading term in $|x|$ at every order in the Maclaurin series expansion is negative, and therefore, the upper bound $\lim_{s \to 0} \phi(s,x)$ on $\phi(s,x)$ holds with distance from the origin. 

\medskip

\paragraph{\em Step 4: Conclude Proof}
We have shown that there exist positive constants $h_c$, $R^{\star}$ and $\gamma$ (all independent of $h$) such that: \begin{equation} \label{eq:target}
 { Q V(x) \over V(x)  }  < - \gamma \quad \text{for all $x \in \mathbb{R}^n$ satisfying $|x| > R^{\star}$} 
\end{equation}   
and for all $h<h_c$.  Define $K$ in \eqref{Qc_infinitesimal_drift_condition} to be: \[
K =   \sup_{|x| \le R^{\star}} | L V(x) + \gamma V(x) | + \sup_{0<h<h_c} \sup_{|x| \le R^{\star}}  | (Q-L) V(x) |  \;.
\]  The first term of $K$ is independent of $h$ and bounded, since $LV(x) + \gamma V(x)$ is smooth and the set $\{ |x| \le R^{\star} \}$ is compact.  According to Lemma~\ref{lem:Qc_accuracy_model} below, the difference $|(Q - L) V(x)|$ is $\mathcal{O}(h^2)$, involves four derivatives of $V(x)$ and the drift field $\mu$.  Since this function is continuous and the set $\{ |x| \le R^{\star} \}$ is compact, the last term can be bounded uniformly in $h$.   
\end{proof}

We are now in position to invoke the countable state space version of Harris Theorem~\ref{countable_harris}, which implies geometric ergodicity of the approximation.

%
%

\begin{theorem}[Geometric Ergodicity of Approximation] \label{thm:geometric_ergodicity_Qc}
Suppose the drift field of the underlying SDE \eqref{eq:model_sde} satisfies Assumption~\ref{assumptions_on_drift}.   Then the Markov process induced by $Q$ has a unique invariant probability measure $\Pi^h$ supported on the grid $S_h$.  Furthermore, there exist positive constants $\lambda$ and $C$ such that \begin{equation} \label{eq:ell1bound}
 \| \Pi^h_{t,x} - \Pi^h \|_{\TV} \le C \; V(x) \; e^{- \lambda t} 
\end{equation}
for  all $t \ge 0$ and all $x \in S_h$.
\end{theorem}

Recall that on a countable state space the TV distance between probability measures is equivalent to the $\ell_1$ distance between their probability densities.   Thus, \eqref{eq:ell1bound} is a bound on the $\ell_1$ distance between the transition and stationary probability densities of the approximation.

\begin{proof}
The Markov process induced by $Q$ is irreducible.  Indeed, at every point on the grid, the transition rate to nearest neighbors is strictly positive. Thus, for any two grid points $x$ and $y$,  it is possible to find a finite sequence of jumps from $x$ to $y$, and for any $t>0$, there is a strictly positive probability that this finite sequence of jumps will occur.  Moreover, the generator $Q$ satisfies an infinitesimal drift condition according to Lemma~\ref{Qdrift}.  Hence, Theorem~\ref{countable_harris} implies that (i) the process possesses a stationary distribution, and (ii) the transition probability of the process converges to this stationary distribution geometrically fast in the TV metric.   
\end{proof}

\begin{rem} \label{rem:v_norm_convergence}
Let $\| \cdot \|_V$ denote the $V$-weighted supremum norm defined as: \[
\| f \|_V = \sup_{x \in S_h} \frac{ | f(x) |}{1+ \left. V \right|_{S_h} (x) } 
\]
where $f: S_h \to \mathbb{R}$ and $V: \mathbb{R}^n \to \mathbb{R}$ is a Lyapunov function of $Q$ from Lemma~\ref{Qdrift}.  Define the following set of grid functions:  \[
\ell_V = \{ g: S_h \to \mathbb{R}~ \mid ~ \| g \|_V < \infty \} \;.
\] The conditions of Theorem~\ref{thm:geometric_ergodicity_Qc}  imply exponential convergence in this function space: \begin{equation} \label{Vnorm_convergence}
\| P^h_t f  - \Pi^h(f) \|_V \le  C \| f - \Pi^h(f) \|_V e^{- \tilde \lambda t} \;,
\end{equation}  for some $\tilde \lambda>0$ and $C>0$, for all $t\ge0$ and for all $f \in \ell_V$ \cite{RoRo1997,KoMe2003,KoMe2012}.  Here we used the notation: \[
\Pi^h(f) = \sum_{x \in S_h} f(x) \nu^h(x) \;.
\]
\end{rem}

\begin{rem} \label{rem:l2_convergence}
Consider the Hilbert space \[
\ell^2(\Pi^h) = \{ f: S_h \to \mathbb{R} ~ \mid ~ \sum_{x \in S_h} |f(x)|^2 \nu^h(x) < \infty \}
\] with inner product: \[
\langle f, g \rangle_{\ell^2(\Pi^h)} = \sum_{x \in S} f(x) g(x) \nu^h(x) \;.
\]
If $Q$ is self-adjoint, in the sense that \[
\langle  f, Q g \rangle_{\ell^2(\Pi^h)}  = \langle Q f, g  \rangle_{\ell^2(\Pi^h)} \quad \text{for all $f, g \in \ell_2(\Pi^h)$}
\] then \eqref{eq:ell1bound} holds if and only if there exists a $\tilde \lambda >0$ such that: \begin{equation} \label{ell2_convergence}
\|  P_t^h f - \Pi^h(f) \|_{\ell^2(\Pi^h)} \le C \| f - \Pi^h(f) \|_{\ell^2(\Pi^h)} e^{- \tilde \lambda t} \;,
\end{equation}
for all $t\ge0$ and for all $f \in \ell_2(\Pi^h)$ \cite{RoRo1997,chen2000equivalence}.  We stress that this equivalence does not hold in general if the Markov process is not self-adjoint.  
\end{rem}

\begin{rem} \label{rem:spectral_gap}
Let $\sigma(Q)$ denote the spectrum of $Q$: the set of all points $\lambda \in \mathbb{C}$ for which $Q - \lambda \; I$ is not invertible.  If $Q$ possesses a spectral gap in $\ell^2(\Pi^h)$, i.e., a gap between its zero eigenvalue and the eigenvalue of second smallest magnitude, then its transition probabilities converge exponentially fast in $\ell^2(\Pi^h)$.  Indeed, by the spectral mapping theorem, \[
\sigma( P_t^h ) = \{ e^{t \lambda} ~\mid~ \lambda \in \sigma(Q) \} 
\]  and hence, a spectral gap in $Q$ implies that $P_t^h$ has a gap between its unit eigenvalue and its second largest eigenvalue \cite{RoRo1997,chen2000equivalence}.   
\end{rem}

%
%

\section{Properties of Realizations} \label{sec:properties}

\subsection{Basic Properties of SDE Solution}

Here we review some basic properties of the SDE solution under Assumption~\ref{assumptions_on_drift} on the drift field. 

%
%

\begin{lemma} \label{lem:moments_of_process}
Suppose the drift field of the SDE~\eqref{eq:model_sde}  satisfies Assumption~\ref{assumptions_on_drift}.  For every $\varphi: \mathbb{R}^n \to \mathbb{R}$ satisfying $|\varphi| \le  V + C_0$ for some $C_0 > 0$ and $\varphi \in \dom(L)$, we have that: \[
| \Ex_x \varphi(Y(t)) | \le   V(x) + \frac{K}{\gamma} + C_0 
\] for all $t \ge 0$ and for all $x \in \mathbb{R}^n$. Here $K$ and $\gamma$ are taken from Lemma~\ref{SDEdrift}.
\end{lemma}

As a consequence of this lemma, all moments of the SDE solution are bounded uniformly in time.  

\begin{proof}
Since $\varphi$ is dominated by $V$, this result follows directly from Lemma~\ref{SDEdrift}.  Indeed, \begin{align*}
| P_t \varphi(x) | \le  P_t | \varphi(x) |  \le P_t V(x) + C_0 \le e^{-\gamma t} V(x) + \frac{K}{\gamma} + C_0 \;.
\end{align*}
\end{proof}

%
%

The next lemma bounds the mean-squared displacement of the SDE solution.

\begin{lemma} \label{lem:displacement_of_process}
Suppose the drift field of the SDE~\eqref{eq:model_sde}  satisfies Assumption~\ref{assumptions_on_drift}.  Then the displacement of the SDE solution $Y$ satisfies: \[
\Ex_x | Y(t) - x|^2  \le 2 \; t^2  \; \left(  V(x) +  \frac{K}{\gamma} +  C_0 \right) + 4 \; n \; t 
\] for all $t \ge 0$ and for all $x \in \mathbb{R}^n$.   Here $K$, $\gamma$ and $C_0$ are taken from  Lemma~\ref{lem:moments_of_process}.
\end{lemma}

\begin{proof}
Apply Jensen and Cauchy-Schwarz inequalities to obtain \begin{align*}
| Y(t) - x |^2 &\le 2 | \int_0^t \mu(Y(s)) ds |^2 + 4 | \int_0^t dW(s) |^2\\
 & \le 2 t \int_0^t | \mu(Y(s)) |^2 ds + 4 | \int_0^t dW(s) |^2 \;.
\end{align*}  Take expectations, use the polynomial growth of $\mu$, apply Lemma~\ref{lem:moments_of_process}, and invoke the It\={o} isometry to derive the desired result.
\end{proof}

%
%

The one-sided Lipschitz property implies that the SDE solution is almost surely Lipschitz with respect to its initial condition as the next lemma states.

\begin{lemma} \label{lem:Lipschitz_property}
Suppose the drift field of the SDE~\eqref{eq:model_sde}  satisfies Assumption~\ref{assumptions_on_drift}.   For $s \le t$, let $Y_{t,s}(x)$ denote the evolution map of the SDE, which satisfies $Y_{s,s}(x)=x$ and $Y_{t,s} \circ Y_{s,r}(x) = Y_{t,r}(x)$ for $r \le s \le t$.  This evolution map satisfies: \[
| Y_{s+t,s}(x) -  Y_{s+t,s}(y)|^2  \le |x - y|^2 \; \exp(2 \; C_L^{\mu} \; t) \quad \text{a.s.}
\] for all $t \ge s$ and for all $x, y \in \mathbb{R}^n$.  Here $C_L^{\mu}$ is the one-sided Lipschitz constant of the drift field from Assumption~\ref{assumptions_on_drift} (A1).
\end{lemma}

\begin{proof}
By Ito-Taylor's formula: \begin{align*}
&| Y_{s+t,s}(x) -  Y_{s+t,s}(y)|^2  \\
&\qquad \le |x -y|^2 + 2 \int_0^t \langle Y_{s+r,s}(x) -  Y_{s+r,s}(y), \mu( Y_{s+r,s}(x) ) -  \mu(Y_{s+r,s}(y)) \rangle \; dr \\
&\qquad \le  |x -y|^2  + 2 \; C_L^{\mu} \; \int_0^t  | Y_{s+r,s}(x) -  Y_{s+r,s}(y) |^2 \; dr
\end{align*} where we used the one-sided Lipschitz property of the drift field with Lipschitz constant $C_L^{\mu}$.   Gronwall's Lemma implies the desired result.
\end{proof}

We now turn to estimates related to the semigroup of the SDE solution.

%
%

\begin{lemma} \label{lem:semigroup_is_lipschitz}
Suppose Assumption~\ref{assumptions_on_drift} holds.   Then there exists $c>0$ such that \begin{align*}
| P_t \varphi(x) - P_t \varphi(y) | \le c \; (1 \wedge t)^{-1/2} \; \| \varphi \|_{\infty} \; |y -x|
\end{align*} for all $x, y \in \mathbb{R}^n$, $\varphi \in B_b(\mathbb{R}^n)$ and $t>0$.  
\end{lemma}

We stress that the test function in this Lemma is not necessarily continuous.  Assuming that $t>0$, this Lemma shows that $P_t$ regularizes test functions, and specifically that the function $P_t \varphi$ is Lipschitz continuous for any $\varphi \in B_b(\mathbb{R}^n)$.  Since $P_t \varphi$ is also bounded, it follows that the semigroup $P_t$ is strongly Feller, see Remark~\ref{rem:minorization} for a definition of the strong Feller property.  This property of $P_t$ relies on the generator $L$ of the SDE being of elliptic type.  For a detailed proof of this result that only requires local Lipschitz continuity and a polynomial growth condition on the drift field, we refer to \S2.3.1 of \cite{cerrai2001second}.

\medskip

Now we make use of the fact that the drift field is four times differentiable and these derivatives satisfy a polynomial growth condition.

\begin{lemma} \label{lem:derivatives_of_solution}
Suppose Assumption~\ref{assumptions_on_drift} holds.  For any $p \ge 1$, for any $1 \le j \le 4$, for any $t \ge s$, and for all $x \in \mathbb{R}^n$, the evolution map $Y_{t,s}(x)$ is four times mean-square differentiable and   \[
 \sup_{x \in \mathbb{R}^n} \Ex \| D^j Y_{t,s}(x)  \|^p  \le c_p(t) \; \| \varphi \|_{\infty} 
\] where $c_p(t)>0$ is a continuous increasing function.  
\end{lemma}

See Theorem 1.3.6 of \cite{cerrai2001second} for a detailed proof of Lemma~\ref{lem:derivatives_of_solution} under the stated assumptions.  We apply this result in the next lemma.

\begin{lemma} \label{lem:derivatives_of_semigroup}
Suppose Assumption~\ref{assumptions_on_drift} holds.  For any $1 \le j \le 4$, for any $t \ge 0$, and for any $\varphi \in C^4_b(\mathbb{R}^n)$, the function $P_t \varphi$ is $j$ times differentiable and there exists $C(t)>0$ such that  \[
 \| P_t \varphi \|_j \le C(t)  \; \| \varphi \|_j
\] where $\| \cdot \|_j$ is the norm introduced in \eqref{eq:ksupnorm}.
\end{lemma}

In short, this lemma states that $P_t$ maps $C_b^4(\mathbb{R}^n)$ into itself for any $t \ge 0$.  

\begin{proof}
For any $ 1\le j \le 4$, differentiate  $j$ times with respect to $x$ the probabilistic representation of the semigroup: \[
P_t \varphi(x)  = \Ex \varphi( Y_{t,0}(x) ) 
\]  and use Lemma~\ref{lem:derivatives_of_solution} to bound these derivatives of $Y_{t,0}(x)$.  For example, for $j=1$ the chain rule implies that: \[
D P_t \varphi(x) (\eta) = \Ex D \varphi( Y_{t,0}(x) )  (D Y_{t,0}(x) (\eta) ) 
\] for any $\eta \in \mathbb{R}^n$, for all $x \in \mathbb{R}^n$ and for all $t \ge 0$.  Hence, \begin{align*}
| D P_t \varphi(x) (\eta) | &\le  |\Ex D \varphi( Y_{t,0}(x) )  (D Y_{t,0}(x) (\eta) ) | \\
&\le  \Ex |D \varphi( Y_{t,0}(x) )  (D Y_{t,0}(x) (\eta) ) | \\
&\le  \Ex \; \|D \varphi( Y_{t,0}(x) )  \|  \; \| D Y_{t,0}(x) \|  \; | \eta | \\
&\le C(t)  | \eta |  ( \sup_{x \in \mathbb{R}^n} \| D \varphi(x) \| )  
\end{align*} for some $C(t)>0$.  Here we used Lemma~\ref{lem:derivatives_of_solution} and the fact that $\varphi \in C^4_b(\mathbb{R}^n)$.  Since this upper bound is uniform in $x$, we obtain the desired estimate.
\end{proof}

\begin{rem}
As an alternative to this probabilistic proof of Lemma~\ref{lem:derivatives_of_semigroup} which is due to \cite{cerrai2001second}, an analytic approach to bounding the derivatives of $P_t \varphi$ is available in \cite{lorenzi2006analytical,bertoldi2005estimates}.    This approach is based on approximating $P_t \varphi$ by the solutions to Cauchy problems on compact spaces and assumes conditions which are similar to Assumption~\ref{assumptions_on_drift}.
\end{rem}

%
%

\subsection{A Martingale Problem}  The process $X$ induced by $Q$ has the following classification.
\begin{itemize}
\item $X$ is right-continuous in time with left limits (c\`adl\`ag).   
\item $X$ is adapted to the filtration associated to the driving noise.
\item $X$ is a pure jump process, since the process only changes its state by jumps.  
\item $X$ has the Markov property, since conditional on $X(t_0)=x$, the time to the next jump is exponentially distributed with parameter that depends only on the state $x$; and, the jump itself has a distribution that only depends on the state $x$ too.    
\end{itemize}
As such the process admits the following semimartingale representation: \begin{equation} \label{eq:semimartingale}
X(t) = X(0) + \int\limits_0^t \mu^h( X(s) ) ds + \overbrace{\mathcal{M}(t)}^{\mathclap{\text{local martingale}}} \;.
\end{equation}
The stochastic equation~\eqref{eq:Xsde} is shorthand for this integral form of $X$.  Let us prove that the local martingale in this representation is a global martingale.  The next Lemma is useful for this purpose.

\begin{lemma} \label{lem:martingale_condition}
Suppose Assumption~\ref{assumptions_on_drift} holds on the drift field, let $V(x)$ be the stochastic Lyapunov function from Lemma~\ref{Qdrift}, and let $\varphi: \mathbb{R}^n \to \mathbb{R}$.  Consider the local martingale: \[
M^{\varphi}(t) = \varphi(X(t)) - \varphi(x) - \int\limits_{0}^t Q \varphi(X(s) ) ds, \quad t \ge 0\;, \quad X(0) = x \;.
\] If $\varphi$ satisfies the growth condition: \begin{equation} \label{eq:martingale_condition}
Q \varphi^2(x) - 2 \varphi(x) Q \varphi(x) \le V(x) + C  \quad \text{for all $x \in \mathbb{R}^n$}
\end{equation} for some $C>0$, then there exists a positive constant $h_c$ such that $M^{\varphi}(t) $ is a global martingale for all $h<h_c$.
\end{lemma}

It is not clear how to relax the martingale condition \eqref{eq:martingale_condition}, since it is based on the integrability of a sharp stochastic Lyapunov function for the process.

\begin{proof}
The idea of the proof is to show that the local martingale $M^{\varphi}(t)$ is dominated by an integrable random variable, namely the maximum process $M^*(t) =\sup_{0 \le s \le t} | M^{\varphi}(s) |$.  (Recall that a local martingale dominated by an integrable random variable is a uniformly integrable global martingale \cite{ReYo1999,protter2004stochastic,Kl2012}.)  To show that the maximum process is integrable, we apply the Burkholder-Davis-Gundy inequality with parameter $p=1$, which implies that there exists a positive constant $C_1$ such that \begin{align*}
\Ex_x M^*(t) \le  C_1 \Ex_x \sqrt{\langle M^{\varphi}, M^{\varphi} \rangle(t)}  \le C_1 \sqrt{\Ex_x \langle M^{\varphi}, M^{\varphi} \rangle(t) }
\end{align*} where we used Jensen's inequality with respect to the (concave) square root function.  Since the quadratic variation of $M^{\varphi}(t)$ is given by \eqref{eq:quadratic_variation_Mphi}, the hypothesis \eqref{eq:martingale_condition} implies that there exist positive constants $C_2$ and $C_3$ such that \begin{align*}
\Ex_x M^*(t) \le C_1 \sqrt{\Ex_x  \int_0^t (V(X(s)) + C_2 ) ds} \le  ( V(x) +  C_3 ) t
\end{align*} Here we used Lemma~\ref{Qdrift}, which implies that $V(x)$ is a stochastic Lyapunov function for $X$.   Hence, $M^*(t)$ is integrable and $M^{\varphi}(t)$ is a uniformly integrable global martingale,  for all $t \ge 0$.
\end{proof}

Thus Dynkin's formula \eqref{eq:dynkins_formula} holds for any function satisfying the growth condition \eqref{eq:martingale_condition}.  The next Lemma illustrates that this class of functions includes locally Lipschitz functions.

\begin{lemma} \label{lem:martingale_solution}
Suppose Assumption~\ref{assumptions_on_drift} holds on the drift field.  Suppose $\varphi: \mathbb{R}^n \to \mathbb{R}$  is a locally Lipschitz continuous function  \[
| \varphi (y) - \varphi(x) | \le L_{\varphi}(x,y) | y - x | \quad \forall~ x,y \in \mathbb{R}^n 
\] with local Lipschitz constant $L_{\varphi}(x,y)$ that satisfies the following condition \[
 L_{\varphi}(x,y) \le C_{\varphi} ( 1 + \exp(|x|^{2 m+1}) +  \exp(|y|^{2 m+1}) )  \quad \forall~ x,y \in \mathbb{R}^n 
\] for some $C_{\varphi}>0$.  Then $\varphi$ satisfies \eqref{eq:martingale_condition}.
\end{lemma}

Since the identity map is globally Lipschitz, this Lemma enables studying solutions to the stochastic equation~\eqref{eq:semimartingale} at any $t \ge 0$.  

\begin{proof}
Let $G^{\varphi}(x) = Q \varphi^2(x) - 2 \varphi(x) Q \varphi(x)$. Use \eqref{eq:Qc2} to write  \begin{equation}
\begin{aligned}
&G^{\varphi}(x) = \sum_{i=1}^n \frac{1}{h^2} \exp( \mu_i h/2) \left( \varphi(x + h e_i) - \varphi(x) \right)^2 \\
& \qquad \qquad +  \frac{1}{h^2} \exp( - \mu_i h/2) \left( \varphi(x - h e_i) - \varphi(x) \right)^2 
\end{aligned}
\end{equation}
Since $\varphi$ is locally Lipschitz continuous by hypothesis, \begin{align*}
G^{\varphi}(x) &\le \sum_{i=1}^n C_{\varphi} \; \left( 1 +  \exp(|x|^{2 m+1}) + \exp(|x + e_i h |^{2 m+1})  \right) \; \exp(\mu_i h /2 )  \\
& \qquad +  C_{\varphi} \; \left( 1 +  \exp(|x|^{2 m+1}) + \exp(|x - e_i h |^{2 m+1})  \right) \; \exp(-\mu_i h /2 )  \\
&\le 2 \; C_1 \; \sum_{i=1}^n  \left( 1 +  \exp(|x|^{2 m+1})  \right) \cosh( \mu_i h / 2 ) \\
&\le 2 \; n \; C_1 \; ( 1 + \exp(|x|^{2 m+1})  )  \; \cosh( C^{\mu}_P ( 1+ |x|^{2m+1}) h / 2)  \\
& \le V(x) + C_2
\end{align*}
where $C_1$ and $C_2$ are positive constants.   Thus, $\varphi(x)$ satisfies the growth condition \eqref{eq:martingale_condition}, as required.  Note that in the input to the cosh function, we used the bound: \[
| \mu_i(x) | \le | \mu(x) | \le C^{\mu}_P (1 + |x|^{2 m +1} ) \quad \forall~x \in \mathbb{R}^n
\] which follows from Assumption~\ref{assumptions_on_drift} (A2).   Also in the last step, to obtain the bound by the stochastic Lyapunov function introduced in \eqref{eq:VforQc} we used the following inequality: \[
|x|^{2 m +1} \le |x|^{2 m +2}  \le C_0 |x|^{2 m + 2}_{2 m+2} 
\] which is valid since all norms are equivalent on $\mathbb{R}^n$.  
\end{proof}

As a tool to analyze realizations, we quote l'H\^{o}pital's monotonicity principle as a lemma (without proof).  

\begin{lemma} \label{lem:LMR}
Let $f, g: [a,b] \to \mathbb{R}$ be continuous functions that are differentiable on $(a,b)$, and suppose that $g'(s) \ne 0$ for every $s \in (a,b)$.  If $f'(s)/g'(s)$ is increasing/decreasing on $(a,b)$, then the function \[
\phi(s) = \frac{f(s) - f(a)}{g(s) - g(a)} 
\] is increasing/decreasing on $(a,b)$.  
\end{lemma}

Next we state an estimate on the mean-squared displacement of the process $X$, which is analogous to Lemma~\ref{lem:moments_of_process}.   

\begin{lemma}
Suppose Assumption~\ref{assumptions_on_drift} holds.  Then there exist $h_c>0$ and $C_1>0$ such that the solution of the stochastic equation~\eqref{eq:semimartingale} satisfies:\[
\Ex_x  |X(t) - x |^2 \le (C_1+V(x)) \; t \; e^{t} \;,  
\] for any $t \ge 0$ and for all $h< h_c$.
\end{lemma}

As a consequence, if $t \sim O(\epsilon)$ then the mean-squared displacement of the approximation is $O(\epsilon)$.   We emphasize that this lemma is uniform in the spatial step size.  This uniformity in $h$ follows from the fact that the stochastic Lyapunov function from Lemma~\ref{Qdrift} satisfies \eqref{Qc_infinitesimal_drift_condition} with constants $K$ and $\gamma$ that  do not depend on $h$.     

\begin{proof}
Define $\varphi(y) = (y-x)^T (y-x)$.  Since \[
| D\varphi(y) \cdot \eta | \le 2 ( |y| + |x| ) |\eta| \quad \text{for every $x, y, \eta \in \mathbb{R}^n$}
\]  it follows that $\varphi$ is locally Lipschitz continuous with a local Lipschitz constant that satisfies the sufficient conditions in Lemma~\ref{lem:martingale_solution}.  Thus, we may conclude that Dynkin's formula holds for $\varphi$ to obtain the bound: \begin{align*}
\Ex_x \varphi(X(t)) &= \varphi(x) + \Ex_x \int_0^t Q \varphi(X(s)) \; ds \\
&= \Ex_x \int_0^t 2 (X(s) - x)^T \mu^h(X(s))  \; ds + \Ex_x \int_0^t  \sum_{i=1}^n 2 \cosh( \mu_i(X(s)) h/2 ) \; ds  \\
&\le \Ex_x  \int_0^t \varphi(X(s)) \; ds + \Ex_x \int_0^t |  \mu^h(X(s))|^2 \; ds  \\
& \quad + \Ex_x \int_0^t \sum_{i=1}^n 2 \cosh( \mu_i(X(s)) h/2 ) \; ds  
\end{align*}
Here we used the elementary inequality $2 x^T y \le |x|^2 + |y|^2$.  To bound the term $(\mu^h_i)^2$, apply Lemma~\ref{lem:LMR} with $f(s) = (2 \sinh(s \mu_i/2))^2$ and $g(s)=s^2$ to obtain \[
\frac{f'(s)}{g'(s)} = \frac{2 \mu_i \cosh(s \mu_i/2) \sinh(s \mu_i/2)}{s} \;.
\] Apply Lemma~\ref{lem:LMR} once more to obtain \[
\frac{f''(s)}{g''(s)} = \mu_i^2 ( \cos(s \mu_i/2)^2 + \sinh(s \mu_i/2)^2)
\] which is monotonically increasing with respect to $s$.   Hence, there exists an $h_c>0$ such that: \[
\lim_{s \to 0} ( \mu_i^s)^2 =  \mu_i^2 \le ( \mu_i^h )^2   \le  \frac{4}{h_c^2}  (\sinh(\mu_i \dfrac{h_c}{2}))^2
\] for all $h < h_c$ and for any $1 \le i \le n$.  This inequality enables us to bound $(\mu_i^h)^2$ uniformly with respect to $h$.  In particular, there exist positive constants $C_0$ and $C_1$ such that \begin{align*}
\Ex_x \varphi(X(t)) &\le \int_0^t \Ex_x \varphi(X(s)) \; ds + \int_0^t (C_0 + \Ex_x V(X(s)) ) \; ds  \\
&\le \int_0^t \Ex_x \varphi(X(s)) \; ds + ( C_1  + V(x) ) \; t
\end{align*}
Here we used  the polynomial growth condition on the drift field from Assumption~\ref{assumptions_on_drift} (A2) and the stochastic Lyapunov function of the process  from Lemma~\ref{Qdrift}.  Apply Gronwall's Lemma to this last inequality to complete the proof.
\end{proof}

\begin{rem}
In \S\ref{sec:accuracy}, we show that for any $\varphi \in C^4_b(\mathbb{R}^n)$, $P_t^h \left. \varphi \right|_{S_h}$ is accurate with respect to $P_t \varphi$ on both finite and infinite time intervals.  Since functions in this space are Lipschitz continuous, Lemma~\ref{lem:martingale_solution} implies that any $\varphi \in C^4_b(\mathbb{R}^n)$ satisfies \eqref{eq:martingale_condition}, and therefore, Dynkin's formula holds for any $\varphi \in C^4_b(\mathbb{R}^n)$. 
\end{rem}

The next lemma quantifies the complexity of the approximation.  

\begin{lemma} \label{lem:barN}
Suppose Assumption~\ref{assumptions_on_drift} holds on the drift field.   Let $\bar N(t)$ denote the average number of jumps of the approximation $X(t)$ on the interval $[0,t]$.  Then there exist positive constants $h_c$, $C_1$ and $C_2$ such that \[
\bar N(t) \le \frac{C_1 \; t}{h^2} (V(x) +  C_2 ) 
\] for all $t \ge 0$ and $h<h_c$.  
\end{lemma}

This lemma is not a surprising result: it states that on average the number of steps of the SSA method in Algorithm~\ref{algo:ssa} is about $\lfloor t / \bar{\delta t} \rfloor$, where $t$ is the time span of simulation and $\bar{\delta t}$ is the mean holding time, which is proportional to $h^2$.   Since each step requires a function evaluation, this lemma provides a crude estimate for the total cost of the approximation.  Figures~\ref{fig:ln_complexity} and~\ref{fig:bd_simulation_Nbar} verify this estimate in the context of a log-normal process that uses adaptive mesh refinement and a $39$-dimensional Brownian Dynamics simulation, respectively.

\begin{proof}
Let $\{ t_i \}_{i=0}^{\infty}$ be a sequence of jump times for the process $X$.  The time lag between jumps $t_{i+1} - t_i$ is exponentially distributed with parameter $\lambda(x)$: \[
\lambda(x) = \frac{2}{h^2} \sum_{i=1}^n \cosh(\mu_i(x) h/2) \;.
\]  On each time interval $(t_i, t_{i+1})$ the process is constant, which implies that \[
\int_0^t \lambda(X(s)) ds = \sum_{i=0}^{\infty} \lambda(X(t_i)) ( t \wedge t_{i+1} - t \wedge t_i ) \;.
\]  Since the drift field $\mu$ satisfies a polynomial growth condition from Assumption~\ref{assumptions_on_drift} (A2), it is dominated by the stochastic Lyapunov function $V(x)$ from Lemma~\ref{Qdrift}, and hence, \begin{equation} \label{eq:fubini_condition}
\Ex_x \int_0^t \lambda(X(s)) ds \le \frac{2}{h^2} \int_0^t \Ex_x V(X(s)) ds \le \frac{2 \; t}{h^2} \; ( V(x) + \frac{K}{\gamma} ) \;.
\end{equation}
Thus, we can invoke Fubini's theorem to interchange  expectation (conditional on $X(0)=x$) and summation to get: \[
\Ex_x \int_0^t \lambda(X(s)) ds = \sum_{i=0}^{\infty}  \Ex_x \lambda(X(t_i))  ( t \wedge t_{i+1} - t \wedge t_i ) \;,
\] Since the random variable $( t \wedge t_{i+1} - t \wedge t_i ) $ is independent of $\lambda(X(t_i))$,   \begin{align*}
\Ex_x \int_0^t \lambda(X(s)) ds  &=  \sum_{i=0}^{\infty} \Ex_x \left( \vphantom{\frac{h}{2}} \Ex_x ( \lambda(X(t_i))  ( t \wedge t_{i+1} - t \wedge t_i ) \; \mid \; \mathcal{F}_{t_i} ) \right) \\
  &=  \sum_{i=0}^{\infty} \Ex_x ( \lambda(X(t_i))   \lambda(X(t_i))^{-1}  \chi_{t_i \le t} ) \\
  &= \bar{N}(t) \;.
\end{align*} Here  $\chi$ is an indicator function, $\mathcal{F}_t$ is the natural filtration of the process $X(t)$, and we used the law of iterated expectation to write the total expectation up to time $t_{i+1}$ in terms of the conditional expectation up to time $t = t_i$ using $\mathcal{F}_t$.  The desired estimate is obtained by combining the last identity with \eqref{eq:fubini_condition}.
\end{proof}

\section{Generator Accuracy}  

The following lemma improves upon Prop.~\ref{prop:Qc_accuracy} by giving a remainder term estimate.   

%
%

\begin{lemma}  \label{lem:Qc_accuracy_model}
Suppose Assumption~\ref{assumptions_on_drift} holds.   Let $f \in C^4(\mathbb{R}^n)$ and define $R f(x)$ to satisfy \[
R f(x) =   L f(x)  - Q  f (x)  \;.
\]  Then there exists $C_f(x)>0$ such that: \begin{align*}
| R f(x) | \le  h^2 C_f(x)   \quad \text{for every $x \in \mathbb{R}^n$} \;.
\end{align*}    
\end{lemma}

\begin{proof}
Retracing (and strengthening) the proof of Prop.~\ref{prop:Qc_accuracy} using Taylor's integral formula (in place of Taylor expansions) yields: \begin{equation} \label{eq:Qc_expansion}
\begin{aligned}
& Q f(x) = L f(x)  +  \underset{=-R f(x)}{\underbrace{h^2 \sum_{i=1}^n ( A_i + B_i + C_i + D_i^+ + D_i^- )}}  \;,
 \end{aligned}
\end{equation} where we have introduced: \begin{align*}
A_i &= \frac{\mu_i^3}{8}  \frac{\partial f}{\partial x_i} \int_0^1\cosh\left( s \frac{\mu_i h}{2} \right) (1-s)^2 ds    \\
B_i &=  \frac{\mu_i^2}{4} \frac{\partial^2 f}{\partial x_i^2} \int_0^1\cosh\left( s \frac{\mu_i h}{2} \right) (1-s) ds    \\
C_i &=  \frac{\mu_i}{6} \frac{\partial^3 f}{\partial x_i^3}  \int_0^1\cosh\left( s \frac{\mu_i h}{2} \right) ds  \\
D_i^{\pm} &=   \frac{1}{6} \exp\left( \pm \frac{\mu_i h}{2} \right) \int_0^1 \frac{\partial^4 f}{\partial x_i^4}(x \pm s e_i h) (1-s)^3  ds   
\end{align*} Applying the elementary inequality: \begin{align*}
& \cosh(s w) \le \cosh(w) \quad \text{for any $s \in [0,1]$ and $w \in \mathbb{R}$}  
\end{align*}
and simplifying yields the desired estimate with \begin{equation} \label{eq:Cf_of_x}
\begin{aligned}
& C_f(x) = 
C  \sum_{i=1}^n  \cosh(\mu_i h/2) \\
& \qquad \times \left( |\mu_i|^3 \left| \frac{\partial f}{\partial x_i} \right|  +  |\mu_i|^2 \left| \frac{\partial^2 f}{\partial x_i^2} \right| 
+ | \mu_i| \left| \frac{\partial^3 f}{\partial x_i^3} \right| + \sup_{|s|<1} \left| \frac{\partial^4 f}{\partial x_i^4} \right|(x+s e_i h) \right) 
\end{aligned}
\end{equation}
where $C$ is a positive constant.
\end{proof}

\section{Global Error Analysis}  \label{sec:accuracy}

In this section it is shown that the Markov process induced by the generator $Q$ in \eqref{eq:Qc2} is weakly accurate on finite time intervals.  Moreover, it is shown that the stationary probability density of $Q$ is nearby the true one.  These statements are made in the context of the model SDE~\eqref{eq:model_sde} under Assumption~\ref{assumptions_on_drift}.    The main tools used in this analysis are properties of the SDE solution and the following variation of constants formula for a mild solution of an inhomogeneous, linear differential equation on a Banach space. 

\begin{defn} \label{lem:variation_of_constants_formula}
Let $A$ be the generator of a semigroup on a Banach space $X$.   For a given initial condition $x \in X$ and a function $f: \mathbb{R}_+ \to X$, consider the differential equation: \begin{equation} \label{eq:iACP}
\begin{dcases}
\dot{v}(t) = A v(t) + f(t) \quad \text{for $t \ge 0$} \;, \\
v(0) = x \;.
\end{dcases}
\end{equation}
Then the variation of constants formula \begin{equation} \label{eq:varofconstants}
v(t) = \exp(A t) x + \int_0^t \exp(A (t-s)) f(s) ds \;,
\end{equation}
is called the mild solution of \eqref{eq:iACP}.  
\end{defn}

For more background on differential equations on Banach spaces see \cite{Pa1983}.    In the remainder of this chapter, we use \eqref{eq:varofconstants} to derive a formula for the global error of the approximation.  This formula is then estimated to determine the weak and long-time accuracy of the approximation. 

%
%

\subsection{Accuracy Theorems}   We are now in position to estimate the global error of the numerical method.

\begin{theorem}[Finite-Time Accuracy] \label{thm:finite_time_accuracy}
Suppose Assumption~\ref{assumptions_on_drift} holds.   
For every $\varphi \in \mathcal{C}_b^4(\mathbb{R}^n)$,  there exists a positive constant $C_{\varphi,t}$ such that \begin{equation} \label{eq:finite_time_accuracy}
  | P_t^h \varphi(x) -  P_t \varphi(x) | \le  h^2  \; C_{\varphi,t} \; V(x) 
\end{equation}
for any $t \ge 0$ and for all $x \in S_h$.
\end{theorem}

The main issue in the subsequent proof is that the generators of the SDE solution and numerical approximation may be unbounded.

\begin{proof}
Define the global error of the approximation as: \begin{equation} 
\epsilon(x,t) =  P_t \varphi(x) - P_t^h \varphi(x) \quad \text{for any $x \in S_h$ and $t \ge 0$} \;.
\end{equation}
Note that this global error $\epsilon(x,t)$ is a mild solution of the following inhomogeneous linear differential equation on the Banach space $\ell^{\infty}$: \begin{equation} \label{eq:global_error_ode}
\dot \epsilon(x,t) - Q \epsilon (x, t) = ( L - Q) P_t \varphi(x)  \quad \text{for any $x \in S_h$ and $t \ge 0$} 
\end{equation} with initial condition: \[
\epsilon(x,0) = 0 \quad \text{for all $x \in S_h$} \;.
\] From Definition~\ref{lem:variation_of_constants_formula}, a mild solution to this equation is given by the variation of constants formula: \begin{equation}
\epsilon(x,t) = \int_0^t P^h_{t-s} ( L - Q) P_s \varphi(x) ds \;.
\end{equation} 

In order to estimate the global error, it helps to express this formula in terms of transition probabilities: \[
\epsilon(x,t) = \int_0^t \sum_{z \in S_h} \Pi^h_{t-s,x}(z) ( L - Q) P_s \varphi(z) ds \;.
\]
Apply Lemma~\ref{lem:Qc_accuracy_model} to obtain \begin{align*}
| \epsilon(x,t) | &= | \int_0^t \sum_{z \in S_h} \Pi^h_{t-s,x}(z) ( L - Q) P_s \varphi(z) ds | \\
&\le \int_0^t \sum_{z \in S_h} \Pi^h_{t-s,x}(z) | ( L - Q) P_s \varphi(z) | ds \\
&\le h^2 \int_0^t \sum_{z \in S_h} \Pi^h_{t-s,x}(z) C_{P_s \varphi}(z)  ds
\end{align*}
where $C_{P_s \varphi}(x) $ is given in \eqref{eq:Cf_of_x} with $f=P_s \varphi$.   Note that there exists $C_0(t)>0$ such that, \[
C_{P_s \varphi}(x)  < V(x) + C_0(t)
\] for all $x \in \mathbb{R}^n$ and for all $s \in [0,t]$.  Indeed, from Lemma~\ref{lem:derivatives_of_semigroup} if $\varphi \in C^4_b(\mathbb{R}^n)$ then $P_s \varphi \in C^4_b(\mathbb{R}^n)$ for any $s \ge 0$, and by Assumption~\ref{assumptions_on_drift} (A2) on the drift field the terms that involve the drift field in \eqref{eq:Cf_of_x} are dominated by the Lyapunov function $V(x)$.
Hence, we obtain \begin{align*}
| \epsilon(x,t) | &\le h^2 \int_0^t \sum_{z \in S_h} \Pi^h_{t-s,x}(z) ( V(z) + C_0(t) ) ds  \\
&\le t \; h^2 \; ( V(x) + \tilde C_0(t) )
\end{align*}
for some $\tilde C_0(t) > 0$.  
\end{proof}


In addition, the Markov process generated by $Q$ accurately represents the stationary distribution of the SDE.

%
%

\begin{theorem}[Stationary Distribution Accuracy] \label{thm:nu_accuracy}
Suppose Assumption~\ref{assumptions_on_drift} holds.   
For every $\varphi \in \mathcal{C}_b^4(\mathbb{R}^n)$,  we have that   \begin{equation} \label{eq:nu_accuracy}
 | \Pi^h( \left. \varphi \right|_{S_h}) -  \Pi( \varphi ) | \le h^2 \; C_{\varphi} 
\end{equation}
where \[
C_{\varphi} = \sup_{0 < h \le h_c} \Pi^{h} \left( \int_0^{\infty} C_{P_s \varphi} ds \right) 
\] and $C_{P_s \varphi}(x) $ is given in \eqref{eq:Cf_of_x} with $f=P_s \varphi$.
\end{theorem}

\begin{proof}
Note that this result is not a corollary of Theorem~\ref{thm:finite_time_accuracy}, since the latter gives an upper bound on the global error that is increasing with $t$.  In order to obtain the desired estimate, we use a slightly different approach that begins with averaging the global error with respect to the invariant probability measure of the approximation $\Pi^h$ (which exists by Theorem~\ref{thm:geometric_ergodicity_Qc}): \begin{align*}
| \Pi^h( \epsilon(t) ) | &=| \sum_{z \in S_h} \Pi^h(z) \epsilon(t,z) | \\ 
&= | \Pi^h \left( \int_0^t (L-Q) P_s \varphi ds \right) | \\ 
&\le  \Pi^h \left( \int_0^t | (L-Q) P_s \varphi | ds \right) \\
&\le h^2 \Pi^h \left( \int_0^t C_{P_s \varphi} ds \right) 
\end{align*}  
where $C_{P_s \varphi}(x) $ is given in \eqref{eq:Cf_of_x} with $f=P_s \varphi$.  Here we used: (i) the Fubini Theorem to interchange space and time integrals; (ii) the remainder term estimate from  Lemma~\ref{lem:Qc_accuracy_model}; and (iii)  the fact that $ \Pi^h ( P^h_t \varphi ) = \Pi^h (\varphi)$.  
Passing to the limit as $t \to \infty$, and using Theorem~\ref{SDEgeometricallyergodic}, gives the desired estimate.
\end{proof}

\begin{rem}
One may be able to prove Theorem~\ref{thm:nu_accuracy} by adapting to this setting the version of Stein's method developed in \S6.2 of \cite{mattingly2010convergence}.  
\end{rem}

\begin{rem}
In the context of smooth initial conditions and a smooth drift field with at most polynomial growth at infinity, an estimate of the type: \[ 
 \| D^j P_t \varphi(x) \|  \le C ( 1 + |x|^s ) e^{-\gamma t} \;, \quad \forall~ t \ge 0 \;, \quad \forall ~j \ge 1 \;, \quad \forall ~x \in \mathbb{R}^n \;,
\] is available in Theorem~3.1 of \cite{Ta2002}.   Roughly, the proof of this Theorem shows exponential convergence of the first $n$ derivatives of the SDE semigroup $P_t$ in $L^2(\mathbb{R}^n)$ and then uses the Sobolev Embedding Theorem to transfer these estimates to $L^{\infty}(\mathbb{R}^n)$, i.e., trade regularity for integrability.  This procedure leads to point wise estimates on the derivatives of the semigroup.  One may be able to derive similar estimates in this context by assuming that the first $n$ derivatives of the drift field exist, but we wish to avoid making this assumption.
\end{rem}

\chapter{Analysis on Gridless State Spaces} \label{chap:gridless}

To study issues that may arise when the noise is multiplicative (or state-dependent) or if the spatial step size is adaptive, consider the generator in \eqref{eq:Qc2} but, with random spatial step size equal to $\xi h$ where $\xi \in U(1/2,1)$.    Figure~\ref{fig:gridvsgridless} illustrates the difference between points spaced in this variable way and equally spaced points on a grid.   In particular, with random spacing the state space of the approximation becomes gridless.  In this chapter we adapt the stability/accuracy results from Chapter~\ref{chap:analysis}  to this setting.   See \S\ref{sec:gridded_vs_gridless} for the definition of a gridless state space.

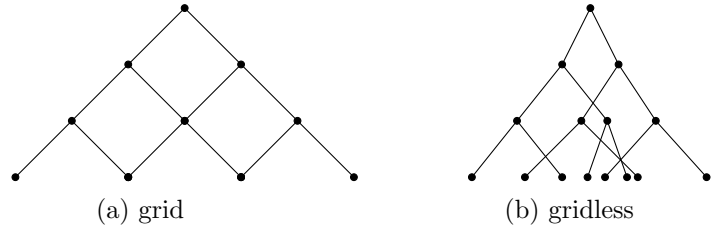
\begin{figure}[ht!]
\centering
\begin{tabular}{ll} 
\begin{tikzpicture}[scale=1.5]
\filldraw[color=black,fill=black] (3.0,2) circle (0.03);	
\filldraw[color=black,fill=black] (2.5,1.5) circle (0.03);	
\filldraw[color=black,fill=black] (3.5,1.5) circle (0.03);	
\filldraw[color=black,fill=black] (2.0,1.0) circle (0.03);	
\filldraw[color=black,fill=black] (3.0,1.0) circle (0.03);	
\filldraw[color=black,fill=black] (3.0,1.0) circle (0.03);	
\filldraw[color=black,fill=black] (4.0,1.0) circle (0.03);	
\filldraw[color=black,fill=black] (1.5,0.5) circle (0.03);	
\filldraw[color=black,fill=black] (2.5,0.5) circle (0.03);	
\filldraw[color=black,fill=black] (2.5,0.5) circle (0.03);	
\filldraw[color=black,fill=black] (3.5,0.5) circle (0.03);	
\filldraw[color=black,fill=black] (2.5,0.5) circle (0.03);	
\filldraw[color=black,fill=black] (3.5,0.5) circle (0.03);	
\filldraw[color=black,fill=black] (3.5,0.5) circle (0.03);	
\filldraw[color=black,fill=black] (4.5,0.5) circle (0.03);	
\draw[-, thin](3.0,2) -- (2.5,1.5);	
\draw[-, thin](3.0,2) -- (3.5,1.5);	
\draw[-, thin](2.0,1.0) -- (2.5,1.5);	
\draw[-, thin](3.0,1.0) -- (2.5,1.5);	
\draw[-, thin](3.0,1.0) -- (3.5,1.5);	
\draw[-, thin](4.0,1.0) -- (3.5,1.5);	
\draw[-, thin](2.0,1.0) -- (2.5,0.5);	
\draw[-, thin](2.0,1.0) -- (1.5,0.5);	
\draw[-, thin](3.0,1.0) -- (3.5,0.5);
\draw[-, thin](3.0,1.0) -- (2.5,0.5);	
\draw[-, thin](4.0,1.0) -- (3.5,0.5);
\draw[-, thin](4.0,1.0) -- (4.5,0.5);
\end{tikzpicture}    \hspace{0.5in}  &  
\begin{tikzpicture}[scale=1.5]
\filldraw[color=black,fill=black] (3.0,2) circle (0.03);	
\filldraw[color=black,fill=black] (2.75,1.5) circle (0.03);	
\filldraw[color=black,fill=black] (3.25,1.5) circle (0.03);	
\filldraw[color=black,fill=black] (2.35,1.0) circle (0.03);	
\filldraw[color=black,fill=black] (3.15,1.0) circle (0.03);	
\filldraw[color=black,fill=black] (3.58,1.0) circle (0.03);	
\filldraw[color=black,fill=black] (2.92,1.0) circle (0.03);	
\filldraw[color=black,fill=black] (2.75,0.5) circle (0.03);	
\filldraw[color=black,fill=black] (1.95,0.5) circle (0.03);	
\filldraw[color=black,fill=black] (3.325,0.5) circle (0.03);	
\filldraw[color=black,fill=black] (2.975,0.5) circle (0.03);	
\filldraw[color=black,fill=black] (3.13,0.5) circle (0.03);	
\filldraw[color=black,fill=black] (4.03,0.5) circle (0.03);	
\filldraw[color=black,fill=black] (2.42,0.5) circle (0.03);	
\filldraw[color=black,fill=black] (3.42,0.5) circle (0.03);	
\draw[-, thin](3.0,2.0)-- (2.75,1.5);	
\draw[-, thin](3.0,2.0) -- (3.25,1.5);	
\draw[-, thin](2.92,1.0) --  (3.25,1.5);	
\draw[-, thin](3.58,1.0) --  (3.25,1.5);	
\draw[-, thin](2.35,1.0)--(2.75,1.5);	
\draw[-, thin](3.15,1.0) -- (2.75,1.5);	
\draw[-, thin](3.15,1.0) -- (3.325,0.5);	
\draw[-, thin](3.15,1.0)  -- (2.975,0.5);	
\draw[-, thin](2.35,1.0) -- (2.75,0.5);
\draw[-, thin](2.35,1.0) -- (1.95,0.5);	
\draw[-, thin](3.58,1.0) -- (3.13,0.5);
\draw[-, thin](3.58,1.0)-- (4.03,0.5);
\draw[-, thin](2.92,1.0)-- (2.42,0.5);
\draw[-, thin](2.92,1.0)-- (3.42,0.5);
\end{tikzpicture} \\
\hbox{ \hspace{0.35in} (a) grid    }
& \hbox{   \hspace{0.1in} (b) gridless   } 
\end{tabular}
\caption{\small { \bf Gridded vs Gridless State Spaces.}
This figure illustrates the difference between gridded and gridless state spaces.  Points in these spaces are produced from a given seed in $\mathbb{R}$ using fixed spatial step size $h$ (left panel) and random spatial step size $ \xi h$ (right panel) where $ \xi \sim U(1/2,1)$.  The set of all such states is the state space of the Markov jump process.  In the case of a gridded state space, every pair of points is connected by a finite path on the grid.  Thus, if the jump rates are positive at every point, then the Markov process is irreducible on this gridded state space.  In contrast, in the gridless case, there exist pairs of nodes such that no finite path on the graph has those nodes as endpoints.  Thus, even if the jump rates are positive at every point, the Markov process is not irreducible on this collection of points.   }
  \label{fig:gridvsgridless}
\end{figure}


\section{A Random Walk in a Random Environment} \label{sec:context}

Here we introduce a realizable discretization that uses variable step sizes.   To introduce this discretization, in addition to the notation presented in Chapter~\ref{chap:analysis}: let $\tilde \xi: \mathbb{R}^n \to [1/2,1]$ with $\tilde \xi(x) \sim U(1/2,1)$ for all $x \in \mathbb{R}^n$, i.e., $\tilde \xi$ assigns to each point $x \in \mathbb{R}^n$ a uniformly distributed random variable on the interval $[1/2,1]$.  To ensure that the coefficients of the resulting generator are continuous, we regularize this function by convolving it with the standard mollifier  $\rho$ on $\mathbb{R}^n$ to obtain: \begin{equation} \label{eq:random_environment}
 \xi(x) = \int_{\mathbb{R}^n} \rho(x-y) \tilde \xi(y) dy \quad \text{for all $x \in \mathbb{R}^n$} \;.
 \end{equation}   Note that \[
 h_{\min} \le  \xi(x) h \le h_{\max} \quad \text{for all $x \in \mathbb{R}^n$}
 \]
where $ h_{\min} = h/2$ and $h_{\max}=h$.   The statements in this chapter are conditional on $\xi$.  In particular, we do not average over the random function $\xi$.  

Let $r \in \mathbb{N}_0$.  Below we derive conditions on this parameter in order to guarantee that the space-discrete operator is bounded and that the approximation has a stochastic Lyapunov function.  Given the field $\xi$, the spatial step size parameter $h$, and a test function $f: \mathbb{R}^n \to \mathbb{R}$, define the second-order generator:

\medskip

\begin{equation} \label{eq:Qq}
\boxed{
\begin{aligned}
& Q_q  f(x)  =    \frac{\exp\left( -  \; \xi(x)^2  h^2 \; |x|^{2r} \right)}{\xi(x)^2 h^2} \\
& \qquad \times \sum_{i=1}^n \left( \exp\left(+ \xi(x) h  \frac{\mu_i(x) }{2}  \right) ( f(x+\xi(x) h e_i) - f(x) ) \right. \\
& \qquad  \qquad   + \left. \exp\left( - \xi(x) h  \frac{\mu_i(x) }{2} \right) ( f(x-\xi(x) h e_i) - f(x) )  \right)
\end{aligned}
}
\end{equation}

\medskip

This generator induces an approximation on a gridless state space, which is a continuous-time random walk in a random environment \cite{sinai1983limiting,zeitouni2004part}.   Since a sum of discrete random variables is discrete, the state space of this approximation is countable.  However, the process is not irreducible on this gridless state space as illustrated in Figure~\ref{fig:gridvsgridless}.   In addition, if we embed the state space of the approximation into $\mathbb{R}^n$, the process is still not irreducible with respect to the standard topology on $\mathbb{R}^n$, since at any finite time there may be zero probability for the process to reach an open neighborhood of a designated point in state space starting from any other point.   Even though the approximation may lack irreducibility in the standard sense, it turns out the approximation is still long-time stable, weakly accurate, and accurate with respect to equilibrium expectations.

\section{Feller Property} \label{sec:feller}

Here we prove a regularity property of the semigroup associated to the gridless generator $Q_q$.   

%
%

\begin{lemma}  \label{lem:Qq_is_bounded}
Suppose Assumption~\ref{assumptions_on_drift} holds.    If the exponent $r$ in \eqref{eq:Qq} satisfies $r \ge m + 1$, then  for any  $h>0$ the linear operator $Q_q$ is bounded.
\end{lemma}

\begin{proof}
It suffices to show that there exists $C>0$ such that \begin{equation} \label{Qc_is_bounded}
\| Q_q f \|_{\infty}  \le C \; \| f \|_{\infty} \;,
\end{equation} for an arbitrary $f \in B_b(\mathbb{R}^n)$.   In order to prove this statement, let $J_i^{\pm}(x)$ be the reaction rate from $x$ to the state $y_i^{\pm}(x) = x\pm e_i \xi(x) h$.   By Assumption~\ref{assumptions_on_drift} (A2) the drift field satisfies a polynomial growth condition with parameter $C^{\mu}_P$, and since $r \ge m+1$, we have that \begin{align*}
J_i^{\pm}(x) &=   \frac{1}{\xi(x)^2 h^2} \exp\left( \pm \frac{h \xi(x)}{2} \mu_i(x)  -  \;  \xi(x)^2  h^2 \; |x|^{2 r} \right) \\
&\le  \frac{1}{\xi(x)^2 h^2} \exp\left( h  C^{\mu}_P \; (1 + |x|^{2 m +1} ) -  \; \xi(x)^2  h^2 \; |x|^{2 r} \right)  \\
&\le C(h)
\end{align*}
for all $x \in \mathbb{R}^n$ and for any $1 \le i \le n$.  Apply this bound on the reaction rates to obtain the following bound on $Q_q f$, \begin{align*}
\| Q_q f \|_{\infty} &= \sup_{x \in \mathbb{R}^n}  \left| \sum_{i=1}^{n} J_i^{+}(x) ( f(y_i^{+}) - f(x) ) + J_i^{-}(x) ( f(y_i^{-}) - f(x) )\right| \\
&\le 2  \| f \|_{\infty} \sup_{x \in \mathbb{R}^n} \sum_{i=1}^{n} ( J_i^{+}(x) + J_i^-(x) ) \\
&\le 4 \; n \; C(h) \; \| f \|_{\infty}
\end{align*}
which shows that the generator $Q_q$ is bounded.
\end{proof}

Since $Q_q$ is a bounded linear operator, the semigroup it induces is Feller, as the next Lemma states.

%
%

\begin{lemma} \label{lem:feller}
Suppose Assumption~\ref{assumptions_on_drift} holds.   For any $t \ge 0$ and $r \ge m+1$, the semigroup $P_t^h$ generated by $Q_q$ is Feller, which means that for each $f \in C_b(\mathbb{R}^n)$ we have that $P_t^h f \in C_b(\mathbb{R}^n)$.  
\end{lemma}

\begin{proof}
Since $Q_q$ is a bounded linear operator on the Banach space $C_b(\mathbb{R}^n)$, then the operator $P_t^h = \exp(Q_q t)$ is bounded on $C_b(\mathbb{R}^n)$, and the homogeneous Cauchy problem associated to $Q_q$ has a unique classical solution for every initial condition in $C_b(\mathbb{R}^n)$.  Finally, this solution belongs to $C_b(\mathbb{R}^n)$ for all $t \ge 0$. 
\end{proof}

\section{Generator Accuracy} \label{sec:generator_accuracy}

Here we state that $Q_q$ is accurate with respect to the infinitesimal generator of the SDE.  

%
%

\begin{lemma}  \label{lem:Qq_accuracy_model}
Suppose Assumption~\ref{assumptions_on_drift} holds.   
Let $f \in C^4(\mathbb{R}^n)$ and set \[
 R_q f(x) =     L f(x)  - Q_q  f (x)  \;.
\]  Then there exists $C_f(x)>0$ (uniform in $\xi$) such that: \begin{align*}
| R_q f(x) | \le  h^2 \; C_f(x)  \quad \text{for every $x \in \mathbb{R}^n$} \;.
\end{align*}    
\end{lemma}

\begin{proof}
This result follows immediately from Lemma~\ref{lem:Qc_accuracy_model}, since applying the bound from that lemma with $h$ replaced by $\xi h$, and bounding the mollifier in $Q_q f(x)$ by unity, yields: \[
| R_q f(x) | \le \xi^2 h^2 \tilde C_f(x)  \le h_{\max}^2 C_f(x) 
\] since $\xi h < h_{\max}$.   Note that the bounding positive function $C_f(x)$ is the same as the one appearing in Lemma~\ref{lem:Qc_accuracy_model} \begin{equation} \label{eq:Cf_of_x_q}
\begin{aligned} 
& C_f(x) = 
C  \sum_{i=1}^n  \cosh(\mu_i h/2) \\
& \qquad \times \left( |\mu_i|^3 \left| \frac{\partial f}{\partial x_i} \right|  +  |\mu_i|^2 \left| \frac{\partial^2 f}{\partial x_i^2} \right| 
+ | \mu_i| \left| \frac{\partial^3 f}{\partial x_i^3} \right| + \sup_{|s|<1} \left| \frac{\partial^4 f}{\partial x_i^4} \right|(x+s e_i h) \right) 
\end{aligned}
\end{equation}
where $C$ is a positive constant.
\end{proof}

\section{Stability by Stochastic Lyapunov Function} \label{sec:lyapunov_gridless}

Here we show that the approximation on a gridless state space has a stochastic Lyapunov function.  

%
%

\begin{lemma}\label{Qq_drift_gridless}
Suppose the drift field of the SDE~\eqref{eq:model_sde} satisfies Assumption~\ref{assumptions_on_drift}.  Let \begin{equation}
V(x) = \exp\left( a |x|^{2m+2}_{2m+2}  \right)  \quad \text{for any $0 < a < \frac{\beta_0}{2 (m+1)} $}
\end{equation}
where $m$ and $\beta_0$ are the constants appearing in Assumption~\ref{assumptions_on_drift} (A2).   Then, there exist positive constants $K$ and $\gamma$ (uniform in $h$ and $\xi$) such that 
\begin{equation} \label{Qq_infinitesimal_drift_condition_gridless}
Q_q V(x)  \le K - \gamma V(x)
\end{equation}
holds for all $x \in \mathbb{R}^n$, $h$ sufficiently small, and $r < 2 m+1$.
\end{lemma}

\begin{proof}
The argument in this proof is the similar to the gridded state space context, see Lemma~\ref{Qdrift}.  Specifically, the proof shows that there exist $h_c$, $R^{\star}$ and $\gamma>0$  such that: \begin{equation} \label{eq:Qq_target}
 { Q_q V(x) \over V(x)  }  < - \gamma \quad \text{for all $x \in \mathbb{R}^n$ satisfying $|x| > R^{\star}$} 
\end{equation}   
and for all $h< h_c$.  We then choose $K$ in \eqref{Qq_infinitesimal_drift_condition_gridless} to satisfy: \[
K \ge  \sup_{|x| \le R^{\star}} | L V(x) + \gamma V(x) | + \sup_{0<h<h_c} \sup_{|x| \le R^{\star}}  | (Q_q-L) V(x) |  \;.
\]  The first two terms of $K$ are independent of $h$ and $\xi(x)$ since $V(x)$ is smooth and the set $\{ |x| \le R^{\star} \}$ is compact.  According to Lemma~\ref{lem:Qq_accuracy_model}, the difference $|(Q_q - L) V(x)|$ is $\mathcal{O}(h^2_{\max})$, involves four derivatives of $V(x)$, and is independent of the random environment.  Since $V(x)$ is smooth and the set $\{ |x| \le R^{\star} \}$ is compact, the last term can be bounded uniformly in $h$.    Thus, we obtain the desired estimate if we can find $R^{\star}$ (uniform in $h$ and $\xi$) such that \eqref{eq:Qq_target} holds.  

Apply the generator $Q_q$ in \eqref{eq:Qq} to $V(x)$ to obtain: \begin{equation} \label{eq:Qq_on_V_over_V}
\begin{aligned}
 { Q_q V(x) \over V(x) } = \sum_{i=1}^n A_i(x) \exp\left( -   \xi(x)^2 \;  h^2 \; |x|^{2r} \right)
 \end{aligned}
\end{equation} where we have defined: \begin{align*}
&A_i(x) = \frac{1}{\xi^2 h^2} \exp\left(    \frac{\xi h}{2} \mu_i   \right)  \left( \exp\left( \vphantom{\frac{h}{2}}  a ((x_i+ \xi h)^2)^{m+1} - a ((x_i)^2)^{m+1}  \right) -1 \right) \\
& \qquad +  \frac{1}{\xi^2 h^2} \exp\left(  -  \frac{\xi h}{2} \mu_i   \right)  \left( \exp\left( \vphantom{\frac{h}{2}}   a ((x_i- \xi h)^2)^{m+1} - a ((x_i)^2)^{m+1}  \right) -1 \right) 
\end{align*} 
Note that $A_i$ involves positive and negative terms, and in order to obtain the estimate \eqref{eq:Qq_target} we must show that the negative terms dominate in a way that is uniform in $h$ and $\xi$.   For any $x \in \mathbb{R}^n$, note that  Lemma~\ref{lem:bernoulli}  implies that $A_i$ satisfies \begin{align*}
  A_i(x)& \le 
\frac{1}{\xi^2 h^2}\exp\left(   \frac{\xi h}{2} \mu_i \right)   \left( \exp\left( \vphantom{\frac{h}{2}} a ( 2 x_i \xi h + \xi^2 h^2  ) (m+1) (x_i + \xi h)^{2m} \right)  -1 \right) \\
&       + 
\frac{1}{\xi^2 h^2}\exp\left( - \frac{\xi h}{2} \mu_i \right)  \left( \exp\left( \vphantom{\frac{h}{2}} a ( -2 x_i \xi h + \xi^2 h^2  ) (m+1) (x_i - \xi h) ^{2m} \right)  -1 \right) 
\end{align*}
for  $i=1,\dotsc,n$.

These bounds on $A_i$ imply that: \begin{equation} \label{eq:QqVoV}
{Q_q V(x) \over V(x)} \le \phi(h,x) =  \sum_{i=1}^n \phi_i(h,x) 
\end{equation} where we have introduced: \begin{equation}
\phi_i(s,x) = \frac{2}{\xi^2 s^2} \left( \exp( c_i(x) \xi^2 s^2) \cosh( a_i(x) \xi s) - \cosh( b_i(x) \xi s) \right) \exp\left( - \xi^2 s^2 d(x) \right)
\end{equation} and the following functions of position only: \begin{equation}
\begin{cases}
 a_i(x) = 2 \; a \; (m+1) \; x_i^{2 m+1} + \frac{1}{2} \mu_i \\
 b_i(x) =  \frac{1}{2} \mu_i   \\
 c_i(x) =  C_1 x_i^{2m} +  C_0 \\
 d(x) =   \; |x|^{2 r}
\end{cases}
\end{equation}
where $C_1$ and $C_0$ are positive constants.  

We now derive an upper bound on the right-hand-side of \eqref{eq:QqVoV} from the limit of $ \phi(h,x)$ as $h \to 0$.  By l'H\^{o}pital's Rule this limit is given by:  \begin{align} 
\lim_{s \to 0} \phi(s,x) &= \sum_{i=1}^n \frac{ ( a_i^2 -  b_i^2 ) + 2  c_i}{2}  \nonumber \\
&=  2 a (m+1) (2 a (m+1) - \beta) |x|_{4 m+2}^{4 m+2} + \text{lower order in $|x|$ terms} \label{eq:limit_of_phi_2}
\end{align}  
Note that if $|x|$ is large enough, $\lim_{s \to 0} \phi(s,x)$ is large and negative.  Another application of l'H\^{o}pital's Rule shows that \begin{equation} \label{eq:limit_of_phiprime_2}
\lim_{s \to 0} \frac{\partial \phi}{\partial s}(s,x) = 0 
\end{equation} and hence, the point $s=0$ is a critical point of $\phi(s,x)$.    To determine the properties of this critical point, we apply  l'H\^{o}pital's Rule a third time to obtain: \begin{align}
\lim_{s \to 0} \frac{\partial^2 \phi}{\partial s^2}(s,x) &= \sum_{i=1}^n \frac{1}{12} (a_i^4 - b_i^4 + 12 a_i^2 c_i +12 c_i^2 + 3 (b_i^2 - a_i^2) d ) \xi^2 \;.
\end{align}
Since $\xi$ is uniformly bounded from below  and  $r < 2 m+1$, from the proof of Lemma~\ref{Qdrift} it follows that this term can be made negative if $|x|$ is large enough.  By the point-wise second derivative rule, $s=0$ is a local maximum of the function $\phi(s,x)$, and therefore, we can find $h_c$ such that: for any $h<h_c$, $\phi(h,x)$ is bounded by $\lim_{h \to 0} \phi(h,x) <0$, as required by the estimate in \eqref{eq:Qq_target}.  Following the proof of Lemma~\ref{Qdrift}, we can also prove that the higher order terms in a Maclaurin series expansion of $\phi(h,x)$ in powers of $h$ are dominated by a negative term if $|x|$ is large enough.  Thus, the bound $\lim_{h \to 0} \phi(h,x)$ holds as the distance from the origin increases.
\end{proof}

\begin{rem}
We will henceforth assume that $r$ in \eqref{eq:Qq} satisfies: \[
 r = m+1
\] This hypothesis ensures that the semigroup associated to the approximation is Feller (by Lemma~\ref{lem:feller}), and that the approximation has a stochastic Lyapunov function (by Lemma~\ref{Qq_drift_gridless}).
\end{rem}

Next we apply this stochastic Lyapunov function to study a martingale problem associated to the Markov process induced by $Q_q$.  

%
%

\begin{lemma} \label{lem:martingale_condition_Qq}
Suppose Assumption~\ref{assumptions_on_drift} holds on the drift field, let $V(x)$ be the Lyapunov function from Lemma~\ref{Qq_drift_gridless}, and let $\varphi: \mathbb{R}^n \to \mathbb{R}$.  Consider the local martingale: \[
M^{\varphi}(t) = \varphi(X(t)) - \varphi(X(0)) - \int_{0}^t Q_q \varphi(X(s) ) ds, \quad t \ge 0 \;.
\] If $\varphi$ satisfies the growth condition: \begin{equation} \label{eq:martingale_condition_Qq}
Q_q \varphi^2(x) - 2 \varphi(x) Q_q \varphi(x) \le V(x) + C  \quad
\end{equation} for some positive constant $C$ and for all $x \in \mathbb{R}^n$, then $M^{\varphi}(t) $ is a global martingale.  
\end{lemma}

The proof of Lemma~\ref{lem:martingale_condition} can be used to prove Lemma~\ref{lem:martingale_condition_Qq}.  It follows from Lemma~\ref{lem:martingale_condition_Qq} that Dynkin's formula \eqref{eq:dynkins_formula} holds for any function satisfying the growth condition \eqref{eq:martingale_condition_Qq}.  Lemma~\ref{lem:martingale_solution} (with $Q$ replaced by $Q_q$) provides a practical way to test if a function satisfies this condition.

\section{Accuracy} \label{sec:accuracy_gridless}

Here we consider finite-time and long-time accuracy of the approximation.

%
%

\begin{theorem}[Finite-Time Accuracy] \label{thm:finite_time_accuracy_Qq}
Suppose Assumption~\ref{assumptions_on_drift} holds.   
For every $\varphi \in \mathcal{C}_b^4(\mathbb{R}^n)$,  there exists a positive constant $C_{\varphi,t}$ such that \begin{equation} \label{eq:finite_time_accuracy_Qq}
  \| P_t^h  \varphi  -  P_t \varphi  \|_{\infty} \le  h^2 \; C_{\varphi, t}
\end{equation}
for any $t>0$.  
\end{theorem}

As in the proof of Theorem~\ref{thm:finite_time_accuracy}, the main issue in the subsequent proof is how to deal with the unbounded property of the drift field.

\begin{proof}
Defined the global error as: \begin{equation} 
\epsilon(x,t) = P_t \varphi(x) - P_t^h \varphi(x) \quad \text{for any $x \in \mathbb{R}^n$ and $t\ge0$} \;.
\end{equation}
This global error $\epsilon(x,t)$ is a mild solution of the following inhomogeneous linear differential equation on the Banach space $B_b(\mathbb{R}^n)$: \begin{equation} \label{eq:global_error_ode_Qq}
\dot \epsilon(x,t) - Q_q \epsilon (x, t) = ( L - Q_q) P_t \varphi(x)  \quad \text{for any $x \in \mathbb{R}^n$  and $t\ge0$}  
\end{equation} with initial condition: \[
\epsilon(x,0) = 0 \quad \text{for all $x \in \mathbb{R}^n$} \;.
\]   Definition~\ref{lem:variation_of_constants_formula} implies this solution can be written in terms of a variation of constants formula: \begin{equation}
\epsilon(x,t) = \int_0^t P^h_{t-s} ( L - Q_q) P_s \varphi(x) ds \;.
\end{equation} 
The desired estimate follows from the proof of Theorem~\ref{thm:finite_time_accuracy} with Lemma~\ref{lem:Qc_accuracy_model} replaced with Lemma~\ref{lem:Qq_accuracy_model}.
\end{proof}

%
%

\begin{lemma} \label{gridless_IM}
Suppose Assumption~\ref{assumptions_on_drift} holds on the drift field.  Then the Markov process generated by $Q_q$ has at least one invariant measure.
\end{lemma}

\begin{proof}
This result follows by the classical Krylov-Bogoliubov argument \cite{DaZa1996}, since the semigroup associated to $Q_q$ is Feller and the process generated by $Q_q$ has a stochastic Lyapunov function.
\end{proof}

%
%

\begin{theorem}[Stationary Distribution Accuracy] \label{thm:nu_accuracy_gridless}
Suppose Assumption~\ref{assumptions_on_drift} holds.   For every $\varphi \in \mathcal{C}_b^4(\mathbb{R}^n)$, we have   \begin{equation} \label{eq:nu_accuracy_gridless}
 | \Pi^h( \varphi ) -  \Pi( \varphi ) | \le  \; h^2 \; C_{\varphi} 
\end{equation}
where $C_{\varphi}$ is defined as: \[
C_{\varphi} = \sup_{0 < h \le h_c} \Pi^{h} \left( \int_0^{\infty} C_{P_s \varphi} ds \right) 
\] and $C_{P_s \varphi}(x) $ is given in \eqref{eq:Cf_of_x_q} with $f=P_s \varphi$.
\end{theorem}

\begin{proof}
 In order to obtain the desired estimate, we average the global error with respect to the invariant probability measure of the approximation $\Pi^h$ (which exists by Lemma~\ref{gridless_IM}): \begin{align*}
\Pi^h( \epsilon ) &= \sum_{z \in S_h} \Pi^h(z) \epsilon(t,z) \\ 
&= \Pi^h \left( \int_0^t (L-Q) P_s \varphi ds \right)  
\end{align*}  
Passing to the limit $t \to \infty$, and using ergodicity of the SDE solution, gives the desired estimate.

\end{proof}

\chapter{Tridiagonal Case}  \label{chap:tridiagonal} 

For scalar SDEs, the operators given in \S\ref{sec:scalar_case} have an infinite tridiagonal matrix representation.  We use this representation in this chapter to show that the invariant density, mean first passage time, and exit probability of the approximation satisfy a tridiagonal system of equations that can be recursively solved.   Explicit formulas are given for these solutions, which were used to benchmark the simulations of the scalar SDEs appearing in \S\ref{sec:cubic_oscillator}, \S\ref{sec:downhill_time} and \S\ref{sec:cir_process}.   
 Most of these results are classical, but their application to numerical methods for SDEs seems somewhat new.

\section{Invariant Density} \label{sec:nu_1D}

Let $S = \{ x_i \}_{i \in \mathbb{Z}}$ be a collection of grid points in $\mathbb{R}$ ordered by index so that: \[
\dotsc < x_{i-1} < x_i < x_{i+1} < \dotsc 
\]  The  one-dimensional generators in \eqref{eq:tQu_1d}, \eqref{eq:tQc_1d}, and \eqref{eq:Qe_1D}  can all be written as a tridiagonal operator $Q$ whose action on a grid function $f: S \to \mathbb{R}$ is given by: \begin{equation} \label{eq:tridiagonal_generator}
(Q f)_i = Q_{i, i+1} f_{i+1} - ( Q_{i, i+1} + Q_{i, i-1} ) f_i + Q_{i, i-1} f_{i-1} 
\end{equation}
where we used the shorthand notation: $f_i = f(x_i)$ and $Q_{i,i+1} = Q(x_i, x_{i+1})$  for grid points $x_i, x_{i+1} \in S$.   Let $Q^*$ be the transpose of $Q$, so that $(Q^*)_{i,j} = Q_{j,i}$ for all $i,j \in \mathbb{Z}$.  Recall, that an invariant (not necessarily integrable) density of $Q$ is a grid function $\nu_d: S \to \mathbb{R}$ that solves: \begin{equation} \label{invariant_density_recurrence}
 (Q^* \nu_d)_i = Q_{i+1,i} (\nu_d)_{i+1} - (Q_{i,i+1} + Q_{i,i-1}) (\nu_d)_i + Q_{i-1,i} (\nu_d)_{i-1} = 0 
\end{equation}
for every $i \in \mathbb{Z}$.   A generator of the form \eqref{eq:tridiagonal_generator} is always `$\nu_d$-symmetric' as the next Proposition states.  

\begin{prop} \label{prop:nu_symmetry_1d}
Suppose $\nu_d$ satisfies \eqref{invariant_density_recurrence} with seed values $(\nu_d)_0 = 1$ and $(\nu_d)_1 = Q_{0,1}/Q_{1,0}$. 
Then $Q$ is $\nu_d$-symmetric, in the sense that \begin{equation} \label{rho_symmetry}
 (\nu_d)_i Q_{i,i+1} = (\nu_d)_{i+1} Q_{i+1, i} \quad \text{for all $i \in \mathbb{Z}$} \;.
\end{equation}
The converse is also true.
\end{prop}

\begin{proof}
The $\nu_d$-symmetry condition \eqref{rho_symmetry} immediately implies the invariant density condition \eqref{invariant_density_recurrence}.  For the converse, use induction: the hypothesis on the seed values implies that \eqref{rho_symmetry} holds for $i=0$.  Assume \eqref{rho_symmetry}  holds for $i = j-1 > 0$.  Since $\nu_d$ satisfies \eqref{invariant_density_recurrence}, it follows that \[
Q_{j+1,j} (\nu_d)_{j+1} - Q_{j,j+1} (\nu_d)_j =  Q_{j,j-1} (\nu_d)_j - Q_{j-1,j} (\nu_d)_{j-1} 
\]  The inductive hypothesis implies that the right-hand-side of this equation is zero, and therefore, \eqref{rho_symmetry}  holds for all $i \ge 0$.  A similar proof works for $i<0$ and is therefore omitted.
\end{proof}

In addition to $\nu_d$-symmetry, a closed form solution to the recurrence relation in \eqref{invariant_density_recurrence} is available.

\begin{prop}[Formula for Invariant Density]
A formula for the solution to  \eqref{invariant_density_recurrence} with seed values $(\nu_d)_0 = 1$ and $(\nu_d)_1 = Q_{0,1}/Q_{1,0}$
is given by  \begin{equation} \label{rho_formula}
(\nu_d)_{i} = \begin{dcases}
\prod_{k=0}^{|i|-1} \frac{Q_{k,k+1}}{Q_{k+1,k}} &  \text{if $i > 0$} \\
\prod_{k=0}^{|i|-1} \frac{Q_{-k,-k-1}}{Q_{-k-1,-k}}  & \text{if $i< 0$} \\
\vphantom{\prod_{k=0}^{-i-1}} 1 & \text{if $i=0$} 
\end{dcases}
\end{equation}
\end{prop}

This expression for $\nu_d$ in terms of the entries of $Q$ can be easily implemented in e.g.~MATLAB.  

\begin{proof}
The proof shows that $\nu_d$ in \eqref{rho_formula} satisfies $\nu_d$-symmetry, which by Prop.~\ref{prop:nu_symmetry_1d} implies that $\nu_d$ solves  \eqref{invariant_density_recurrence}.  Evaluating \eqref{invariant_density_recurrence} at $i=0$ implies that the formula \eqref{rho_formula} holds for $i= \pm 1$, since $(\nu_d)_{\pm 1} =   Q_{0,\pm 1} / Q_{\pm 1, 0}$.  For $i > 1$,  \begin{align*}
 (\nu_d)_{i+1} Q_{i+1,i} &=  \overset{(\nu_d)_{i+1}}{\overbrace{\left(  \prod_{k=0}^i \frac{Q_{k,k+1}}{Q_{k+1,k}}  \right)}} Q_{i+1, i} =  \underset{(\nu_d)_{i}}{\underbrace{\left(  \prod_{k=0}^{i-1} \frac{Q_{k,k+1}}{Q_{k+1,k}}  \right)}} Q_{i,i+1} = (\nu_d)_i Q_{i, i+1} 
\end{align*}
which implies the desired result for $i \ge 0$.   A similar result holds for $i<-1$.
\end{proof}

A Markov process on a countable state space is ergodic (or positive recurrent) if it is: \begin{itemize}
\item irreducible; 
\item recurrent (or persistent); and,
\item has an invariant probability distribution.
\end{itemize}
Thus, we have the following result \cite{Ki2015}.

\begin{theorem} [Sufficient Conditions for Ergodicity]
Let $X(t)$ be the Markov process with tridiagonal generator $Q$ given by \eqref{eq:tridiagonal_generator}, Markov semigroup $\{ P_{ij}(t) \}_{t \ge0}$ defined as: \[
P_{ij}(t) = \P( X(t) = x_j ~|~ X(0) = x_i ) \;. 
\] 
and invariant density $\nu_d$ whose formula is given in \eqref{rho_formula}.  This process is ergodic if it satisfies the following conditions: 
\begin{description}
\item[(C1)] For every integer $i$ \[
Q_{i,i+1}>0  \quad \text{and} \quad  Q_{i,i-1}>0 
\]
\item[(C2)] The following series converges \[
Z = \sum_{i\in \mathbb{Z}} (\nu_d)_i < \infty
\]
\end{description}
Let $\pi_d = \nu_d/Z$.  Then we have that: \[
\sum_{j \in \mathbb{Z}} | P_{ij}(t) - (\pi_d)_j | \to 0  \quad \text{ as $t \to \infty$ for any $i \in \mathbb{Z}$} \;.
\]
\end{theorem}

Here is a simple sufficient condition for (C2) to hold.

\begin{lemma} \label{lemma_ratio_test}
Assume that $\nu_d$ satisfies  \eqref{invariant_density_recurrence}  with seed values
$(\nu_d)_0 = 1$ and $(\nu_d)_1 = Q_{0,1}/Q_{1,0}$.  If the following inequalities are satisfied:
\[
\lim_{k \to \infty} \frac{Q_{k,k+1}}{Q_{k+1,k}} < 1 \quad \text{and} \quad 
\lim_{k \to -\infty} \frac{Q_{k,k-1}}{Q_{k-1,k}} < 1  
\]  then $\nu_d$ is normalizable.
\end{lemma}

\begin{proof}
This lemma is an application of the ratio test, which gives a sufficient condition to establish convergence of a series.  
From \eqref{rho_formula} the quotient this test uses can be written as a ratio of jump rates: \[
\frac{(\nu_d)_{k+1}}{(\nu_d)_k} =  \frac{Q_{k,k+1}}{Q_{k+1,k}} \;. 
\] Hence, the ratio test for convergence of the series $\sum_{i \in \mathbb{N}} (\nu_d)_i$  reduces to confirming that: \[
\lim_{k \to \infty} \frac{Q_{k,k+1}}{Q_{k+1,k}} < 1 \;,
\]   and similarly for $i<0$.
\end{proof}

\section{Stationary Density Accuracy}  \label{sec:nu_accuracy_1D}

As an application of Formula~\eqref{rho_formula}, we obtain explicit estimates for the error in (sup norm) of the numerical stationary density of the generator $Q_c$ introduced in \eqref{eq:tQc_1d}.  

\begin{prop} \label{nu_tQc_1d}
Consider the SDE \eqref{sde_1d} with $\Omega=\mathbb{R}$.  Assume that:
\medskip
\begin{itemize}
\item $U(x)$ is three times differentiable for all $x \in \mathbb{R}$, and \[
\sign(x) U'(x) \to \infty \quad \text{as $|x| \to \infty$} \;.
\]
\item  $M(x)$ is differentiable and $M(x)>0$, for all $x \in \mathbb{R}$.
\end{itemize}
\medskip
Let $\nu(x) = \exp(-U(x))$.  Suppose that the spatial step size is constant and that $x_i = i h$ for all $i \in \mathbb{Z}$.
Consider the second-order scheme \eqref{eq:tQc_1d} with off-diagonal entries given by: \[
( Q_c)_{i,i+1} = \frac{1}{h^2} \frac{ M_i + M_{i+1} }{2} \exp\left( -U_i^{\prime} \frac{h}{2} \right) \;, \quad ( Q_c)_{i,i-1} = \frac{1}{h^2} \frac{ M_i + M_{i-1} }{2} \exp\left( U_i^{\prime} \frac{h}{2} \right) \;.
\]   Let $\nu_c$ be an invariant density of $ Q_c$ normalized so that $\nu_c(x_0) = \nu(x_0)$.  Then, \begin{equation}
\sup_{x_i \in S} | \nu(x_i) - \nu_c(x_i) | <  C h^2
\end{equation}
where \[
C = \frac{1}{12} \sup_{x} \nu(x)  |x| \sup_{\xi < |x|} | U'''(\xi) |  \exp\left( \frac{h^2}{12} |x| \sup_{\xi < |x|} | U'''(\xi) | \right) \;.
\]
\end{prop}

\begin{proof}
First, we apply the test given in Lemma~\ref{lemma_ratio_test}  to check that an invariant density of $Q_c$ is normalizable.  The hypothesis on $U'(x)$ implies that: \[
\lim_{k \to \infty} \frac{(Q_c)_{k,k+1}}{(Q_c)_{k+1,k}}  =\lim_{k \to \infty} \exp\left( - \frac{h}{2} ( U'_k + U'_{k+1} ) \right) = 0
\] and similarly for $k < 0$.  Thus, $Q_c$ has a stationary density.   

Next, we use formula~\ref{rho_formula} to quantify the point-wise error in the approximation to the density $\nu(x)$.   Let $\tilde \nu_c(x)$ denote the stationary density of $Q_c$.  (This density will be rescaled in order to obtain $\nu_c(x)$.)  Substituting the rates appearing in~\eqref{eq:tQc_1d} into formula~\ref{rho_formula} and simplifying yields: \[
\tilde \nu_c(x_i) = \begin{dcases}
 \exp\left( - \frac{h}{2} \sum_{k=0}^{|i|-1} ( U_k' + U_{k+1}' )  \right) &  \text{if $i > 0$} \\
 \exp\left( \frac{h}{2} \sum_{k=0}^{|i|-1} ( U_{-k}' + U_{-k-1}' )  \right)   & \text{if $i< 0$} \\
1 & \text{if $i=0$} 
\end{dcases}
\]  
We consider two cases depending on the value of $x_i$.  Set $\nu_c = \tilde \nu_c / \exp( U_0)$, so that $\nu_c(x_0) = \nu(x_0)$.

\begin{itemize}
\item When $x_i > 0$, \[
 \nu_c(x_i) = \exp\left( - \int_0^{x_i} U'(s) ds - U_0 + \epsilon_h^+ \right) = \exp\left( - U_i + \epsilon_h^+ \right) \;,
\] where we have introduced\begin{align*}
\epsilon_h^+ &= \int_0^{x_i} U'(s) ds - \frac{h}{2}  \sum_{k=0}^{i-1} ( U'_k + U'_{k+1} ) \\
&= \sum_{k=0}^{i-1} \underset{\delta_k^+}{\underbrace{\left( \int_{x_k}^{x_{k+1}} U'(s) ds - \frac{h}{2} ( U'_k + U'_{k+1} ) \right)}} = \sum_{k=0}^{i-1} \delta_k^+
\end{align*}
Use integration by parts twice to obtain the following expression for $\delta_k^+$: \[
\delta_k^+ = \frac{1}{2} \int_{x_k}^{x_{k+1}} \left( (s-x_{k+\frac{1}{2}} )^2 - \left( \frac{h}{2} \right)^2 \right) U'''(s) ds = - \frac{h^3}{12} U'''(\xi_k^+)
\] where $x_{k+\frac{1}{2}} = (x_k + x_{k+1})/2$ and $x_k \le \xi_k^+ \le x_{k+1}$.   
 Next we apply the elementary inequality: \begin{equation} \label{eq:exp_inequality}
 |e^a - 1| \le e^{|a|} - 1 \le |a| e^{|a|} \quad \forall~~a \in \mathbb{R}
 \end{equation}
to obtain \begin{align*}
 | \nu(x_i) - \nu_c(x_i) |  &=   \nu(x_i) \left| \exp\left( - \frac{h^3}{12} \sum_{k=0}^{i-1} U'''( \xi_k^+ ) \right) -1 \right| \\
&  \le  \frac{h^2}{12}   \nu(x_i)  |x_i| \sup_{\xi < |x_i|} | U'''(\xi) |  \exp\left( \frac{h^2}{12}  |x_i| \sup_{\xi < |x_i|} | U'''(\xi) | \right) \;.
\end{align*}

\medskip

\item When $x_i < 0$, we obtain:\[
 \nu_c(x_i) = \exp\left( \int_{x_i}^{0} U'(s) ds - U_0 + \epsilon_h^- \right) = \exp\left( - U_i  + \epsilon_h^- \right) \;,
\] and as in the previous case, \[
 \epsilon_h^- =  - \sum_{k=0}^{i-1} \frac{h^3}{12} U'''(\xi_{k}^-) 
\] where $\xi_{k}^- \in [x_{k-1}, x_k]$.
Applying \eqref{eq:exp_inequality} we obtain, \begin{align*}
| \nu(x_i) - \nu_c(x_i) |  &=   \nu(x_i) \left| 1 - \exp\left( - \frac{h^3}{12} \sum_{k=0}^{i-1} U'''( \xi_k^- ) \right) \right| \\
&  \le  \frac{h^2}{12}   \nu(x_i)  |x_i| \sup_{\xi < |x_i|} | U'''(\xi) |  \exp\left( \frac{h^2}{12}  |x_i| \sup_{\xi < |x_i|} | U'''(\xi) | \right) \;.
\end{align*}
\end{itemize}

Taking the sup over $x_i$ in the last inequalities in each case gives the desired estimate.
\end{proof}

\section{Exit Probability} \label{sec:committor_1D}

Here we demonstrate how to compute the exit probability or committor function of the approximation induced by $Q$.  As a side remark,  the committor function can be used to analytically describe the statistical properties of trajectories between two metastable states of an SDE \cite{EVa2004, Va2006, EVa2006, EVa2010}.   Given `reactant' and `product' states at the points $a$ and $b$ (resp.) with $a<b$, the committor function is defined as the probability that the process initiated at a point $x \in [a,b]$ reacts in the sense that it first hits $b$ before it hits $a$.  Assume that  the reactant and product states are the left and right endpoints of the grid $S$ and suppose that the grid points are arranged so that: \[
 x_0 = a  < \dotsb <  x_N = b 
 \]
Let $q_d = \{ (q_d)_i \}_{i=0}^N$ denote the committor function associated to the generator $Q$ which satisfies: \begin{equation}
\label{committor_recurrence}
\begin{aligned}
(Q q_d)_i &= 0 \quad  \text{if $0< i < N$} \;, \\
 (q_d)_0 &=0 \;, \quad (q_d)_N =1 \;.
\end{aligned}
\end{equation}
The following proposition gives a formula for $q_d$ in terms of the reaction rates and invariant density $\nu_d$ of the approximation.

\begin{prop}[Formula for Exit Probability]  \label{prop:committor}
The committor function which solves \eqref{committor_recurrence} can be written as: \begin{equation} \label{eq:committor_1d}
(q_d)_i = \begin{dcases}
Z \sum_{j=0}^{i-1}  \frac{(\nu_d)_1 Q_{1,0}}{(\nu_d)_{j+1} Q_{j+1,j}} &  \text{if $0< i \le N$} \\
\vphantom{\sum_{j=0}^{i-1}} 0 & \text{if $i=0$} 
\end{dcases}
\end{equation} where $Z$ is a constant selected so that $(q_d)_N = 1$.  
\end{prop}

\begin{proof}
Let $\tilde q_d$ solve the recurrence relation in \eqref{committor_recurrence}, which implies that \begin{equation} \label{tilde_s_recurrence}
Q_{i,i+1} ((\tilde q_d)_{i+1} - (\tilde q_d)_i) =  Q_{i,i-1} ( (\tilde q_d)_i -(\tilde q_d)_{i-1} )     \quad \text{(for $i \ge 1$)}
\end{equation} with seed values $(\tilde q_d)_0 = 0$ and $(\tilde q_d)_1 = 1$.   Then $q_d = \tilde q_d / (\tilde q_d)_N$ solves the same recurrence relation with the desired boundary conditions: $( q_d)_0=0$ and $( q_d)_N = 1$. Hence, this proof derives a solution to \eqref{tilde_s_recurrence}.  Multiply \eqref{tilde_s_recurrence} by $(\nu_d)_i$ and use $\nu_d$-symmetry to obtain: \[
(\nu_d)_{i+1} Q_{i+1,i} ((\tilde q_d)_{i+1} - (\tilde q_d)_i) = (\nu_d)_i Q_{i,i-1} ( (\tilde q_d)_i -(\tilde q_d)_{i-1} ) 
\]  Sum these equations over the index $i$ from $1$ to any $k \ge 1$ and use  summation by parts to obtain: \begin{gather*}
 \sum_{i=1}^k (\nu_d)_{i+1} Q_{i+1, i} ((\tilde q_d)_{i+1} - (\tilde q_d)_i ) = \sum_{i=1}^k  (\nu_d)_i Q_{i,i-1} ((\tilde q_d)_i - (\tilde q_d)_{i-1} ) \\
 \sum_{i=2}^{k+1} (\nu_d)_{i} Q_{i,i-1} ((\tilde q_d)_i - (\tilde q_d)_{i-1}) =  \sum_{i=1}^k  (\nu_d)_i Q_{i,i-1} ((\tilde q_d)_i - (\tilde q_d)_{i-1} )   \\
\vphantom{ \sum_{i=2}^{k+1} } (\nu_d)_{k+1} Q_{k+1,k} ((\tilde q_d)_{k+1} -(\tilde q_d)_k ) =  (\nu_d)_1 Q_{1,0} ((\tilde q_d)_1 - (\tilde q_d)_0) \;.
\end{gather*}
Combine this result with a telescoping sum to obtain: \begin{align*}
(\tilde q_d)_{k+1} &= \vphantom{ \sum_{j=0}^k} (\tilde q_d)_{k+1} - (\tilde q_d)_0 + (\tilde q_d)_0 \\
& = \sum_{j=0}^k ((\tilde q_d)_{j+1} -(\tilde q_d)_j ) + (\tilde q_d)_0 \\
& = \sum_{j=0}^k  \frac{ (\nu_d)_1 Q_{1,0}}{ (\nu_d)_{j+1} Q_{j+1,j} } ((\tilde q_d)_1 -(\tilde q_d)_0) + (\tilde q_d)_0  \quad \text{for $k \ge 0$} \;.
\end{align*} 
Substituting the seed values $(\tilde q_d)_0=0$ and $(\tilde q_d)_1=1$ gives the solution to \eqref{tilde_s_recurrence}.
The desired formula is obtained by setting $q_d = \tilde q_d / (\tilde q_d)_N$ or $Z=1 / (\tilde q_d)_N$ in \eqref{eq:committor_1d}.
\end{proof}

\section{Mean First Passage Time} \label{sec:mfpt_1D}

A similar recurrence relation can be derived for the mean first passage time (MFPT) of the approximation from any point in $[x_L, x_R]$ to $x_L$ or $x_R$.  
Define $\tau$ as \[
\tau = \min \{ t \ge 0 ~|~ X(t) \in \{x_L, x_R\} \;,~~ X(0) = x_i \} \;.
\]  
Let $u_d = \{ (u_d)_i \}_{i=0}^N$ be the MFPT, which satisfies $u_d(x_i) = \Ex_{x_i}(\tau)$ and, \begin{equation} \label{mfpt_recurrence}
\begin{aligned}
(Q u_d)_i &= -1 \quad  \text{if $0< i < N$} \;, \\
u_d(x_L) &= u_d(x_R) = 0
\end{aligned}
\end{equation}   
The following proposition gives a formula for $u_d$ in terms of the reaction rates and invariant density $\nu_d$ of the approximation.

\begin{prop}[Formula for MFPT]  \label{prop:mfpt}
The MFPT which solves \eqref{mfpt_recurrence} can be written as: \begin{equation}
(u_d)_i = \begin{dcases}
Z \sum_{j=0}^{i-1}  \frac{(\nu_d)_1 Q_{1,0}}{(\nu_d)_{j+1} Q_{j+1,j}} - \sum_{j=0}^k \sum_{i=1}^j \frac{(\nu_d)_i}{(\nu_d)_{j+1} Q_{j+1,j}}&  \text{if $0< i < N$} \\
\vphantom{\sum_{j=0}^{i-1}}  0 & \text{if $i=N$ or $i=0$} 
\end{dcases}
\end{equation} where $Z$ is a constant selected so that $(u_d)_N = 0$.
\end{prop}

\begin{proof}
This proof is similar to the proof of Proposition~\ref{prop:committor}, except that we do not introduce an auxiliary solution.   Multiply \eqref{mfpt_recurrence} by $(\nu_d)_i$ and use $\nu_d$-symmetry to obtain: \[
(\nu_d)_{i+1} Q_{i+1,i} ((u_d u_d)_{i+1} - (u_d)_i) - (\nu_d)_i Q_{i,i-1} ( (u_d)_i -(u_d)_{i-1} ) = -( \nu_d)_i
\]  Sum these equations over the index $i$ from $1$ to any $k \ge 1$ and use summation by parts to obtain: \begin{gather*}
 \sum_{i=1}^k (\nu_d)_{i+1} Q_{i+1, i} ((u_d)_{i+1} - (u_d)_i ) = \sum_{i=1}^k  (\nu_d)_i Q_{i,i-1} ((u_d)_i - (u_d)_{i-1} )  - \sum_{i=1}^k (\nu_d)_i\\\
 \sum_{i=2}^{k+1} (\nu_d)_{i} Q_{i,i-1} ((u_d)_i - (u_d)_{i-1}) =  \sum_{i=1}^k  (\nu_d)_i Q_{i,i-1} ((u_d)_i - (u_d)_{i-1} ) - \sum_{i=1}^k (\nu_d)_i\  \\
 \vphantom{\sum_{i=2}^{k+1}} (\nu_d)_{k+1} Q_{k+1,k} ((u_d)_{k+1} -(u_d)_k ) =  (\nu_d)_1 Q_{1,0} ((u_d)_1 - (u_d)_0) - \sum_{i=1}^k (\nu_d)_i\;.
\end{gather*}
Combine this result with a telescoping sum to obtain: \begin{align*}
(u_d)_{k+1} &= \vphantom{\sum_{i=2}^{k+1}} (u_d)_{k+1} - (u_d)_0 + (u_d)_0 \\
& = \sum_{j=0}^k ((u_d)_{j+1} -(u_d)_j ) + (u_d)_0 \\
& =(u_d)_1  \sum_{j=0}^k  \frac{ (\nu_d)_1 Q_{1,0}}{ (\nu_d)_{j+1} Q_{j+1,j} }  - \sum_{j=0}^k \sum_{i=1}^j \frac{(\nu_d)_i}{ (\nu_d)_{j+1} Q_{j+1,j} }
\end{align*} 
To enforce the boundary condition at $x_N = b$ we solve \[
(u_d)_N = (u_d)_1  \sum_{j=0}^{N-1}  \frac{ (\nu_d)_1 Q_{1,0}}{ (\nu_d)_{j+1} Q_{j+1,j} }  - \sum_{j=0}^{N-1} \sum_{i=1}^j \frac{(\nu_d)_i}{ (\nu_d)_{j+1} Q_{j+1,j} } = 0 
\] for $(u_d)_1$ to obtain \[
(u_d)_1= \frac{ \sum_{j=0}^{N-1} \sum_{i=1}^j \dfrac{(\nu_d)_i}{ (\nu_d)_{j+1} Q_{j+1,j} } }{  \sum_{j=0}^{N-1}  \dfrac{ (\nu_d)_1 Q_{1,0}}{ (\nu_d)_{j+1} Q_{j+1,j} }   } 
\]
which gives the desired formula with $Z=(u_d)_1$.
\end{proof}

By construction, the mean holding time of the generator $Q_e$ in \eqref{eq:Qe_1D} at every grid point $x_i$ is equal to the mean first passage time of $Y(t)$ to $(x_{i-1}, x_{i+1})^c$ conditional on $Y(0) = x_i$.   Here we show that this property implies that the Markov process associated to $Q_e$ also preserves the mean first passage time of $Y(t)$ to $(x_k, x_j)^c$ for any $x_k, x_j \in S$ with $x_k < Y(0) = x_i < x_j$.  Abstractly speaking, this result manifests the additive property of mean first passage times.

\begin{lemma}
Let $Q_e$ denote the generator in \eqref{eq:Qe_1D} that uses exact jump rates and transition probabilities.  Let $v(x)$ be the exact solution to: \begin{equation} \label{eq:bvp_kj}
L v(x) = - 1 \quad \text{and} \quad v(x_k) = v(x_j ) = 0 
\end{equation} where $L$ is the generator of the SDE and $x_k, x_j \in S$ with $x_k < x_j$.  Then $v(x)$ satisfies: \[
Q_e v(x_i) = -1 \;,
\] for all grid points $x_i \in (x_k, x_j)$.
\end{lemma}

\begin{proof}
The unique solution to the boundary value problem \eqref{eq:bvp_kj} is: \begin{equation} \label{eq:mfpt_kj}
v(x) = -  \int_{x_k}^x \int_{x_k}^y \frac{\nu(r)}{M(y) \nu(y)} dr dy + K_1 \int_{x_k}^x  \frac{dy}{M(y) \nu(y)}  
\end{equation}
where $K_1$ is the following constant: \[
K_1 = \dfrac{  \int_{x_k}^{x_j} \int_{x_k}^y \dfrac{\nu(r)}{M(y) \nu(y)} dr dy}{ \int_{x_k}^{x_j}  \dfrac{dy}{M(y) \nu(y)}} \;.
\]
Apply $Q_e$ on $v(x)$ to obtain: \begin{align*}
 (Q_e v)_i &= \frac{z_i - z_{i-1}}{z_{i+1} - z_{i-1}} \frac{1}{u_i} v_{i+1} - \frac{v_i}{u_i} + \frac{z_{i+1} - z_{i}}{z_{i+1} - z_{i-1}} \frac{1}{u_i} v_{i-1} \\
&=  \frac{1}{u_i}  \frac{1}{z_{i+1} - z_{i-1}} \left[ (z_{i+1} - z_{i-1}) (v_{i-1} -v_i) + (z_i - z_{i-1}) (v_{i+1} - v_{i-1}) \right] \;. 
\end{align*}
To simplify this expression, recall that the exact mean first passage time to $(x_{i-1}, x_{i+1})^c$ is given by: \begin{equation} \label{eq:mfpt_iip1}
u(x) = -  \int_{x_{i-1}}^x \int_{x_{i-1}}^y \frac{\nu(r)}{M(y) \nu(y)} dr dy + \tilde K_1 \int_{x_{i-1}}^x  \frac{dy}{M(y) \nu(y)}  
\end{equation}
where $\tilde K_1$ is the following constant: \[
\tilde K_1 = \dfrac{  \displaystyle \int_{x_{i-1}}^{x_{i+1}} \int_{x_{i-1}}^y \dfrac{\nu(r)}{M(y) \nu(y)} dr dy}{\displaystyle  \int_{x_{i-1}}^{x_{i+1}}  \dfrac{dy}{M(y) \nu(y)}} \;.
\]
To be sure, $u(x)$ equals $v(x)$ when $x_k = x_{i-1}$ and $x_j = x_{i+1}$.  From \eqref{eq:mfpt_kj} and \eqref{eq:mfpt_iip1} it follows that:  \begin{align*}
  (Q_e v)_i &=  \frac{1}{u_i}  \frac{1}{z_{i+1} - z_{i-1}}  \left[ (z_{i+1} - z_{i-1}) \int_{x_{i-1}}^{x_i} \int_{x_k}^y \frac{\nu(r)}{M(y) \nu(y)} dr dy \right. \\
& \quad \quad \quad\quad\quad\quad\quad\quad\quad - \left. (z_i - z_{i-1})   \int_{x_{i-1}}^{x_{i+1}} \int_{x_k}^y \frac{\nu(r)}{M(y) \nu(y)} dr dy \right]  \\
 &= \frac{1}{u_i} \left[ - u_i + u_i +  \int_{x_{i-1}}^{x_i} \int_{x_k}^y \frac{\nu(r)}{M(y) \nu(y)} dr dy \right. \\
& \quad \quad\quad \left. - \frac{z_i - z_{i-1}}{z_{i+1} - z_{i-1}}   \int_{x_{i-1}}^{x_{i+1}} \int_{x_k}^y \frac{\nu(r)}{M(y) \nu(y)} dr dy \right]  \\
 &= \frac{1}{u_i} \left[ - u_i + u_i +  \int_{x_{i-1}}^{x_i} \int_{x_{i-1}}^y \frac{\nu(r)}{M(y) \nu(y)} dr dy \right. \\
& \quad \quad\quad \left. - \left( \frac{ \displaystyle \int_{x_{i-1}}^{x_{i+1}} \int_{x_{i-1}}^y \frac{\nu(r)}{M(y) \nu(y)} dr dy}{\displaystyle \int_{x_{i-1}}^{x_{i+1}}  \dfrac{dy}{M(y) \nu(y)} }   \right) \;  \int_{x_{i-1}}^{x_i}  \frac{dy}{M(y) \nu(y)}  \right]  \\
&= -1
\end{align*}
as required.   Thus, the truncation error of $Q_e$ with respect to the exact MFPT $v(x)$ is zero.
\end{proof}

\chapter{Conclusion} \label{chap:conclusion}

This paper presented methods to simulate SDEs that permit fixing the spatial step size of the approximation.  This feature of the approximation stabilized the method for both finite and long-time simulations.  To construct solvers with this capability, we considered spatial discretizations of the infinitesimal generator of the SDE, which constitutes a shift in perspective from the standard approach based on time-discretization of realizations of the SDE.  In order to obtain scalable approximations, we required that these spatial discretizations satisfy a realizability condition.  Intuitively speaking, this condition is a non-negativity requirement on the weights used in the finite difference approximation of partial derivatives, and implied that the spatial discretizations generate a Markov jump process.  Realizations of this process can be produced without time discretization error using the SSA algorithm.  Roughly speaking, the SSA produces a simulation of a continuous-time random walk with a user-defined jump size, jump directions given by the columns of the (local) noise coefficient matrix, and a mean holding time that is determined by the generator.   Since we used a Monte-Carlo method to simulate the approximation, our approach applies to problems that practically speaking are beyond the reach of standard numerical PDE methods.  Moreover, our approach did not require discretizing any part of the domain of the SDE problem.  Thus, we were able to construct realizable discretizations for SDE problems with general domains using simple upwinded and central difference methods that are not necessarily tailored to a grid, though other types of discretizations (like finite volume methods) are also possible.  These realizable discretizations were tested on a variety of SDE problems with additive and multiplicative noise.

In the numerical tests, we found that the mean holding time of the approximation in any state adjusts according to the stiffness of the SDE coefficients.  In particular, the mean holding time was inversely proportional to the magnitude of drift.  We also found that once a bound on the spatial step size is set, the mean holding time at any state is determined by the discretization of the infinitesimal generator.  As we illustrated in the Cox-Ingersoll-Ross and Lotka-Volterra SDE problems, this spatial step size can be chosen adaptively according to some measure of the local Lipschitz constant of the drift field of the SDE.  These examples also showed that the method does not produce moves outside the domain of definition of the SDE.  This property is a consequence of the spatial updates being defined to lie in the domain of the SDE.  Numerical tests on non-self-adjoint, planar diffusions illustrated that the integrator is stable and accurate for SDEs with locally Lipschitz continuous drift fields, unknown stationary distributions, and irregular coefficients (internal discontinuities or singularities).  We also considered a Brownian dynamics simulation of a cluster of 13 particles immersed in an unbounded solvent.  The solvent induced interactions between particles, which were long-range and modeled using multiplicative noise.  This example demonstrated that the spatial step size can be set according to a characteristic length scale in the SDE problem -- namely, the Lennard-Jones radius -- in order to avoid exploding trajectories.

The first part of the theory analyzed the properties of a second-order accurate, realizable discretization on a gridded state space with fixed spatial step size.  The main issues in this analysis were that the state space may not be finite-dimensional, and the linear operator associated to the approximation may not be bounded with respect to bounded functions.  To deal with the unbounded nature of the state space and the generator, we used probabilistic tools to show that the approximation is stable.   In particular, we showed that under a weak dissipativity condition on the drift field, a sharp stochastic Lyapunov function was inherited by the approximation.   A key result in this analysis was that the infinitesimal drift condition we derived is uniform in the spatial step size of the approximation.   We then used Harris theorem to prove that the approximation is geometrically ergodic with respect to a unique invariant probability measure.  The stochastic Lyapunov function of the approximation was also used to solve a martingale problem associated to the approximation.  Using a global semimartingale representation of realizations, we analyzed some strong properties of the approximation, including the complexity of the approximation as a function of the spatial step size. To estimate the global error of the approximation, we used a continuous-time analog of the Talay-Tubaro expansion.  This analog was derived using tools from the theory of semigroups produced by linear operators on Banach spaces.  Specifically, we derived an expression for the global error using a variation of constants formula.  An analysis of this formula revealed that the accuracy of the approximation at finite and infinite-time is determined by the accuracy of the generator of the approximation.  

Variable step sizes were considered in the second part of the theory.  We analyzed a second-order accurate, realizable discretization that uses randomly rescaled spatial steps.  The main issue in this analysis turned out to be lack of irreducibility in the standard sense, and also lack of a strong Feller property.  These issues were not surprising since the approximation is a pure jump process with a countable state space, even though it may not be confined to a grid.  Therefore the transition probabilities of the approximation were not absolutely continuous with respect to Lebesgue measure, and the noise driving the approximation did not have a regularizing effect.  That said, we were still able to derive a sharp stochastic Lyapunov function, prove finite-time weak accuracy, prove existence of an invariant probability measure, and prove accuracy with respect to equilibrium expectations.   The latter results required that the semigroup of the approximation be Feller.  To obtain this property, we mollified the generator of the approximation in such a way that its jump rates were bounded and the approximation still had a sharp stochastic Lyapunov function.  

Finally, we analyzed the one-dimensional case where the space-discrete generators admit an infinite, tridiagonal matrix representation.  By solving tridiagonal systems of equations, we verified the $n$-dimensional theory on gridded state spaces, and obtained explicit expressions for the stationary density, mean first passage time, and exit probability of the approximation.  These formulas were used in the numerical tests to validate the approximations.

\backmatter
\bibliographystyle{amsplain}
\bibliography{nawaf}



\end{document}